\documentclass[11pt]{amsart} 
\usepackage{amsmath} \usepackage{amsxtra}
\usepackage{amscd} \usepackage{amsthm}  
\usepackage{amsfonts}\usepackage{amssymb} 
\usepackage{eucal}\usepackage{bm}
\usepackage{graphicx} 
\makeindex

\usepackage{hyperref}

\newtheorem{Cor}[subsubsection]{Corollary}
\newtheorem{Lm}[subsubsection]{Lemma}
\newtheorem{Pp}[subsubsection]{Proposition}
\newtheorem{Con}[subsubsection]{Conjecture}
\newtheorem{Thm}[subsubsection]{Theorem}
\newtheorem{Def}[subsubsection]{Definition}
\newtheorem{Rem}[subsubsection]{Remark}

\theoremstyle{definition}

\theoremstyle{remark}

\newcommand{\nc}{\newcommand}
\nc{\renc}{\renewcommand}
\nc{\ssec}{\subsection}
\nc{\sssec}{\subsubsection}
\nc{\on}{\operatorname}

\nc\ol{\overline}
\nc\wt{\widetilde}
\nc\tboxtimes{\wt{\boxtimes}}

\newcommand{\cA}{{\mathcal A}}
\newcommand{\cB}{{\mathcal B}}
\newcommand{\cC}{{\mathcal C}}
\newcommand{\cD}{{\mathcal D}}
\newcommand{\cH}{{\mathcal H}}
\newcommand{\cE}{{\mathcal E}}
\newcommand{\cG}{{\mathcal G}}
\newcommand{\cI}{{\mathcal I}}
\newcommand{\cJ}{{\mathcal J}}
\newcommand{\cO}{{\mathcal O}}
\newcommand{\cL}{{\mathcal L}}
\newcommand{\cM}{{\mathcal M}}

\newcommand{\cF}{{\mathcal F}}
\newcommand{\cK}{{\mathcal K}}

\newcommand{\cR}{{\mathcal R}}
\newcommand{\cS}{{\mathcal S}}
\newcommand{\cT}{{\mathcal T}}
\newcommand{\cU}{{\mathcal U}}
\newcommand{\cV}{{\mathcal V}}

\newcommand{\cX}{{\mathcal X}}
\newcommand{\cY}{{\mathcal Y}}
\newcommand{\cZ}{{\mathcal Z}}

\newcommand{\GG}{{\mathbb G}}
\newcommand{\NN}{{\mathbb N}}
\newcommand{\ZZ}{{\mathbb Z}}

\newcommand{\PP}{{\mathbb P}}

\newcommand{\EE}{{\mathbb E}}

\nc{\gi}{\mathfrak{i}}

\newcommand{\gq}{\mathfrak{q}}
\newcommand{\gu}{\mathfrak{u}}

\newcommand{\gt}{\mathfrak{t}}
\newcommand{\gr}{\mathfrak{r}}
\newcommand{\gs}{\mathfrak{s}}

\newcommand{\gU}{\mathfrak{U}}

\nc{\gM}{\mathfrak{M}}
\nc{\gV}{\mathfrak{V}}
\nc{\gE}{\mathfrak{E}}
\nc{\gL}{\mathfrak{L}}
\nc{\gG}{\mathfrak{G}}

\nc{\bC}{{\mathbf C}}
\nc{\uZ}{\underline{\cZ}}

\newcommand{\Rep}{{\on{Rep}}}
\newcommand{\Sch}{{\on{Sch}}}

\newcommand{\Gm}{\mathbb{G}_m}

\newcommand{\toup}[1]{\stackrel{#1}{\to}}

\newcommand{\hook}[1]{\stackrel{#1}{\hookrightarrow}}

\newcommand{\getsup}[1]{\stackrel{#1}{\gets}}

\newcommand{\IC}{\on{IC}}

\newcommand{\Hom}{\on{Hom}}
\newcommand{\Mod}{\on{Mod}}

\newcommand{\Sym}{\on{Sym}}

\newcommand{\Ker}{\on{Ker}}

\newcommand{\triv}{\on{triv}}

\newcommand{\Bun}{\on{Bun}}

\newcommand{\Bunb}{\on{\overline{Bun}} }
\newcommand{\Bunt}{\on{\widetilde\Bun}}

\newcommand{\Spec}{\on{Spec}}

\newcommand{\supp}{\on{supp}}

\newcommand{\HOM}{{{\mathcal H}om}}

\newcommand{\Gr}{\on{Gr}}

\newcommand{\Fun}{{\on{Fun}}}

\newcommand{\pr}{\on{pr}}
\newcommand{\id}{\on{id}}

\newcommand{\QED}{$\square$} 
  
\newcommand{\iso}{{\widetilde\to}}

\newcommand{\comp}{\circ}

\renewcommand{\H}{{\on{H}}}   

\newcommand{\DD}{\mathbb{D}}  
\newcommand{\D}{\on{D}}       
\newcommand{\ov}[1]{\overline{#1}}
\newcommand{\select}[1]{{\it{#1}}}

\newcommand{\und}[1]{\underline{#1}}

\renewcommand{\div}{\on{div}}

\newcommand{\<}{\langle}
\renewcommand{\>}{\rangle}

\newcommand{\Av}{\on{Av}}
\newcommand{\ev}{\mathit{ev}}

\newcommand{\Conv}{\on{Conv}}
\newcommand{\Loc}{\on{Loc}}
\newcommand{\Lie}{\on{Lie}}
\newcommand{\Sph}{\on{Sph}}
\newcommand{\Res}{\on{Res}}

\newcommand{\ttimes}{\tilde\times}

\newcommand{\act}{\on{act}}
\newcommand{\dimrel}{\on{dim.rel}}

\newcommand{\codim}{\on{codim}}

\newcommand{\SL}{\on{SL}}

\newcommand{\Vect}{\on{Vect}}

\newcommand{\Ind}{\on{Ind}}

\newcommand{\ra}{\rightarrow}
\newcommand{\la}{\leftarrow}
\newcommand{\diag}{\on{diag}}

\nc{\Perv}{\on{Perv}}

\nc{\Gra}{\on{Gra}}
\nc{\PPerv}{\on{{\PP}erv}}

\nc{\oX}{\overset{\scriptscriptstyle\circ}{X}}
\nc{\ocL}{\overset{\scriptscriptstyle\circ}{\cL}}
\nc{\gRes}{\on{gRes}}
\nc{\Sign}{\on{Sign}}
\nc{\goodat}{\rm{good\, at}}
\nc{\Whit}{\on{Whit}}
\nc{\add}{\on{add}}
\nc{\FS}{\on{FS}}
\nc{\oo}[1]{\overset{\scriptscriptstyle\circ}{#1}}
\nc{\can}{\on{can}}
\nc{\summ}{\on{sum}}
\nc{\SiSu}{\on{SS}}
\nc{\Irr}{\on{Irr}}
\nc{\Hecke}{\on{Hecke}}
\nc{\oHecke}{\overset{\scriptstyle\bullet}{\Hecke}}
\nc{\og}[1]{\overset{\scriptscriptstyle\bullet}{#1}}
\nc{\of}{\overset{\scriptstyle\bullet}{f}}
\nc{\Exp}{\on{{\mathcal E}xp}}
\nc{\Chain}{\on{Chain}}
\nc{\Map}{\on{Map}}
\nc{\cSet}{\on{{\mathcal S}et}}
\nc{\Cat}{\on{\mathcal{C}at}}
\nc{\bfitDelta}{\bm{\mathit{\Delta}}}
\nc{\Grpd}{\on{Grpd}}
\nc{\Kan}{\on{{\mathcal K}an}}
\nc{\Spc}{\on{Spc}}
\nc{\Yon}{\on{Yon}}
\nc{\colim}{\on{colim}}
\nc{\Fin}{\on{{\mathcal F}in}}
\nc{\Alg}{\on{Alg}}
\nc{\Triv}{\on{Triv}}
\nc{\Grp}{\on{{\mathcal G}rp}}
\nc{\Mon}{\on{Mon}}
\nc{\Disc}{\on{Disc}}
\nc{\PreStk}{\on{PreStk}}
\nc{\Stk}{\on{Stk}}
\nc{\DGCat}{\on{DGCat}}
\nc{\Ab}{\on{Ab}}
\nc{\ComGrp}{\on{{\mathcal C}omGrp}}
\nc{\Ptd}{\on{Ptd}}
\nc{\Shv}{\on{Shv}}
\nc{\DG}{\on{DG}}
\nc{\cLS}{\on{\mathcal L\mathcal S}}
\nc{\Ge}{\on{Ge}}
\nc{\Ran}{\on{Ran}}
\nc{\Surj}{\on{Surj}}
\nc{\FactGe}{\on{FactGe}}
\nc{\Quad}{\on{Quad}}
\nc{\detrel}{\on{detrel}}
\nc{\Fact}{\on{Fact}}
\nc{\Sptr}{\on{Sptr}}
\nc{\cPr}{\on{{\mathcal P}r}}
\nc{\Tw}{\on{Tw}}
\nc{\QCoh}{\on{QCoh}}
\nc{\uMap}{\und{\Map}}
\nc{\cTw}{{\mathcal T}w}
\nc{\FactPic}{\on{FactPic}}
\nc{\oblv}{\on{oblv}}
\nc{\com}{\on{com}}
\nc{\Glue}{\on{Glue}}
\nc{\Sat}{\on{Sat}}
\nc{\coind}{\on{coind}}
\nc{\counit}{\on{coun}}
\nc{\SI}{\on{SI}}
\nc{\Tot}{\on{Tot}}
\nc{\Conf}{\on{Conf}}
\nc{\coInd}{\on{coInd}}
\nc{\bvartriangle}{\boldsymbol{\vartriangle}}
\nc{\Inv}{\on{Inv}}
\nc{\ind}{\on{ind}}

\makeatletter
\newcommand*{\doublerightarrow}[2]{\mathrel{
  \settowidth{\@tempdima}{$\scriptstyle#1$}
  \settowidth{\@tempdimb}{$\scriptstyle#2$}
  \ifdim\@tempdimb>\@tempdima \@tempdima=\@tempdimb\fi
  \mathop{\vcenter{
    \offinterlineskip\ialign{\hbox to\dimexpr\@tempdima+1em{##}\cr
    \rightarrowfill\cr\noalign{\kern.5ex}
    \rightarrowfill\cr}}}\limits^{\!#1}_{\!#2}}}
\newcommand*{\triplerightarrow}[1]{\mathrel{
  \settowidth{\@tempdima}{$\scriptstyle#1$}
  \mathop{\vcenter{
    \offinterlineskip\ialign{\hbox to\dimexpr\@tempdima+1em{##}\cr
    \rightarrowfill\cr\noalign{\kern.5ex}
    \rightarrowfill\cr\noalign{\kern.5ex}
    \rightarrowfill\cr}}}\limits^{\!#1}}}
\makeatother

\usepackage{lipsum}

\begin{document}
\author{G. Dhillon}
\address{UCLA, Mathematical Sciences Building, 520 Portola Plaza, Los Angeles, CA 90095}
\email{gsd@math.ucla.edu}
\author{S. Lysenko}
\address{Institut Elie Cartan Lorraine, Universit\'e de Lorraine, 
B.P. 239, F-54506 Vandoeuvre-l\`es-Nancy Cedex, France}
\email{Sergey.Lysenko@univ-lorraine.fr}
\thanks{We are grateful to Misha Finkelberg for regular fruitful discussions and to Sam Raskin for answering the second author's questions.}
\title{Semi-infinite parabolic $\IC$-sheaf II: the Ran space version}
\date{12 November 2026}
\begin{abstract}
This paper is a sequel to \cite{DL}. Let $G$ be a split reductive group. We define the parabolic semi-infinite category on the Ran version of the affine Grassmanian $\Gr_{G,\Ran}$. We study the semi-infinite parabolic $\IC$-sheaf in this category denoted $\IC^{\frac{\infty}{2}}_{P,\Ran}$. We provide several descriptions of this object, one as certain colimit, another as an intermediate extension in certain category. We relate the global and local semi-infinite categories of sheaves. We also relate the intersection cohomology sheaf $\IC_{\Bunt_P}$ of the Drinfeld compactification $\Bunt_P$ of $\Bun_P$ with $\IC^{\frac{\infty}{2}}_{P,\Ran}$.
\end{abstract}
\maketitle
\tableofcontents

\section{Introduction}

\sssec{} This paper is a second one (a sequel to \cite{DL}) in a series whose aim is to construct the analogs of residues of Eisenstein series in the geometric Langlands program. We refer the reader to the introduction of  \select{loc. cit.} for some motivational and contextual discussion of the present series. 

\sssec{} Work over an algebraically closed field $k$. Let $G$ be a split connected reductive group, $P\subset G$ a parabolic subgroup with unipotent radical $U(P)$ and the standard Levi subgroup $M$. For $\cO=k[[t]]\subset F=k((t))$ in \select{loc.cit.} we considered the group ind-scheme $H=M(\cO)U(P)(F)$ acting on the affine Grassmanian $\Gr_G$. Let $S^0_P\subset \Gr_G$ be the $H$-orbit of $G(\cO)$. In \select{loc.cit.} we introduced the object $\IC^{\frac{\infty}{2}}_P$ of the semi-infinite category $Shv(\Gr_G)^H$ playing the role of the intersection cohomology sheaf of the closure $\bar S^0_P$ of $S^0_P$.  

 Fix a smooth projective connected curve $X$, let $\Ran$ denote the Ran space of $X$ (cf. Section~\ref{Sect_1.1.1_now}). In this paper we introduce and study the Ran versions of $\IC^{\frac{\infty}{2}}_P$, they live on the Ran version $\Gr_{G,\Ran}$ of the affine Grassmanian of $G$. We should acknowledge right away that many our techniques are inspired by the case of $P=B$ studied in Gaitsgory's paper \cite{Gai19SI}. However, the results that we get are formally different even in the case $P=B$. In addition, we introduce important new ingredients, which were absent in Gaitsgory's paper, they will be used for further papers in this series (cf. Section~\ref{Sect_intro_Graded Satake functors} for details). 
 
\sssec{} We introduce the Ran versions $\bar S^0_{P,\Ran}\subset \Gr_{G,\Ran}$ (resp., $H_{\Ran}$) of $\bar S^0_P$ (resp., of $H$). We define two versions of the semi-infinite category of sheaves on $\bar S^0_{P,\Ran}$ denoted $\SI^{\le 0}_{P,\Ran}(S)$ and $\cS\cI^{\le 0}_{P,\Ran}(S)$, the former corresponds to $H_{\Ran}$-equivariance, and the latter to $\gL(U(P))_{\Ran}$-equivariance. To get reasonable t-structures, we further impose some unitality conditions and get full subcategories 
$$
\SI^{\le 0}_{P,\Ran, untl}(S)\;\;\;\mbox{and}\;\;\; \cS\cI^{\le 0}_{P,\Ran, untl}(S)
$$
of the above categories. We equip these two categories with t-structures. One has the oblivion functor $\oblv: \SI^{\le 0}_{P,\Ran, untl}(S)\to\cS\cI^{\le 0}_{P,\Ran, untl}(S)$, which is right t-exact (we do not know if it is left t-exact)\footnote{The t-structure on $\SI^{\le 0}_{P,\Ran, untl}(S)$ might turn out sick. In particular, we do not know it is it compatible with filtered colimits, while this property is clear for $\cS\cI^{\le 0}_{P,\Ran, untl}(S)$}. 

\sssec{} We define $\IC^{\frac{\infty}{2}}_{P,\Ran}\in \cS\cI^{\le 0}_{P,\Ran, untl}(S)
$ as the intermediate extension of the dualizing object from the open stratum. 

\sssec{} 
\label{Sect_1.0.5_now}
In (\cite{Ly10}, Section~6) we developed some version of the Drinfeld-Pl\"ucker formalism for commutative factorization categories (more general than the formalism of \cite{Gai19SI}). In Section~\ref{Sect_Drinfeld-Plucker formalism: the Ran space version} we apply it to `spread' the Drinfeld-Plucker formalism from its form established in \cite{DL} to the Ran space. This is just a specialization of the ideas from \cite{Ly10}.

 We apply the resulting Drinfeld-Plucker formalism in its $\Bunt_P$-version to construct an object $'\IC^{\frac{\infty}{2}}_{P,\Ran}\in \SI^{\le 0}_{P,\Ran, untl}(S)$, which is another version of the semi-infinite parabolic $\IC$-sheaf. Our Corollary~\ref{Cor_3.3.7} shows that $\oblv('\IC^{\frac{\infty}{2}}_{P,\Ran})\,\iso\, \IC^{\frac{\infty}{2}}_{P,\Ran}$. 

\sssec{} Let $\Lambda_{G,P}$ denote the fundamental group of $M$, and $\Lambda_{G,P}^{pos}$ the $\ZZ_+$-span of positive $G$-coroots in $\Lambda_{G,P}$. Set $\Lambda_{G,P}^{pos, *}=\Lambda_{G,P}^{pos}-\{0\}$.

The ind-scheme $\bar S^0_{P,\Ran}$ is stratified by locally closed ind-subschemes $S^{\theta}_{P,\Ran}$ for $\theta\in -\Lambda_{G,P}^{pos}$ (cf. Section~\ref{Sect_1.3.27_now}). We calculate the $*$-restrictions of $'\IC^{\frac{\infty}{2}}_{P,\Ran}$ to these strata. The answer is given in terms of the Satake functors for $M$, and even more precisely, in terms of \select{graded Satake functors}, which we introduce. This is an new ingredient reviewed in the next subsection.

\sssec{Graded Satake functors} 
\label{Sect_intro_Graded Satake functors}
In \cite{Ly10} we associated to a non-unital commutative algebra $C$ in $\DGCat_{cont}$ a commutative factorization sheaf of categories $\Fact(C)$ on $\Ran$ and studied many related constructions and properties. 

 We consider the examples of $C$
\begin{equation}
\label{examples_C_introduction}
\Rep(\check{M})_{neg}, \; \Rep(\check{M})_{neg}^{untl}, \; \Rep(\check{M})_{<0}, \; \Rep(\check{M})_{<0}^-
\end{equation} 
defined as follows. Write $\check{M}$ for the Langlands dual of $M$, let $Z(\check{M})$ denote its center. For $\theta\in \Lambda_{G,P}$ let $\Rep(\check{M})_{\theta}\subset \Rep(\check{M})$ be the full subcategory of those representations of $\check{M}$, on which $Z(\check{M})$ acts by $\theta$. Set
$$
\Rep(\check{M})_{neg}=\underset{\theta\in -\Lambda_{G,P}^{pos}}{\oplus} \Rep(\check{M})_{\theta},\;\;\;\;\; \Rep(\check{M})_{<0}=\underset{\theta\in -\Lambda_{G,P}^{pos, *}}{\oplus} \Rep(\check{M})_{\theta}
$$
Let 
$$
\Rep(\check{M})_{neg}^{untl}=\Vect\oplus \Rep(\check{M})_{<0}
$$
Set $\Lambda^-_{M, G}=\Lambda^+_M\cap (-\Lambda^{pos})$, here $\Lambda$ is the set of coweights of $G$, $\Lambda^{pos}$ is the $\ZZ_+$-span of positive coroots in $\Lambda$, and $\Lambda^+_M$ is the set of $M$-dominant coweights. Then 
$$
\Rep(\check{M})_{<0}^-=\underset{0\ne \lambda\in \Lambda^-_{M, G}}{\oplus} \Vect\otimes U^{\lambda},
$$
where $U^{\lambda}$ denotes the irreducible $\check{M}$-module with highest weight $\lambda$. We have the full embeddings
$$
\Rep(\check{M})_{<0}^-\subset \Rep(\check{M})_{<0}\subset \Rep(\check{M})_{neg}^{untl}\subset \Rep(\check{M})_{neg}.
$$
The categories (\ref{examples_C_introduction}) are graded (in a way compatible with the symmetric monoidal structures) by $-\Lambda_{G,P}^{pos}$ in the unital (resp., by
$-\Lambda_{G,P}^{pos,*}$ in the non-unital) case. In this situation $\Fact(C)$ inherits a grading by the same semi-group (cf. \cite{Ly10}, Appendix~E).

 Moreover, in the unital (resp., non-unital) case each graded component $\Fact(C)_{\theta}$ is a naturally a module category over $Shv((X^{-\theta}\times\Ran)^{\subset})$ (resp., over $Shv(X^{-\theta})$). Here for $\eta\in\Lambda_{G,P}^{pos}$ we denote by $X^{\eta}$ the moduli scheme of $\Lambda_{G,P}^{pos, *}$-valued divisors on $X$ of degree $\eta$. The prestack $(X^{\eta}\times \Ran)^{\subset}$ is defined in Section~\ref{Sect_1.3.13_now}.
 
 The usual Satake functor for $M$ is the functor 
$$
\Sat_{M,\Ran}: \Fact(\Rep(\check{M}))\to \Sph_{M,\Ran}
$$ 
taking values in spherical sheaves on $\Gr_{M,\Ran}$ (cf. \cite{Ly10}, Section~7). We define versions of $\Sat_{M,\Ran}$ for $\Rep(\check{M})$ replaced by 
 every category in (\ref{examples_C_introduction}). In the unital case the resulting functors 
 denoted $\Sat_{M,\Ran, +}$, $\Sat_{M, \Ran}^{untl}$ take values in some full subcategories of $\Sph_{M,\Ran}$, which are also graded by $-\Lambda_{G,P}^{pos}$. The $\theta$-component of $\Sat_{M,\Ran, +}$, $\Sat_{M, \Ran}^{untl}$ is $Shv((X^{-\theta}\times\Ran)^{\subset})$-linear. We study these functors systematically in Appendix~\ref{Sect_append_Graded Satake functors}.  

 Let $\Conf^*=\underset{\eta\in \Lambda_{G,P}^{pos,*}}{\sqcup} X^{\eta}$. We have the version of the group scheme of arcs $\gL^+(M)_{\Conf^*}$ over $\Conf^*$, some versions of the affine Grasslmanian $\Gr_{M, \Conf^*, +}$ over $\Conf^*$, and the corresponding category of spherical sheaves $\Sph_{M,\Conf^*, +}$ defined in Section~\ref{Sect_3.2.7_now}. In the non-unital case we get the functor
$$
\Sat_{M,\Conf^*}: \Fact(\Rep(\check{M})_{<0})\to \Sph_{M,\Conf^*, +}
$$
graded by $-\Lambda_{G,P}^{pos,*}$. For $\theta\in -\Lambda_{G,P}^{pos, *}$ the $\theta$-component of the latter functor is $Shv(X^{-\theta})$-linear.

 For the full subcategory $\Rep(\check{M})_{<0}^-\subset  \Rep(\check{M})_{<0}$ the corresponding graded Satake functor is the restriction of $\Sat_{M,\Conf^*}$. It takes values in a full subcategory of $\Sph_{M,\Conf^*, +}$ essentially given by the property that the corresponding spherical sheaf is the extension by zero from the negative part of the affine grassmanian $\Gr_{M,\Conf^*, +}$ (with respect to $G$).  
 
 Our key technical result about graded Satake functors is Proposition~\ref{Pp_3.2.5_Satake}. It provides a relation between the image of a commutative factorization algebra $\Fact(A)$ (attached to a free commutative algebra $A$ in $\Rep(\check{M})_{neg}^{untl}$) under $\Sat_{M,\Ran}^{untl}$ and the image of $\Fact(A_{<0})$ under $\Sat_{M, \Conf^*}$. Here $A_{<0}$ is obtained from $A$ by removing the unit. This result is inspired by a result of Raskin (\cite{Ras2}, Proposition-Construction~4.10.1) in the case of a torus. It is our Proposition~\ref{Pp_3.2.5_Satake} that allows to prove the unitality of our object $'\IC^{\frac{\infty}{2}}_{P,\Ran}$. 
 
\sssec{Main results} The $*$-restrictions of $'\IC^{\frac{\infty}{2}}_{P,\Ran}$ to the strata $S^{\theta}_P$, $\theta\in -\Lambda_{G,P}^{pos}$ is calculated in Corollary~\ref{Cor_3.3.5}. The answer is given in terms of the Satake functor $Sat_{M, \Conf^*}$ applied to the factorization algebra attached to the non-unital commutative algebra $\cO(U(\check{P}))_{<0}$. Here $U(\check{P})$ is the unipotent radical of the dual parabolic $\check{P}$ of the Langlands dual group $\check{G}$, and $\cO(U(\check{P}))$ is the space of regular functions on $U(\check{P})$. We denoted by $\cO(U(\check{P}))_{<0}$ the algebra obtained from $\cO(U(\check{P}))$ by removing the unit, so 
$$
\cO(U(\check{P}))_{<0}\in \Rep(\check{M})_{<0}^-.
$$
This allows to prove Corollary~\ref{Cor_3.3.7} already mentioned in Section~\ref{Sect_1.0.5_now}. 

  Consider the Drinfeld compactification $\Bunt_P$ of $\Bun_P$, we introduce also its version allowing some poles $_{\Ran, \infty}\Bunt_P$ (cf. Section~\ref{Sect_1.3.5_now}).  
We define global semi-infinite categories $\Inv(_{\Ran, \infty}\Bunt_P)$, $\Inv(\Bunt_P)$. Our main results relate the local object $'\IC^{\frac{\infty}{2}}_{P,\Ran}$ with the global one, the $\IC$-sheaf $\IC_{\Bunt_P}$ of $\Bunt_P$. For this we introduce a diagram
$$
\gL^+(M)_{\Ran}\backslash \Gr_{G,\Ran}\getsup{\pi_{loc,\Ran}}
\cY_{\Ran}\toup{\pi_{glob,\Ran}} {_{\Ran,\infty}\Bunt_P}
$$
in Section~\ref{Sect_The stack cY_Ran}. For us $\cY_{\Ran}$ is a tool allowing to formulate the desired relation.

 The stack $\cY_{\Ran}$ classifies: $M$-torsor $\cF_M$ on $X$, $G$-torsor $\cF_G$ on $X$, a point $\cI\in\Ran$ and an isomorphism $\cF_M\times_M G\,\iso\, \cF_G\mid_{X-\Gamma_{\cI}}$ of $G$-torsors. Here $\Gamma_{\cI}$ is the union of graphs of the points of $X$ that constitute $\cI$.  
 
 We define a closed prestack $\cG_{\Ran, M}\times^{\gL^+(M)_{\Ran}} \bar S^0_{P,\Ran}\hook{} \cY_{\Ran}$ and the corresponding "almost local" version of the semi-infinite category $\Inv(\cG_{\Ran, M}\times^{\gL^+(M)_{\Ran}} \bar S^0_{P,\Ran})_{untl}$. The map $\pi_{glob,\Ran}$ restricts to a morphism 
$$
\bar\pi^0_{glob}:
 \cG_{\Ran, M}\times^{\gL^+(M)_{\Ran}} \bar S^0_{P,\Ran}\to \Bunt_P.
$$ 
Our Theorem~\ref{Thm_loc_glob_equiv} claims that one has a canonical equivalence
$$
(\bar\pi^0_{glob})^!: \Inv(\Bunt_P)\,\iso\,\Inv(\cG_{\Ran, M}\times^{\gL^+(M)_{\Ran}} \bar S^0_{P,\Ran})_{untl}.
$$

 Our Theorem~\ref{Th_4.2.12_global_compatibility} establishes a canonical isomorphism
$$
(\bar \pi^0_{loc})^!('\IC^{\frac{\infty}{2}}_{P,\Ran})\,\iso\, (\bar\pi^0_{glob})^!\IC_{\Bunt_P}[\dim\Bun_P]
$$ 
in $\Inv(\cG_{\Ran, M}\times^{\gL^+(M)_{\Ran}}\bar S^0_{P,\Ran})_{untl}$.  

  The stack $\Bunt_P$ is stratified by locally closed substacks $_{\theta}\Bunt_P$ for $\theta\in -\Lambda_{G,P}^{pos}$ (cf. Section~\ref{Sect_1.3.19_now}). An important application of the above results is our Theorem~\ref{Th_5.5.2}. It provides a canonical description of the $*$-restriction of $\IC_{\Bunt_P}$ to the stratum $_{\theta}\Bunt_P$ in terms of the graded Satake functors for $M$. 
  
  This result has a long history of previous partial results: after some further stratification the corresponding $*$-restriction has been calculated \select{non-canonically} in (\cite{BFGM}, Theorem~1.12). In the case of $B=P$ this is the result of Gaitsgory (\cite{Gai19Ran}, Theorem~3.9.3). The canonicity of the isomorphism of Theorem~\ref{Th_5.5.2} will be essential for our construction of the residues of geometric Eisenstein series in subsequent papers in this series.

\sssec{Other results} We "almost calculate" also the !-restriction of $'\IC^{\frac{\infty}{2}}_{P,\Ran}$ to a stratum $S^{\theta}_{P,\Ran}$. More precisely, we formulate Conjecture~\ref{Con_3.3.8}, which is "almost proved" as in explained in Remark~\ref{Rem_4.3.10_now}. 

 Another new result, which is absent in Gaitsgory's paper \cite{Gai19Ran}, is Proposition~\ref{Pp_3.3.6}. It is based on some unexpected ULA property established in 
Proposition~\ref{Pp_ULA_and_perverse_appendix} for some morphism of configuration spaces. This is a part of Appendix~\ref{Sect_The ULA property and perversity}, where we prove more general results relating the ULA property and perversity.  

 In Appendix~\ref{Sect_Invariants_under category objects} we propose a general point of view on the category of invariants under the action of a "category object" in prestacks. In Section~\ref{Sect_Placid group schemes over Ran} we explain how to define the perverse t-structure on the category of invariants $Shv(Y)^G$, where $Y\to\Ran$ is a relative ind-scheme of ind-finite type, and $G\to \Ran$ is a placid prosmooth group scheme over $\Ran$ acting on $Y$. 
 
 In Appendix~\ref{Sect_append_Graded Satake functors} we systematically study the graded Satake functors for $M$. In particular, we define the ranification functor for $M$ (an analog of a similar functor considered by Raskin for a torus in \cite{Ras2}), and discuss the related factorization algebras/coalgebras and the chiral symmetric monoidal structures.  
 
\sssec{}  We made special efforts to obtain the results of this paper for both $\cD$-modules and the constructible context. This required some preliminary work: the paper \cite{Ly10} by the second author has been written, in particular, as a companion to the present one 
to provide some results about the factorizable Satake functors and commutative factorization categories uniformy for $\cD$-modules and the constructible context. The results of (\cite{Ly10}, Section~7 and Appendix~E) are essentially used in the present paper. 
 
\begin{Rem} A different approach to  $\IC^{\frac{\infty}{2}}_{P,\Ran}$ and its applications has appeared in the paper of Faergeman-Hayash \cite{FH}, which was released essentially at the same time as the present paper. The approach of \cite{FH} uses Hecke structures to define $\IC^{\frac{\infty}{2}}_{P,\Ran}$ rather than employing the colimit construction directly or the t-structure on the semi-infinite category. Yet another definition of $\IC^{\frac{\infty}{2}}_{P,\Ran}$ via Zastava spaced has previously appeared in \cite{Ha}, it differs also from our approach. Besides, the results of \cite{FH}, \cite{Ha} are obtained in the context of $\cD$-modules only. 
\end{Rem}

\ssec{Conventions}

\sssec{} 
\label{Sect_0.1.1_now}
Our conventions are those of \cite{DL}, except that our sheaf theory is either in the constructible context or the theory of $\cD$-modules. 

 We work over an algebraically closed field $k$. Write $\Sch^{aff}$ for the category of classical affine schemes, $\Sch_{ft}$ for the category of classical schemes of finite type. Write $\PreStk=\Fun((\Sch^{aff})^{op}, \Spc)$ for the category of prestacks, $\PreStk_{lft}$ for its full subcategory of prestacks locally of finite type. Our stacks are considered for the etale topology. \index{$k, \Sch^{aff}, \Sch_{ft}, \PreStk, \PreStk_{lft}, X, \Omega_X, \Omega^{\otimes\frac{1}{2}}$, Section~\ref{Sect_0.1.1_now}} 
 
 For a scheme $S$ and a group scheme $\cG$ over $S$, a trivial $\cG$-torsor on $S$ is denoted $\cF^0_{\cG}$. For an affine algebraic group $L$ write $\Bun_L$ for the moduli stack of $L$-torsors on $X$. 
 
 Let $X$ be a smooth projective connected curve. We assume $char(k)\ne 2$ to use the results of (\cite{GLys}, Section~5), namely we need to use the corrected Jacquet functors and their compatibility with the factorization. For this reason we also fix a square root $\Omega^{\otimes\frac{1}{2}}$ of the canonical line bundle $\Omega_X$, which was used in (\cite{GLys}, 5.2.2) for the construction of the corrected Jacquet functors. 
 
  The field of coefficients of our sheaf theories $e$ is algebraically closed of characteristic zero. For a scheme $S$ over $e$ we denote by $\cO(S)\in\Vect^{\heartsuit}$ the space of regular functions on $S$.

\sssec{} 
\label{Sect_0.1.2_now}
As in \cite{DL}, we let $G$ be a connected reductive group, $T\subset B\subset G$ a maximal torus and Borel subgroups, $B^-$ an opposite Borel with $B\cap B^-=T$. We write $U$ (resp., $U^-$) for the unipotent radical of $B$ (resp., of $B^-$).  
\index{$G, T, B, B^-, U, U^-, P, P^-, M$, Section~\ref{Sect_0.1.2_now}} 
 
 Let $P\subset G$ be a standard parabolic, $P^-$ an opposite parabolic with common Levi subgroup $M=P\cap P^-$. Write $w_0$ for the longest element of the Weyl group $W$, and similarly for $w_0^M\in W_M$, where $W_M$ is the Weyl group of $(M, T)$. Write $U(P)$ (resp., $U(P^-)$) for the unipotent radical of $P$ (resp., of $P^-$). Set $B_M=B\cap M$, $B_M^-=B^-\cap M$. 
\index{$W, W_M, w_0, w_0^M, U(P), U(P^-)$, Section~\ref{Sect_0.1.2_now}} 

 Let $\cI_G$ be the set of vertices of the Dynkin diagram. For $i\in\cI_G$ we write $\alpha_i$ (resp., $\check{\alpha}_i$) for the corresponding simple coroot (resp., simple root). Let $\cI_M\subset \cI_G$ correspond to the Dynkin diagram of $M$. Write $\check{P}, \check{B}, \check{T}$, $U(\check{P}), U(\check{P}^-)$ for the corresponding dual objects.
\index{$B_M, B_M^-, \cI_G, \alpha_i, \check{\alpha}_i, \cI_M$, Section~\ref{Sect_0.1.2_now}} 
\index{$\check{P}, \check{B}, \check{T}, U(\check{P}), U(\check{P}^-), $, Section~\ref{Sect_0.1.2_now}} 
\index{$\Lambda_{G,P}^{pos}, \Lambda_{G,P}^{pos, *}, \cO, F$, Section~\ref{Sect_0.1.2_now}} 
 
 Write $\Lambda$ (resp., $\check{\Lambda}$) for the coweights (resp., weights) lattice of $T$, $\Lambda^+$ (resp., $\check{\Lambda}^+$) for the dominant coweights (resp. weights) of $(G,T)$. Write $\Lambda^+_M$ for the dominant coweights of $B_M$. Let $\Lambda^{pos}\subset\Lambda$ be the $\ZZ_+$-span of positive coroots. Let $\Lambda_M^{pos}\subset\Lambda$ be the $\ZZ_+$-span of $\alpha_i$, $i\in \cI_M$. 
 
  Let $\Lambda_{G,P}$ be the lattice of cocharacters of $M/[M,M]$, $\check{\Lambda}_{G,P}$ be the dual lattice. Write $\Lambda_{G,P}^{pos}$ for the $\ZZ_+$-span of $\alpha_i$ for $i\in \cI_G-\cI_M$ in $\Lambda_{G,P}$. Set $\Lambda_{G,P}^{pos, *}=\Lambda_{G,P}^{pos}-\{0\}$. Set $\cO=k[[t]]\subset F=k((t))$. 
\index{$\Lambda, \check{\Lambda}, \Lambda^+, \check{\Lambda}^+, \Lambda^+_M, \Lambda^{pos}, \Lambda_M^{pos}, \Lambda_{G,P}, \check{\Lambda}_{G,P}$, Section~\ref{Sect_0.1.2_now}}   

\sssec{} 
\label{Sect_0.1.3_now}
For $\check{\lambda}\in \check{\Lambda}^+$ write $\cV^{\check{\lambda}}$ for the Weyl $G$-module with a fixed highest weight vector $v^{\check{\lambda}}$ of weight $\check{\lambda}$. For $\check{\lambda}\in \check{\Lambda}^+$ write $\tilde\cV^{\check{\lambda}}$ for the corresponding dual Weyl $G$-module, it is equipped with a distinguished linear functional $u^{\check{\lambda}}: \tilde\cV^{\check{\lambda}}\to k$.  
\index{$\cV^{\check{\lambda}}, v^{\check{\lambda}}, \tilde\cV^{\check{\lambda}}, u^{\check{\lambda}}$, Section~\ref{Sect_0.1.3_now}}

 For the convenience of the reader, recall their definitions. Let 
$\Rep(B)^{\heartsuit, fd}$ (resp., $\Rep(G)^{\heartsuit, fd}$) be the abelian category of finite-dimensional representations of $B$ (resp., $G$). Let 
$$
\Ind_B^G: \Rep(B)^{\heartsuit, fd}\to \Rep(G)^{\heartsuit, fd}\;\;\;\;\mbox{and}\;\;\;\; \coInd_B^G: \Rep(B)^{\heartsuit, fd}\to \Rep(G)^{\heartsuit, fd}
$$ 
be the left and right adjoints to the restriction. For $\check{\lambda}\in \check{\Lambda}^+$ we let $\cV^{\check{\lambda}}=\Ind_B^G(k^{\check{\lambda}})$, where $k^{\check{\lambda}}$ is the 1-dimensional space $k$, on which $T$ acts by $\check{\lambda}$ and $U$ acts trivially. For $\check{\lambda}_i\in\check{\Lambda}^+$ we have a canonical map $\cV^{\check{\lambda}_1+\check{\lambda}_2}\to \cV^{\check{\lambda}_1}\otimes \cV^{\check{\lambda}_2}$ sending $v^{\check{\lambda}_1+\check{\lambda}_2}$ to $v^{\check{\lambda}_1}\otimes v^{\check{\lambda}_2}$. 
\index{$\Rep(B)^{\heartsuit, fd}, \Rep(G)^{\heartsuit, fd}, \Ind_B^G, \coInd_B^G$, Section~\ref{Sect_0.1.3_now}}

 For $\check{\lambda}\in \check{\Lambda}^+$ set $\tilde\cV^{\check{\lambda}}=\coInd_{B^-}^G(e^{\check{\lambda}})$. For $\check{\lambda}_i\in\check{\Lambda}^+$ we have a canonical map $\tilde\cV^{\check{\lambda}_1}\otimes \tilde\cV^{\check{\lambda}_2}\to \tilde \cV^{\check{\lambda}_1+\check{\lambda}_2}$ whose composition with $u^{\check{\lambda}_1+\check{\lambda}_2}$ is $u^{\check{\lambda}_1}\otimes u^{\check{\lambda}_2}$. 
 
 Recall that $\cV^{\check{\lambda}}$ has a unique irreducible quotient $V^{\check{\lambda}}$ of highest weight $\check{\lambda}$. We have a canonical inclusion $V^{\check{\lambda}}\hook{} \tilde\cV^{\check{\lambda}}$, hence a canonical map $\cV^{\check{\lambda}}\to \tilde\cV^{\check{\lambda}}$. 
 
\sssec{} Starting from Section~\ref{Sect_Relation between local and global: geometry} we will assume $[G,G]$ simply-connected.\footnote{This is done to simplify the definitions of Drinfeld compactifications $\Bunb_P, \Bunt_P$ and use the results of \cite{BG}. This also simplifies the definitions of the strata of $\Gr_{G,\Ran}$.}
 
\sssec{} 
\label{Sect_0.1.5_now}
Our conventions and notations related to higher categories are those of \cite{DL, Ly10}. In particular, for $A\in Alg(\DGCat_{cont})$ we write $A-mod$ for $A-mod(\DGCat_{cont})$. The category of $\infty$-categories is denoted by $1-\Cat$. 
\index{$A-mod, 1-\Cat$, Section~\ref{Sect_0.1.5_now}}

\sssec{} 
\label{Sect_0.1.6_now}
We denote by $Alg$ (resp., $Alg^{nu}$) the operad of unital (resp., non-unital) associative algebras. The corresponding operads for commutative algebras are denoted $CAlg, CAlg^{nu}$. 
\index{$Alg, Alg^{nu}, CAlg, CAlg^{nu}$, Section~\ref{Sect_0.1.6_now}}

\section{Parabolic semi-infinite category of sheaves over $\Ran$}

In this section we introduce our main geometric objects, define several versions of the semi-infinite parabolic category of sheaves in its $\Ran$ version and study its properties related, in particular, to natural stratifications of the underlying prestacks. 

\ssec{Basic geometric objects}
\sssec{} 
\label{Sect_1.1.1_now}
For $x\in X$ write $\cO_x$ for the completed local ring of $X$ at $x$, $F_x$ for its fraction field. Write $fSets$ for the category of finite non-empty sets and surjections. Recall the following geometric objects from \cite{GLys2, Gai19Ran}. 

 The Ran space of $X$ is the prestack $\Ran\in\PreStk_{lft}$ sending $S\in\Sch^{aff}$ to the set of finite non-empty subsets $\cI\subset \Map(S, X)$. One equivalently has $\Ran\,\iso\,\underset{I\in fSets^{op}}{\colim} X^I$. 
 
  For $S\in\Sch^{aff}$, an $S$-point $\cI\in\Map(S, \Ran)$, and $i\in \cI$ we denote by $\Gamma_i\subset S\times X$ the graph of the corresponding map. Let $\Gamma_{\cI}$ denote the sum of $\Gamma_i$ for $i\in\cI$ as a relative effective Cartier divisor on $S\times X$ over $S$. For $\cI\subset \Map(S, X)$ write $\hat\cD_{\cI}$ for the formal completion of $S\times X$ along $\Gamma_{\cI}$ viewed as a formal scheme. Let $\cD_{\cI}$ denote the affine scheme obtained from $\hat\cD_{\cI}$, the image of $\hat\cD_{\cI}$ under $\colim: \Ind(\Sch^{aff})\to \Sch^{aff}$, see (\cite{GLys}, 7.1.2). For $n\ge 0$ we will also use $n\Gamma_{\cI}$ as an effective relative Cartier divisor on $S\times X$ over $S$.    
 
 Note that $\Gamma_{\cI}\subset \cD_{\cI}$ is a closed subscheme, set $\oo{\cD}_{\cI}=\cD_{\cI}-\Gamma_{\cI}$. Note that $\oo{\cD}_{\cI}\in\Sch^{aff}$. 
\index{$fSets, \Ran, \Gamma_i, \Gamma_{\cI}, \hat\cD_{\cI}, \cD_{\cI}, n\Gamma_{\cI}, \oo{\cD}_{\cI}$, Section~\ref{Sect_1.1.1_now}} 
 
\sssec{} 
\label{Sect_1.1.2_now}
For $J\in fSets$ we write $\Ran^J_d\subset \Ran^J$ for the open subfunctor classifying for $S\in\Sch^{aff}$ collections $\cI_j\subset \Map(S, X), j\in J$ such that for any $j_1\ne j_2$ one has 
$$
\Gamma_{\cI_{j_1}}\cap \Gamma_{\cI_{j_2}}=\emptyset
$$ 

 Write $\PreStk_{corr}$ for the category of correspondences in $\PreStk$ (cf. \cite{R}). Recall that $\Ran$ is an object of $CAlg^{nu}(\PreStk_{corr})$, where for $J\to *$ the product is given by the diagram
$$
\Ran^J\gets \Ran^J_d\to \Ran,
$$  
here the right arrow is the union of the corresponding finite subsets. 

For a notion of a factorization prestack over $\Ran$ we refer to (\cite{GLys}, 2.2). A concise definition is as follows. This is a map $Z_{\Ran}\to\Ran$ in $\PreStk$, which is lifted to a morphism in $CAlg^{nu}(\PreStk_{corr})$ and such that for any $J$ the induced morphism 
$$
Z^J_{\Ran}\times_{\Ran^J} \Ran^J_d\to Z_{\Ran}\times_{\Ran}  \Ran^J_d
$$ 
is an isomorphism. 
\index{$\Ran^J_d, \PreStk_{corr}$, Section~\ref{Sect_1.1.2_now}} 
 
\sssec{} 
\label{Sect_1.1.3}
For $Y\in\PreStk$ denote by $\gL^+(Y)_{\Ran}\to \Ran$ the prestack over $\Ran$ whose fibre over $(S, \cI)$ as above is $\Map(\hat\cD_{\cI}, Y)$. If $Y$ is an affine scheme, this fibre coincides with $\Map(\cD_{\cI}, Y)$. For $I\in fSets, n>0$
let $\gL^+(Y)_{n, I}$ be the prestack over $X^I$ whose fibre for $S\in\Sch^{aff}$ over an $S$-point $\cI\in\Map(S, X^I)$ is $\Map(n\Gamma_{\cI}, Y)$. 

 Denote by $\gL(Y)_{\Ran}\to \Ran$ the prestack whose fibre over $(S, \cI)$ is $\Map(\oo{\cD}_{\cI}, Y)$. The prestacks $\gL(Y)_{\Ran}, \gL^+(Y)_{\Ran}$ have natural factorization structure over $\Ran$, see (\cite{GLys2}, 1.2.5). 
 
 For a prestack $Z_{\Ran}\to \Ran$ over $\Ran$ we use the notation $Z_I$ to denote the base change $Z_{\Ran}\times_{\Ran} X^I$ for $I\in fSets$, that is, the subscript $\Ran$ is replaced by the subscript $I$. In particular, for $I\in fSets$ we get
$$
\gL(Y)_I=\gL(Y)_{\Ran}\times_{\Ran} X^I, \; \gL^+(Y)_I=\gL^+(Y)_{\Ran}\times_{\Ran} X^I.
$$ 
\index{$\gL^+(Y)_{\Ran}, \gL^+(Y)_{n, I}, \gL(Y)_{\Ran}, Z_I$, Section~\ref{Sect_1.1.3}}
 
\sssec{} 
\label{Sect_def_Gr_G,Ran}
We write $\Gr_{G, \Ran}$ for the Ran version of the affine Grassmanian of $G$. It is a prestack over $\Ran$ whose fibre over $(S,\cI)$ as above is the groupoid of pairs: $(\cF_G, \beta)$, where $\cF_G$ is a $G$-torsor over $S\times X$, $\beta: \cF_G\,\iso\, \cF^0_G\mid_{S\times X-\Gamma_{\cI}}$ is a trivialization over $S\times X-\Gamma_{\cI}$.

 The Beauville-Laszlo theorem says that the fibre of $\Gr_{G, \Ran}$ over $(S, \cI)$ as above can equivalently be described as the groupoid of pairs: $(\cF_G,\beta)$, where $\cF_G$ is a $G$-torsor on $\cD_{\cI}$, and $\beta: \cF_G\,\iso\, \cF^0_G\mid_{\oo{\cD}_{\cI}}$ is a trivialization. Namely, the restriction from the first datum to the second one defines this equivalence. If $I\in fSets$ then $\Gr_{G, I}$ identifies with the stack quotient $\gL(G)_I/\gL^+(G)_I$, and similarly for $\Gr_{G,\Ran}$.
 
Note that $\Gr_{G,\Ran}$ is naturally a factorization prestack over $\Ran$.
\index{$\Gr_{G, \Ran}$, Section~\ref{Sect_def_Gr_G,Ran}}

\sssec{Unital structures} 
\label{Sect_1.1.5_now}
Let
$$
(\Ran\times\Ran)^{\subset} \subset\Ran\times\Ran
$$
be the subfunctor sending $S\in\Sch^{aff}$ to $(\cI,\cI')\in \Map(S, \Ran)$ such that $\Gamma_{\cI}$ is set-theoretically contained in $\Gamma_{\cI'}$, which means by definition that
\begin{equation}
\label{inclusion_for_Ran_times_Ran_subset} 
(S\times X)-\Gamma_{\cJ'}\subset (S\times X)-\Gamma_{\cJ}.
\end{equation}

  Let $\varphi_s,\varphi_b: (\Ran\times\Ran)^{\subset}\to\Ran$ be the map sending the above point to $\cI$ and $\cI'$ respectively, the letters $s$ and $b$ stand for small and big.
We view $(\Ran\times\Ran)^{\subset}$ as a category object in $\PreStk_{lft}$ acting on $\Ran$ on the right in the sence of Section~\ref{Sect_Invariants_under category objects}. Here $\varphi_b$ is the action map. 
  
  Now given a map $Z\to\Ran$ in $\PreStk$, a \select{unital structure} on $Z$ is a right $(\Ran\times\Ran)^{\subset}$-action on $Z$ such that the map $Z\to \Ran$ is equivariant with respect to the right actions of $(\Ran\times\Ran)^{\subset}$.  
\index{$(\Ran\times\Ran)^{\subset}, \varphi_s,\varphi_b$, Section~\ref{Sect_1.1.5_now}}  
  
\sssec{} Note that $(\Ran\times\Ran)^{\subset}\toup{\varphi_b}\Ran$ is not a factorization prestack over $\Ran$. Indeed, given $S\in\Sch^{aff}$ and $\{\cI'_j\}_{j\in J}\in\Map(S,\Ran^J_d)$, let $\cI'$ be its image under 
$$
\Map(S, \Ran^J_d)\to \Map(S, \Ran),
$$ 
and $(\cI, \cI')\in \Map(S, (\Ran\times\Ran)^{\subset})$. Then for $j\in J$ we naturally have the `piece' $\cI_j$ of $\cI$ which is set-theoretically contained in $\cI'_j$. Though $\cI=\underset{j\in J}{\sqcup} \cI_j$, for some $j$ the set $\cI_j$ could be empty. 
  
  However, $(\Ran\times\Ran)^{\subset}\in CAlg^{nu}(\PreStk_{corr})$ naturally in such a way that 
$$
\varphi_b: (\Ran\times\Ran)^{\subset}\to\Ran
$$ 
is a map in $CAlg^{nu}(\PreStk_{corr})$. The product for $J\to *$ in $(\Ran\times\Ran)^{\subset}$ is given by the diagram
$$
((\Ran\times\Ran)^{\subset})^J\gets ((\Ran\times\Ran)^{\subset})^J\times_{(\phi_b)^J,\Ran^J} \Ran^J_d\to (\Ran\times\Ran)^{\subset},
$$
where the right arrow is given by the union of the corresponding finite subsets of $\Map(S, X)$ for $S\in\Sch^{aff}$. 
  
Let $Z_{\Ran}\to\Ran$ be a factorization prestack. Similarly, $Z_{\Ran}\times_{\Ran, \varphi_s} (\Ran\times\Ran)^{\subset}$ is not a factorization prestack over $\Ran$ via the map
\begin{equation}
\label{map_phi_b_composed_with_pr}
Z_{\Ran}\times_{\Ran, \varphi_s} (\Ran\times\Ran)^{\subset}\toup{\pr} (\Ran\times\Ran)^{\subset}\toup{\varphi_b}\Ran.
\end{equation}
However, 
$$
Z_{\Ran}\times_{\Ran, \varphi_s} (\Ran\times\Ran)^{\subset}\in CAlg^{nu}(\PreStk_{corr})
$$ 
naturally in such a way that (\ref{map_phi_b_composed_with_pr}) is a map in $CAlg^{nu}(\PreStk_{corr})$.

 Assume given a unital structure on $Z_{\Ran}$. By definition, the unital and factorization structures on $Z_{\Ran}$ are \select{compatible} if the action map 
$$
\varphi_b: Z_{\Ran}\times_{\Ran, \varphi_s} (\Ran\times\Ran)^{\subset}\to Z_{\Ran}
$$ 
is a morphism in $CAlg^{nu}(\PreStk_{corr})$. 

\sssec{} We equip $\Gr_{G,\Ran}$ with a unital structure by letting 
$$
\varphi_b: \Gr_{G,\Ran}\times_{\Ran} (\Ran\times\Ran)^{\subset}\to \Gr_{G,\Ran}
$$ 
be the map that, in terms of Section~\ref{Sect_def_Gr_G,Ran} sends for $S\in \Sch^{aff}$ the collection $(\cI, \cI', \cF_G,\beta)$ to $(\cI',\cF_G, \beta')$, where $\beta'$ is the restriction of $\beta$ under (\ref{inclusion_for_Ran_times_Ran_subset}). 

 The unital and factorization structures on $\Gr_{G,\Ran}$ are compatible.

\ssec{Categories of sheaves} 

\sssec{} 
\label{Sect_1.2.1_now}
If $Z\to \Ran$ is a map in $\PreStk_{lft}$, and $Z$ is equipped with a unital structure, we define $Shv(Z)_{untl}\in\DGCat_{cont}$ as the category of $(\Ran\times\Ran)^{\subset}$-equivariant objects in $Shv(Z)$ as in Section~\ref{Sect_A.0.4}. 
In particular, we get $Shv(\Gr_{G,\Ran})_{untl}$. 

 By (\cite{GLys2}, 1.6.4) and (\cite{Gai19Ran}, 4.1.2) the map $\varphi_s: (\Ran\times\Ran)^{\subset}\to \Ran$ is universally homologically contructible in the sense of (\cite{Gai19Ran}, A.1.8). So, by Section~\ref{Sect_A.0.6}, $Shv(Z)_{untl}\subset Shv(Z)$ is the full subcategory of $K\in Shv(Z)$ such that for the diagram
$$
Z\getsup{\varphi_s} Z\times_{\Ran} (\Ran\times\Ran)^{\subset}\toup{\varphi_b} Z
$$
the object $\varphi_b^! K$ lies in the image of the full embedding 
$$
\varphi_s^!: Shv(Z)\to Shv(Z\times_{\Ran} (\Ran\times\Ran)^{\subset}).
$$ 
\index{$Shv(Z)_{untl}, Shv(\Gr_{G,\Ran})_{untl}$, Section~\ref{Sect_1.2.1_now}} 
 
\sssec{} For a notion of a sheaf of categories on $Y\in\PreStk_{lft}$ we refer to (\cite{Ly10}, Appendix~A). In the case of $\cD$-modules it is also given in (\cite{GLys}, 1.6.2). So, for $Y\in\PreStk_{lft}$ one has $ShvCat(Y)\in 1-\Cat$ defined in (\cite{Ly10}, A.1). For a morphism $f: Y\to Y'$ in $\PreStk_{lft}$ the corresponding restriction functor is denoted 
$$
f^!: ShvCat(Y')\to ShvCat(Y)
$$ 

In particular, 
$$
ShvCat(\Ran)\,\iso\, \underset{I\in fSets}{\lim} ShvCat(X^I), 
$$
and $ShvCat(X^I)\,\iso\, Shv(X^I)-mod$. Recall that $\Ran$ is 1-affine for our sheaf theories by (\cite{Ly10}, 2.1.2). 
 
\sssec{} The theory of placid group ind-shemes acting on $\DG$-categories is developed for $\cD$-modules in (\cite{Chen}, Appendix B) and in (\cite{Ly9}, Sections 1.3.2 - 1.3.24) in the constructible context, see also (\cite{DL}, Section~A.4). In particular, the corresponding categories of invariants are defined in \select{loc.cit.}
  
\sssec{} 
\label{Sect_1.2.4_now}
For $I\in fSets$ we have 
$$
\Gr_{G, I}=X^I\times_{\Ran}\Gr_{G, \Ran},
$$
and similarly for other groups. For a map $\phi: I\to J$ in $fSets$ we write $\vartriangle=\vartriangle^{(I/J)}: X^J\to X^I$ for the diagonal closed immersion. 

 Recall that $\gL^+(G)_{\Ran}$ (resp. $\gL(G)_{\Ran}$) is a placid group scheme (resp., placid group ind-scheme) over $\Ran$. Set 
$$
H_{\Ran}=\gL^+(M)_{\Ran}\gL(U(P))_{\Ran}\subset \gL(P)_{\Ran}, \; H_I=X^I\times_{\Ran} H_{\Ran}
$$

 For $I=*$ and $x\in X$ write $\Gr_{G, x}$ for the fibre of $\Gr_{G, *}$ over $x\in X$, and similarly for other groups.
\index{$\Gr_{G, I}, H_{\Ran}, H_I, \Gr_{G, x}$, Section~\ref{Sect_1.2.4_now}} 
  
\sssec{} 
\label{Sect_1.2.5}
For $\lambda\in\Lambda$, a $T$-torsor $\cF_T$ on $X$, and $x\in X$ 
denote by $\cF_T(\lambda x)$ the $T$-torsor on $X$ equipped with isomorphisms for $\check{\mu}\in\check{\Lambda}$
$$
\cL^{\check{\mu}}_{\cF_T(\lambda x)}\,\iso\, \cL^{\check{\mu}}_{\cF_T}(\<\lambda, \check{\mu}\>x).
$$
 
Given $\und{\lambda}: I\to\Lambda$ let $s^{\und{\lambda}}: X^I\to \Gr_{T, I}$ be the map sending $(x_i)\in X^I$ to 
$$
\cI=(x_i)\in X^I, \; \cF_T=\cF^0_T(-\sum_i \lambda_i x_i), \; \beta: \cF_T\,\iso\, \cF^0_T\mid_{X-\Gamma_{\cI}},
$$ 
where $\beta$ is the evident trivialization. 
 
 For $I\in fSets$ write $\oo{X}{}^I\subset X^I$ for the complement to all the diagonals.
Let $I\to J$ be a map in $fSets$, $(x_j)\in \oo{X}{}^J(k)$ and $(x_i)\in X^I(k)$ be its image under $\vartriangle: X^J\to X^I$. Let $\und{\mu}: J\to \Lambda$ is given by $\mu_j=\underset{i\in I_j}{\sum}\lambda_i$. The fibre of $\Gr_{T, I}$ over $(x_i)$ is $\underset{j\in J}{\prod} \Gr_{T, x_j}$, and 
$$
s^{\und{\lambda}}(x_i)=\{t^{\mu_j}_{x_j}T(\cO_{x_j})\}\in \underset{j\in J}{\prod} \Gr_{T, x_j}.
$$
Here $t_{x_j}\in \cO_{x_j}$ is a uniformizer. More precisely, the compositions $X^J\to X^I\toup{s^{\und{\lambda}}}\Gr_{T, I}$ and $X^J\toup{s^{\und{\mu}}}\Gr_{T, J}\hook{} \Gr_{T, I}$ coincide.

Write $\und{0}: I\to \Lambda$ for the constant map with value $0$. 

\index{$\cF_T(\lambda x), s^{\und{\lambda}}, \oo{X}{}^I, \und{0}$, Section~\ref{Sect_1.2.5}}

\sssec{} 
\label{Sect_1.2.6_now}
As in \cite{DL}, set $\Lambda_{M, ab}=\{\lambda\in\Lambda\mid \<\lambda, \check{\alpha}_i\>=0\;\mbox{for}\; i\in\cI_M\}$ and $\Lambda_{M, ab}^+=\Lambda_{M, ab}\cap \Lambda^+$.
\index{$\Lambda_{M, ab}, \Lambda_{M, ab}^+$, Section~\ref{Sect_1.2.6_now}}

\sssec{} 
\label{Sect_1.2.7_now}
For $\und{\lambda}: I\to \Lambda^+_{M, ab}$ let $H_{\und{\lambda}}$ (resp., $U(P)_{\und{\lambda}}$) denote the stabilizer in $\gL(P)_I$ (resp., in $\gL(U(P))_I$)
 of the composition 
$$
X^I\toup{s^{-\und{\lambda}}} \Gr_{T, I}\to \Gr_{P, I}.
$$ 
Note that $H_{\und{\lambda}}\,\iso\, \gL^+(M)_I\rtimes U(P)_{\und{\lambda}}$ is the semi-direct product, where $U(P)_{\und{\lambda}}$ is a normal subgroup.

 If $I\to J$ is a map in $fSets$ then $U(P)_{\und{\lambda}}\times_{X^I} X^J\,\iso\, U(P)_{\und{\mu}}$, where $\und{\mu}: J\to \Lambda^+_{M, ab}$ is given by $\mu_j=\sum_{i\in I_j} \lambda_i$. 
 
\index{$H_{\und{\lambda}}, U(P)_{\und{\lambda}}$, Section~\ref{Sect_1.2.7_now}} 
 
\sssec{} 
\label{Sect_1.2.8_now}
Equip $\Map(I, \Lambda^+_{M, ab})$ with the pointwise sum. Equip $\Map(I, \Lambda^+_{M, ab})$ with the relation $\und{\lambda}_1\le\und{\lambda}_2$ iff there is $\und{\lambda}\in\Lambda^+_{M, ab}$ with $\und{\lambda}_2-\und{\lambda}_1=\und{\lambda}$. With this relation $\Map(I, \Lambda^+_{M, ab})$ becomes a filtered category.

 If $\und{\lambda}_1\le\und{\lambda}_2$ then $H_{\und{\lambda}_1}\subset H_{\und{\lambda}_2}$, and 
$$
H_I\,\iso\, \underset{\und{\lambda}\in (\Map(I, \Lambda^+_{M, ab}), \le)}{\colim} H_{\und{\lambda}},\;\;\;\;\; \gL(U(P))_I\,\iso\, \underset{\und{\lambda}\in (\Map(I, \Lambda^+_{M, ab}), \le)}{\colim} U(P)_{\und{\lambda}}.
$$ 
Since $H_{\und{\lambda}}\to X^I$ is a placid group scheme, this shows that $H_I\to X^I$ is a placid group ind-scheme. So, by definition, $H_{\Ran}$ is a placid group ind-scheme over $\Ran$. 

 For $I\in fSets$ set
$$
\SI_{P, I}=Shv(\Gr_{G, I})^{H_I}\in Shv(X^I)-mod.
$$
If $I\to J$ is a map in $fSets$ then one has canonically 
$$
\SI_{P, I}\otimes_{Shv(X^I)} Shv(X^J)\,\iso\, \SI_{P, J}
$$ 
by (\cite{Ly9}, Section 1.3.24, Claim 2), because each $H_{\und{\lambda}}$ is prosmooth over $X^I$. So, the compatible collection $\{\SI_{P, I}\}_{I\in fSets}$ forms a sheaf of categories on $\Ran$ that we denote $\SI_{P,\Ran}$.
 
 By definition, 
$$
\SI_{P,\Ran}\,\iso\, \underset{I\in fSet}{\lim} \SI_{P, I}
$$ 
calculated in $Shv(\Ran)-mod$ or, equivalently, in $\DGCat_{cont}$, as $\Ran$ is 1-affine for our sheaf theories (by \cite{Ly10}, 2.1.2). We write informally
$$
\SI_{P,\Ran}=Shv(\Gr_{G,\Ran})^{H_{\Ran}}\in Shv(\Ran)-mod.
$$
\index{$\SI_{P, I}, \SI_{P,\Ran}$, Section~\ref{Sect_1.2.8_now}}

\sssec{} 
\label{Sect_1.2.9_now}
Let $I\in fSets$, $\und{\lambda}: I\to \Lambda^+_{M, ab}$. As in (\cite{DL}, A.1), the functor 
\begin{equation}
\label{functor_oblv_SI_PI_to_gL^+(M)_I-invariants_on_Gr_G,I}
\oblv: \SI_{P, I}\to Shv(\Gr_{G, I})^{\gL^+(M)_I}
\end{equation}
is fully faithful, and if $\und{\lambda}_i\in\Map(I, \Lambda^+_{M, ab})$ with $\und{\lambda}_1\le\und{\lambda}_2$ then  
$$
Shv(\Gr_{G, I})^{H_{\und{\lambda}_2}}\subset Shv(\Gr_{G, I})^{H_{\und{\lambda}_1}} 
$$
is a full subcategory. By (\cite{LyWhit_loc_glob}, 1.3.8) and (\cite{Ly}, 2.7.7), 
$$
\SI_{P, I}\,\iso\, \underset{\und{\lambda}\in (\Map(I, \Lambda^+_{M, ab}), \le)^{op}}{\lim} Shv(\Gr_{G, I})^{H_{\und{\lambda}}}
$$
coincides with $\underset{\und{\lambda}\in \Map(I, \Lambda^+_{M, ab})}{\cap} Shv(\Gr_{G, I})^{H_{\und{\lambda}}}$ taken inside $Shv(\Gr_{G, I})^{\gL^+(M)_I}$. 

 For $\und{\lambda}\in \Map(I, \Lambda^+_{M, ab})$ the functor $\oblv: Shv(\Gr_{G, I})^{H_{\und{\lambda}}}\to Shv(\Gr_{G, I})^{\gL^+(M)_I}$ has a partially defined left adjoint $\Av^{U(P)_{\und{\lambda}}}_!$, which is everywhere defined in the constructible context by Lemma~\ref{Lm_A.1.3}. The partially defined left adjoint $\Av^{\gL(U(P))_I}_!$
 to (\ref{functor_oblv_SI_PI_to_gL^+(M)_I-invariants_on_Gr_G,I}) is then given as 
$$
 \underset{\und{\lambda}\in \Map(I, \Lambda^+_{M, ab})}{\colim} \Av^{U(P)_{\und{\lambda}}}_!
$$ 
by (\cite{DL}, A.3.3). It is everywhere defined in the constructible context. 

\index{$\Av^{U(P)_{\und{\lambda}}}_{^^21}, \Av^{\gL(U(P))_I}_{^^21}$, Section~\ref{Sect_1.2.9_now}}

\sssec{$\Sph_{M, I}$-action} 
\label{Sect_Sph_M_I-action}
Let $\gL(P)_I$ act on $\Gr_{M, I}$ through its quotient $\gL(M)_I$. Let $I\in fSets$. Write $s^{\und{0}}_M$ for the composition $X^I\toup{s^{\und{0}}} \Gr_{T, I}\to\Gr_{M, I}$. Since $s^{\und{0}}_M$ is $H_I$-invariant, 
$$
(s^{\und{0}}_M)_*: Shv(X^I)\to Shv(\Gr_{M, I})
$$ 
is a map in $Shv(H_I)-mod$. By adjointness, it gives the $Shv(\gL(P)_I)$-linear
functor 
\begin{equation}
\label{functor_coinvariants_for_Sect_1.2.9}
Shv(\gL(P)_I)\otimes_{Shv(H_I)} Shv(X^I)\to Shv(\Gr_{M, I})
\end{equation}
As in (\cite{DL}, Lemma~A.2.2) one checks that (\ref{functor_coinvariants_for_Sect_1.2.9}) is an equivalence.

 Now for any $C\in Shv(\gL(P)_I)-mod$ we get
$$
C^{H_I}=\Fun_{Shv(H_I)}(Shv(X^I), C)\,\iso\, \Fun_{Shv(\gL(P)_I)}(Shv(\gL(P)_I)\otimes_{Shv(H_I)} Shv(X^I), C).
$$
Thus, the monoidal category $\Fun_{Shv(\gL(P)_I)}(Shv(\Gr_{M, I}), Shv(\Gr_{M,I}))$ acts on $C^{H_I}$. We have
$$
\Fun_{Shv(\gL(P)_I)}(Shv(\Gr_{M,I}), Shv(\Gr_{M,I}))\,\iso\, Shv(\Gr_{M, I})^{H_I}\,\iso\, Shv(\Gr_{M, I})^{\gL^+(M)_I},
$$
as monoidal categories, because $\gL(U(P))_I$ is ind-prounipotent. So, 
$$
\Sph_{M, I}:=Shv(\Gr_{M, I})^{\gL^+(M)_I}
$$
acts naturally on $C^{H_I}$ on the left. In particular, $\Sph_{M, I}$ acts on $\SI_{P, I}$. 



\index{$s^{\und{0}}_M, \Sph_{M, I}$, Section~\ref{Sect_Sph_M_I-action}}

\sssec{} 
\label{Sect_1.2.11_now}
As in \cite{Ly10}, we equip $\Sph_{M, I}$ with a perverse t-structure following the convention of (\cite{DL}, A.5). In particular, $\oblv: \Sph_{M, I}\to Shv(\Gr_{M, I})$ is 
t-exact for the perverse t-structures.

\sssec{} 
\label{Sect_1.2.12_now}
We use the following notations from \cite{Ly10}. 
Given $C\in CAlg^{nu}(\DGCat_{cont})$ we consider $C(X)=C\otimes Shv(X)\in CAlg(Shv(X)-mod)$ and the corresponding commutative factorization category $\Fact(C)\in Shv(\Ran)-mod$ defined in (\cite{Ly10}, Section~2). For $I\in fSets$ we write $C_{X^I}\in Shv(X^I)-mod$ for the category of sections of $\Fact(C)$ over $X^I$.

\index{$C(X), \Fact(C), C_{X^I}$, Section~\ref{Sect_1.2.12_now}}

\sssec{} 
\label{Sect_1.2.13_now}
In (\cite{Ly10}, Section~7) we introduced a sheaf of categories $\Sph_{M, \Ran}$ over $\Ran$ as the compatible collection of categories $\Sph_{M, I}$ for $I\in fSets$.
For $I\in fSets$ we defined the Satake functors
$$
\Sat_{M, I}: \Rep(\check{M})_{X^I}\to \Sph_{M, I},
$$
which glue into the functor
$$
\Sat_{M, \Ran}: \Fact(\Rep(\check{M}))\to \Sph_{M, \Ran}.
$$
Then $\Sat_{M, I}$ (resp., $\Sat_{M, \Ran}$) is a map in $Alg(Shv(X^I)-mod)$ (resp., in $Alg(Shv(\Ran)-mod))$. We also consider the corresponding objects for $M$ replaced by $G$.

\index{$\Sph_{M, \Ran}, \Sat_{M, I}, \Sat_{M, \Ran}$, Section~\ref{Sect_1.2.13_now}}

\begin{Rem} 
\label{Rem_Sat_and_factorization}
To make $\Sat_{M, \Ran}$ compatible with factorization, one needs to twist $\Rep(\check{M})$ as in (\cite{Ly10}, 7.3.17). Namely, in this case we let $\epsilon\in Z(\check{M})$ be the image of $-1$ under $2\check{\rho}_M: \Gm\to \check{T}$ and replace
$\Rep(\check{M})$ by $\Rep(\check{M})^{\epsilon}\in CAlg(\DGCat_{cont})$ as in \select{loc.cit.}. The sheaves of categories $\Fact(\Rep(\check{M}))$ and $\Fact(\Rep(\check{M})^{\epsilon})$ on $\Ran$ are isomorphic, however, this isomorphism is not compatible with the factorization structures. The functor $\Sat_{M, \Ran}: \Fact(\Rep(\check{M})^{\epsilon})\to \Sph_{M, \Ran}$ is compatible with factorization structures. 
\index{$\Rep(\check{M})^{\epsilon}$, Remark~\ref{Rem_Sat_and_factorization}}
\end{Rem}

\begin{Rem} 
\label{Rem_1.2.11}
Note that $\Sph_{M, I}$ acts by convolutions on the left on $Shv(\Gr_{G, I})^{\gL^+(M)_I}$. 
Globalizing the argument in the local situation from (\cite{DL}, 3.3.20), one shows that this action preserves the full subcategory $\SI_{P,I}$. It is compatible with the !-restrictions along $\vartriangle^{(I/J)}: X^J\to X^I$ for a map $I\to J$ in $fSets$. So, they glue to a $\Sph_{M,\Ran}$-action on $\SI_{P,\Ran}$. For $K\in \SI_{P, I}, F\in \Sph_{M, I}$ we write $F\ast K\in\SI_{P, I}$ for this action. 
\end{Rem}

\sssec{} 
\label{Sect_1.2.15}
For $I\in fSets$ the monoidal category $\Sph_{G,I}$ acts naturally on the right on $Shv(\Gr_{G,I})^{\gL^+(M)_I}$ and preserves the full subcategory $\SI_{P,I}$. For a map $I\to J$ in $fSets$ these action are compatible with the !-restrictions along $\vartriangle^{(I/J)}: X^J\to X^I$. They glue to a $\Sph_{G,\Ran}$-action on the right on $\SI_{P, \Ran}$. 
For $K\in \SI_{P, I}, F\in \Sph_{G, I}$ we write $K\ast F\in \SI_{P, I}$ for this action.

\index{$\vartriangle^{(I/J)}, K\ast F$, Section~\ref{Sect_1.2.15}}

\sssec{Hecke functors} 
\label{Sect_1.2.14}
Pick $I\in fSets$. For the convenience of the reader, recall the right action of $\Sph_{G,I}$ on $Shv(X^I\times \Bun_G)$. 

 Let $\cH_{G,I}$ be the Hecke stack whose $S$-points for $S\in\Sch^{aff}$ is the groupoid of collections: $G$-torsors $\cF_G, \cF'_G$ on $S\times X$, $\cI\in \Map(S,X^I)$ and an isomorphism $\cF_G\,\iso\,\cF'_G\mid_{S\times X-\Gamma_{\cI}}$. We have the diagram of projections
$$
X^I\times \Bun_G\getsup{h^{\la}_G} \cH_{G,I}\toup{h^{\ra}_G} X^I\times \Bun_G,
$$
where $h^{\la}_G$ (resp., $h^{\ra}_G$) sends the above point to $(\cF_G, \cI)$ (resp., $(\cF'_G,\cI)$). Let $\cG_I\to X^I\times \Bun_G$ be the torsor under $\gL^+(G)_I$ classifying $\cF_G\in\Bun_G, \cI\in X^I$ and a trivialization $\cF_G^0\,\iso\, \cF_G\mid_{\cD_{\cI}}$. We write $\cG_{I, G}$ when we need to express its dependence on $G$. 

We have two identifications $\id^l, \id^r: \cH_{G, I}\,\iso\, \Gr_{G, I}\times^{\gL^+(G)_I} \cG_I$ for which the projections $h^{\la}_G, h^{\ra}_G$ correspond to the projections 
on the second factor.

We have the natural diagram
$$
\Gr_{G, I}\times^{\gL^+(G)_I} \cG_I\toup{q_I} (\gL^+(G)_I\backslash \Gr_{G,I})\times \Bun_G\toup{\diag_I} (\gL^+(G)_I\backslash \Gr_{G,I})\times (X^I\times \Bun_G).
$$
Given $\cS\in Shv(\gL^+(G)_I\backslash \Gr_G), K\in Shv(X^I\times \Bun_G)$ denote by 
$$
(\cS\tboxtimes K)^l, (\cS\tboxtimes K)^r\in Shv(\cH_{G, I})
$$ 
the image of $q_I^!\diag_I^!(\cS\boxtimes K)$ under $\id^l, \id^r$ respectively. The result of the right action of $\cS$ on $K$ is 
$$
\H^{\ra}_G(\cS, K)=(h^{\ra}_G)_*(\cS\tboxtimes K)^l.
$$
Write for brevity $K\ast \cS=\H^{\ra}_G(\cS, K)$. This generalizes the notation from (\cite{BG}, 3.2.4). Restricting the above action under $Shv(X^I)\to \Sph_{G, I}$, one gets the usual action of $(Shv(X^I), \otimes^!)$ on $Shv(X^I\times\Bun_G)$. 

 The compatibility of the $\Sph_{G, I}$-action on $Shv(X^I\times \Bun_G)$ with the monoidal structure on $\Sph_{G, I}$ uses the base change isomorphisms established in \cite{Ly4}. 

 The stack quotient $\gL^+(G)_I\backslash \Gr_{G,I}$ can be seen as the prestack classifying $\cI\in X^I$, $G$-torsors $\cF_G, \cF'_G$ on $\cD_{\cI}$ together with an isomorphism $\cF_G\,\iso\, \cF'_G\mid_{\oo{\cD}_{\cI}}$ of $G$-torsors. It carries an involution interchanging $\cF_G$ and $\cF'_G$. It gives rise to a covariant functor $\Sph_{G,I}\to \Sph_{G,I}, \cS\mapsto \ast\cS$. 
 
\index{$\cH_{G,I}, h^{\la}_G, h^{\ra}_G, \cG_I, \cG_{I, G}, \id^l, \id^r$, Section~\ref{Sect_1.2.14}}

\index{$(\cS\tboxtimes K)^l, (\cS\tboxtimes K)^r, \H^{\ra}_G, {K\ast \cS}, \ast\cS$, Section~\ref{Sect_1.2.14}}  
  
\ssec{Relation between local and global: geometry}
\label{Sect_Relation between local and global: geometry}

\sssec{} 
\label{Sect_1.3.1_now}
From now on for the rest of the paper we assume $[G,G]$ simply-connected. 
For $\theta\in \Lambda_{G,P}^{pos}$ write $X^{\theta}$ for the moduli space of $\Lambda_{G,P}^{pos}$-valued divisor of degree $\theta$.

Let $I\in fSets$. Let $\ov{\Gr}_{P, I}^0\subset \Gr_{G, I}$ be the closed ind-subscheme classifying $(\cI\in X^I, \cF_G\in\Bun_G,\beta: \cF_G\,\iso\, \cF^0_G\mid_{X-\Gamma_{\cJ}})$ such that for any $\check{\lambda}\in \check{\Lambda}^+\cap \check{\Lambda}_{G,P}$ the map
$$
\cL^{\check{\lambda}}_{\cF^0_{M/[M,M]}}\to \cV^{\check{\lambda}}_{\cF_G}
$$
initially defined over $X-\Gamma_{\cI}$ is regular over $X$. Let $\Gr_{P, I}^0\subset \ov{\Gr}_{P, I}^0$ be the open subscheme given by the property that the above maps are vector subbundles on $X$.  

 Note that $\ov{\Gr}_{G, I}^0=\Gr_{G,I}^0\,\iso\, \Gr_{[G,G], I}$ canonically.
 
\index{$X^{\theta}, \ov{\Gr}_{P, I}^0, \Gr_{P, I}^0$, Section~\ref{Sect_1.3.1_now}} 

\sssec{} 
\label{Sect_1.3.2_now}
For $I\in fSets$ write $_{I, \infty}\Bunb_P$ for the stack whose $S$-points for $S\in\Sch^{aff}$ is a collection:  
$\cI\in \Map(S, X^I)$, a $G$-torsor $\cF_G$ on $S\times X$, a $M/[M,M]$-torsor $\cF_{M/[M,M]}$ on $S\times X$, and a collection of maps
$$
\kappa^{\check{\lambda}}: \cL^{\check{\lambda}}_{\cF_{M/[M,M]}}\to \cV^{\check{\lambda}}_{\cF_G}, \;\;\;\check{\lambda}\in \check{\Lambda}^+\cap \check{\Lambda}_{G,P}
$$
over $S\times X-\Gamma_{\cI}$ satisfying the Pl\"ucker relations, and such that for any geometric point $s\in S$ the restriction of $\kappa^{\check{\lambda}}$ to $s\times X$ is injective.

 Let $\Bunb_P$ be the stack defined in (\cite{BG}, 1.3.2). So, $X^I\times \Bunb_P\subset {_{I, \infty}\Bunb_P}$ is a closed substack.

 We have the map $\pi_I: \Gr_{G, I}\to {_{I, \infty}\Bunb_P}$ sending $(\cI\in X^I, \cF_G,\beta: \cF_G\,\iso\, \cF^0_G\mid_{X-\Gamma_{\cI}})$ to $(\cI, \cF^0_{M/[M,M]},\cF_G,\kappa)$, where $\kappa$ is induced by the $P$-structure on the trivial $G$-torsor. 
 
 Note that $\pi_I^{-1}(X^I\times \Bunb_P)=\ov{\Gr}_{P, I}^0$.   
 
\index{${_{I, \infty}\Bunb_P}, \Bunb_P, \pi_I$, Section~\ref{Sect_1.3.2_now}} 
 
\sssec{} As in Section~\ref{Sect_1.2.14}, one defines a right action of $\Sph_{G, I}$ on $Shv(_{I, \infty}\Bunb_P)$. The functor $\pi_I^!: Shv(_{I, \infty}\Bunb_P)\to Shv(\Gr_{G, I})$ commutes with the $\Sph_{G, I}$-actions on the right. We have a partially defined left adjoint $(\pi_I)_!$ of $\pi_I^!$, which is everywhere defined in the constructible context. 

 In the constructible context, $(\pi_I)_!$ commutes with the $\Sph_{G, I}$-actions.
 
\sssec{} 
\label{Sect_1.3.4_now}
Write $\bar S^0_{P, I}\subset \Gr_{G, I}$ for the closed ind-subscheme 
classifying 
\begin{equation}
\label{point_of_Gr_G,Ran}
(\cI\in X^I, \cF_G\in\Bun_G,\beta: \cF_G\,\iso\, \cF^0_G\mid_{X-\Gamma_{\cI}})\in \Gr_{G, I}
\end{equation}
such that for any $V\in\Rep(\check{G})^{\heartsuit}$ finite-dimensional, 
$$
V^{U(P)}_{\cF^0_M}\to V_{\cF_G}  
$$
initially defined over $X-\Gamma_{\cI}$ is regular over $X$. Let $v^0_{S, I}: S^0_{P, I}\subset \bar S^0_{P, I}$ be the open subscheme given by the property that the above maps have no zeros on $X$ (are morphisms of vector bundles).  

 Note that if $P=G$ then $S^0_{G, I}=\bar S^0_{G, I}\,\iso\, X^I$ is the unit section of $\Gr_{G,I}\to X^I$. 
 
\index{$\bar S^0_{P, I}, S^0_{P, I}, v^0_{S, I}$, Section~\ref{Sect_1.3.4_now}} 
 
\sssec{} 
\label{Sect_1.3.5_now}
Let $\Bunt_P$ be defined as in (\cite{BG}, 4.1.1).\footnote{If $P=G$ then the canoncal maps $\Bun_P\to \Bunt_P\to \Bunb_P$ are isomorphisms.} Let $_{I, \infty}\Bunt_P$ be the stack whose $S$-point for $S\in\Sch^{aff}$ is a collection $(\cI\in\Map(S, X^I), \cF_G, \cF_M, \kappa)$, where $\cF_G$ (resp., $\cF_M$) is a $G$-torsor (resp., a $M$-torsor) on $S\times X$, and $\kappa$ is a collection of maps
$$
\kappa^V: V^{U(P)}_{\cF_M}\to V_{\cF_G}
$$
over $S\times X-\Gamma_{\cI}$ for $V\in\Rep(G)^{\heartsuit}$ finite-dimensional, which satisfy the Pl\"ucker relations as in \select{loc.cit.} and such that for any geometric point $s\in S$ the restriction of $\kappa^V$ to $s\times X$ is an injection. 

 We have the closed embedding $X^I\times\Bunt_P\hook{} {_{I, \infty}\Bunt_P}$ given by requiring that all $\kappa^V$ are regular over $S\times X$.
 
 Let $\tilde\pi_I: \Gr_{G,I}\to {_{I, \infty}\Bunt_P}$ be the map sending (\ref{point_of_Gr_G,Ran}) to $(\cI, \cF^0_M, \cF_G, \kappa)$, where $\kappa$ comes from the canonical $P$-structure on $\cF^0_G$. One has 
$$
\tilde\pi_I^{-1}(X^I\times\Bunt_P)=\bar S^0_{P, I},\;\;\;\; \tilde\pi_I^{-1}(X^I\times\Bun_P)= S^0_{P, I}.
$$ 
  
  Let $\gr_I: {_{I,\infty}\Bunt_P}\to {_{I,\infty}\Bunb_P}$ be the map sending $(\cI, \cF_M, \cF_G,\kappa)$ to 
$$
(\cI, \cF_{M/[M,M]}, \cF_G,\kappa),
$$ 
where $\cF_{M/[M,M]}$ is the extension of scalars of $\cF_M$. We have $\gr_I\comp\tilde\pi_I=\pi_I$.

\index{$\Bunt_P, {_{I, \infty}\Bunt_P}, \tilde\pi_I, \gr_I$, Section~\ref{Sect_1.3.5_now}}
  
\sssec{} 
\label{Sect_1.3.6_Hecke_action}
Let us give some details for the definition of the right action of $\Sph_{G, I}$ on $Shv(_{I, \infty}\Bunt_P)$. Set 
$$
Z_{P, G, I}={_{I,\infty}\Bunt_P}\times_{X^I\times \Bun_G} \cH_{G,I},
$$ 
where we used $h^{\la}_G$ to form the fibred product. We get the diagram
$$
_{I, \infty}\Bunt_P\getsup{'h^{\la}_G} Z_{P, G, I}\toup{'h^{\ra}_G}{_{I, \infty}\Bunt_P},
$$ 
where $'h^{\la}_G$ (resp., $'h^{\ra}_G$) sends a collection 
$$
(\cI, \cF_M, \cF_G,\kappa, \cF_G\,\iso\,\cF'_G\mid_{X-\Gamma_{\cI}})
$$ 
to $(\cI, \cF_M, \cF_G, \kappa)$ (resp., $(\cI, \cF_M, \cF'_G, \kappa')$, where $\kappa'$ is induced by $\kappa$.

 Let $\cG_{I, P, G}={_{I,\infty}\Bunt_P}\times_{X^I\times \Bun_G} \cG_I$. We get two identifications 
$$
\id^l_{G}, \id^r_{G}: Z_{P, G, I}\,\iso\, \Gr_{G, I}\times^{\gL^+(G)_I} \cG_{I,P,G}
$$ 
for which the projections $'h^{\la}_G, {'h^{\ra}_G}$ correspond to the projection 
on the second factor. 

 We get the diagram
\begin{multline*}
\Gr_{G, I}\times^{\gL^+(G)_I} \cG_{I,P,G}\toup{q_{I,G}} (\gL^+(G)_I\backslash \Gr_{G,I})\times_{X^I} (_{I,\infty}\Bunt_P)\toup{\diag_{I,G}} \\ (\gL^+(G)_I\backslash \Gr_{G,I})\times (_{I,\infty}\Bunt_P)
\end{multline*}
Given $\cS\in \Sph_{G,I}, K\in Shv(_{I,\infty}\Bunt_P)$ we denote by $(\cS\tboxtimes K)^l, (\cS\tboxtimes K)^r\in Shv(Z_{P,G,I})$ the image of $q_{I,G}^*[\dimrel(q_{I,G})]
\diag_{I,G}^!(\cS\boxtimes K)$ under $\id^l_G, \id^r_G$. The right action of $\cS$ on $K$ is
defined as
$$
\H^{\ra}_{P,G}(\cS, K)=('h^{\ra}_G)_*(\cS\tboxtimes K)^l
$$
Write for brevity $K\ast \cS=\H^{\ra}_{P,G}(\cS, K)$, the meaning should be clear from the context.

\index{$Z_{P, G, I}, {'h^{\la}_G}, {'h^{\ra}_G}, \cG_{I, P, G}, q_{I,G}, \diag_{I,G}, \H^{\ra}_{P,G}, {K\ast \cS}$, Section~\ref{Sect_1.3.6_Hecke_action}}

\sssec{} 
\label{Sect_1.3.7_Hecke_action}
Similarly to (\cite{DL}, 3.2.11), we define a left action of $\Sph_{M, I}$ on $Shv(_{I,\infty}\Bunt_P)$ as follows. Set
$$
Z_{P, M, I}=\cH_{M, I}\times_{X^I\times\Bun_M} (_{I,\infty}\Bunt_P),
$$
where we used $h^{\ra}_M$ to form the fibred product. We get a diagram
$$
 _{I,\infty}\Bunt_P\getsup{'h^{\la}_M} Z_{P, M, I}\toup{'h^{\ra}_M} {_{I,\infty}\Bunt_P},
$$
where $'h^{\la}_M$ (resp., $h^{\ra}_M$) sends 
$$
(\cI, \cF'_M, \cF'_G, \kappa', \cF_M\,\iso\, \cF'_M\mid_{X-\Gamma_{\cI}})
$$ 
to $(\cI, \cF_M, \cF'_G, \kappa)$ (resp., $(\cI, \cF'_M, \cF'_G, \kappa')$. Here $\kappa$ is obtained from $\kappa'$ naturally. Set
$$
\cG_{I, P, M}={_{I, \infty}\Bunt_P}\times_{X^I\times \Bun_M} \cG_{I, M}.
$$
As above, we get two identifications 
$$
\id^l_M, \id^r_M: Z_{P, M, I}\,\iso\, \Gr_{M, I}\times^{\gL^+(M)_I} \cG_{I,P, M}
$$
for which the projections $'h^{\la}_M, {'h^{\ra}_M}$ correspond to the projection on the second factor. 

 We get the diagram
\begin{multline*}
\Gr_{M, I}\times^{\gL^+(M)_I} \cG_{I,P, M}\toup{q_{I, M}} (\gL^+(M)_I\backslash \Gr_{M,I})\times_{X^I}(_{I,\infty}\Bunt_P)\toup{\diag_{I, M}} \\ (\gL^+(M)_I\backslash \Gr_{M,I})\times (_{I,\infty}\Bunt_P).
\end{multline*}
Given $\cS\in \Sph_{M, I}, K\in Shv(_{I,\infty}\Bunt_P)$ we denote by $(\cS\tboxtimes K)^l, (\cS\tboxtimes K)^r\in Shv(Z_{P,M,I})$ the image of $q_{I,M}^*[\dimrel(q_{I,M})]\diag_{I,M}^!(\cS\boxtimes K)$ under $\id^l_M, \id^r_M$. The result of the left action of $\cS$ on $K$ is
$$
\H^{\la}_{P,M}(\cS, K)=('h^{\la}_M)_*(\ast\cS\tboxtimes K)^r.
$$
Write for brevity $\cS\ast K=\H^{\la}_{P,M}(\cS, K)$, the meaning should be clear from the context.

\index{$Z_{P, M, I}, {'h^{\la}_M}, {'h^{\ra}_M}, \cG_{I, P, M}, \id^l_M, \id^r_M, q_{I, M}$, Section~\ref{Sect_1.3.7_Hecke_action}}

\index{$\diag_{I, M}, \H^{\la}_{P,M}, \cS\ast K$, Section~\ref{Sect_1.3.7_Hecke_action}}

\sssec{} The functors 
$$
\gr_I^!: Shv(_{I,\infty}\Bunb_P)\to Shv(_{I,\infty}\Bunt_P)\;\;\;\mbox{and}\;\;\; \tilde\pi_I^!: Shv(_{I,\infty}\Bunt_P)\to Shv(\Gr_{G,I})
$$ 
naturally commute with these $\Sph_{G, I}$-actions. Their partially defined left adjoints $(\gr_I)_!, (\tilde\pi_I)_!$ are everywhere defined in the constructible context and commute with the $\Sph_{G, I}$-actions.     
  
\sssec{} 
\label{Sect_1.3.7}
Replacing $X^I$ by $\Ran$, we get the $\Ran$ versions of all the above objects. In particular, we have $_{\Ran,\infty}\Bunt_P, {_{\Ran,\infty}\Bunb_P}$.
The factorization (resp., unital) structure on $\Gr_{G,\Ran}$ yields one on the
prestacks $\Gr_{P,\Ran}^0, \ov{\Gr}_{P,\Ran}^0, S^0_{P,\Ran}, \bar S^0_{P,\Ran}$. The corresponding unital and factorization structures are compatible.

 Write $\eta_{\Ran}: \bar S^0_{P,\Ran}\hook{} \ov{\Gr}^0_{P,\Ran}$ for the corresponding closed immersion.
 
\index{$_{\Ran,\infty}\Bunt_P, {_{\Ran,\infty}\Bunb_P}, \eta_{\Ran}, \pi_{\Ran}$, Section~\ref{Sect_1.3.7}} 

\sssec{} 
\label{Sect_1.3.10}
For $\lambda\in\Lambda$ the connected component $\Bun_T^{\lambda}$ of $\Bun_T$ is defined as in (\cite{DL}, 3.2.5). Namely, for $\check{\lambda}\in\check{\Lambda}$, $\deg\cL^{\check{\lambda}}_{\cF_T}=-\<\lambda, \check{\lambda}\>$. For $\theta\in\Lambda_{G,P}$ the component $\Bun_M^{\theta}$ is defined as the preimage of $\Bun_{M/[M,M]}^{\theta}$ under $\Bun_M\to \Bun_{M/[M,M]}$. The component $\Bun_P^{\theta}, \Bunt_P^{\theta}, \Bunb_P^{\theta}$ denotes the preimage of the corresponding component $\Bun_M^{\theta}$ of $\Bun_M$. 

 For $\theta\in\Lambda_{G,P}$ we use the notations $\Gr_M^{\theta}, \Gr_M^+,
 \Gr_M^{+,\theta}, \Lambda^+_{M, G}$ introduced in (\cite{DL}, Section~3.2). In particular, $\Gr_M^{\theta}$ is the connected component of $\Gr_M$ containing $t^{\lambda}M(\cO)$ for $\lambda\in\Lambda$ over $\theta$, and 
$$
\Lambda^+_{M, G}=\Lambda^+_M\cap w_0^M(\Lambda^{pos}).
$$

View $\Gr_M$ as the ind-scheme classifying a $M$-torsor $\cF_M$ on $\Spec\cO$ together with a trivialization $\beta_M: \cF_M\,\iso\, \cF^0_M\mid_{\Spec F}$. For $\theta\in -\Lambda^{pos}_{G, P}$ write $\Gr^{-, \theta}_M$ for the scheme defined in (\cite{DL}, 4.2.11). Namely, it is the subscheme of those $(\cF_M, \beta_M: \cF_M\,\iso\, \cF^0_M\mid_{\Spec F})\in \Gr_M^{\theta}$ for which for any $V\in \Rep(G)^{\heartsuit}$ finite-dimensional the natural map
$$
V^{U(P)}_{\cF^0_M}\to V^{U(P)}_{\cF_M}
$$
is regular over $\Spec\cO$.  

 By (\cite{BG}, 6.2.3), for $\lambda\in\Lambda^+_M$ over $\theta\in -\Lambda_{G,P}^{pos}$ the property $\Gr_M^{\lambda}\subset \Gr_M^{-,\theta}$ is equivalent to $\lambda\in \Lambda^+_M\cap (-\Lambda^{pos})$. Set 
$$
\Lambda^-_{M,G}=\Lambda^+_M\cap (-\Lambda^{pos}).
$$

\index{$\Bun_T^{\lambda}, \Bun_M^{\theta}, \Bun_P^{\theta}, \Bunt_P^{\theta}, \Bunb_P^{\theta}, \Gr_M^{\theta}$, Section~\ref{Sect_1.3.10}}
 
\index{$\Gr_M^+,\Gr_M^{+,\theta}, \Lambda^+_{M, G}, \Lambda^-_{M,G}$, Section~\ref{Sect_1.3.10}}
 
\sssec{}  
\label{Sect_1.3.11_now}
For $\theta\in\Lambda_{G,P}$ write $\Gr_{M,\Ran,\theta}$ for the preimage of $\Bun_M^{\theta}$ under the map $\Gr_{M,\Ran}\to \Bun_M$ sending 
$(\cI,\cF_M,\beta: \cF_M\,\iso\, \cF^0_M\mid_{X-\Gamma_{\cI}})$ to $\cF_M$. One similarly defines the prestack $\Gr_{P,\Ran,\theta}$ and its versions with $\Ran$ replaced by $X^I$. For $\theta\in\Lambda_{G,G}$ we similarly define the prestack $\Gr_{G,\Ran,\theta}$.

\index{$\Gr_{M,\Ran,\theta}, \Gr_{P,\Ran,\theta}, \Gr_{M, I, \theta}, \Gr_{P, I, \theta}, \Gr_{G,\Ran,\theta}$, Section~\ref{Sect_1.3.11_now}}
    
\sssec{} 
\label{Sect_1.3.12_now}
For $d\ge 0$ write $X^{(d)}$ for the $d$-th symmetric power of $X$. For $\theta\in\Lambda^{pos}_{G,P}$, $S\in\Sch^{aff}$ and $D\in \Map(S, X^{\theta})$ we have a well-defined scheme $S\times X-\supp(D)$ as in (\cite{DL}, 4.2.14). Namely, let $\Theta:\Lambda_{G,P}^{pos}\to\ZZ_+$ be the homomorphism of semigroups sending each $\alpha_i, i\in \cI_G-\cI_M$ to $1\in\ZZ_+$. Then $S\times X-\supp(D)$ is defined as $S\times X-\supp(\Theta(D))$, where by $\supp(\Theta(D))$ we mean the pullback under $\Theta(D)\times\id: S\times X\to X^{(d)}\times X$ of the incidence divisor on $X\times X^{(d)}$, here $d=\Theta(\theta)$.   

 Write $\hat\cD_D$ for the formal completion of $S\times X$ along $\supp(\Theta(D))$. Write $\cD_D$ for the affine scheme obtained from $\hat\cD_D$, the image of $\hat\cD_D$ under $\colim: \Ind(\Sch^{aff})\to \Sch^{aff}$. Set $\oo{\cD}_D=\cD_D-\supp(\Theta(D))$. 
 
 For $n\ge 0$ and $D\in \Map(S, X^{\theta})$ denote by $D_n$ the $n$-th thickening of $\supp(\Theta(D))$, that is, the closed subscheme of $S\times X$ whose sheaf of ideals is the $n$-th power of the sheaf of ideals in $\cO_{S\times X}$ defining $\supp(\Theta(D))$. 
 
 For $\theta\in\Lambda_{G,P}^{pos}$ write $\Gr_{M, X^{\theta}}$ for the prestack classifying $D\in X^{\theta}$, a $M$-torsor $\cF_M$ on $X$ with a trivialization $\cF^0_M\,\iso\, \cF_M\mid_{X-\supp(D)}$. Our convention is that $\Gr_{M, X^0}=X^0=\Spec k$. For $\theta'\in\Lambda_{G,P}$ write $\Gr_{M, X^{\theta}, \theta'}$ for the preimage of $\Bun_M^{\theta'}$ under the map $\Gr_{M, X^{\theta}}\to \Bun_M$ sending the above point to $\cF_M$.  
 
\index{$X^{(d)}, \Theta, \hat\cD_D, \cD_D, \oo{\cD}_D, D_n, \Gr_{M, X^{\theta}}, \Gr_{M, X^{\theta}, \theta'}$, Section~\ref{Sect_1.3.12_now}} 
   
\sssec{} 
\label{Sect_1.3.13_now}
For $\theta\in\Lambda_{G,P}^{pos}$ let $(X^{\theta}\times\Ran)^{\subset}$ be the prestack, whose $S$-points for $S\in\Sch^{aff}$ are pairs $(\cI\in\Map(S,\Ran), D\in \Map(S, X^{\theta}))$ such that
$$
S\times X-\Gamma_{\cI}\subset S\times X-\supp(D).
$$

 Let $\pr_{\Ran}^{\theta}: (X^{\theta}\times\Ran)^{\subset}\to X^{\theta}$ be the projection. As in (\cite{Gai19Ran}, Lemma~1.3.3), $\pr_{\Ran}^{\theta}$ is universally homologically contractible in the sense of (\cite{Gai19Ran}, A.1.8).
 
\index{$(X^{\theta}\times\Ran)^{\subset}, \pr_{\Ran}^{\theta}$, Section~\ref{Sect_1.3.13_now}} 
 

From (\cite{Stack}, Lemma 31.18.9) one gets the following.
\begin{Rem} 
\label{Rem_1.3.14_now}
For $S\in\Sch^{aff}$, the groupoid of $S$-points
of $(X^{\theta}\times\Ran)^{\subset}$ identifies with the groupoid classifying $\cF_{M/[M,M]}\in \Map(S, \Bun_{M/[M,M]}^{-\theta})$, $\cI\in\Map(S,\Ran)$ with a trivialization $\beta: \cF_{M/[M,M]}\,\iso\, \cF^0_{M/[M,M]}$ on $S\times X-\Gamma_{\cI}$ such that for $\check{\lambda}\in\check{\Lambda}^+\cap\check{\Lambda}_{G,P}$ the map
$$
\cL^{\check{\lambda}}_{\cF^0_{M/[M,M]}}\to \cL^{\check{\lambda}}_{\cF_{M/[M,M]}}
$$
initially defined over $S\times X-\Gamma_{\cI}$ is regular over $S\times X$. Write 
$$
q_{\Ran}^+: (X^{\theta}\times\Ran)^{\subset}\to \Gr_{M/[M,M],\Ran}
$$ 
for the map sending $(\cI, D)$ to $(\cI, \cF_{M/[M,M]}, \beta)$, here $\cF_{M/[M,M]}=\cF^0_{M/[M,M]}(D)$. 

\index{$q_{\Ran}^+$, Remark~\ref{Rem_1.3.14_now}}
\end{Rem} 
 
\sssec{} 
\label{Sect_1.3.15_now}
For $\theta\in\Lambda_{G,P}^{pos}$ let 
$$
\ov{\Gr}^{\, -\theta}_{P,\Ran}\subset (X^{\theta}\times\Ran)^{\subset}\times_{\Ran} \Gr_{G,\Ran}
$$ 
be the closed subfunctor whose $S$-points are $S$-points $(D, \cI, \cF_G,\beta: \cF_G\,\iso\,\cF^0_G\mid_{X-\Gamma_{\cI}})$ of the RHS such that for any $\check{\lambda}\in\check{\Lambda}^+\cap \check{\Lambda}_{G,P}$ the map
$$
\cL^{\check{\lambda}}_{\cF^0_{M/[M,M]}}(\<D, \check{\lambda}\>)\to \cV^{\check{\lambda}}_{\cF_G}
$$
initially defined over $S\times X-\Gamma_{\cI}$ is regular on $S\times X$. Let $j^{-\theta}_{P,\Ran}: \Gr^{-\theta}_{P,\Ran}\subset \ov{\Gr}^{\, -\theta}_{P,\Ran}$ be the open subscheme where the above maps have no zeros on $S\times X$ (are maps of vector bundles on $S\times X$). 

 Note that $\ov{\Gr}^{\, -\theta}_{P,\Ran}, \Gr^{-\theta}_{P,\Ran}$ do not change if one replaces $(G,P)$ by $([G,G], P\cap[G,G])$. For $\theta=0$ these objects are those defined in Section~\ref{Sect_1.3.7}. 
 
The projection
$
\ov{\Gr}^{\, -\theta}_{P,\Ran}\to \Gr_{G,\Ran}
$
is ind-proper and factors through 
$$
\bar v^{-\theta}_{P,\Ran}: \ov{\Gr}^{\, -\theta}_{P,\Ran}\to
\ov{\Gr}^0_{P,\Ran}.
$$
Its restriction 
$$
v^{-\theta}_{P,\Ran}=\bar v^{-\theta}_{P,\Ran}\comp j^{-\theta}_{P,\Ran}
: \Gr^{-\theta}_{P,\Ran}\to \ov{\Gr}^0_{P, \Ran}
$$ 
is a locally closed embedding. The ind-schemes $\Gr^{-\theta}_{P, \Ran}$, $\theta\in\Lambda^{pos}_{G,P}$ form a stratification of $\ov{\Gr}^0_{P,\Ran}$. Note that $\Gr^{-\theta}_{P, \Ran}\hook{} \Gr_{P,\Ran, -\theta}$ is closed.

\index{$\ov{\Gr}^{\theta}_{P,\Ran}, j^{\theta}_{P,\Ran}, \Gr^{\theta}_{P,\Ran}, \bar v^{\theta}_{P,\Ran}, v^{\theta}_{P,\Ran}$, Section~\ref{Sect_1.3.15_now}}

\sssec{} 
\label{Sect_1.3.16_now}
For $\theta\in\Lambda_{G,P}^{pos}$ let 
$$
_{\le -\theta}\Bunb_P\subset X^{\theta}\times\Bunb_P
$$ 
be the closed substack whose $S$-points are those $S$-points 
$(D, \cF_{M/[M,M]}, \cF_G,\kappa)$ of the RHS such that for any $\check{\lambda}\in\check{\Lambda}^+\cap \check{\Lambda}_{G,P}$ the map
$$
\cL^{\check{\lambda}}_{\cF_{M/[M,M]}}(\<D, \check{\lambda}\>)\to \cV^{\check{\lambda}}_{\cF_G}
$$
is regular over $S\times X$.  Let $_{-\theta}\Bunb_P\subset {_{\le -\theta}\Bunb_P}$ be the open substack where all the above maps have no zeros on $S\times X$. The projection $_{-\theta}\Bunb_P\to\Bunb_P$ is a locally closed embedding. The substacks $_{-\theta}\Bunb_P$ for $\theta\in\Lambda^{pos}_{G,P}$ form a stratification of $\Bunb_P$.  

The map $\pi_{\Ran}$ yields the projection
$$
(X^{\theta}\times\Ran)^{\subset}\times_{\Ran} \Gr_{G,\Ran}\to
X^{\theta}\times (_{\Ran,\infty}\Bunb_P).
$$ 
The preimage of $\Ran\times{_{\le -\theta}\Bunb_P}$ (resp., of $\Ran\times (_{-\theta}\Bunb_P)$) under this map is $\ov{\Gr}^{\, -\theta}_{P,\Ran}$ and $\Gr^{\, -\theta}_{P,\Ran}$ respectively. 

  One has an isomorphism $X^{\theta}\times\Bun_P\,\iso\, {_{-\theta}\Bunb_P}$ sending $(D, \cF_P)$ to 
$$
(D, \cF_{M/[M,M]}(-D), \cF_G), 
$$
where $\cF_{M/[M,M]}, \cF_G$ are obtained from $\cF_P$ by extension of scalars. 

\index{$_{\le \theta}\Bunb_P, {_{\theta}\Bunb_P}$, Section~\ref{Sect_1.3.16_now}}
 
\sssec{}  
\label{Sect_1.3.17_now}
As in \cite{BFGM} for $\theta\in\Lambda^{pos}_{G,P}$ write $\Mod^{+,\theta}_{\Bun_M}$ for the stack classifying $D\in X^{\theta}, \cF_M,\cF'_M\in\Bun_M$ with an isomorphism $\beta: \cF_M\,\iso\,\cF'_M\mid_{X-\supp(D)}$ such that for any $V\in\Rep(G)^{\heartsuit}$ finite-dimensional, the natural map
$$
V^{U(P)}_{\cF_M}\to V^{U(P)}_{\cF'_M}
$$
is regular over $X$, and the induced map $\cF_{M/[M,M]}\,\iso\, \cF'_{M/[M,M]}$ over $X-\supp(D)$ extends to an isomorphism $\cF_{M/[M,M]}(D)\,\iso\, \cF'_{M/[M,M]}$ over $X$. 
 
 The scheme $\Mod_M^{+,\theta}$ is obtained from $\Mod^{+,\theta}_{\Bun_M}$ by base change posing $\cF'_M=\cF^0_M$. 
 
 For $\theta\in -\Lambda^{pos}_{G,P}$ denote by $\Mod_M^{-,\theta}$ the scheme defined in (\cite{DL}, 4.2.12). Namely, it is obtained from $\Mod^{+,-\theta}_{\Bun_M}$ by the base change posing $\cF_M=\cF^0_M$. Note that $\Mod_M^{-,\theta}\hook{}\Gr_{M, X^{-\theta}}$ is closed.
 
\index{$\Mod^{+,\theta}_{\Bun_M}, \Mod_M^{+,\theta}, \Mod_M^{-,\theta}$, Section~\ref{Sect_1.3.17_now}} 
 
\begin{Rem} 
\label{Rem_1.3.17}
Let $\theta\in -\Lambda^{pos}_{G,P}$.\\ 
i) The scheme $\Mod_M^{-,\theta}$ can be seen as the scheme classifying $D\in X^{-\theta}$, a $M$-torsor $\cF_M$ on $\cD_D$ together with a trivialization $\beta: \cF_M\,\iso\, \cF^0_M$ over $\oo{\cD}_D$ such that for any $V\in\Rep(G)^{\heartsuit}$ finite-dimensional, the natural map
$$
V^{U(P)}_{\cF^0_M}\to V^{U(P)}_{\cF_M}
$$
is regular over $\cD_D$, and the induced map $\cF_{M/[M,M]}\,\iso\, \cF^0_{M/[M,M]}\mid_{\oo{\cD}_D}$ extends to an isomorphism $\cF_{M/[M,M]}\,\iso\,  \cF^0_{M/[M,M]}(D)$ over $\cD_D$.
 
 Let $\gL^+(M)_{\theta}$ be the group scheme over $X^{-\theta}$ whose fibre over $D$ is $\Map(\cD_D, M)$. It acts naturally on $\Mod_M^{-,\theta}\hook{}\Gr_{M, X^{-\theta}}$. For $n\ge 0$ let $\gL^+(M)_{n,\theta}$ be the group scheme over $X^{-\theta}$ whose fibre over $D$ is $\Map(D_n, M)$. The notation $D_n$ is that of Section~\ref{Sect_1.3.12_now}.
 
 For $\theta=0$ we get $\Mod_M^{-,\theta}=\Spec k=X^{-\theta}$, and the group scheme $\cL^+(M)_{\theta}$ is trivial. 
 
\smallskip\noindent 
ii) The stack quotient $\Mod_M^{-,\theta}/\cL^+(M)_{\theta}$ classifies: $D\in X^{-\theta}$, $M$-torsors $\cF_M,\cF'_M$ on $\cD_D$ with a trivialization $\beta_D: \cF_M\,\iso\, \cF'_M$ over $\oo{\cD}_D$ such that for any $V\in\Rep(G)^{\heartsuit}$ finite-dimensional, the natural map
$$
V^{U(P)}_{\cF_M}\to V^{U(P)}_{\cF'_M}
$$
is regular over $\cD_D$, and the reduction of $\beta_D$ to $M/[M,M]$ extends to an isomorphism $\cF_{M/[M,M]}(D)\,\iso\,  \cF'_{M/[M,M]}$ over $\cD_D$.

\index{$\gL^+(M)_{\theta}, \gL^+(M)_{n,\theta}$, Remark~\ref{Rem_1.3.17}}
\end{Rem}
  
\sssec{} 
\label{Sect_1.3.19_now}
For $\theta\in\Lambda_{G,P}^{pos}$ let $_{-\theta}\Bunt_P$ be the preimage of $_{-\theta}\Bunb_P\hook{} \Bunb_P$ under $\gr: \Bunt_P\to\Bunb_P$. By (\cite{BFGM}, Section~3.3), one has an isomorphism
$$
_{-\theta}\Bunt_P\,\iso\,\Bun_P\times_{\Bun_M} \Mod^{+,\theta}_{\Bun_M}
$$
where we used the map $h^{\ra}: \Mod^{+,\theta}_{\Bun_M}\to \Bun_M$, $(D,\cF_M, \cF'_M,\beta)\mapsto \cF'_M$ to form the fibred product. We have the diagram
$$
\Mod^{+,\theta}_{\Bun_M}\;\getsup{p^{-\theta}_{glob}}\;
{_{-\theta}\Bunt_P}\;\toup{v^{-\theta}_{glob}}\; \Bunt_P,
$$
where $v^{-\theta}_{glob}$ is the inclusion, and $p^{-\theta}_{glob}$ is the projection. Write 
$$
\bar v^{-\theta}_{glob}: \Bunt_P\times_{\Bun_M} \Mod^{+,\theta}_{\Bun_M}\to \Bunt_P
$$ 
for the map defined as follows. First, we used $h^{\ra}$ to form the fibred product. Then $\bar v^{-\theta}_{glob}$ sends $(\cF_M, \cF'_M, D,\beta)\in \Mod^{+,\theta}_{\Bun_M}$, $(\cF'_M, \cF'_G,\kappa)\in\Bunt_P$ to $(\cF_M, \cF'_G, \kappa')$, here $\kappa'$ are obtained naturally from $\kappa, \beta$. The map $\bar v^{-\theta}_{glob}$ is proper. 
  
\index{$_{\theta}\Bunt_P, v^{\theta}_{glob}, p^{\theta}_{glob}, \bar v^{\theta}_{glob}$, 
Section~\ref{Sect_1.3.19_now}}  
  
\sssec{}  
\label{Sect_1.3.20_now}
For $\theta\in\Lambda^{pos}_{G,P}$ set 
$$
\Gr_{M, \Ran,+}^{-\theta}=(X^{\theta}\times\Ran)^{\subset}\times_{\Gr_{M/[M,M],\Ran}} \Gr_{M,\Ran}, 
$$
where we used the map $q_{\Ran}^+$ to form the fibred product. So, for $S\in\Sch^{aff}$ an $S$-point of $\Gr_{M, \Ran,+}^{-\theta}$ is a collection 
$$
(\cI\in\Map(S,\Ran), D\in \Map(S, X^{\theta}), \cF_M, \beta_M),
$$ 
where $\cF_M$ is a $M$-torsor on $S\times X$, $\beta_M: \cF_M\,\iso\, \cF^0_M\mid_{S\times X-\Gamma_{\cI}}$ is a trivialization such that the induced trivialization $\cF_{M/[M,M]}\,\iso\, \cF^0_{M/[M,M]}\mid_{S\times X-\Gamma_{\cI}}$ extends to an isomorphism $\cF_{M/[M,M]}\,\iso\, \cF^0_{M/[M,M]}(D)$ over $S\times X$. Set
$$
\Gr_{M,\Ran,+}=\underset{\theta\in\Lambda^{pos}_{G,P}}{\sqcup} \Gr_{M,\Ran,+}^{-\theta}.
$$ 

\index{$\Gr_{M, \Ran,+}^{\theta}, \Gr_{M,\Ran,+}$, Section~\ref{Sect_1.3.20_now}}

 As in (\cite{BG}, 4.3.2), one gets the following.
\begin{Lm} 
\label{Lm_1.3.21_now}
For $\theta\in\Lambda_{G,P}^{pos}$ there is a unique morphism $\gt^{-\theta}_{P,\Ran}: \Gr^{-\theta}_{P,\Ran}\to \Gr_{M, \Ran,+}^{-\theta}$ with the following property. For $(\cI, D,\cF_G, \beta: \cF_G\,\iso\, \cF^0_G\mid_{X-\Gamma_{\cJ}})\in \Gr^{-\theta}_{P,\Ran}$ its image $(D, \cI, \cF_M,\beta_M)$ is the unique point of $\Gr^{-\theta}_{M,\Ran, +}$ for which the composition
$$
V^{U(P)}_{\cF_M}\toup{\beta_M} V^{U(P)}_{\cF^0_M}\to V_{\cF^0_G}\toup{\beta} V_{\cF_G}
$$
is a map of vector bundles over $X$ for any $V\in\Rep(G)^{\heartsuit}$ finite-dimensional. 
In fact, $\Gr^{-\theta}_{P,\Ran}\,\iso\, \Gr_{P,\Ran}\times_{\Gr_{M,\Ran}} \Gr_{M, \Ran,+}^{-\theta}$ naturally. 
\QED
\index{$\gt^{\theta}_{P,\Ran}$, Lemma~\ref{Lm_1.3.21_now}}
\end{Lm} 

\sssec{} 
\label{Sect_1.3.22_now}
For $\theta\in\Lambda_{G,P}^{pos}$ write $q^+_{\theta}: X^{\theta}\to \Gr_{M/[M,M], X^{\theta}}$ for the map sending $D$ to $\cF_{M/[M,M]}=\cF^0_{M/[M,M]}(D)$ with the evident trivialization 
$$
\cF_{M/[M,M]}\,\iso\,\cF^0_{M/[M,M]}\mid_{X-\supp(D)}.
$$ 
Set 
$$
\Gr^{-\theta}_{M, X^{\theta}, +}=X^{\theta}\times_{\Gr_{M/[M,M], X^{\theta}}} \Gr_{M, X^{\theta}},
$$
where we used the map $q^+_{\theta}$. Our convention is that $\Gr^0_{M, X^0, +}=\Spec k$.  

\index{$q^+_{\theta}, \Gr^{\theta}_{M, X^{\theta}, +}$, Section~\ref{Sect_1.3.22_now}}

\sssec{} If $\theta\in -\Lambda_{G,P}^{pos}$ then the square is cartesian
$$
\begin{array}{ccc} 
\Gr_{M, \Ran, +}^{\theta} & \to & \ov{\Gr}_{P,\Ran}^0\\
\downarrow && \downarrow\\
\Gr_{M, \Ran, \theta} & \to & \Gr_{G,\Ran},
\end{array}
$$
where the vertical arrows are closed immersions.

\sssec{} For $-\theta\in\Lambda_{G,P}^{pos}$ the prestacks $(X^{\theta}\times\Ran)^{\subset}, \ov{\Gr}^{\theta}_{P,\Ran}, \Gr^{\theta}_{P, \Ran}, \Gr_{M,\Ran, +}^{\theta}$ have natural unital structures.

Note that if $G=P$ then $\Gr^0_{G,\Ran, +}=\Gr_{[G,G], \Ran}=\Gr^0_{G,\Ran}$.

\sssec{} 
\label{Sect_1.3.25_now}
For $-\theta\in\Lambda^{pos}_{G,P}$ let $i^{\theta}_{P, \Ran}: \Gr_{M, \Ran,+}^{\theta}\to \Gr^{\theta}_{P,\Ran}$ be the section obtained from the section $\Gr_{M,\Ran}\to \Gr_{P,\Ran}$. 

\index{$i^{\theta}_{P, \Ran}$, Section~\ref{Sect_1.3.25_now}}

\sssec{} 
\label{Sect_1.3.26_now}
For $\theta\in\Lambda_{G,P}^{pos}$ we define the versions $\ov{\Gr}^{\, -\theta}_{P^-, \Ran}, \Gr^{-\theta}_{P^-,\Ran}$ of the above objects with $P$ replaced by $P^-$ as follows\footnote{We warn the reader that these are not the exact analogs of the corresponding objects for $P$, as in our definition the divisor appearing is $D$ and not $-D$.}
. Let 
$$
\ov{\Gr}^{\,-\theta}_{P^-, \Ran}\subset (X^{\theta}\times\Ran)^{\subset}\times_{\Ran} \Gr_{G,\Ran}
$$
be the closed subfunctor whose $S$-points for $S\in\Sch^{aff}$ are $S$-points $(D,\cI, \cF_G, \beta: \cF_G\,\iso\, \cF^0_G\mid_{X-\Gamma_{\cI}})$ of the RHS such that for any $\check{\lambda}\in\check{\Lambda}^+\cap \check{\Lambda}_{G,P}$ the map
$$
\tilde\cV^{\check{\lambda}}_{\cF_G}\to \cL^{\check{\lambda}}_{\cF^0_{M/[M,M]}}(\<D, \check{\lambda}\>)
$$
initially defined over $S\times X-\Gamma_{\cI}$ is regular over $S\times X$. Let $j^{-\theta}_{P^-, \Ran}: \Gr^{-\theta}_{P^-,\Ran}\subset \ov{\Gr}^{\,-\theta}_{P^-, \Ran}$ be the open subscheme where the above maps are morphisms of vector bundles on $S\times X$. 
Note that $\Gr^{-\theta}_{P^-,\Ran}\hook{}\Gr_{P^-,\Ran, -\theta}$ is closed.

 As for $P$, one defines the map $\gt^{-\theta}_{P^-, \Ran}: \Gr^{-\theta}_{P^-,\Ran}\to \Gr^{-\theta}_{M,\Ran, +}$.  
 
\index{$\ov{\Gr}^{\theta}_{P^-, \Ran}, \Gr^{\theta}_{P^-,\Ran}, j^{\theta}_{P^-, \Ran}, \gt^{\theta}_{P^-, \Ran}$, Section~\ref{Sect_1.3.26_now}} 

\sssec{} 
\label{Sect_1.3.27_now}
For $\theta\in -\Lambda^{pos}_{G,P}$ write $S^{\theta}_{P, \Ran}$ for the preimage of $\Gr_{P, \Ran}^{\theta}$ under $\bar S^0_{P, \Ran}\hook{} \ov{\Gr}^0_{P, \Ran}$. So, $\{S^{\theta}_{P, \Ran}\}$, $\theta\in -\Lambda^{pos}_{G, P}$ is a stratification of $\bar S^0_{P, \Ran}$. Write $v^{\theta}_{S, \Ran}: S^{\theta}_{P, \Ran}\to \bar S^0_{P, \Ran}$ for the locally closed immersion. Write $\eta^{\theta}_{\Ran}: S^{\theta}_{P,\Ran}\to \Gr^{\theta}_{P,\Ran}$ for the natural closed immersion.
  
 Consider the prestack $\Mod_{M,\Ran}^{-,\theta,\subset}:=\Mod_M^{-,\theta}\times_{X^{-\theta}}(X^{-\theta}\times\Ran)^{\subset}$ classifying $(D, \cI)\in (X^{-\theta}\times\Ran)^{\subset}$ and 
$$
(D, \cF'_M,\beta_M: \cF_M^0\,\iso\, \cF'_M\mid_{X-\supp(D)})\in \Mod_M^{-,\theta}.
$$ 
We have the closed immersion 
\begin{equation}
\label{inclusion_Mod_M_Ran^theta^subset}
\Mod_{M,\Ran}^{-,\theta,\subset}\hook{} \Gr^{\theta}_{M,\Ran, +}
\end{equation}
sending the above collection to $(D, \cI, \cF'_M, \beta_M: \cF_M^0\,\iso\, \cF'_M\mid_{X-\Gamma_{\cI}})$. This closed subscheme is preserved under the action of $\gL^+(M)_{\Ran}$ on $\Gr^{\theta}_{M,\Ran, +}$. Define the map $\gt^{\theta}_{S,\Ran}$
by the cartesian square
$$
\begin{array}{ccc}
S^{\theta}_{P,\Ran} & \hook{} & \Gr^{\theta}_{P,\Ran}\\ \\
\downarrow\lefteqn{\scriptstyle \gt^{\theta}_{S,\Ran}} && \downarrow\lefteqn{\scriptstyle \gt^{\theta}_{P,\Ran}}\\ \\
\Mod_{M,\Ran}^{-,\theta,\subset} & \hook{} &\Gr^{\theta}_{M,\Ran, +},
\end{array}
$$  
where the horizontal arrows are closed immersions. For $I\in fSets$ let $\Mod_{M,I}^{-,\theta,\subset}=\Mod_{M,\Ran}^{-,\theta,\subset}\times_{\Ran} X^I$. Let 
$
i^{\theta}_{S,\Ran}: \Mod_{M,\Ran}^{-,\theta,\subset}\to S^{\theta}_{P,\Ran}
$ 
be the restriction of $i^{\theta}_{P,\Ran}$.  

\index{$S^{\theta}_{P, \Ran}, v^{\theta}_{S, \Ran}, \eta^{\theta}_{\Ran}, \Mod_{M,\Ran}^{-,\theta,\subset}, \gt^{\theta}_{S,\Ran}$, Section~\ref{Sect_1.3.27_now}}
\index{$\Mod_{M,I}^{-,\theta,\subset}, i^{\theta}_{S,\Ran}$, Section~\ref{Sect_1.3.27_now}}

\sssec{} 
\label{Sect_1.3.28_now}
For $-\theta\in\Lambda_{G,P}^{pos}$ set $\bar S^{\theta}_{P,\Ran}=\bar S^0_{P,\Ran}\times_{\ov{\Gr}^0_{P,\Ran}} \ov{\Gr}^{\,\theta}_{P,\Ran}$. Let $\bar v^{\theta}_{S,\Ran}: \bar S^{\theta}_{P,\Ran}\to \bar S^0_{P,\Ran}$ be the projection. Denote by 
$$
j^{\theta}_{S,\Ran}: S^{\theta}_{P,\Ran}\to \bar S^{\theta}_{P,\Ran}
$$ 
the natural open immersion.

\index{$\bar S^{\theta}_{P,\Ran}, \bar v^{\theta}_{S,\Ran}, j^{\theta}_{S,\Ran}$, Section~\ref{Sect_1.3.28_now}} 

\sssec{} For $\theta\in -\Lambda^{pos}_{G,P}$ the prestacks $\bar S^{\theta}_{P,\Ran}$, $S^{\theta}_{P,\Ran}, \Mod_{M,\Ran}^{-,\theta,\subset}$ have natural unital structures, and the maps 
$$
\eta_{\Ran}, j^{\theta}_{S,\Ran}, \bar v^{\theta}_{S,\Ran}, \eta^{\theta}_{\Ran}, \gt^{\theta}_{P,\Ran}, i^{\theta}_{P,\Ran}
$$ 
are equivariant under the actions of the category object $(\Ran\times\Ran)^{\subset}$.


\sssec{} 
\label{Sect_1.3.27}
For $\theta\in -\Lambda^{pos}_{G,P}$ the stack quotient $\Gr^{\theta}_{M,\Ran, +}/\gL^+(M)_{\Ran}$ classifies $(D,\cI)\in (X^{-\theta}\times\Ran)^{\subset}$, $M$-torsors $\cF_M, \cF'_M$ on $\cD_{\cI}$ with an isomorphism $\beta: \cF_M\,\iso\,\cF'_M\mid_{\oo{\cD}_{\cI}}$ whose reduction to $M/[M,M]$ extends to an isomorphism on $\cD_{\cI}$
$$
\cF_{M/[M,M]}(D)\,\iso\, \cF'_{M/[M,M]}.
$$

 The closed substack $\Mod_{M,\Ran}^{-,\theta,\subset}/\gL^+(M)_{\Ran}$ of the latter classifies a point as above for which $\beta$ extends to an isomorphism
\begin{equation}
\label{iso_beta_extended_for_Sect_1.3.27}
\beta: \cF_M\,\iso\, \cF'_M\mid_{\cD_{\cI}-\supp(D)}
\end{equation}
and such that for any $V\in \Rep(G)^{\heartsuit}$ finite-dimensional the natural map
\begin{equation}
\label{map_U(P)_invariants_for_Sect_1.3.27}
V^{U(P)}_{\cF_M}\to V^{U(P)}_{\cF'_M}
\end{equation}
is regular over $\cD_{\cI}$. We used here the open immersion $\oo{\cD}_{\cI}\hook{} \cD_{\cI}-\supp(D)$. 

\sssec{} 
\label{Sect_1.3.31_now}
Let $\theta\in -\Lambda^{pos}_{G,P}$. Recall that 
$$
Shv(\Gr^{\theta}_{M,\Ran, +})^{\gL^+(M)_{\Ran}}\,\iso\, \underset{I\in fSets}{\lim} Shv(\Gr^{\theta}_{M, I, +})^{\gL^+(M)_I}.
$$ 
For each $I$ we have the functor $\oblv[\dimrel]: Shv(\Gr^{\theta}_{M, I, +})^{\gL(M)_I}\to  Shv(\Gr^{\theta}_{M, I, +})$ defined in (\cite{DL}, A.5.3). By Section~\ref{Sect_A.4.1}, these functors commute with the $!$-restrictions along $X^J\to X^I$ for maps $I\to J$ in $fSets$. Passing to the limit over $I\in fSets$ in the above functors, one gets the functor denoted
$$
\oblv[\dimrel]: Shv(\Gr^{\theta}_{M,\Ran, +})^{\gL^+(M)_{\Ran}}\to Shv(\Gr^{\theta}_{M,\Ran, +}).
$$ 

 Similarly, one defines the functor
$$
\oblv[\dimrel]: Shv(\Mod_{M,\Ran}^{-,\theta,\subset})^{\gL^+(M)_{\Ran}}\to Shv(\Mod_{M,\Ran}^{-,\theta,\subset}).
$$  

\index{$\oblv[\dimrel]$, Section~\ref{Sect_1.3.31_now}}

\sssec{Actions of $\Map(I,  \Lambda_{M, ab})$} 
\label{Sect_1.3.32_now}
Consider the action of $\Lambda_{M, ab}$ on $\Gr_{G, x}$ such that $\lambda\in\Lambda_{M, ab}$ acts on $\Gr_{G, x}$ as multiplcation by $t_x^{\lambda}$, where $t_x\in \cO_x$ is a uniformizer. The global analog of this action is as follows. Let $I\in fSets$. 

Let us define an action of $\Map(I, \Lambda_{M, ab})$ on the diagram over $X^I$
$$
\Gr_{M, I}/\gL^+(M)_I\gets \Gr_{P, I}/\gL^+(M)_I\to
\Gr_{G, I}/\gL^+(M)_I\to X^I/\gL^+(M)_I.
$$ 

 Set $\bar T=\Lambda_{M, ab}\otimes\Gm$. Note that the natural map $\Lambda_{M, ab}\otimes \GG\to \Lambda\otimes\Gm=T$ coming from $\Lambda_{M, ab}\subset \Lambda$ takes values in the center $Z(M)$ of $M$. 
 
 The stack quotient $X^I/\gL^+(M)_I$ classifies $\cI\in X^I$ and a $M$-torsor $\cF_M$ on $\cD_{\cI}$. We let $\und{\lambda}\in \Map(I, \Lambda_{M, ab})$ act on $X^I/\gL^+(M)_I$ sending the point $(\cI, \cF_M)$ to $(\cI, \cF'_M)$. Here 
$$
\cF_{\bar T}=\cF^0_{\bar T}(-\sum_{i\in I} \lambda_i \Gamma_i)
$$ 
is the $\bar T$-torsor on $\cD_{\cI}$, and $\cF'_M=\cF_M\times^{\bar T} \cF_{\bar T}$ is the $M$-torsor on $\cD_{\cI}$, the quotient of $\cF_M\times \cF_{\bar T}$ by the diagonal action of $\bar T$. We used that $\Gamma_i$ is an effective Cartier divisor on $\cD_{\cI}$.    

 The stack quotient $\Gr_{G, I}/\gL^+(M)_I$ classifies: $\cI\in X^I$, $M$-torsor $\cF_M$ on $\cD_{\cI}$, $G$-torsor $\cF_G$ on $\cD_{\cI}$ with an isomorphism $\beta: \cF_M\times^M G\,\iso\,\cF_G\mid_{\oo{\cD}_{\cI}}$ of $G$-torsors. 
 
 Let $\und{\lambda}\in \Map(I, \Lambda_{M, ab})$ act on $\Gr_{G, I}/\gL^+(M)_I$ sending the above point to 
$(\cI, \cF'_M, \cF_G, \beta')$, where $\cF'_M$ is as above, and
$$
\beta': \cF'_M\times^M G\,\iso\,\cF_G\mid_{\oo{\cD}_{\cI}}
$$
is obtained from $\beta$ by composing with $\cF_M\,\iso\, \cF'_M\mid_{\oo{\cD}_{\cI}}$. 

 The $\Map(I, \Lambda_{M, ab})$-action on $\Gr_{M, I}/\gL^+(M)_I,\; \Gr_{P, I}/\gL^+(M)_I$ is defined similarly. 
 
\index{$t_x,\bar T$, Section~\ref{Sect_1.3.32_now}} 
 
\begin{Rem} 
\label{Rem_action_on_conn_components}
For $\und{\lambda}\in \Map(I, \Lambda_{M, ab})$ let $\lambda=\sum_i \und{\lambda}(i)$, $\theta\in\Lambda_{G,P}$. The action map 
$$
\und{\lambda}: \Gr_{M, I, \theta}/\gL^+(M)_I\to \Gr_{M, I, \theta+\lambda}/\gL^+(M)_I
$$ 
is an isomorphism, and similarly for $\und{\lambda}: \Gr_{P, I, \theta}/\gL^+(M)_I\to \Gr_{P, I, \theta+\lambda}/\gL^+(M)_I$.
\end{Rem}

\sssec{} 
\label{Sect_1.3.30}
For $\und{\lambda}\in \Map(I, \Lambda_{M, ab})$, $K\in Shv(\Gr_{M, I}/\gL^+(M)_I)$ we denote by $\und{\lambda}K\in Shv(\Gr_{M, I}/\gL^+(M)_I)$ 
the direct image of $K$ under the action map 
$$
\und{\lambda}: \Gr_{M, I}/\gL^+(M)_I\to \Gr_{M, I}/\gL^+(M)_I.
$$

If $K\in Shv(\Gr_{G, I}/\gL^+(M)_I)$ then $\und{\lambda}K\in Shv(\Gr_{G, I}/\gL^+(M)_I)$ is defined similarly. This action of $\Map(I, \Lambda_{M, ab})$ on $Shv(\Gr_{G, I})^{\gL^+(M)_I}$ preserves the full subcategory $\SI_{P, I}$. 

\index{$\und{\lambda}K$, Section~\ref{Sect_1.3.30}}

\sssec{Example} 
\label{Sect_1.3.31}
Given $I\in fSets$, $\und{\lambda}: I\to \Lambda_{M, ab}$ let $s^{\und{\lambda}}_G: X^I\to \Gr_{G, I}$ be the composition $X^I\toup{s^{\und{\lambda}}}\Gr_{T, I}\to \Gr_{G, I}$, here $s^{\und{\lambda}}$ is defined in Section~\ref{Sect_1.2.5}. 
 This map is $\gL^+(M)_I$-equivariant. So, we may view $(s^{\und{\lambda}}_G)_*\omega$ as an object of $Shv(\Gr_{G,I})^{\gL^+(M)_I}$. Then $\und{\lambda}(s^{\und{0}}_G)_*\omega\,\iso\, (s^{\und{\lambda}}_G)_*\omega$. We set for brevity $\delta_I=(s^{\und{0}}_G)_*\omega$. 
 
\index{$s^{\und{\lambda}}_G, \delta_I$, Example~\ref{Sect_1.3.31}} 
 
\sssec{} The $\Map(I, \Lambda_{M, ab})$-action on $Shv(\Gr_{G, I}/\gL^+(M)_I)$ is compatible with $!$-restric\-tions under $\vartriangle: X^J\to X^I$ for a map $\phi: I\to J$ in $fSets$. Namely, for $\und{\mu}=\phi_*\und{\lambda}$ one has 
$$
\vartriangle^!(\und{\lambda}K)\,\iso\, \und{\mu}(\vartriangle^! K)
$$ 
in $Shv(\Gr_{G, J}/\gL^+(M)_J)$ functorially on $K\in Shv(\Gr_{G, I}/\gL^+(M)_I)$.

\ssec{Structure of $\SI^{\le 0}_{P,\Ran}$}
\label{Sect_Structure_of_SI^le0_P,Ran}

\sssec{} 
\label{Sect_1.4.1_now}
If $-\theta\in\Lambda^{pos}_{G,P}$ then $H_{\Ran}$ acts naturally on 
$\Gr_{P,\Ran}^{\theta}, \ov{\Gr}_{P,\Ran}^{\, \theta}$ over $(X^{-\theta}\times\Ran)^{\subset}$, so one similarly gets for $I\in fSets$ the categories
$$
\SI^{\le \theta}_{P, I}=Shv(\ov{\Gr}_{P, I}^{\, \theta})^{H_I}, \;\;\;\; \SI^{\theta}_{P, I}=Shv(\Gr_{P, I}^{\theta})^{H_I}
$$
and their Ran versions
$$
\SI^{\le \theta}_{P,\Ran}=\underset{I\in fSets}{\lim} \SI^{\le \theta}_{P, I},\;\;\;\; \SI^{\theta}_{P,\Ran}=\underset{I\in fSets}{\lim}\SI^{\theta}_{P, I}.
$$

\index{$\SI^{\le \theta}_{P, I}, \SI^{\theta}_{P, I}, \SI^{\le \theta}_{P,\Ran}, \SI^{\theta}_{P,\Ran}$, Section~\ref{Sect_1.4.1_now}}

\sssec{} 
\label{Sect_1.4.2}
For $-\theta\in\Lambda^{pos}_{G,P}$, $I\in fSets$ we get the adjoint pairs in $Shv(X^I)-mod$
$$
(j^{\theta}_{P,I})^*: \SI^{\le \theta}_{P,I}\leftrightarrows \SI^{\theta}_{P,I}: (j^{\theta}_{P,I})_*
$$
and
$$
(\bar v^{\theta}_{P,I})_!: \SI^{\le \theta}_{P,I}\leftrightarrows \SI^{\le 0}_{P,I}: (\bar v^{\theta}_{P,I})^!
$$

 For a map $I\to J$ in $fSets$ these functors are compatible with the $!$-restrictions along $\vartriangle^{(I/J)}: X^J\to X^I$, so yield similar functors between the Ran versions of the corresponding categories. 

\sssec{} 
Pick $\lambda_0\in\Lambda^+_{M, ab}$ satisfying $\<\lambda_0, \check{\alpha}_i\>>0$ for $i\in \cI_G-\cI_M$. 

\begin{Lm}
\label{Lm_1.4.3} 
Let $-\theta\in\Lambda_{G,P}^{pos}$, $I\in fSets$. \\
i) The functor 
\begin{equation}
\label{funct_for_Lm_1.4.3}
(\gt^{\theta}_{P, I})^!: Shv(\Gr_{M, I, +}^{\theta})^{\gL^+(M)_I}\to Shv(\Gr_{P, I}^{\theta})^{\gL^+(M)_I}
\end{equation} 
is fully faithful, its essential image is the full subcategory $\SI_{P, I}^{\theta}$. The composition
$$
\SI_{P, I}^{\theta}\hook{} \Shv(\Gr_{P, I}^{\theta})^{\gL^+(M)_I}\toup{(i^{\theta}_{P, I})^!} Shv(\Gr_{M, I, +}^{\theta})^{\gL^+(M)_I}
$$
is an equivalence.

\smallskip\noindent
ii) The partially defined left adjoint $(\gt^{\theta}_{P, I})_!$ to (\ref{funct_for_Lm_1.4.3}) is defined on the full subcategory 
$
\SI_{P, I}^{\theta},
$
and 
$
(\gt^{\theta}_{P, I})_!\,\iso\, (i^{\theta}_{P, I})^!
$ 
canonically as functors 
$$
\SI_{P, I}^{\theta}\to Shv(\Gr_{M, I, +}^{\theta})^{\gL^+(M)_I}.
$$ 

 The partially defined left adjoint $(i^{\theta}_{P, I})^*$ to 
$$
(i^{\theta}_{P, I})_*: Shv(\Gr_{M, I, +}^{\theta})^{\gL^+(M)_I}\to Shv(\Gr_{P, I}^{\theta})^{\gL^+(M)_I}
$$ 
is defined on the full subcategory $\SI_{P, I}^{\theta}$, and one has canonically
$(\gt^{\theta}_{P, I})_*\,\iso\, (i^{\theta}_{P, I})^*$ as functors 
$$
\SI_{P, I}^{\theta}\to Shv(\Gr_{M, I, +}^{\theta})^{\gL^+(M)_I}.
$$
\end{Lm}
\begin{proof} i) The group ind-scheme $\gL(U(P))_I\to X^I$ is ind-prounipotent, and acts transitively on the fibres of $\gt^{\theta}_{P, I}: \Gr_{P, I}^{\theta}\to \Gr^{\theta}_{M, I, +}$.  
So, the claim follows from (\cite{DL}, Lemma~ A.4.4, A.4.5) whose proof works for both constructible context and $\cD$-modules. In the case of $\cD$-modules we may also invoke (\cite{Chen}, Lemma~B.4.1).

\medskip\noindent
ii) Let $\Gm$ act on $\Gr_{P, I}^{\theta}$ over $X^I$ via $\lambda_0$, so $\Gm$ contracts $\Gr_{P, I}^{\theta}$ to $\Gr_{M, I, +}^{\theta}$. This action commutes with that of $\gL^+(M)_I$. Our claim follows now from a version of a theorem of Braden (\cite{DG1}, Proposition~3.2.2).  
\end{proof}

\begin{Lm} 
\label{Lm_1.4.5}
Let $I\in fSets$. Consider the $\Gm$-action on $\ov{\Gr}_{P, I}^0$ via $\lambda_0$. The attracting (resp. repelling) locus of this $\Gm$-action is
$$
\underset{\theta\in -\Lambda_{G,P}^{pos}}{\sqcup} \Gr^{\theta}_{P, I}\;\;\;\;\;\mbox{and}\;\;\;\;\;\underset{\theta\in -\Lambda_{G,P}^{pos}}{\sqcup} \Gr^{\theta}_{P^-, I}\cap \ov{\Gr}_{P, I}^0
$$
respectively. \QED
\end{Lm} 

\begin{Lm} 
\label{Lm_1.4.6}
Let $I\in fSets$, $-\theta\in\Lambda^{pos}_{G,P}$. The partially defined left adjoint of $(v^{\theta}_{P, I})_*: \SI_{P,I}^{\theta}\to \SI_{P,I}^{\le 0}$ is defined and denoted $(v^{\theta}_{P, I})^*: \SI_{P,I}^{\le 0}\to \SI_{P,I}^{\theta}$. Besides, the natural left-lax $Shv(X^I)$-structure on $(v^{\theta}_{P, I})^*$ is strict. For a map $I\to J$ in $fSets$ the canonical natural transformation 
$$
(v^{\theta}_{P, J})^*\!\vartriangle^!\to \vartriangle^!\!(v^{\theta}_{P, I})^*,\;\;\;\; \SI^{\le 0}_{P, I} \to \SI^{\theta}_{P, J}
$$ 
is an isomorphism, here $\vartriangle: X^J\to X^I$. 
\end{Lm} 
\begin{proof}
This is analogous to (\cite{Gai19Ran}, 1.5.3). Using Lemma~\ref{Lm_1.4.3}, it suffices to show that the composition 
$$
\Shv(\Gr^{\theta}_{M, I, +})^{\gL^+(M)_I}\toup{(\gt_{P,I}^{\theta})^!} \SI_{P,I}^{\theta}\toup{(v^{\theta}_{P, I})_*} \SI_{P,I}^{\le 0}
$$ 
admits a left adjoint. For this, it suffices to show that the partially defined left adjoint to the composition
$$
\Shv(\Gr^{\theta}_{M, I, +})^{\gL^+(M)_I}\toup{(\gt_{P,I}^{\theta})^!} Shv(\Gr_{P, I}^{\theta})^{\gL^+(M)_I}\toup{(v^{\theta}_{P, I})_*} Shv(\ov{\Gr}^0_{P, I})^{\gL^+(M)_I}
$$
is defined on the full subcategory $\SI_{P,I}^{\le 0}\subset Shv(\ov{\Gr}^0_{P, I})^{\gL^+(M)_I}$. 

 Consider the $\Gm$-action on $\ov{\Gr}^0_{P, I}$ by the coweight $\lambda_0$. 
Apply Braden's theorem (\cite{DG1}, 3.4.3). Using Lemma~\ref{Lm_1.4.5}, we conclude that the desired left adjoint is given by 
$$
(\gt^{\theta}_{P^-, I})_*(v^{\theta}_{P^-, I})^!
$$
for the diagram
$$
\Gr_{M, I,+}^{\theta}\getsup{\gt^{\theta}_{P^-, I}} \Gr_{P^-, I}^{\theta}\cap \ov{\Gr}^0_{P, I}\toup{v^{\theta}_{P^-, I}} \ov{\Gr}^0_{P, I}.
$$
Our claims follow.
\end{proof}

\begin{Cor} 
\label{Cor_1.4.7_left_adjoint}
Let $I\in fSets$, $-\theta\in\Lambda^{pos}_{G,P}$. The left adjoint of $(v^{\theta}_{P, I})^!: \SI^{\le 0}_{P, I}\to \SI^{\theta}_{P, I}$ is defined and denoted $(v^{\theta}_{P, I})_!: \SI^{\theta}_{P, I}\to \SI^{\le 0}_{P, I}$. Moreover, the natural left-lax $Shv(X^I)$-structure on $(v^{\theta}_{P, I})_!$ is strict. For a map $I\to J$ and the corresponding diagonal $\vartriangle: X^J\to X^I$ the canonical natural transformation
$$
(v^{\theta}_{P, J})_!\vartriangle^!\to \vartriangle^!(v^{\theta}_{P, I})_!, \;\;\; \SI^{\theta}_{P, I}\to \SI^{\le 0}_{P, J}
$$
is an equivalence.
\end{Cor}
\begin{proof}
This is analogous to (\cite{Gai19Ran}, 1.5.6). Namely, this follows from Lemma~\ref{Lm_1.4.6} using for example (\cite{Ly9}, Lemma~1.8.15, 1.8.16).  
\end{proof}

 Passing to the limit over $I\in fSets$, we obtain the following.
 
\begin{Cor} 
\label{Cor_1.4.8}
Let $-\theta\in\Lambda^{pos}_{G,P}$. \\
i) The left adjoint of 
$
(v^{\theta}_{P, \Ran})_*: \SI_{P,\Ran}^{\theta}\to \SI_{P,\Ran}^{\le 0}
$ 
is defined and denoted 
\begin{equation}
\label{functor_(v^theta_P,Ran)^*}
(v^{\theta}_{P, \Ran})^*: \SI_{P,\Ran}^{\le 0}\to \SI_{P,\Ran}^{\theta}.
\end{equation} 
The natural left-lax $Shv(\Ran)$-structure on $(v^{\theta}_{P, \Ran})^*$ is strict. 

\medskip\noindent
ii) The left adjoint of $(v^{\theta}_{P, \Ran})^!: \SI^{\le 0}_{P, \Ran}\to \SI^{\theta}_{P, \Ran}$ is defined and denoted 
$$
(v^{\theta}_{P, \Ran})_!: \SI^{\theta}_{P, \Ran}\to \SI^{\le 0}_{P, \Ran}.
$$
The natural left-lax $Shv(\Ran)$-structure on $(v^{\theta}_{P, \Ran})_!$ is strict. \QED
\end{Cor} 

\sssec{} 
\label{Sect_1.4.9_now}
For $-\theta\in\Lambda_{G,P}^{pos}$ the $H_{\Ran}$-action on $\ov{\Gr}^{\theta}_{P,\Ran}$ preserves the subschemes $\bar S^{\theta}_{P,\Ran}, S^{\theta}_{P,\Ran}$. For $I\in fSets$ we get the corresponding versions of the semi-infinite categories
$$
\SI^{\le\theta}_{P,I}(S)=Shv(\bar S^{\theta}_{P, I})^{H_I},\;\;\;\; \SI^{\theta}_{P, I}(S)=Shv(S^{\theta}_{P, I})^{H_I}
$$
and their Ran versions
$$
\SI^{\le\theta}_{P,\Ran}(S)=\underset{I\in fSets}{\lim} \SI^{\le\theta}_{P,I}(S),\;\;\;\; \SI^{\theta}_{P, \Ran}(S)=\underset{I\in fSets}{\lim} \SI^{\theta}_{P,I}(S)
$$

 We often view $\SI^{\le\theta}_{P,I}(S)$ (resp., $\SI^{\theta}_{P, I}(S)$) as a full subcategory of $\SI^{\le\theta}_{P,I}$ (resp., of $\SI^{\theta}_{P, I}$) via extension by zero under $\bar S^{\theta}_{P,I}\hook{} \ov{\Gr}^{\theta}_{P, I}$ (resp., under $S^{\theta}_{P, I}\hook{} \Gr^{\theta}_{P, I}$). 
 
   All the functors introduced in Sections~\ref{Sect_1.4.2}-\ref{Cor_1.4.8} preserve the corresponding $S$-versions of the semi-infinite categories and induce similar adjoint pairs between them. In particular, the following version of Lemma~\ref{Lm_1.4.3} is obtained similarly.
   
\index{$\SI^{\le\theta}_{P,I}(S), \SI^{\theta}_{P, I}(S), \SI^{\le\theta}_{P,\Ran}(S), \SI^{\theta}_{P, \Ran}(S)$, Section~\ref{Sect_1.4.9_now}}   
   
\begin{Lm} 
\label{Lm_1.4.10_now}
Let $-\theta\in\Lambda^{pos}_{G,P}$, $I\in fSets$. \\
i) The functor
\begin{equation}
\label{functor_for_Lm_1.4.10}
(\gt^{\theta}_{S, I})^!: Shv(\Mod_{M, I}^{-,\theta,\subset})^{\gL^+(M)_I}\to Shv(S^{\theta}_{P,I})^{\gL^+(M)_I}
\end{equation}
is fully faithful, its essential image is the full subcategory $\SI^{\theta}_{P,I}(S)$. The composition
$$
\SI^{\theta}_{P,I}(S)\hook{} Shv(S^{\theta}_{P, I})^{\gL^+(M)_I} \toup{(i^{\theta}_{S, I})^!} Shv(\Mod_{M, I}^{-,\theta,\subset})^{\gL^+(M)_I}
$$
is an equivalence. \\
ii) The partially defined left adjoint $(\gt^{\theta}_{S, I})_!$ to (\ref{functor_for_Lm_1.4.10})
is defined on the full subcategory $\SI^{\theta}_{P,I}(S)$, and $(\gt^{\theta}_{S, I})_!\,\iso\,(i^{\theta}_{S, I})^!$ as functors 
$$
\SI^{\theta}_{P, I}(S)\to Shv(\Mod_{M, I}^{-,\theta,\subset})^{\gL^+(M)_I}
$$

 The partially defined left adjoint $(i^{\theta}_{S, I})^*$ to
$$
(i^{\theta}_{S, I})_*:  Shv(\Mod_{M, I}^{-,\theta,\subset})^{\gL^+(M)_I}\to Shv(S^{\theta}_{P, I})^{\gL^+(M)_I} 
$$
is defined on the full subcategory $\SI^{\theta}_{P, I}(S)$, and one has canonically
$(\gt^{\theta}_{S, I})_*\,\iso\, (i^{\theta}_{S, I})^*$ as functors
$$
\SI^{\theta}_{P, I}(S)\to Shv(\Mod_{M, I}^{-,\theta,\subset})^{\gL^+(M)_I}.
$$
\QED
\end{Lm}   

\begin{Cor} 
\label{Cor_1.4.11_now_adjoints}
Let $\theta\in -\Lambda^{pos}_{G,P}$. \\
i) The left adjoint to $(v^{\theta}_{S,\Ran})_*: \SI^{\theta}_{P, \Ran}(S)\to \SI^{\le 0}_{P,\Ran}(S)$ is defined and denoted 
\begin{equation}
\label{functor_(v^theta_S,Ran)^*}
(v^{\theta}_{S,\Ran})^*: \SI^{\le 0}_{P,\Ran}(S)\to \SI^{\theta}_{P, \Ran}(S).
\end{equation} 
The natural left-lax $Shv(\Ran)$-structure on $(v^{\theta}_{S,\Ran})^*$ is strict.

\smallskip\noindent
ii) The left adjoint to $(v^{\theta}_{S,\Ran})^!: \SI^{\le 0}_{P,\Ran}(S)\to \SI^{\theta}_{P,\Ran}(S)$ is defined and denoted
\begin{equation}
\label{functor_(v^theta_S,Ran)_!}
(v^{\theta}_{S,\Ran})_!: \SI^{\theta}_{P,\Ran}(S)\to  \SI^{\le 0}_{P,\Ran}(S)
\end{equation}
The functor (\ref{functor_(v^theta_S,Ran)_!}) is fully faithful, and 
its natural left-lax $Shv(\Ran)$-structure is strict.

\index{$(v^{\theta}_{S,\Ran})^*, (v^{\theta}_{S,\Ran})_{^^21}$, Corollary~\ref{Cor_1.4.11_now_adjoints}}
\end{Cor}
\begin{proof} This follows from Corollary~\ref{Cor_1.4.8}.
\end{proof}  

\begin{Rem} i) Let $I\in fSets$. The objects $(v^{\theta}_{P, I})_! K$ for $\theta\in -\Lambda^{pos}_{G,P}, K\in \SI^{\theta}_{P, I}$ generate $\SI^{\le 0}_{P, I}$.

\smallskip\noindent
ii) The objects of the form $(v^{\theta}_{S, I})_!K$ for $\theta\in -\Lambda^{pos}_{G,P}, K\in \SI^{\theta}_{P, I}(S)$ generate $\SI^{\le 0}_{P, I}(S)$.
\end{Rem}

\begin{Lm} 
\label{Lm_1.4.13_existence}
Let $\theta\in -\Lambda^{pos}_{G,P}$, $K\in Shv(\Mod_{M, I}^{-,\theta, \subset})^{\gL^+(M)_I}$ such that 
$$
(v^{\theta}_{S, I}i^{\theta}_{S, I})_!K\in Shv(\bar S^0_{P, I})^{\gL^+(M)_I}
$$ 
is defined (this is always the case for $\theta=0$). Then $\Av_!^{\gL(U(P))_I}((v^{\theta}_{S, I}i^{\theta}_{S, I})_!K)\in \SI_{P, I}^{\le 0}(S)$ is defined, and
$$
\Av_!^{\gL(U(P))_I}((v^{\theta}_{S, I}i^{\theta}_{S, I})_!K)\,\iso\, (v^{\theta}_{S, I})_!(\gt^{\theta}_{S, I})^!K.  
$$
\end{Lm}
\begin{proof}
For $F\in \SI_{P, I}^{\le 0}(S)$ we have
\begin{multline*}
\HOM_{\SI_{P, I}^{\le 0}(S)}((v^{\theta}_{S, I})_!(\gt^{\theta}_{S, I})^!K, F)\,\iso\, \HOM_{\SI^{\theta}_{P, I}}((\gt^{\theta}_{S, I})^!K, (v^{\theta}_{S, I})^!F)\,\iso\\\HOM_{Shv(\Mod_{M, I}^{-,\theta, \subset})^{\gL^+(M)_I}}(K, (i^{\theta}_{S, I})^!(v^{\theta}_{S, I})^!F)\,\iso\, \HOM_{Shv(\bar S^0_{P, I})^{\gL^+(M)_I}}((v^{\theta}_{S, I}i^{\theta}_{S, I})_!K, F),
\end{multline*}
here the second isomorphism comes from Lemma~\ref{Lm_1.4.10_now}. 
\end{proof}

\ssec{Relation with the unitality}
\sssec{} 
\label{Sect_1.4.12}
Let $-\theta\in\Lambda_{G,P}^{pos}$. Recall the maps 
$$
\Gr_{P,\Ran}^{\theta}\hook{j^{\theta}_{P,\Ran}} \ov{\Gr}^{\theta}_{P,\Ran}\toup{\bar v^{\theta}_{P,\Ran}} \ov{\Gr}^0_{P,\Ran}\getsup{\eta_{\Ran}}\bar S^0_{P,\Ran}.
$$
From Remark~\ref{Rem_A.1.7} we conclude that the functors 
\begin{equation}
\label{functors_for_Sect_1.4.12}
(j^{\theta}_{P,\Ran})^!, (j^{\theta}_{P,\Ran})_*, (\bar v^{\theta}_{P,\Ran})^!, (\bar v^{\theta}_{P,\Ran})_*, \eta_{\Ran}^!, (\eta_{\Ran})_*
\end{equation}
preserve the corresponding unital categories of sheaves.  

\begin{Lm} 
\label{Lm_1.4.12}
An object $K\in Shv(\ov{\Gr}^0_{P,\Ran})$ lies in the full subcategory $Shv(\ov{\Gr}^0_{P,\Ran})_{untl}$ iff for any $\theta\in -\Lambda^{pos}_{G,P}$, $(v^{\theta}_{P,\Ran})^!K\in Shv(\Gr^{\theta}_{P,\Ran})_{untl}$. 
\end{Lm}
\begin{proof} Let us prove the if direction. Equip $-\Lambda_{G,P}^{pos}$ with the order so that $\theta_1\le\theta_2$ iff $\theta_2-\theta_1\in \Lambda_{G,P}^{pos}$. Let $\{\theta_i\}$ be a finite collection of elements of $-\Lambda_{G,P}^{pos}$ such that $-\Lambda^{pos}_{G,P}-\cup_i \{\theta\in -\Lambda^{pos}_{G,P}\mid \theta\le \theta_i\}$ is finite. Let $\ov{\Gr}_{P,\Ran}^{\{\theta_i\}}$ be the image of $\sqcup_i \ov{\Gr}_{P,\Ran}^{\theta_i}\to \ov{\Gr}_{P,\Ran}^0$. Let $U^{\{\theta_i\}}=\ov{\Gr}_{P,\Ran}^0-\ov{\Gr}_{P,\Ran}^{\{\theta_i\}}$, it is open in $\ov{\Gr}_{P,\Ran}^0$. By our assumption, $K\in Shv(U^{\{\theta_i\}})_{untl}$ for any such finite collection.

 Our claim follows now from the Zariski descent for $Shv$. Namely, consider the diagram
$$
\ov{\Gr}^0_{P,\Ran}\getsup{\phi_s} \ov{\Gr}^0_{P,\Ran}\times_{\Ran} (\Ran\times\Ran)^{\subset}\toup{\phi_b} \ov{\Gr}^0_{P,\Ran}
$$
 For any collection $\{\theta_i\}$ as above, $\phi_b^!K$ over $U^{\{\theta_i\}}\times_{\Ran}(\Ran\times\Ran)^{\subset}$ is of the form $\phi_s^!K_{\{\theta_i\}}$ for some $K_{\{\theta_i\}}\in Shv(U^{\{\theta_i\}})$. We have
$$
Shv(\ov{\Gr}^0_{P,\Ran})\,\iso\, \lim Shv(U^{\{\theta_i\}})
$$ 
taken over the collections $\{\theta_i\}$ as above. The objects $K_{\{\theta_i\}}$ organize naturally into an object $K'$ of the latter limit equipped with $\phi_b^! K\,\iso\, \phi_s^! K'$. We are done.
\end{proof}

\sssec{} 
\label{Sect_1.4.13}
For $\theta\in -\Lambda^{pos}_{G,P}$ the functors $(\eta^{\theta}_{\Ran})^!, (\eta^{\theta}_{\Ran})_*$ also preserve the corresponding unital categories of sheaves by Remark~\ref{Rem_A.1.7}.

\begin{Cor} 
\label{Cor_1.4.15_now}
An object $K\in Shv(\bar S^0_{P,\Ran})$ lies in the full subcategory $Shv(\bar S^0_{P,\Ran})_{untl}$ iff for any $\theta\in -\Lambda^{pos}_{G,P}$, $(v^{\theta}_{S,\Ran})^!K\in Shv(S^{\theta}_{P,\Ran})_{untl}$.
\end{Cor}
\begin{proof} Combine Section~\ref{Sect_1.4.13} and Lemma~\ref{Lm_1.4.12}.
\end{proof}

\sssec{Unital categories} 
\label{Sect_1.5.5_now}
We define the unital versions of the corresponding category
$$
\SI_{P,\Ran, untl},\;\; 
\SI^{\le\theta}_{P,\Ran, untl}(S), \;\; \SI^{\theta}_{P, \Ran, untl}(S), \;\; Shv(\Mod^{-,\theta, \subset}_{M, \Ran})^{\gL^+(M)_{\Ran}}_{untl}
$$ 
as the preimage of the corresponding unital category of sheaves under $\oblv$. Namely, $Shv(\Mod^{-,\theta, \subset}_{M, \Ran})^{\gL^+(M)_{\Ran}}_{untl}$ is the preimage of $Shv(\Mod^{-,\theta, \subset}_{M, \Ran})_{untl}$ under 
$$
\oblv: Shv(\Mod^{-,\theta, \subset}_{M, \Ran})^{\gL^+(M)_{\Ran}}\to Shv(\Mod^{-,\theta, \subset}_{M, \Ran}).
$$ 
The category $\SI^{\theta}_{P, \Ran, untl}(S)$ is defined as the preimage of $Shv(S^{\theta}_{P,\Ran})_{untl}$ under
$$
\oblv: \SI^{\theta}_{P, \Ran}(S)\to Shv(S^{\theta}_{P,\Ran})
$$
The category $\SI_{P,\Ran, untl}$ is defined as the preimage of $Shv(\Gr_{G,\Ran})_{untl}$ under 
$$
\oblv: \SI_{P,\Ran}\to Shv(\Gr_{G,\Ran})
$$ 
The category $\SI^{\le\theta}_{P, \Ran, untl}(S)$ is defined as the preimage of $Shv(\bar S_{P,\Ran}^{\theta})_{untl}$ under
$$
\oblv: \SI^{\le\theta}_{P,\Ran}(S)\to Shv(\bar S^{\theta}_{P,\Ran}).
$$
One similarly defines
$$
\SI^{\le\theta}_{P,\Ran, untl},\;\; \SI^{\theta}_{P,\Ran, untl}, \;\; Shv(\Gr^{\theta}_{M,\Ran, +})^{\gL^+(M)_{\Ran}}_{untl}
$$

\index{$\SI_{P,\Ran, untl}, \SI^{\le\theta}_{P,\Ran, untl}(S), \SI^{\theta}_{P, \Ran, untl}(S)$, Section~\ref{Sect_1.5.5_now}}
\index{$Shv(\Mod^{-,\theta, \subset}_{M, \Ran})^{\gL^+(M)_{\Ran}}_{untl}$, Section~\ref{Sect_1.5.5_now}}
\index{$\SI^{\le\theta}_{P,\Ran, untl}, \SI^{\theta}_{P,\Ran, untl}, Shv(\Gr^{\theta}_{M,\Ran, +})^{\gL^+(M)_{\Ran}}_{untl}$, Section~\ref{Sect_1.5.5_now}}

\begin{Rem} One may equivalently define $Shv(\Mod^{-,\theta, \subset}_{M, \Ran})^{\gL^+(M)_{\Ran}}_{untl}$ as the preimage of $Shv(\Mod^{-,\theta, \subset}_{M, \Ran})_{untl}$ under 
$$
\oblv[\dimrel]: Shv(\Mod^{-,\theta, \subset}_{M, \Ran})^{\gL^+(M)_{\Ran}}\to Shv(\Mod^{-,\theta, \subset}_{M, \Ran}).
$$ 
\end{Rem}

\sssec{} From Section~\ref{Sect_1.4.12} we see that the functors (\ref{functors_for_Sect_1.4.12}) preserve the corresponding semi-infinite unital categories of sheaves. Using Section~\ref{Sect_1.4.13}, for $\theta\in -\Lambda^{pos}_{G,P}$ this gives the adjoint pairs
$$
(\eta_{\Ran})_!: \SI^{\le 0}_{P,\Ran, untl}(S)
\leftrightarrows \SI^{\le 0}_{P,\Ran, untl}: \eta_{\Ran}^!, 
$$
$$
(\eta^{\theta}_{\Ran})_!: \SI^{\theta}_{P,\Ran, untl}(S)\leftrightarrows \SI^{\theta}_{P,\Ran, untl}: (\eta^{\theta}_{\Ran})^!,
$$
and
$$
(j^{\theta}_{P,\Ran})^*: \SI^{\le\theta}_{P,\Ran, untl}\leftrightarrows \SI^{\theta}_{P,\Ran, untl}: (j^{\theta}_{P,\Ran})_*,
$$
$$
(\bar v^{\theta}_{P,\Ran})_!: \SI^{\le \theta}_{P,\Ran, untl}\leftrightarrows \SI^{\le 0}_{P,\Ran, untl}: (\bar v^{\theta}_{P,\Ran})^!.
$$
In the above $(j^{\theta}_{P,\Ran})_*$ is fully faithful.
 
\sssec{} From Lemma~\ref{Lm_1.4.12} we immediately derive the following. 

\begin{Cor} 
\label{Cor_1.5.9_now}
Let $K\in \SI^{\le 0}_{P,\Ran}$. Then $K\in \SI^{\le 0}_{P,\Ran, untl}$ iff for any $\theta\in -\Lambda_{G,P}^{pos}$, $(v^{\theta}_{P,\Ran})^!K\in \SI^{\theta}_{P,\Ran, untl}$. \QED
\end{Cor} 

 Similarly, we get the following version of Corollary~\ref{Cor_1.4.15_now}.
 
\begin{Cor} 
\label{Cor_1.5.10_now}
Let $K\in \SI^{\le 0}_{P,\Ran}(S)$. Then $K\in \SI^{\le 0}_{P,\Ran, untl}(S)$ iff for any $\theta\in -\Lambda_{G,P}^{pos}$, $(v^{\theta}_{S,\Ran})^!K\in \SI^{\theta}_{P,\Ran, untl}(S)$. \QED
\end{Cor} 

From Lemma~\ref{Lm_1.4.3} one immediately derives the following.
\begin{Cor}
\label{Cor_1.5.11_now}
Let $\theta\in -\Lambda^{pos}_{G,P}$. The functor
$$
(\gt^{\theta}_{P,\Ran})^!: Shv(\Gr^{\theta}_{M,\Ran, +})^{\gL^+(M)_{\Ran}}\to Shv(\Gr^{\theta}_{P,\Ran})^{\gL^+(M)_{\Ran}}
$$
is fully faithful, its essential image is the full subcategory $\SI^{\theta}_{P,\Ran}$. It induces an equivalence between the full subcategories $Shv(\Gr^{\theta}_{M,\Ran, +})^{\gL^+(M)_{\Ran}}_{untl}\,\iso\, \SI^{\theta}_{P,\Ran, untl}$. The composition
$$
\SI^{\theta}_{P,\Ran}\hook{} Shv(\Gr_{P,\Ran}^{\theta})^{\gL^+(M)_{\Ran}}\toup{(i^{\theta}_{P,\Ran})^!} Shv(\Gr^{\theta}_{M,\Ran,+})^{\gL^+(M)_{\Ran}}
$$
is an equivalence. It induces an equivalence between the full subcategories 
$$
\SI^{\theta}_{P,\Ran, untl}\,\iso\, Shv(\Gr^{\theta}_{M,\Ran,+})^{\gL^+(M)_{\Ran}}_{untl}.
$$
The functor 
$$
(\gt^{\theta}_{P,\Ran})_!: \SI^{\theta}_{P,\Ran}\to Shv(\Gr^{\theta}_{M,\Ran, +})^{\gL^+(M)_{\Ran}}
$$ 
preserves the corresponding full subcategories of unital objects.
\end{Cor}

\begin{Rem} We don't know if $(\gt^{\theta}_{P,\Ran})_*: \SI^{\theta}_{P,\Ran}\to Shv(\Gr^{\theta}_{M,\Ran, +})^{\gL^+(M)_{\Ran}}$ preserves the corresponding full subcategories of unital objects.
\end{Rem}

\begin{Cor} 
\label{Cor_1.5.13_now}
Let $\theta\in -\Lambda^{pos}_{G,P}$. The functor
$$
(\gt_{S,\Ran}^{\theta})^!: Shv(\Mod^{-,\theta,\subset}_{M,\Ran})^{\gL^+(M)_{\Ran}}\to Shv(S^{\theta}_{P,\Ran})^{\gL^+(M)_{\Ran}}
$$
is fully faithful, its essential image is the full subcategory $\SI^{\theta}_{P,\Ran}(S)$. It induces an equivalence between the full subcategories
$$
Shv(\Mod^{-,\theta,\subset}_{M,\Ran})^{\gL^+(M)_{\Ran}}_{untl}\,\iso\,\SI^{\theta}_{P,\Ran, untl}(S) 
$$
The composition
$$
\SI^{\theta}_{P,\Ran}(S)\to Shv(S^{\theta}_{P,\Ran})^{\gL^+(M)_{\Ran}}\toup{(i^{\theta}_{S,\Ran})^!} Shv(\Mod^{-,\theta,\subset}_{M,\Ran})^{\gL^+(M)_{\Ran}}
$$
is an equivalence. It induces an equivalence between the full subcategories
$$
\SI^{\theta}_{P,\Ran, untl}(S)\,\iso\, Shv(\Mod^{-,\theta,\subset}_{M,\Ran})^{\gL^+(M)_{\Ran}}_{untl}.
$$
The functor
$$
(\gt^{\theta}_{S,\Ran})_!: \SI^{\theta}_{P,\Ran}(S)\to Shv(\Mod^{-,\theta,\subset}_{M,\Ran})^{\gL^+(M)_{\Ran}}
$$
preserves the corresponding full subcategories of unital objects. 
\end{Cor}

\begin{Lm} 
\label{Lm_1.5.14_now}
Let $\theta\in -\Lambda^{pos}_{G,P}$. The composition 
$$
\Gr^{\theta}_{P^-,\Ran}\cap \bar S^0_{P,\Ran}\hook{} \Gr^{\theta}_{P^-,\Ran}\cap \ov{\Gr}^0_{P,\Ran}\toup{\gt^{\theta}_{P^-, \Ran}}\Gr^{\theta}_{M,\Ran, +}
$$ 
factors through the closed immersion (\ref{inclusion_Mod_M_Ran^theta^subset}).  The resulting map is still denoted
$$
\gt^{\theta}_{P^-, \Ran}: \Gr^{\theta}_{P^-,\Ran}\cap \bar S^0_{P,\Ran}\to \Mod^{-,\theta, \subset}_{M,\Ran}
$$
by abuse of notations.

\index{$\gt^{\theta}_{P^-, \Ran}$, Lemma~\ref{Lm_1.5.14_now}}
\end{Lm}  
\begin{proof} The proof is as in (\cite{BFGM}, Proposition~2.6). We give the details for the convenience of the reader. 
 
 For $S\in\Sch^{aff}$ an $S$-point of $\Gr^{\theta}_{P^-,\Ran}\cap \bar S^0_{P,\Ran}$ is a collection: an $S$-point $(D, \cI)$ of $(X^{-\theta}\times\Ran)^{\subset}$, a $P^-$-torsor $\cF_{P^-}$ on $S\times X$ with a trivialization $\beta: \cF_{P^-}\,\iso\,\cF^0_{P^-}\mid_{S\times X-\Gamma_{\cI}}$ inducing an isomorphism $\cF_{M/[M,M]}\,\iso\, \cF^0_{M/[M,M]}(D)$ on $S\times X$, and such that for any $V\in\Rep(G)^{\heartsuit}$ finite-dimensional, 
$$
V^{U(P)}_{\cF^0_M}\hook{} V_{\cF_{P^-}}
$$
is regular over $S\times X$. 

 In particular, for such a collection the composition $V^{U(P)}_{\cF^0_M}\hook{} V_{\cF_{P^-}}\to (V_{U(P^-)})_{\cF_M}$ is regular, where $\cF_M$ is the extension of scalars of $\cF_{P^-}$. The claim follows now from Lemma~\ref{Lm_1.4.17} below.
\end{proof} 
\begin{Lm}[\cite{BFGM}, Lemma~2.7] 
\label{Lm_1.4.17}
For any finite-dimensional $M$-module $\cU$ isomorphic to $\cV^{U(P)}$ for some finite-dimensional $G$-module $V$, the composition
$\cU\to (\Ind(\cU))^{U(P)}\to (\Ind(\cU))_{U(P^-)}$ is an isomorphism. Here the induced $G$-module $\Ind(\cU)$ is characterized by 
$$
\Hom_M(\cU, V^{U(P)})\,\iso\, \Hom_G(\Ind(\cU), V)
$$
for $V\in\Rep(G)^{\heartsuit}$ finite-dimensional.
\QED
\end{Lm}
\begin{Rem} If $k$ is of characteristic zero, we may simplify the above argument as follows: for any $V\in\Rep(G)^{\heartsuit}$ finite-dimensional, the composition $V^{U(P)}\to V\to V_{U(P^-)}$ is an isomorphism.
\end{Rem}

\sssec{} 
\label{Sect_1.5.17}
Let $\theta\in -\Lambda^{pos}_{G,P}$. By the above, 
$$
(v^{\theta}_{S,\Ran})_*: \SI^{\theta}_{P,\Ran}(S)\to \SI^{\le 0}_{P,\Ran}(S)
$$ 
preserves the corresponding full subcategories of unital sheaves.

By construction, the functor (\ref{functor_(v^theta_S,Ran)^*}) identifies with 
$$
(\gt^{\theta}_{S,\Ran})^!(\gt^{\theta}_{P^-, \Ran})_*(v^{\theta}_{P^-,\Ran})^!
$$ 
for the diagram
\begin{equation}
\label{diag_for_Sect_1.5.17}
S^{\theta}_{P,\Ran}\toup{\gt^{\theta}_{S,\Ran}} \;\Mod^{-,\theta, \subset}_{M,\Ran}\;\getsup{\gt^{\theta}_{P^-,\Ran}} \;\Gr^{\theta}_{P^-,\Ran}\cap \bar S^0_{P,\Ran}\;\toup{v^{\theta}_{P^-,\Ran}}\; \bar S^0_{P,\Ran}.
\end{equation}
\begin{Lm} 
\label{Lm_1.4.27}
The functor (\ref{functor_(v^theta_S,Ran)^*}) preserves the corresponding unital subcategories, 
so we get an adjoint pair
$$
(v^{\theta}_{S,\Ran})^*: \SI^{\le 0}_{P,\Ran, untl}(S)\leftrightarrows \SI^{\theta}_{P, \Ran, untl}(S): (v^{\theta}_{S,\Ran})_*.
$$
\end{Lm}
\begin{proof} Clearly, $(v^{\theta}_{P^-,\Ran})^!$ and $(\gt^{\theta}_{S,\Ran})^!$ preserve the corresponding unital subcategories of sheaves.
Let us show that $(\gt^{\theta}_{P^-,\Ran})_*: Shv(\Gr^{\theta}_{P^-,\Ran}\cap \bar S^0_{P,\Ran})\to Shv(\Mod^{-,\theta, \subset}_{M,\Ran})$ sends the corresponding unital subcategories of sheaves to one another. It suffices to show that the square is cartesian
$$
\begin{array}{ccc}
(\Gr^{\theta}_{P^-,\Ran}\cap \bar S^0_{P,\Ran})\times_{\Ran}(\Ran\times\Ran)^{\subset}& \toup{\phi_b} &\Gr^{\theta}_{P^-,\Ran}\cap \bar S^0_{P,\Ran}\\
\downarrow\lefteqn{\scriptstyle \gt^{\theta}_{P^-,\Ran}\times\id} && \downarrow\lefteqn{\scriptstyle \gt^{\theta}_{P^-,\Ran}}\\
\Mod^{-,\theta, \subset}_{M,\Ran}\times_{\Ran}(\Ran\times\Ran)^{\subset}&\toup{\phi_b} & \Mod^{-,\theta, \subset}_{M,\Ran}
\end{array}
$$
This is a local phenomenon. Namely, for $x\in X$ one has $\Gr_{U(P^-), x}\cap \bar S^0_{P, x}=\Spec k$, the unit section of the corresponding Grassmanian. Indeed, this is the claim that the Zastava space $Z^0$ attached to $0\in\Lambda^{pos}_{G,P}$ in the notations of \cite{BFGM} is a point, which is established in \select{loc.cit.}
\end{proof}

\begin{Cor} 
\label{Cor_1.5.19}
Let $\theta\in -\Lambda^{pos}_{G,P}$. Then functor (\ref{functor_(v^theta_S,Ran)_!}) preserves the corresponding full subcategories of unital sheaves.
\end{Cor}
\begin{proof}
This follows formally from Lemma~\ref{Lm_1.4.27} using (\cite{Ly9}, Lemma~1.8.15, 1.8.16).  
\end{proof}

\begin{Rem} In the case $P\ne B$ we do not know if the functor (\ref{functor_(v^theta_P,Ran)^*}) preserves the corresponding full subcategories of unital sheaves.
\end{Rem}

\begin{Cor} 
\label{Cor_1.5.21}
Let $K\in \SI^{\le 0}_{P,\Ran}(S)$. Then $K\in \SI^{\le 0}_{P,\Ran, untl}(S)$ iff for any $\theta\in -\Lambda_{G,P}^{pos}$ one has $(v^{\theta}_{S,\Ran})^*K\in \SI^{\theta}_{P,\Ran, untl}(S)$.
\end{Cor}
\begin{proof} Let us show the if direction. Using Corollary~\ref{Cor_1.5.19}, this is proved as the if direction in Lemma~\ref{Lm_1.4.12}.

The only if direction follows from Lemma~\ref{Lm_1.4.27}.
\end{proof}

\sssec{} 
\label{Sect_1.5.22_now}
For $\theta\in -\Lambda^{pos}_{G,P}$ write 
$$
p^{\theta}_{\Ran}: \Mod^{-,\theta,\subset}_{M, \Ran}\to \Mod_M^{-,\theta}
$$ 
for the projection, it is obtained by base change from $\pr^{-\theta}_{\Ran}: (X^{-\theta}\times\Ran)^{\subset}\to X^{-\theta}$. 

\index{$p^{\theta}_{\Ran}$, Section~\ref{Sect_1.5.22_now}}

\begin{Lm} 
\label{Lm_1.4.31}
For $\theta\in -\Lambda^{pos}_{G,P}$ the functor $(p^{\theta}_{\Ran})^!: Shv(\Mod_M^{-,\theta})\to Shv(\Mod^{-,\theta,\subset}_{M, \Ran})$ is fully faithful  and defines an equivalence
$$
(p^{\theta}_{\Ran})^!: Shv(\Mod_M^{-,\theta})\,\iso\, Shv(\Mod^{-,\theta,\subset}_{M, \Ran})_{untl}.
$$
\end{Lm}
\begin{proof}
The proof of (\cite{Gai19Ran}, Proposition~4.2.7) applies in our situation also.
\end{proof}

\ssec{t-structure on the semi-infinite category}
\label{Sect_t-structure_on_SI_untl}

\sssec{} For each $\theta\in -\Lambda^{pos}_{G,P}$ we equip $Shv(\Mod^{-,\theta, \subset}_{M, \Ran})^{\gL^+(M)_{\Ran}}_{untl}$ with a t-structure as follows. Its object $K$ is declared connective iff the image of $K$ under the composition
\begin{equation}
\label{comp_for_Sect_1.6.1}
Shv(\Mod^{-,\theta, \subset}_{M, \Ran})^{\gL^+(M)_{\Ran}}_{untl}\;\toup{\oblv} \; Shv(\Mod^{-,\theta, \subset}_{M, \Ran})_{untl}\;\iso\; Shv(\Mod_M^{-,\theta})
\end{equation}
lies in perverse degrees $\le 0$. Here the second equivalence is that of Lemma~\ref{Lm_1.4.31}.

 By (\cite{HA}, 1.4.4.11(1)) this defines an accessible t-structure on 
$Shv(\Mod^{-,\theta, \subset}_{M, \Ran})^{\gL^+(M)_{\Ran}}_{untl}$.  

\sssec{} 
\label{Sect_1.6.2}
For $\theta\in -\Lambda^{pos}_{G,P}$ let us construct a map
$$
\tau^{\theta}: \Mod^{-,\theta, \subset}_{M, \Ran}/\gL^+(M)_{\Ran}
\to \Mod^{-,\theta}_M/\gL^+(M)_{\theta}
$$ 

 Consider a point of the source given by a collection as in Section~\ref{Sect_1.3.27}, namely, $(D,\cI)\in (X^{-\theta}\times\Ran)^{\subset}$, $M$-torsors $\cF_M, \cF'_M$ on $\cD_{\cI}$ with an isomorphism (\ref{iso_beta_extended_for_Sect_1.3.27}) whose reduction to $M/[M,M]$ extends to an isomorphism $\cF_{M/[M,M]}(D)\,\iso\,\cF'_{M/[M,M]}$ over $\cD_{\cI}$. Besides, for $V\in \Rep(G)^{\heartsuit}$ finite-dimensional the maps (\ref{map_U(P)_invariants_for_Sect_1.3.27}) are required to be regular over $\cD_{\cI}$. 
 
 For such a point we get a natural map $\cD_D\to \cD_{\cI}$, and the preimage of $\cD_{\cI}-\supp(D)$ under this map is $\oo{\cD}_D$. In terms of the description of the target in Remark~\ref{Rem_1.3.17}, $\tau^{\theta}$ sends the above collection to: $D, \cF_M\mid_{\cD_D}, \cF'_M\mid_{\cD_D}$, the restriction
$$
\beta_D: \cF_M\,\iso\, \cF'_M\mid_{\oo{\cD}_D}
$$
of $\beta$ together with the induced isomorphism $\cF_{M/[M,M]}(D)\,\iso\,\cF'_{M/[M,M]}$ over $\cD_D$.  

\index{$\tau^{\theta}$, Section~\ref{Sect_1.6.2}}

\begin{Lm} 
\label{Lm_1.6.3}
For $\theta\in -\Lambda^{pos}_{G,P}$ the map $\tau^{\theta}$ yields a functor 
$$
(\tau^{\theta})^!: Shv(\Mod^{-,\theta}_M)^{\gL^+(M)_{\theta}}\to Shv(\Mod^{-,\theta, \subset}_{M, \Ran})^{\gL^+(M)_{\Ran}}
$$
taking values in the full subcategory $Shv(\Mod^{-,\theta, \subset}_{M, \Ran})^{\gL^+(M)_{\Ran}}_{untl}$.
\end{Lm}
\begin{proof} Pick $n\in\ZZ_+$ large enough so that the $\gL^+(M)_{\theta}$-action on $\Mod_M^{-,\theta}$ factors through the quotient group scheme $\gL^+(M)_{n,\theta}$ over $X^{-\theta}$, the latter being defined in Remark~\ref{Rem_1.3.17}. We used that $\Mod_M^{-,\theta}\in\Sch_{ft}$. By definitions, 
$$
Shv(\Mod^{-,\theta}_M)^{\gL^+(M)_{\theta}}\,\iso\, Shv(\Mod^{-,\theta}_M/\gL^+(M)_{n,\theta})
$$ 

 The map $\tau^{\theta}$ yields a morphism denoted for short
$$
\tau: \Mod^{-,\theta, \subset}_{M, \Ran}/\gL^+(M)_{\Ran}
\to \Mod^{-,\theta}_M/\gL^+(M)_{n, \theta} 
$$
Pick $I\in fSets$. Let $\tau_I$ denote the composition
$$
\Mod^{-,\theta, \subset}_{M, I}/\gL^+(M)_I\to \Mod^{-,\theta, \subset}_{M, \Ran}/\gL^+(M)_{\Ran}\toup{\tau} \Mod^{-,\theta}_M/\gL^+(M)_{n, \theta}  
$$
Let us construct the functors 
$$
\tau_I^!: Shv(\Mod^{-,\theta}_M/\gL^+(M)_{n, \theta})\to Shv(\Mod^{-,\theta, \subset}_{M, I}/\gL^+(M)_I).
$$  

 Given $I\in fSets$ there is $m>0$ large enough such that for any $S\in\Sch^{aff}$ and an $S$-point $(D,\cI)$ of $(X^{-\theta}\times X^I)^{\subset}$ the composition $D_n\to \cD_D\to \cD_{\cI}$ factors through $m\Gamma_{\cI}\hook{} \cD_{\cI}$. Recall the group scheme $\gL^+(M)_{m, I}$ over $X^I$ defined in Section~\ref{Sect_1.1.3}. The $\gL^+(M)_I$-action on $\Mod^{-,\theta, \subset}_{M, I}$ factors through $\gL^+(M)_{m, I}$. So, $\tau_I$ is the composition
$$
\Mod^{-,\theta, \subset}_{M, I}/\gL^+(M)_I \toup{a} 
\Mod^{-,\theta, \subset}_{M, I}/\gL^+(M)_{m,I}  \toup{b}  \Mod^{-,\theta}_M/\gL^+(M)_{n,\theta}.
$$ 

 Let us identify $Shv(\Mod^{-,\theta, \subset}_{M, I}/\gL^+(M)_{m,I})\,\iso\, Shv(\Mod^{-,\theta, \subset}_{M, I}/\gL^+(M)_I)$ via $a^!$. In these terms the desired functor $\tau_I^!$ is $b^!$. If $m'>m$ then we have a natural map 
$$
a_{m',m}:  \Mod^{-,\theta, \subset}_{M, I}/\gL^+(M)_{m',I}\to \Mod^{-,\theta, \subset}_{M, I}/\gL^+(M)_{m,I}
$$
It is understood that we identify here 
$$
Shv(\Mod^{-,\theta, \subset}_{M, I}/\gL^+(M)_{m,I})\iso\, Shv(\Mod^{-,\theta, \subset}_{M, I}/\gL^+(M)_{m',I})
$$
via $a_{m',m}^!$. 

 The so obtained functors $\tau_I^!$ are compatible with the !-restrictions along $X^J\to X^I$ for maps $I\to J$ in $fSets$, so give the desired functor $(\tau^{\theta})^!$. By construction and Lemma~\ref{Lm_1.4.31}, it takes values in the full subcategory of unital objects.
\end{proof}

\begin{Rem} i) For $\theta=0$ we have $\Mod_M^{-, 0}=\Spec k$, and
$$
(\tau^0)^!: \Vect\to Shv(\Ran)^{\gL^+(M)_{\Ran}}
$$ 
sends $e\in\Vect$ to the object $\omega\in Shv(\Ran)^{\gL^+(M)_{\Ran}}$. \\ 
ii) For $\theta=0$ the category of coconnective objects in $Shv(\Mod^{-,0, \subset}_{M, \Ran})^{\gL^+(M)_{\Ran}}_{untl}$ is the preimage of $\Vect^{\ge 0}$ under (\ref{comp_for_Sect_1.6.1}). 
\end{Rem}

\sssec{Question} Let us equip $Shv(\Mod^{-,\theta}_M)^{\gL^+(M)_{\theta}}$ with the perverse t-structure. Is it true that 
$$
(\tau^{\theta})^!: Shv(\Mod^{-,\theta}_M)^{\gL^+(M)_{\theta}}\to Shv(\Mod^{-,\theta, \subset}_{M, \Ran})^{\gL^+(M)_{\Ran}}_{untl}
$$ 
is t-exact? For $\theta=0$ the answer is yes. 

\sssec{} Let $\theta\in -\Lambda^{pos}_{G,P}$. We equip $\SI^{\theta}_{P,\Ran, untl}(S)$ with a t-structure as follows. We declare an object $K\in \SI^{\theta}_{P,\Ran, untl}(S)$ connective/coconnective iff 
$$
(i^{\theta}_{S,\Ran})^!K[\<\theta, 2\check{\rho}-2\check{\rho}_M\>]\in Shv(\Mod^{-,\theta,\subset}_{M,\Ran})^{\gL^+(M)_{\Ran}}_{untl}
$$
is connective/coconnective. Here we used Corollary~\ref{Cor_1.5.13_now}. 

\sssec{} We define a t-structure on $\SI^{\le 0}_{P, \Ran, untl}(S)$ as follows. The full subcategory of connective objects in $\SI^{\le 0}_{P, \Ran, untl}(S)$ is defined as the smallest full subcategory closed under colimits, closed under extensions and containing for each $\theta\in -\Lambda_{G,P}^{pos}$ the objects $(v^{\theta}_{S,\Ran})_!K$
for $K\in (\SI^{\theta}_{P,\Ran, untl}(S))^{\le 0}$. By (\cite{HA}, 1.4.4.11), this is an accessible t-structure on $\SI^{\le 0}_{P, \Ran, untl}(S)$. 

 Note that $K\in \SI^{\le 0}_{P, \Ran, untl}(S)$ is coconnective iff for any $\theta\in -\Lambda_{G,P}^{pos}$, $(v^{\theta}_{S,\Ran})^!K\in (\SI^{\theta}_{P,\Ran, untl}(S))^{\ge 0}$. 
 
\begin{Lm}
\label{Lm_1.6.8_now}
 An object $K\in \SI^{\le 0}_{P, \Ran, untl}(S)$ is connective iff for any $\theta\in -\Lambda_{G,P}^{pos}$, $(v^{\theta}_{S,\Ran})^*K\in \SI^{\theta}_{P,\Ran, untl}(S)$ is connective.
\end{Lm}
\begin{proof} The proof of (\cite{Gai19Ran}, Lemma~2.1.9) applies in our situation without changes.
\end{proof}

\sssec{} 
\label{Sect_1.6.9_now}
For $\theta\in -\Lambda_{G,P}^{pos}$ let 
$
\SI^{<\theta}_{P,\Ran, untl}(S)\subset \SI^{\le\theta}_{P,\Ran, untl}(S)
$ 
be the full $\DG$-subcategory generated by $(\bar v^{\theta-\alpha_i}_{S,\Ran})_!\SI^{\le \theta-\alpha_i}_{P,\Ran, untl}(S)$ for $i\in \cI_G-\cI_M$. The essential image of 
$$
(v^{\theta}_{S,\Ran})_*: \SI^{\theta}_{P,\Ran, untl}(S)\to \SI^{\le\theta}_{P,\Ran, untl}(S)
$$ 
is the right orthogonal to $\SI^{<\theta}_{P,\Ran, untl}(S)$. In fact,
$\SI^{<\theta}_{P,\Ran, untl}(S)$ is the kernel of
$$
(v^{\theta}_{S,\Ran})^*: \SI^{\le\theta}_{P,\Ran, untl}(S)\to \SI^{\theta}_{P,\Ran, untl}(S).
$$
\index{$\SI^{<\theta}_{P,\Ran, untl}(S)$, Section~\ref{Sect_1.6.9_now}}

\sssec{} For $\theta\in -\Lambda_{G,P}^{pos}$ the functor $(v^{\theta}_{S,\Ran})_*: \SI^{\theta}_{P,\Ran, untl}(S)\to \SI^{\le 0}_{P,\Ran, untl}(S)$ is left t-exact, and
$$
(v^{\theta}_{S,\Ran})_!: \SI^{\theta}_{P,\Ran, untl}(S)\to \SI^{\le 0}_{P,\Ran, untl}(S)
$$ 
is right t-exact. This implies immediately that 
$$
(v^{\theta}_{S,\Ran})^*: \SI^{\le 0}_{P,\Ran, untl}(S)\to  \SI^{\theta}_{P,\Ran, untl}(S)
$$ 
is right t-exact, and 
$$
(v^{\theta}_{S,\Ran})^!: \SI^{\le 0}_{P,\Ran, untl}(S)\to \SI^{\theta}_{P,\Ran, untl}(S)
$$ 
is left t-exact. 

\sssec{} 
\label{Sect_1.6.11_now}
We have $\omega\in \SI^0_{P,\Ran, untl}(S)$. Note that $(v^0_{S,\Ran})_!\omega$ is connective, $(v^0_{S,\Ran})_*\omega$ is coconnective. Let $\wt\IC_{P,\Ran}^{\frac{\infty}{2}}$ 
\index{$\wt\IC_{P,\Ran}^{\frac{\infty}{2}}$, Section~\ref{Sect_1.6.11_now}}
be the image in $\SI^{\le 0}_{P,\Ran, untl}(S)^{\heartsuit}$ of the natural map
$$
\H^0((v^0_{S,\Ran})_!\omega)\to \H^0((v^0_{S,\Ran})_*\omega).
$$
This is not an object we are looking for, as the t-structure on $\SI^0_{P,\Ran, untl}(S)$ is not completely satisfactory. We will study some version of $\wt\IC_{P,\Ran}^{\frac{\infty}{2}}$ given in Definition~\ref{Def_1.9.6}. 
 
\sssec{Example} If $P=G$ then $\SI^{\le 0}_{G,\Ran}(S)\,\iso\, Shv(\Ran)^{\gL^+(G)_{\Ran}}$ canonically. In this case $\wt\IC_{P,\Ran}^{\frac{\infty}{2}}$ is the dualizing sheaf on $\Ran$.  
 

\ssec{$\cS\cI$-versions of the semi-infinite categories}
\label{Sect_cScI-versions}

\sssec{} One may consider a version of the semi-infinite categories replacing everywhere the $H_I$ (resp., $H_{\Ran}$)-equivariance by $\gL(U(P))_I$ (resp., $\gL(U(P))_{\Ran}$)-equivariance. We convent to use for the corresponding versions the notation $\cS\cI$ instead of $\SI$. The purpose of Section~\ref{Sect_cScI-versions} is to explain that, essentially, all the results (and definitions) of Sections~\ref{Sect_Structure_of_SI^le0_P,Ran}-\ref{Sect_t-structure_on_SI_untl} hold for the $\cS\cI$-versions of the corresponding categories. 

\sssec{} 
\label{Sect_1.7.2_now}
In details, for $I\in fSets$ let 
$$
\cS\cI_{P, I}=Shv(\Gr_{G, I})^{\gL(U(P))_I},\;\;\; \cS\cI_{P,\Ran}=\underset{I\in fSets}{\lim} \cS\cI_{P, I}.
$$ 
We simply write $\cS\cI_{P,\Ran}=Shv(\Gr_{G, \Ran})^{\gL(U(P))_{\Ran}}$. If $\theta\in -\Lambda_{G,P}^{pos}$ we set
$$
\cS\cI_{P, I}^{\le \theta}=Shv(\ov{\Gr}_{P,I}^{\theta})^{\gL(U(P))_I}, \;\;\;\cS\cI^{\theta}_{P, I}=Shv(\Gr_{P,I}^{\theta})^{\gL(U(P))_I} 
$$
We also get their Ran versions
$$
\cS\cI_{P, \Ran}^{\le \theta}=\underset{I\in fSets}{\lim} \cS\cI_{P, I}^{\le \theta},\;\;\; \cS\cI^{\theta}_{P, \Ran}=\underset{I\in fSets}{\lim} \cS\cI^{\theta}_{P, I}.
$$
\index{$\cS\cI_{P,\Ran}, \cS\cI_{P, \Ran}^{\le \theta}, \cS\cI^{\theta}_{P, \Ran}$, Section~\ref{Sect_1.7.2_now}} 

\sssec{} For $I\in fSets, \und{\lambda}: I\to \Lambda^+_{M, ab}$ the functor $\oblv: \cS\cI_{P, I}\to Shv(\Gr_{G, I})$ is fully faithful. If $\und{\lambda}_i\in \Map(I, \Lambda^+_{M, ab})$ with
$\und{\lambda}_1\le \und{\lambda}_2$ then
$$
Shv(\Gr_{G, I})^{U(P)_{\und{\lambda}_2}}\subset Shv(\Gr_{G, I})^{U(P)_{\und{\lambda}_1}}
$$
is a full subcategory, and 
$$
\cS\cI_{P, I}\,\iso\, \underset{\und{\lambda}\in (\Map(I, \Lambda^+_{M, ab}), \le)^{op}}{\lim} Shv(\Gr_{G, I})^{U(P)_{\und{\lambda}}}
$$ 
coincides with 
$\underset{\und{\lambda}\in \Map(I, \Lambda^+_{M, ab})}{\cap} Shv(\Gr_{G, I})^{U(P)_{\und{\lambda}}}$ taken inside $Shv(\Gr_{G, I})$.

 For $\und{\lambda}:I\to \Lambda^+_{M, ab}$ the functor $\oblv: Shv(\Gr_{G, I})^{U(P)_{\und{\lambda}}}\to Shv(\Gr_{G, I})$ has a partially defined left adjoint $\Av^{U(P)_{\und{\lambda}}}_!$, which is everywhere defined in the constructible context by Lemma~\ref{Lm_A.1.3}. The partially defined left adjoint $\Av^{\gL(U(P))_I}_!$ to $\oblv: \cS\cI_{P, I}\to Shv(\Gr_{G, I})$ is then given by (\cite{Ly9}, 1.2.15) as 
$$
\underset{\und{\lambda}\in \Map(I, \Lambda^+_{M, ab})}{\colim} \Av^{\gL(U(P))_I}_!.
$$ 
\sssec{} 
\label{Sect_1.7.4_now}
For $\theta\in -\Lambda_{G,P}^{pos}, I\in fSets$ we set
$$
\cS\cI^{\le\theta}_{P, I}(S)=Shv(\bar S^{\theta}_{P,I})^{\gL(U(P))_I}, \;\;\; \cS\cI^{\theta}_{P, I}(S)=Shv(S^{\theta}_{P,I})^{\gL(U(P))_I}.
$$ 
We get also their Ran versions
$$
\cS\cI^{\le\theta}_{P, \Ran}(S)=\underset{I\in fSets}{\lim} Shv(\bar S^{\theta}_{P,I})^{\gL(U(P))_I}, \;\;\;  \cS\cI^{\theta}_{P, \Ran}(S)=\underset{I\in fSets}{\lim} Shv(S^{\theta}_{P,I})^{\gL(U(P))_I}.
$$
\index{$\cS\cI^{\le\theta}_{P, \Ran}(S), \cS\cI^{\theta}_{P, \Ran}(S)$, Section~\ref{Sect_1.7.4_now}}

\sssec{} We use the same notations for the functors between the $\cS\cI$-versions as for $\SI$-versions. So, for $I\in fSets$, $\theta\in -\Lambda_{G,P}^{pos}$ we have the adjoint pairs in $Shv(X^I)-mod$
$$
(j^{\theta}_{P, I})^*: \cS\cI^{\le\theta}_{P, I}\leftrightarrows \cS\cI^{\theta}_{P, I}: (j^{\theta}_{P, I})_*
$$
and
$$
(\bar v^{\theta}_{P, I})_!: \cS\cI^{\theta}_{P, I}\leftrightarrows\cS\cI^{\le 0}_{P, I}: (\bar v^{\theta}_{P, I})^!
$$
For a map $I\to J$ in $fSets$ these functors are compatible with the !-restricitons along $\vartriangle^{(I/J)}: X^J\to X^I$, so yields similar functors between the corresponding Ran versions.  

An analog of Lemma~\ref{Lm_1.4.3} holds:
\begin{Lm} Let $\theta\in -\Lambda_{G,P}^{pos}$, $I\in fSets$. \\
i) The functor
\begin{equation}
\label{(gt^theta_P,I)^!-pullback_for_cScI_version}
(\gt^{\theta}_{P, I})^!: Shv(\Gr_{M, I, +}^{\theta})\to Shv(\Gr_{P, I}^{\theta})
\end{equation}
is fully faithful, its essential image in the full subcategory $\cS\cI^{\theta}_{P, I}$. The composition
$$
\cS\cI^{\theta}_{P, I}\hook{} Shv(\Gr_{P, I}^{\theta})\toup{(i^{\theta}_{P, I})^!} Shv(\Gr_{M, I, +}^{\theta})
$$
is an equivalence.\\
ii)The partially defined left adjoint $(\gt^{\theta}_{P, I})_!$ to (\ref{(gt^theta_P,I)^!-pullback_for_cScI_version}) is defined on the full subcategory $\cS\cI^{\theta}_{P, I}$, and
$(\gt^{\theta}_{P, I})_!\,\iso\, (i^{\theta}_{P, I})^!$ canonically as functors
$$
\cS\cI^{\theta}_{P, I}\to Shv(\Gr_{M, I, +}^{\theta}).
$$
 
  The partially defined left adjoint $(i^{\theta}_{P, I})^*$ to 
$$
(i^{\theta}_{P, I})_*: Shv(\Gr_{M, I, +}^{\theta})\to Shv(\Gr_{P, I}^{\theta})
$$
is defined on the full subcategory $\cS\cI^{\theta}_{P, I}$, and one has canonically
$(i^{\theta}_{P, I})^*\,\iso\, (\gt^{\theta}_{P, I})_*$ as functors 
$$
\cS\cI^{\theta}_{P, I}\to Shv(\Gr_{M, I, +}^{\theta}).
$$
\end{Lm}

\sssec{} The analog of Lemma~\ref{Lm_1.4.6} also holds with the same proof:

\begin{Lm} Let $\theta\in -\Lambda_{G,P}^{pos}$, $I\in fSets$. The partially defined left adjoint of $(v^{\theta}_{P, I})_*: \cS\cI^{\theta}_{P, I}\to \cS\cI^{\le 0}_{P, I}$ is defined and denoted $(v^{\theta}_{P, I})^*: \cS\cI^{\le 0}_{P, I}\to \cS\cI^{\theta}_{P, I}$. The natural left-lax $Shv(X^I)$-structure on $(v^{\theta}_{P, I})^*$ is struct. For a map $I\to J$ in $fSets$ the canonical natural transformation
$$
(v^{\theta}_{P,J})^*\vartriangle^!\to \vartriangle^! (v^{\theta}_{P,I})^*, \cS\cI^{\le 0}_{P, I}\to \cS\cI^{\theta}_{P, I}
$$
is an isomorphism, here $\vartriangle: X^J\to X^I$. 
\end{Lm}
\begin{Cor} Let $\theta\in -\Lambda_{G,P}^{pos}$, $I\in fSets$. The left adjoint of $(v^{\theta}_{P,I})^!: \cS\cI^{\le 0}_{P, I}\to \cS\cI^{\theta}_{P, I}$ is defined and denoted $(v^{\theta}_{P,I})_!: \cS\cI^{\theta}_{P, I}\to \cS\cI^{\le 0}_{P, I}$. The natural left-lax $Shv(X^I)$-structure on $(v^{\theta}_{P,I})_!$ is strict. For a map $I\to J$ in $fSets$ the canonical natural transformation
$$
(v^{\theta}_{P,J})_!\vartriangle^!\to \vartriangle^! (v^{\theta}_{P,I})_!, \cS\cI^{\theta}_{P, I}\to \cS\cI^{\le 0}_{P, I}
$$
is an isomorphism, here $\vartriangle: X^J\to X^I$. 
\end{Cor}

\sssec{} The $\cS\cI$-analog of Corollary~\ref{Cor_1.4.8}, Lemma~\ref{Lm_1.4.10_now}, Corollary~\ref{Cor_1.4.11_now_adjoints} also holds. 

\begin{Rem} i) Let $I\in fSets$. The objects $(v^{\theta}_{P, I})_!K$ for $\theta\in -\Lambda_{G,P}^{pos}$ and $K\in \cS\cI^{\theta}_{P, I}$ generate $\cS\cI^{\le 0}_{P, I}$. 

\smallskip\noindent
ii) The objects of the form $(v^{\theta}_{S, I})_!K$ for $\theta\in -\Lambda_{G,P}^{pos}$ and $K\in \cS\cI^{\theta}_{P, I}(S)$ generate $\cS\cI^{\le 0}_{P, I}(S)$. 
\end{Rem}

\sssec{} The version of Lemma~\ref{Lm_1.4.13_existence} with $\SI$ replaced everywhere by $\cS\cI$ also holds with the same proof. 

\ssec{Relation with the unitality for $\cS\cI$-versions}

\sssec{} 
\label{Sect_1.8.1_now}
We define the unital versions of the $\cS\cI$-semi-infinite categories by intersecting with the full subcategories of unital objects. This way we get the categories
$$
\cS\cI_{P, Ran, untl},\; \cS\cI^{\le \theta}_{P, \Ran, untl}, \; \cS\cI^{\theta}_{P, \Ran, untl}, \;\cS\cI^{\le \theta}_{P, \Ran, untl}(S), \; \cS\cI^{\theta}_{P, \Ran, untl}(S).
$$
In more details,
$$
\cS\cI_{P, Ran, untl}=\cS\cI_{P,\Ran}\cap Shv(\Gr_{G,\Ran})_{untl},\; 
\cS\cI^{\le \theta}_{P, \Ran, untl}=\cS\cI^{\le \theta}_{P, \Ran}\cap Shv(\ov{\Gr}_{P,\Ran}^{\theta})_{untl},
$$
$$
\cS\cI^{\theta}_{P, \Ran, untl}=\cS\cI^{\theta}_{P, \Ran}\cap Shv(\Gr_{P,\Ran}^{\theta})_{untl},\; \cS\cI^{\le \theta}_{P, \Ran, untl}(S)=\cS\cI^{\le \theta}_{P, \Ran}(S)\cap Shv(\bar S^{\theta}_{P,\Ran})_{untl},
$$
$$
\cS\cI^{\theta}_{P, \Ran, untl}(S)=\cS\cI^{\theta}_{P, \Ran}(S)\cap Shv(S^{\theta}_{P,\Ran})_{untl}.
$$
\index{$\cS\cI_{P, Ran, untl}, \cS\cI^{\le \theta}_{P, \Ran, untl}, \; \cS\cI^{\theta}_{P, \Ran, untl}$, Section~\ref{Sect_1.8.1_now}}
\index{$\cS\cI^{\le \theta}_{P, \Ran, untl}(S), \; \cS\cI^{\theta}_{P, \Ran, untl}(S)$,  Section~\ref{Sect_1.8.1_now}}

\sssec{} For $\theta\in -\Lambda_{G,P}^{pos}$ the functors (\ref{functors_for_Sect_1.4.12}) and the functors $(\eta^{\theta}_{\Ran})^!, (\eta^{\theta}_{\Ran})_*$ 
preserve the corresponding $\cS\cI$-versions. We get the adjoint pairs
$$
(\eta_{\Ran})_!: \cS\cI^{\le 0}_{P,\Ran, untl}(S)
\leftrightarrows \cS\cI^{\le 0}_{P,\Ran, untl}: \eta_{\Ran}^!, 
$$
$$
(\eta^{\theta}_{\Ran})_!: \cS\cI^{\theta}_{P,\Ran, untl}(S)\leftrightarrows \cS\cI^{\theta}_{P,\Ran, untl}: (\eta^{\theta}_{\Ran})^!,
$$
and
$$
(j^{\theta}_{P,\Ran})^*: \cS\cI^{\le\theta}_{P,\Ran, untl}\leftrightarrows \cS\cI^{\theta}_{P,\Ran, untl}: (j^{\theta}_{P,\Ran})_*,
$$
$$
(\bar v^{\theta}_{P,\Ran})_!: \cS\cI^{\le \theta}_{P,\Ran, untl}\leftrightarrows \cS\cI^{\le 0}_{P,\Ran, untl}: (\bar v^{\theta}_{P,\Ran})^!.
$$
In the above $(j^{\theta}_{P,\Ran})_*$ is fully faithful.

\sssec{} The analogs of Corollaries~\ref{Cor_1.5.9_now}, \ref{Cor_1.5.10_now} hold with $\SI$ replaced by $\cS\cI$ everywhere.

 The analog of Corollary~\ref{Cor_1.5.11_now} is as follows (with the same proof).
\begin{Cor} 
Let $\theta\in -\Lambda^{pos}_{G,P}$. The functor
$$
(\gt^{\theta}_{P,\Ran})^!: Shv(\Gr^{\theta}_{M,\Ran, +})\to Shv(\Gr^{\theta}_{P,\Ran})
$$
is fully faithful, its essential image is the full subcategory $\cS\cI^{\theta}_{P,\Ran}$. It induces an equivalence between the full subcategories $Shv(\Gr^{\theta}_{M,\Ran, +})_{untl}\,\iso\, \cS\cI^{\theta}_{P,\Ran, untl}$. The composition
$$
\cS\cI^{\theta}_{P,\Ran}\hook{} Shv(\Gr_{P,\Ran}^{\theta})\toup{(i^{\theta}_{P,\Ran})^!} Shv(\Gr^{\theta}_{M,\Ran,+})
$$
is an equivalence. It induces an equivalence between the full subcategories 
$$
\cS\cI^{\theta}_{P,\Ran, untl}\,\iso\, Shv(\Gr^{\theta}_{M,\Ran,+})_{untl}.
$$
The functor 
$$
(\gt^{\theta}_{P,\Ran})_!: \cS\cI^{\theta}_{P,\Ran}\to Shv(\Gr^{\theta}_{M,\Ran, +})
$$ 
preserves the corresponding full subcategories of unital objects.
\end{Cor}

\sssec{} The analog of Corollary~\ref{Cor_1.5.13_now} also holds, it is as follows.

\begin{Cor} 
\label{Cor_2.2.5_for_cScI-versions}
Let $\theta\in -\Lambda^{pos}_{G,P}$. The functor
$$
(\gt_{S,\Ran}^{\theta})^!: Shv(\Mod^{-,\theta,\subset}_{M,\Ran})\to Shv(S^{\theta}_{P,\Ran})
$$
is fully faithful, its essential image is the full subcategory $\cS\cI^{\theta}_{P,\Ran}(S)$. It induces an equivalence between the full subcategories
$$
Shv(\Mod^{-,\theta,\subset}_{M,\Ran})_{untl}\,\iso\,\cS\cI^{\theta}_{P,\Ran, untl}(S) 
$$
The composition
$$
\cS\cI^{\theta}_{P,\Ran}(S)\hook{} Shv(S^{\theta}_{P,\Ran})\toup{(i^{\theta}_{S,\Ran})^!} Shv(\Mod^{-,\theta,\subset}_{M,\Ran})
$$
is an equivalence. It induces an equivalence between the full subcategories
$$
\cS\cI^{\theta}_{P,\Ran, untl}(S)\,\iso\, Shv(\Mod^{-,\theta,\subset}_{M,\Ran})_{untl}.
$$
The functor
$$
(\gt^{\theta}_{S,\Ran})_!: \cS\cI^{\theta}_{P,\Ran}(S)\to Shv(\Mod^{-,\theta,\subset}_{M,\Ran})
$$
preserves the corresponding full subcategories of unital objects. 
\end{Cor}

\sssec{} From Lemma~\ref{Lm_1.4.27} we get an adjoint pair in $\DGCat_{cont}$
$$
(v^{\theta}_{S,\Ran})^*: \cS\cI^{\le 0}_{P,\Ran, untl}(S)\leftrightarrows \cS\cI^{\theta}_{P, \Ran, untl}(S): (v^{\theta}_{S,\Ran})_*.
$$
As in Corollary~\ref{Cor_1.5.19}, we get the adjoint pair in $\DGCat_{cont}$
$$
(v^{\theta}_{S,\Ran})_!: \cS\cI^{\theta}_{P,\Ran, untl}(S)\leftrightarrows \cS\cI^{\le 0}_{P,\Ran, untl}(S): (v^{\theta}_{S,\Ran})^!
$$

An analog of Corollary~\ref{Cor_1.5.21} holds with $\SI$ replaced by $\cS\cI$-versions. 

\ssec{t-structure on the $\cS\cI$-version of the semi-infinite category}

\sssec{} For $\theta\in -\Lambda_{G,P}^{pos}$ we equip $Shv(\Mod^{-,\theta, \subset}_{M, \Ran})_{untl}$ with a t-structure by transferring the perverse t-structure under the equivalence of Lemma~\ref{Lm_1.4.31}. 

\sssec{} For $\theta\in -\Lambda_{G,P}^{pos}$ we equip $\cS\cI^{\theta}_{P,\Ran, untl}(S)$ with a t-structure as follows. We declare an object $K\in \cS\cI^{\theta}_{P,\Ran, untl}(S)$  connective/coconnective iff
$$
(i^{\theta}_{S,\Ran})^!K[\<\theta, 2\check{\rho}-2\check{\rho}_M\>]\in Shv(\Mod^{-,\theta, \subset}_{M, \Ran})_{untl}
$$
is connective/coconnective. Here we used Corollary~\ref{Cor_2.2.5_for_cScI-versions}. 

\sssec{} 
\label{Sect_1.9.3_now}
Define a t-structure on $\cS\cI^{\le 0}_{P,\Ran, untl}(S)$ as follows. The full subcategory of connective objects is defined as the smallest full subcategory closed under colimits, under extensions and containing for each $\theta\in -\Lambda_{G,P}^{pos}$ the objects $(v^{\theta}_{S,\Ran})_!K$ for $K\in (\cS\cI^{\theta}_{P,\Ran, untl}(S))^{\le 0}$. This is an accessible t-structure on $\cS\cI^{\le 0}_{P,\Ran, untl}(S)$. 

 Note that $K\in \cS\cI^{\le 0}_{P,\Ran, untl}(S)$ is coconnective iff for any $\theta\in -\Lambda_{G,P}^{pos}$, $(v^{\theta}_{S,\Ran})^!K$ is coconnective in $\cS\cI^{\theta}_{P,\Ran, untl}(S)$. 
 
\sssec{} The $\cS\cI$-analog of Lemma~\ref{Lm_1.6.8_now} of the paper holds with the same proof.  

\sssec{} For $\theta\in -\Lambda_{G,P}^{pos}$ the functor $(v^{\theta}_{S,\Ran})_*: \cS\cI^{\theta}_{P,\Ran, untl}(S)\to \cS\cI^{\le 0}_{P,\Ran, untl}(S)$ is left t-exact, and
$$
(v^{\theta}_{S,\Ran})_!: \cS\cI^{\theta}_{P,\Ran, untl}(S)\to \cS\cI^{\le 0}_{P,\Ran, untl}(S)
$$ 
is right t-exact. If $\theta\in -\Lambda_{G,P}^{pos}$ then
$$
(v^{\theta}_{S,\Ran})^*: \cS\cI^{\le 0}_{P,\Ran, untl}(S)\to  \cS\cI^{\theta}_{P,\Ran, untl}(S)
$$ 
is right t-exact, and 
$$
(v^{\theta}_{S,\Ran})^!: \cS\cI^{\le 0}_{P,\Ran, untl}(S)\to \cS\cI^{\theta}_{P,\Ran, untl}(S)
$$ 
is left t-exact. 

\begin{Def} 
\label{Def_1.9.6}
One has $\omega\in \cS\cI^0_{P,\Ran, untl}(S)$. Let $\IC^{\frac{\infty}{2}}_{P,\Ran}$ be the image in $\cS\cI^0_{P,\Ran, untl}(S)^{\heartsuit}$ of the natural map
$$
\H^0((v^0_{S,\Ran})_!\omega)\to \H^0((v^0_{S,\Ran})_*\omega).
$$
\end{Def}
 This is one of the main objects of study of this paper. 
 
\sssec{} For $\theta\in -\Lambda_{G,P}^{pos}$ we have the oblivion functor 
$$
Shv(\Mod^{-,\theta, \subset}_{M,\Ran})^{\gL^+(M)_{\Ran}}_{untl}\to Shv(\Mod^{-,\theta, \subset}_{M,\Ran})_{untl},
$$ 
which is right t-exact. In general, there seems no reason for this functor to be left t-exact, though we have this property \select{after further stratification} of $X^{-\theta}$, cf. Lemma~\ref{Lm_3.3.11_coconnective}. 

 Similarly, we have the oblivion functor $\SI^{\le 0}_{P,\Ran, untl}(S)\to \cS\cI^{\le 0}_{P,\Ran, untl}(S)$, which is right t-exact. In general, there seems no reason for this functor to be left t-exact.
 
\section{Drinfeld-Plucker formalism: the Ran space version}
\label{Sect_Drinfeld-Plucker formalism: the Ran space version}

\sssec{} 
\label{Sect_2.0.1}
In this section we explain how to spread the Drinfeld-Plucker formalism in its form established in \cite{DL} to the Ran space. 

We freely use some notations of \cite{DL}. In particular, recall that for $\lambda\in\Lambda^+$ we have the irreducible $\check{G}$-module $V^{\lambda}$ with highest weight $\lambda$ and highest weight vector $v^{\lambda}$. For $\mu\in\Lambda^+_M$ we have the irreducible $\check{M}$-module $U^{\mu}$ with highest weight $\mu$ and a highest weight vector $v^{\mu}$. If $\mu\in\Lambda_{M, ab}$ then we also write $e^{\mu}$ for the 1-dimensional $\check{M}$-module $U^{\mu}$. 

 The results of this section are an application of (\cite{Ly10}, Section~6). We assume the reader is familiar with the construction from \select{loc.cit.}, we also use some notations from \select{loc.cit.}
 
 In particular, for $I\in fSets$ one has the category $\Tw(I)$ defined in (\cite{Ly10}, 2.1.8). Given $C\in CAlg(\DGCat_{cont})$ we consider $C(X)=C\otimes Shv(X)\in CAlg(Shv(X)-mod)$. One then has the object $(C_{X^I}, \otimes^!)\in CAlg(Shv(X^I)-mod)$ defined in (\cite{Ly10}, 2.5.1) and $(\Fact(C),\otimes^!)\in CAlg(Shv(\Ran)-mod)$ defined in (\cite{Ly10}, 2.5.3). We sometimes write $C_I=C_{X^I}$ for short.  
 
  Given $A\in CAlg^{nu}(C(X))$, one associates to it a commutative factorization algebra $\Fact(A)\in\Fact(C)$ as in (\cite{Ly10}, 2.4) as well as its $!$-restriction $A_{X^I}\in C_{X^I}$ under $X^I\to \Ran$. Recall that $A_{X^I}\in CAlg^{nu}(C_{X^I}, \otimes^!)$ by (\cite{Ly10}, 2.4.12). If moreover we have $A\in CAlg(C(X), \otimes^!)$ then $A_{X^I}\in CAlg(C_{X^I}, \otimes^!)$ by (\cite{Ly10}, 2.5.4). The situation for $\Fact(A)\in (\Fact(C),\otimes^!)$ is similar.  
 
\index{$V^{\lambda}, v^{\lambda}, U^{\mu}, e^{\mu}, \Tw(I), (C_{X^I}, \otimes^{^^21})$, Section~\ref{Sect_2.0.1}} 
\index{$(\Fact(C),\otimes^{^^21}), C_I, \Fact(A), A_{X^I}$, Section~\ref{Sect_2.0.1}} 
 
\sssec{} 
\label{Sect_2.0.2_now}
Write $\Fun^{rlaxu}(\Lambda^+, \Rep(\check{G}))$ for the category of right lax symmetric monoidal functors $\Lambda^+\to \Rep(\check{G})$, they preserve units. This category is naturally symmetric monoidal (\cite{Ly10}, 6.1.15).   
 
Let $\Irr: \Lambda^+\to \Rep(\check{G})$ be the functor $\lambda\mapsto V^{\lambda}$. We view it as right-lax symmetric monoidal via the maps $u^{\lambda_1,\lambda_2}: V^{\lambda_1}\otimes V^{\lambda_2}\to V^{\lambda_1+\lambda_2}$ defined in (\cite{DL}, 2.1.2). Let $\Irr^*: \Lambda^+\to  \Rep(\check{G})$ be the functor $\lambda\mapsto (V^{\lambda})^*$. We view it as right-lax symmetric monoidal via the maps $v^{\lambda_1,\lambda_2}: (V^{\lambda_1})^*\otimes (V^{\lambda_2})^*\to (V^{\lambda_1+\lambda_2})^*$ defined in \select{loc.cit.} The unit of the usual duality
\begin{equation}
\label{unit_duality_V_lambda_and_its_dual}
e\to V^{\lambda}\otimes (V^{\lambda})^*
\end{equation}
is functorial in $\lambda\in\Lambda$, and defines a morphism in $\Fun^{rlaxu}(\Lambda^+, \Rep(\check{G}))$. However, the corresponding counit map is not a morphism of right-lax monoidal functors, and $\Irr, \Irr^*$ are not dualizable in the symmetric monoidal category $\Fun^{rlaxu}(\Lambda^+, \Rep(\check{G})$ (unless $G$ is commutative).  

\index{$\Fun^{rlax}, \Irr, \Irr^*$, Section~\ref{Sect_2.0.2_now}}

\ssec{Case of $\Bunb_P$} 
\label{Sect_case_of_Bunb_P}

\sssec{Version for $\check{P}^-$} 
\label{Sect_2.1.1_Bunb_P}
Let $I\in fSets$. Let $C=\Rep(\check{M}_{ab})\otimes \Rep(\check{G})\in CAlg(\DGCat_{cont})$. 

 As in (\cite{DL}, 2.2.3), consider the diagram
\begin{equation}
\label{diag_from_DL_for_Sect_2.1.1}
\begin{array}{ccccc}
\check{G}\backslash(\overline{\check{G}/[\check{P}^-, \check{P}^-]})/\check{M}_{ab} & \getsup{j_{M, ab}} &
B(\check{P}^-) & \getsup{\eta} & B(\check{M})\\
& \searrow\lefteqn{\scriptstyle \bar q_{ab}} &
\downarrow\lefteqn{\scriptstyle q_{ab}} & \swarrow\lefteqn{\scriptstyle q_{M, ab}}\\
&& B(\check{G}\times \check{M}_{ab})
\end{array}
\end{equation}
obtained from the diagonal map $\check{P}^-\to \check{G}\times \check{M}_{ab}$. Here $j_{M, ab}$ is obtained by passing to the stack quotient under $\check{G}\times \check{M}_{ab}$ in
$$
\check{G}/[\check{P}^-, \check{P}^-]\hook{} \overline{\check{G}/[\check{P}^-, \check{P}^-]}.
$$
In Sections~\ref{Sect_2.1.1_Bunb_P}-\ref{Pp_2.1.5} we explain how to spread over $X^I$ the colimit formula for $\eta^*j_{M, ab}^*$ from (\cite{DL}, 2.2.13). 

\index{$q_{ab}, \bar q_{ab}, j_{M, ab}, q_{M, ab}, \eta$, Section~\ref{Sect_2.1.1_Bunb_P}}

\sssec{}
\label{Sect_2.1.2_Bunb_P}
The base change of $\eta$ by $\Spec k\to B(\check{G}\times \check{M}_{ab})$ is the map 
$$
\bar\eta_{ab}: \check{G}/[\check{M}, \check{M}]\to \check{G}/[\check{P}^-, \check{P}^-].
$$ 
It induces a map $A\to B$ in $CAlg(C)$, here $A=\cO(\check{G}/[\check{P}^-, \check{P}^-]), B=\cO(\check{G}/[\check{M}, \check{M}])
 \in CAlg(C)$ are the corresponding spaces of regular functions.  Set $E=A-mod(C)$.
 
The diagram (\ref{diag_from_DL_for_Sect_2.1.1}) yields equivalences 
$$
\QCoh(\check{G}\backslash(\overline{\check{G}/[\check{P}^-, \check{P}^-]})/\check{M}_{ab})
\,\iso\, E\;\;\;\mbox{and}\;\;\;\; \Rep(\check{M})\,\iso\, B-mod(C)
$$ 
by (\cite{DL}, 2.2.5). By Lemma~\ref{Lm_A.5.2}, the maps $q_{ab}, q_{M, ab}$ yield canonical equivalences 
$$
\Rep(\check{P}^-)_I\,\iso\, (q_{ab, *}\cO)_{X^I}-mod(C_{X^I})\;\;\;\mbox{and}\;\;\;\;\Rep(\check{M})_I\,\iso\, B_{X^I}-mod(C_{X^I}).
$$ 
By Section~\ref{Sect_A.5.1_now}, we have canonically
$$
E_{X^I}\,\iso\, A_{X^I}-mod(C_{X^I}). 
$$

\index{$\bar\eta_{ab}$, Section~\ref{Sect_2.1.2_Bunb_P}}
  
\sssec{} 
\label{Sect_2.1.3_Bunb_P}
Let $\und{\lambda}\in \Map(I, \Lambda_{M, ab})$. The functor 
$$
\Lambda_{M, ab}\to \Rep(\check{M}_{ab}), \; \lambda\mapsto e^{-\lambda}
$$ 
is symmetric monoidal. So, the construction of (\cite{Ly10}, 6.1.2) associates to it the object $e^{-\und{\lambda}}\in \Rep(\check{M}_{ab})_I$. The latter objects is dualizable in $(\Rep(\check{M}_{ab})_I,  \otimes^!)$ by (\cite{Ly10}, 6.1.16), and the dual is denoted $e^{\und{\lambda}}$. In fact, $e^{\und{\lambda}}$ is obtained by the same construction out of the symmetric monoidal functor
$$
\Lambda_{M, ab}\to \Rep(\check{M}_{ab}), \; \lambda\mapsto e^{\lambda}.
$$    
 
 Assume $\und{\lambda}\in \Map(I, \Lambda^+_{M, ab})$. The functor 
$$
\Lambda^+_{M, ab}\to \Rep(\check{G}), \; \lambda\mapsto V^{\lambda}
$$ 
is right-lax symmetric monoidal. So, the construction of (\cite{Ly10}, 6.1.2) associates to it the object $V^{\und{\lambda}}\in \Rep(\check{G})_I$. Similarly, the right-lax symmetric monoidal functor
$$
\Lambda^+_{M, ab}\to \Rep(\check{G}), \; \lambda\mapsto (V^{\lambda})^*
$$
yields via (\cite{Ly10}, 6.1.2) the object denoted $(V^{\und{\lambda}})^*\in \Rep(\check{G})_I$. By functoriality, (\ref{unit_duality_V_lambda_and_its_dual})
yields the unit map in $\Rep(\check{G})_I$
$$
1_{\Rep(\check{G})_I}\to V^{\und{\lambda}}\otimes^! (V^{\und{\lambda}})^*.
$$

 We warn the reader that $V^{\und{\lambda}}$ is not dualizable in $\Rep(\check{G})_I$ in general. Consider also the functor 
$$
\Irr_C: \Lambda^+_{M, ab}\to C, \; \lambda\mapsto e^{-\lambda}\otimes V^{\lambda}.
$$ 
It is naturally right-lax symmetric monoidal. So, the construction of (\cite{Ly10}, 6.1.2) associates to it the object 
\begin{equation}
\label{object_for_2.1}
\underset{\Tw(I)}{\colim} \; \cF_{\und{\lambda}, \Irr_C}\in C_{X^I}\,\iso\, \Rep(\check{M}_{ab})_I\otimes_{Shv(X^I)} \Rep(\check{G})_I.
\end{equation} 
Here $\cF_{\und{\lambda}, \Irr_C}: \Tw(I)\to C_{X^I}$ is the functor defined in (\cite{Ly10}, 6.1.2). It sends $(I\to J\to K)\in \Tw(I)$ to the image of 
$$
\omega_{X^K}\otimes (\underset{j\in J}{\otimes} (e^{-\lambda_j}\otimes V^{\lambda_j}))
$$
under the structure map $Shv(X^K)\otimes C^{\otimes J}\to C_{X^I}$. Here for $j\in J$, $\lambda_j=\sum_{i\in I_j} \und{\lambda}(i)$. 

As in (\cite{Ly10}, 6.1.15), one shows that (\ref{object_for_2.1}) identifies canonically with 
$$
e^{-\und{\lambda}}\otimes V^{\und{\lambda}}\in \Rep(\check{M}_{ab})_I\otimes_{Shv(X^I)} \Rep(\check{G})_I.
$$ 

\index{$e^{\und{\lambda}}, V^{\und{\lambda}}, (V^{\und{\lambda}})^*, \Irr_C, \cF_{\und{\lambda}, \Irr_C}$, Section~\ref{Sect_2.1.3_Bunb_P}}

\sssec{} 
\label{Sect_2.1.4}
Recall that 
$$
A\,\iso\, \underset{\lambda\in\Lambda^+_{M, ab}}{\oplus} e^{-\lambda}\otimes V^{\lambda},
$$
and the product in $A$ comes from the right-lax symmetric monoidal structure on the functor $\Irr_C$.

Equip $\Lambda^+_{M, ab}$ with the relation: $\lambda\le\mu$ iff $\mu-\lambda\in\Lambda^+_{M, ab}$. Then $(\Lambda^+_{M, ab}, \le)$ is a filtered category.

 Recall that (\cite{DL}, 2.2.13) gives a canonical isomorphism in $C$
$$
\underset{\lambda\in (\Lambda^+_{M, ab}, \le)}{\colim} A\otimes e^{\lambda}\otimes (V^{\lambda})^*\,\iso\, B,
$$ 
the corresponding functor being defined in \select{loc.cit.}. So, the assumptions of (\cite{Ly10}, Section~6.2) hold, and we apply (\cite{Ly10}, 6.2.12). 

 The natural map $e^{-\lambda}\otimes V^{\lambda}\to A$ is functorial in $\lambda\in \Lambda^+_{M, ab}$, it is a morphism in $\Fun^{rlaxu}(\Lambda^+_{M, ab}, C)$. So, for $\und{\lambda}\in \Map(I, \Lambda^+_{M, ab})$ it yields a natural map in $C_{X^I}$
$$
e^{-\und{\lambda}}\otimes V^{\und{\lambda}}\to A_{X^I}.
$$ 

Equip $\Map(I, \Lambda^+_{M, ab})$ with the relation: $\und{\lambda}_1\le\und{\lambda}_2$ iff there is $\und{\lambda}\in\Map(I, \Lambda^+_{M, ab})$ with $\und{\lambda}_2=\und{\lambda}_1+\und{\lambda}$. Then $(\Map(I, \Lambda^+_{M, ab}),\le)$ is a filtered category. 

\index{$\le$, Section~\ref{Sect_2.1.4}}

\sssec{} 
\label{Sect_2.1.5}
Let $D\in C_{X^I}-mod$. Now (\cite{Ly10}, 6.2.12) gives the following. For any $d\in A_{X^I}-mod(D)$ there is a well-defined functor 
$$
(\Map(I, \Lambda^+_{M, ab}),\le )\to A_{X^I}-mod(D), \; \und{\lambda}\mapsto d\ast (e^{\und{\lambda}}\otimes (V^{\und{\lambda}})^*).
$$
 
 If $\und{\lambda}, \und{\lambda}_i\in \Lambda^+_{M, ab}$ with $\und{\lambda}_2=\und{\lambda}_1+\und{\lambda}$ then the transition map
$$
d\ast (e^{\und{\lambda_1}}\otimes (V^{\und{\lambda_1}})^*)\to d\ast (e^{\und{\lambda_2}}\otimes (V^{\und{\lambda_2}})^*)
$$
is the composition
\begin{multline*}
d\ast (e^{\und{\lambda_1}}\otimes (V^{\und{\lambda_1}})^*)\to d\ast (e^{\und{\lambda_1}}\otimes (V^{\und{\lambda_1}})^*)\otimes (e^{\und{\lambda}}\otimes (V^{\und{\lambda}})^*)\otimes (e^{-\und{\lambda}}\otimes V^{\und{\lambda}})\to \\ d\ast (e^{\und{\lambda}}\otimes (V^{\und{\lambda}})^*)\otimes (e^{\und{\lambda_1}}\otimes (V^{\und{\lambda_1}})^*)\to 
d\ast (e^{\und{\lambda_2}}\otimes (V^{\und{\lambda_2}})^*)
\end{multline*}
Here the first map comes from the unit of the corresponding duality, and
the second map comes from the composition
$$
d\ast (e^{-\und{\lambda}}\otimes V^{\und{\lambda}})\to d\ast A_{X^I}\toup{\act} d,
$$
where $\act$ is the action map. The third map comes from the morphism
$$
(e^{\und{\lambda_1}}\otimes (V^{\und{\lambda_1}})^*)\otimes (e^{\und{\lambda}}\otimes (V^{\und{\lambda}})^*)\to (e^{\und{\lambda_2}}\otimes (V^{\und{\lambda_2}})^*),
$$
which manifests the right-lax symmetric monoidal structure on the functor 
$$
\Map(I, \Lambda^+_{M, ab})\to (C_{X^I}, \otimes^!), \; \und{\mu}\mapsto e^{\und{\mu}}\otimes (V^{\und{\mu}})^*
$$ 
described in (\cite{Ly10}, 6.1.12).  
 
\begin{Pp} 
\label{Pp_2.1.5}
The functor 
\begin{equation}
\label{functor_for_Pp_2.1.5}
A_{X^I}-mod(D)\to B_{X^I}-mod(D), \; d\mapsto B_{X^I}\otimes_{A_{X^I}} d
\end{equation}
identifies canonically with the functor $d\mapsto \underset{\und{\lambda}\in(\Map(I, \Lambda^+_{M, ab}), \le)}{\colim} d\ast (e^{\und{\lambda}}\otimes (V^{\und{\lambda}})^*)$.
\end{Pp} 
\begin{proof} Apply (\cite{Ly10}, 6.2.12).
\end{proof}
By Section~\ref{Sect_2.1.2_Bunb_P}, the functor (\ref{functor_for_Pp_2.1.5}) rewrites as 
$$
D\otimes_{C_{X^I}} E_{X^I}\to D\otimes_{C_{X^I}} \Rep(\check{M})_I.
$$

\sssec{Version for $\check{P}$} 
\label{Sect_2.1.7_now}
We spell the version of Proposition~\ref{Pp_2.1.5} for $\check{P}^-$ replaced by $\check{P}$.

 Let $C, B$ be as in Section~\ref{Sect_2.1.2_Bunb_P}. Set $A=\cO(\check{G}/[\check{P},\check{P}])$. We are spreading over $X^I$ the colimit formula for the pullback under 
$$
 B(\check{M})\to B(\check{P})\to \check{G}\backslash(\overline{\check{G}/[\check{P}, \check{P}]})/\check{M}_{ab}.
$$
 
  Keep notations of Section~\ref{Sect_2.1.3_Bunb_P}. In addition to $\Irr_C$ consider the functor
$$
\Irr_C^*: \Lambda^+_{M, ab}\to C, \; \lambda\mapsto e^{\lambda}\otimes (V^{\lambda})^*.
$$  
\index{$\Irr_C^*$, Section~\ref{Sect_2.1.7_now}}
Recall that 
$$
A\,\iso\, \underset{\lambda\in\Lambda^+_{M, ab}}{\oplus} e^{\lambda}\otimes (V^{\lambda})^*\in C, 
$$
and the product on $A$ comes from the right-lax symmetric monoidal structure on $\Irr_C^*$. For $\und{\lambda}\in\Map(I, \Lambda^+_{M, ab})$ we get a natural map $e^{\und{\lambda}}\otimes (V^{\und{\lambda}})^*\to A_{X^I}$ in $C_{X^I}$ as before.

 By (\cite{DL}, Section~2.2.18), one has canonically 
$$
\underset{\lambda\in (\Lambda^+_{M, ab}, \le)}{\colim} A\otimes e^{-\lambda}\otimes V^{\lambda}\,\iso\, B
$$ 
in $C$, the corresponding functor being defined in \select{loc.cit.}  
 
 Let $D\in C_{X^I}-mod$. Then for $d\in D$ there is a well-defined functor
\begin{equation}
\label{functor_for_version_checkP} 
(\Map(I, \Lambda^+_{M, ab}),\le)\to A_{X^I}-mod(D), \, \lambda\mapsto d\ast (e^{-\und{\lambda}}\otimes V^{\und{\lambda}})
\end{equation} 
 
The version of Proposition~\ref{Pp_2.1.5} in our situation claims that the functor 
\begin{equation}
\label{functor_for_version_checkP_final}
A_{X^I}-mod(D)\to B_{X^I}-mod(D), \, d\mapsto B_{X^I}\otimes_{A_{X^I}} d
\end{equation} 
identifies canonically with 
$$
d\mapsto  \underset{\und{\lambda}\in(\Map(I, \Lambda^+_{M, ab}), \le)}{\colim} d\ast (e^{-\und{\lambda}}\otimes V^{\und{\lambda}}).
$$

 For $E=A-mod(C)$ then (\ref{functor_for_version_checkP_final}) rewrites as 
$$
D\otimes_{C_{X^I}} E_I\to D\otimes_{C_{X^I}} \Rep(\check{M})_I.
$$

\ssec{Case of $\Bunt_P$} 
\label{Sect_2.2_Bunt_P}

\sssec{}
\label{Sect_2.2.1}
Let $I\in fSets$ and $C=\Rep(\check{M})\otimes \Rep(\check{G})\in CAlg(\DGCat_{cont})$. As in (\cite{DL}, 2.1.4), consider the diagram
\begin{equation}
\label{diag_for_Sect_2.2.1}
\begin{array}{ccccc}
\check{G}\backslash (\overline{\check{G}/U(\check{P}^-)})/\check{M} & \getsup{j_M} & 
B(\check{P}^-) & \getsup{\eta} & B(\check{M})\\
& \searrow\lefteqn{\scriptstyle \bar q} &
\downarrow\lefteqn{\scriptstyle q} & \swarrow\lefteqn{\scriptstyle q_M}  \\
&& B(\check{G}\times\check{M}),
\end{array} 
\end{equation}
where the maps come from $\check{M}\to \check{P}^-\to \check{G}\times\check{M}$, the second map being the diagonal one. 

 In Sections~\ref{Sect_2.2.1}-\ref{Pp_2.2.6} we explain how to spread over $X^I$ using  $C$-actions the colimit formula for the functor $\eta^*j_M^*$ from (\cite{DL}, 2.1.11). 
 
\index{$j_M, \bar q$, Section~\ref{Sect_2.2.1}} 

\sssec{} After the base change by $\Spec k\to B(\check{G}\times\check{M})$ the map $\eta$ becomes $\bar\eta: \check{G}\to \check{G}/U(\check{P}^-)$. It induces a map $\cA\to \cB$ in $CAlg(C)$, where $\cA=\cO(\check{G}/U(\check{P}^-)), \cB=\cO(\check{G})$ are the corresponding spaces of regular functions. Set $E=\cA-mod(C)$. 

 The diagram (\ref{diag_for_Sect_2.2.1}) yields equivalences
$$
\QCoh(\check{G}\backslash (\overline{\check{G}/U(\check{P}^-)})/\check{M})\,\iso\, E,\;\;\;\; \Rep(\check{M})\,\iso\, \cB-mod(C),
$$
and $\Rep(\check{P}^-)\,\iso\, (q_*\cO)-mod(C)$ by (\cite{DL}, 2.1.4). By Lemma~\ref{Lm_A.5.2}, the maps $q, q_M$ yield canonical equivalences
$$
\Rep(\check{P}^-)_I\,\iso\, (q_*\cO)_{X^I}-mod(C_{X^I}), \;\;\;\;\; \Rep(\check{M})_I\,\iso\, \cB_{X^I}-mod(C_{X^I}).
$$ 
By Section~\ref{Sect_A.5.1_now}, we have canonically $E_I\,\iso\, \cA_{X^I}-mod(C_{X^I})$. 

\sssec{} 
\label{Sect_2.2.3_now}
Let the functors $\Lambda^+_{M, ab}\to \Rep(\check{M}), \lambda\mapsto e^{-\lambda}$ and $\Lambda^+_{M, ab}\to \Rep(\check{G}), \lambda\mapsto V^{\lambda}$ be as in Section~\ref{Sect_2.1.3_Bunb_P}. For $\und{\lambda}\in \Map(I, \Lambda^+_{M, ab})$  they produce 
$$
e^{-\und{\lambda}}\in \Rep(\check{M})_I, V^{\und{\lambda}}\in \Rep(\check{G})_I.
$$

Write $\Irr_C: \Lambda^+_{M, ab}\to C$ for the functor $\lambda\mapsto e^{-\lambda}\otimes V^{\lambda}$. It is naturally right-lax symmetric monoidal. As in Section~\ref{Sect_2.1.3_Bunb_P} it gives the object $e^{-\und{\lambda}}\otimes V^{\und{\lambda}}\in C_{X^I}$. Similarly, the functor 
$$
\Irr_C^*: \Lambda^+_{M, ab}\to C, \;\lambda\mapsto e^{\lambda}\otimes (V^{\lambda})^*
$$ 
yields $e^{\und{\lambda}}\otimes (V^{\und{\lambda}})^*\in C_{X^I}$.

\index{$\Irr_C$, Section~\ref{Sect_2.2.3_now}}

\sssec{} Equip $\Lambda^+_{M, ab}$, $\Map(I, \Lambda^+_{M, ab})$ with the relations $\le$ as in Section~\ref{Sect_2.1.4}. Recall that by (\cite{DL}, 2.1.11), there is a canonical isomorphism in $C$
$$
\underset{\lambda\in (\Lambda^+_{M, ab}, \le)}{\colim} \cA\otimes e^{\lambda}\otimes (V^{\lambda})^*\,\iso\, \cB.
$$
It comes from the canonical isomorphism $\cA\otimes_A B\,\iso\, \cB$ in $\Rep(\check{G}\times\check{M}_{ab})$ obtained in  (\cite{DL}, 2.1.12), here $A,B$ are those of Section~\ref{Sect_2.1.2_Bunb_P}. 


\sssec{} Let $D\in C_{X^I}-mod$. Now (\cite{Ly10}, 6.2.12) gives the following. For any $d\in \cA_{X^I}-mod(D)$ there is a well-defined functor
\begin{equation}
\label{functor_for_Sect_2.2.5} 
(\Map(I, \Lambda^+_{M, ab}),\le)\to \cA_{X^I}-mod(D), \; \und{\lambda}\mapsto d\ast e^{\und{\lambda}}\otimes (V^{\und{\lambda}})^* 
\end{equation}
defined along the lines of Section~\ref{Sect_2.1.5}.

\begin{Pp} 
\label{Pp_2.2.6}
The functor
\begin{equation}
\label{functor_for_Bunt_P_case_final}
\cA_{X^I}-mod(D)\to \cB_{X^I}-mod(D), \; d\mapsto \cB_{X^I}\otimes_{\cA_{X^I}} d
\end{equation}
identifies canonically with the functor 
$$
d\mapsto \underset{\und{\lambda}\in (\Map(I, \Lambda^+_{M, ab}),\le)}{\colim}  d\ast e^{\und{\lambda}}\otimes (V^{\und{\lambda}})^*.
$$
\end{Pp}  
\begin{proof} Apply (\cite{Ly10}, 6.2.12). 
\end{proof}

 Note that (\ref{functor_for_Bunt_P_case_final}) rewrites as the functor
$$
D\otimes_{C_{X^I}} E_I \to D\otimes_{C_{X^I}} \Rep(\check{M})_I
$$  
\begin{Rem} Assume $\cD\in \Fact(C)-mod$ and $D=\Gamma(X^I, \cD)$ is the category of its sections over $X^I$. Then the results of Propositions~\ref{Pp_2.1.5}, \ref{Pp_2.2.6} are compatible with $!$-restrictions along $X^J\to X^I$ for a map $I\to J$ in $fSets$.
\end{Rem}

\begin{Rem} 
\label{Rem_2.2.8}
Write $\Fact(C)-mod^{\Fact}$ for the category of factorizable sheaves of $\Fact(C)$-module categories on $\Ran$. So, for $\cD\in \Fact(C)-mod^{\Fact}$, its sections over $X^I$ is $\Gamma(X^I, \cD)\in C_{X^I}-mod$. Then both $\Fact(\cA)-mod(\cD), \Fact(\cB)-mod(\cD)$ are factorization sheaves of categiories over $\Ran$, and the functors
$$
\cA_{X^I}-mod(\Gamma(X^I, \cD))\to \cB_{X^I}-mod(\Gamma(X^I, \cD)), \; d\mapsto \cB_{X^I}\otimes_{\cA_{X^I}} d
$$
are naturally compatible with factorization.
\index{$\Fact(C)-mod^{\Fact}$, Remark~\ref{Rem_2.2.8}}
\end{Rem}

\section{Properties of the semi-infinite $\IC$-sheaf}
\label{Sect_properties_semi-infinite_IC}

In Section~\ref{Sect_properties_semi-infinite_IC} we apply the Drinfeld-Pl\"ucker formalism (in its $\Bunt_P$-version) to construct our main object of study $'\!\IC^{\frac{\infty}{2}}_{P,\Ran}$. We also introduce
in Section~\ref{Sect_Graded Satake functor} one of the main new technical ingredients of this paper, the graded Satake functors. They are used in Section~\ref{Sect_Restrictions to strata} to describe the $*$-restrictions of $'\!\IC^{\frac{\infty}{2}}_{P,\Ran}$ to the strata. 
One of the main results of Section~\ref{Sect_properties_semi-infinite_IC} is Corollary~\ref{Cor_3.3.7}, which relates $'\!\IC^{\frac{\infty}{2}}_{P,\Ran}$ and $\IC^{\frac{\infty}{2}}_{P,\Ran}$. 

\ssec{Presentation of $\IC_{P,\Ran}^{\frac{\infty}{2}}$ as a colimit}

\sssec{} Let $I\in fSets$. Recall the action of $\Sph_{M, I}$ (resp., of $\Sph_{G, I}$) on $Shv(\Gr_{G, I})^{\gL^+(M)_I}$ by convolutions on the left (resp., on the right). These actions preserve the full subcategory $\SI_{P, I}$ (cf. Remark~\ref{Rem_1.2.11} and Section~\ref{Sect_1.2.15}).

 Restricting under $\Sat_{M, I}: \Rep(\check{M})_I\to \Sph_{M, I}$ and $\Sat_{G, I}: \Rep(\check{G})_I\to \Sph_{G, I}$ we get an action of $(\Rep(\check{M})\otimes\Rep(\check{G}))_I$ on $Shv(\Gr_{G, I})^{\gL^+(M)_I}$.
 
\sssec{} 
\label{Sect_3.1.2_now}
Let $\und{\lambda}: I\to \Lambda_{M, ab}$. Consider $e^{\und{\lambda}}\in \Rep(\check{M}_{ab})_I$ defined in Section~\ref{Sect_2.1.3_Bunb_P}. It is easy to see that for $K\in Shv(\Gr_{G, I})^{\gL^+(M)_I}$ one has $\Sat_{M, I}(e^{\und{\lambda}})\ast K\,\iso\, \und{\lambda}K$, the latter being defined in Section~\ref{Sect_1.3.30}.  

\sssec{Shifted Satake action} 
\label{Sect_3.1.3_now}
For $\theta\in\Lambda_{G,P}$, $V\in\Rep(\check{M})$ let $V_{\theta}\subset V$ be the direct summand on which $Z(\check{M})$ acts by $\theta$. 
Consider the functor 
$$
Sh: \Rep(\check{M})\to \Rep(\check{M}),\;\; \underset{\theta\in\Lambda_{G,P}}{\oplus} V_{\theta}\mapsto\underset{\theta\in\Lambda_{G,P}}{\oplus} V_{\theta}[-\<\theta, 2\check{\rho}-2\check{\rho}_M\>].
$$
This is a map in $CAlg(\DGCat_{cont})$. So, for $I\in fSets$ it induces a morphism $Sh_I: \Rep(\check{M})_I\to \Rep(\check{M})_I$ in $CAlg(Shv(X^I)-mod)$ by functoriality, as well as 
$$
\Fact(\Rep(\check{M}))\to \Fact(\Rep(\check{M}))
$$ 
in $CAlg(Shv(\Ran)-mod)$. 

 Let $\Sat'_{M, I}$ be the composition $\Rep(\check{M})_I\toup{Sh_I} \Rep(\check{M})_I\toup{\Sat_{M, I}} \Sph_{M, I}$. We refer to it as the shifted Satake functor for $M$.  
 
\index{$V_{\theta}, Sh, Sh_I, \Sat'_{M, I}$, Section~\ref{Sect_3.1.3_now}} 
 
\sssec{} We need the following generalization of (\cite{DL}, Lemma~4.1.15).

\begin{Lm} 
\label{Lm_3.1.5_now}
Let $\eta, \lambda\in\Lambda^+$ with $\ov{\Gr}_G^{\lambda}\subset \ov{\Gr}_G^{\eta}$. Let $i^{\eta}_M: \ov{\Gr}_M^{\eta}\hook{} \ov{\Gr}_G^{\eta}$ be the closed immersion. The complex $(i^{\eta}_M)^!\Sat_G(V^{\lambda})$ is placed in perverse degrees $\ge \<\eta, 2\check{\rho}-2\check{\rho}_M\>$. 
\end{Lm} 
\begin{proof}
For $\lambda=\eta$ this is (\cite{DL}, 4.1.15). Assume $\lambda\ne\eta$. 
Let $\tau\in \Lambda^+_M$ with $\Gr_M^{\tau}\subset \ov{\Gr}^{\eta}_M\cap \ov{\Gr}_G^{\lambda}$. Let $\nu\in\Lambda^+$ such that $\tau\in W\lambda$, so $\Gr_M^{\tau}\subset \Gr_G^{\nu}\subset \ov{\Gr}_G^{\lambda}$. By (\cite{DL}, 4.1.15), 
$$
\Gr_G^{\nu}\cap \ov{\Gr}_M^{\eta}\ne \emptyset
$$ 
gives $\eta-\nu\in\Lambda^{pos}_M$ and $\Gr_G^{\nu}\cap \ov{\Gr}_M^{\eta}=\Gr_M^{\nu}$. Thus, $\tau=\nu\in\Lambda^+$. 

The $!$-restriction of $\Sat_G(V^{\lambda})$ to $\Gr_G^{\nu}$ is placed in perverse degrees $\ge 0$ and has smooth perverse cohomology sheaves. 

 For any bounded complex on $\Gr_G^{\nu}$ placed in perverse degrees $\ge 0$ and having smooth perverse cohomology sheaves, its $!$-restrictiion to $\Gr_M^{\nu}$ is placed in perverse degrees $\ge \codim(\Gr_M^{\nu}, \Gr_G^{\nu})=\<\nu, 2\check{\rho}-2\check{\rho}_M\>=\<\eta, 2\check{\rho}-2\check{\rho}_M\>$, because the image of $\eta-\nu$ in $\Lambda_{G,P}$ vanishes. 
\end{proof}

\sssec{} 
\label{Sect_3.1.6_now}
Let $\und{\lambda}\in \Map(I, \Lambda^+)$ and $\lambda=\sum_{i\in I} \lambda_i$. The fibre of $\Gr_{G,I}\to X^I$ over $(x_i)\in \oo{X}{}^I$ is $\prod_{i\in I} \Gr_{G, x_i}$ canonically. Consider the closed subscheme $\ov{\Gr}^{\und{\lambda}}_{\, G, I}\mid_{\oo{X}{}^I}\subset \Gr_{G,I}\mid_{\oo{X}{}^I}$ whose fibre over $(x_i)\in \oo{X}{}^I$ is $\prod_{i\in I} \ov{\Gr}^{\lambda_i}_{G, x_i}$. Write $\ov{\Gr}^{\und{\lambda}}_{\, G,I}$ for the closure of $\ov{\Gr}^{\und{\lambda}}_{\, G, I}\mid_{\oo{X}{}^I}$ in $\Gr_{G, I}$. 

\index{$\ov{\Gr}^{\und{\lambda}}_{\, G,I}, \ov{\Gr}^{\und{\lambda}}_{\, M, I}$, Section~\ref{Sect_3.1.6_now}}

 Define $\ov{\Gr}^{\und{\lambda}}_{\, M, I}$ similarly replacing $G$ by $M$. 
\begin{Lm} 
\label{Lm_3.1.7_nice_map}
Consider the closed immersion $i: \ov{\Gr}^{\und{\lambda}}_{\, M, I}\to \ov{\Gr}^{\und{\lambda}}_{\, G, I}$. \\
i) The complex $i^!\IC(\ov{\Gr}^{\und{\lambda}}_{\, G, I})$ is placed in perverse degrees $\ge \<\lambda, 2\check{\rho}-2\check{\rho}_M\>$.\\
ii) there is a canonical map 
\begin{equation}
\label{map_for_Lm_3.1.5_Step1}
\IC(\ov{\Gr}^{\und{\lambda}}_{\, M, I})[-\<\lambda, 2\check{\rho}-2\check{\rho}_M\>]\to 
\IC(\ov{\Gr}^{\und{\lambda}}_{\, G, I})
\end{equation}
in $Shv(\Gr_{G, I})^{\gL^+(M)_I}$, which after applying $i^!$ becomes an isomorphism in perverse degree $\<\lambda, 2\check{\rho}-2\check{\rho}_M\>$. 
\end{Lm}
\begin{proof}
Recall the functor $\Loc: Shv(X^I)\otimes \Rep(\check{G})^{\otimes I}\to \Rep(\check{G})_I$ from (\cite{Ly10}, 2.2.1). View $\omega_{X^I}\otimes(\underset{i\in I}{\otimes} V^{\lambda_i})$ as an object of $\Rep(\check{G})_I$, namely the image of this object under $\Loc$. Then
$$
\Sat_{G,I}(\omega_{X^I}\otimes(\underset{i\in I}{\otimes} V^{\lambda_i}))\,\iso\, \IC(\ov{\Gr}^{\und{\lambda}}_{\, G,I})[\mid I\mid].
$$

 For each map $\phi: I\to J$ in $fSets$ we have the locally closed immersion $\oo{X}{}^J\hook{} X^I$. Consider similarly $\omega_{X^J}\otimes (\underset{j\in J}{\otimes} (\underset{i\in I_j}{\otimes} V^{\lambda_i}))\in \Rep(\check{G})_J$. Using (\cite{Ly10}, 7.3.16), one gets
$$
\Sat_{G,I}(\omega_{X^I}\otimes(\underset{i\in I}{\otimes} V^{\lambda_i}))\mid_{\Gr_{G,I}\times_{X^I} \oo{X}{}^J}\,\iso\, \Sat_{G, J}(\omega_{X^J}\otimes (\underset{j\in J}{\otimes} (\underset{i\in I_j}{\otimes} V^{\lambda_i}))),
$$
where the LHS denotes the !-restriction. 

 Now (\cite{DL}, Lemma~4.1.15) combined with Lemma~\ref{Lm_3.1.5_now} show that the !-restriction of $i^!\IC(\ov{\Gr}^{\und{\lambda}}_{\, G, I})$ to each stratum $\ov{\Gr}^{\und{\lambda}}_{M,I}\times_{X^I} \oo{X}{}^J$ is placed in perverse degrees $\ge \<\lambda, 2\check{\rho}-2\check{\rho}_M\>$, and the inequality is strict unless $\phi: I\to J$ is an isomorphism. The construction of (\ref{map_for_Lm_3.1.5_Step1}) follows now from (\cite{DL}, Lemma~4.1.15) over the open stratum first, hence over the whole of $\ov{\Gr}^{\und{\lambda}}_{\, M, I}$.
\end{proof} 

\begin{Rem} 
\label{Rem_3.1.8}
Assume for a moment that $M=T$. In this case Lemma~\ref{Lm_3.1.7_nice_map} can be strengthened as follows. Let $\und{\lambda}\in\Map(I, \Lambda^+_{M, ab})$. Then the map $i: \ov{\Gr}^{\und{\lambda}}_{\, M, I}\to \ov{\Gr}^{\und{\lambda}}_{\, G, I}$ becomes the map $s^{\und{\lambda}}_G: X^I\to \ov{\Gr}^{\und{\lambda}}_{\, G, I}$ from Section~\ref{Sect_1.3.31}. In this case one has a canonical isomorphism 
$$
\IC(X^I)[-\<\lambda, 2\check{\rho}\>]\,\iso\,(s^{\und{\lambda}}_G)^!\IC(\ov{\Gr}^{\und{\lambda}}_{\, G, I})
$$
in $Shv(\Gr_{G, I})^{\gL^+(T)_I}$. Indeed, the morphism is already constructed in Lemma~\ref{Lm_3.1.7_nice_map}. One checks that its $!$-restriction to each stratum $\oo{X}{}^J\hook{} X^I$ is an isomorphism (for any map $I\to J$ in $fSets$). 
\end{Rem}

\sssec{} 
\label{Sect_3.1.4} Set $C=\Rep(\check{M})\otimes\Rep(\check{G})$. Consider the $C_{X^I}$-action on $Shv(\Gr_{G, I})^{\gL^+(M)_I}$, where the action of $\Rep(\check{G}))_I$ is as above, and the action of $\Rep(\check{M})_I$ is the shifted one. Write $\ast'$ for the action maps of the shifted action. 

 Let $\cA=\cO(\check{G}/U(\check{P}))\in CAlg(C)$. Recall the object $\delta_I\in Shv(\Gr_{G, I})^{\gL^+(M)_I}$ from Section~\ref{Sect_1.3.31}.
 
\index{$\ast'$, Section~\ref{Sect_3.1.4}} 
 
\begin{Lm} 
\label{Lm_3.1.5}
In the situation of Section~\ref{Sect_3.1.4}, $\delta_I$ is naturally promoted to an object of $\cA_{X^I}-mod(Shv(\Gr_{G, I})^{\gL^+(M)_I})$.
\end{Lm} 
\begin{proof}
Consider the functor $\cF_{I, \cA}: \Tw(I)\to C_{X^I}$ defined as in (\cite{Ly10}, 2.4.2). So, for $(I\to J\to K)\in \Tw(I)$, $\cF_{I, \cA}(I\to J\to K)$ is the image of $\cA^{\otimes J}\otimes \omega_{X^K}$ under the structure map 
$$
C^{\otimes J}\otimes Shv(X^K)\to C_{X^I}.
$$ 
By definition, $\cA_{X^I}\,\iso\,\underset{\Tw(I)}{\colim} \; \cF_{I, \cA}$. We have the natural map $(I\to J\to K)\to (I\to K\to K)$ in $\Tw(I)$. So, it suffices to construct a compatible system of maps 
$$
(\cA^{\otimes K}\otimes \omega_{X^K})\ast' \delta_I\to \delta_I
$$ 
in $Shv(\Gr_{G, I})^{\gL^+(M)_I}$ for morphisms $I\to K$ in $fSets$. The latter map rewrites as $(\cA^{\otimes K}\otimes \omega_{X^K})\ast' \delta_K\to \delta_K$, so we may assume $K=I$ up to verifying the corresponding compatibilities.

 Recall that 
$$
\cA\,\iso\, \underset{\lambda\in\Lambda^+}{\oplus} U^{\lambda}\otimes (V^{\lambda})^*.
$$
Let $\und{\lambda}\in \Map(I, \Lambda^+)$ and $\lambda=\sum_{i\in I} \lambda_i$. It suffices to define the corresponding maps
\begin{equation}
\label{map_for_Lm_3.1.5}
\Sat_{M, I}(\omega_{X^I}\otimes (\underset{i\in I}{\otimes} U^{\lambda_i}))\ast \delta_I[-\<\lambda, 2\check{\rho}-2\check{\rho}_M\>]\to \delta_I\ast \Sat_{G,I}(\omega_{X^I}\otimes(\underset{i\in I}{\otimes} V^{\lambda_i}))
\end{equation}
This is the map (\ref{map_for_Lm_3.1.5_Step1}) constructed in Lemma~\ref{Lm_3.1.7_nice_map}. The compatibilities are left to a reader. 
\end{proof}

\sssec{} So, in the situation of Section~\ref{Sect_3.1.4} we apply the Drinfeld-Plucker formalism of Section~\ref{Sect_2.2_Bunt_P} (with $\check{P}^-$ replaced by $\check{P}$) to the object $\delta_I$ and get the following. Set $\cB=\cO(\check{G})\in CAlg(C)$.   
 
\begin{Def} For $I\in fSets$ set 
$$
\IC^{\frac{\infty}{2}}_{P, I}=\underset{\und{\lambda}\in (\Map(I, \Lambda^+_{M, ab}),\le)}{\colim} 
\Sat_{M, I}(e^{-\und{\lambda}})\ast\delta_I\ast \Sat_{G, I}(V^{\und{\lambda}})[\<\lambda, 2\check{\rho}\>]
$$
in $Shv(\Gr_{G, I})^{\gL^+(M)_I}$. By Sectiion~\ref{Sect_3.1.2_now}, it rewrites as the colimit in $Shv(\Gr_{G, I})^{\gL^+(M)_I}$
\begin{equation}
\label{formula_Def_of_IC_infty/2_P_I}
\IC^{\frac{\infty}{2}}_{P, I}\,\iso\, \underset{\und{\lambda}\in (\Map(I, \Lambda^+_{M, ab}),\le)}{\colim} (-\und{\lambda}) \Sat_{G, I}(V^{\und{\lambda}})[\<\lambda, 2\check{\rho}\>]
\end{equation} 
\end{Def} 
 
\sssec{} From Proposition~\ref{Pp_2.2.6} we conclude that 
$$
\IC^{\frac{\infty}{2}}_{P, I}\,\iso\, \cB_{X^I}\otimes_{\cA_{X^I}} \delta_I\in (Shv(\Gr_{G, I})^{\gL^+(M)_I})\otimes_{C_{X^I}} \Rep(\check{M})_I,
$$ 
where we used the shifted action of $C_{X^I}$ on $Shv(\Gr_{G, I})^{\gL^+(M)_I}$. 

\begin{Pp}
\label{Pp_3.1.11}
i) $\IC^{\frac{\infty}{2}}_{P, I}$ belongs to $\SI_{P, I}\subset Shv(\Gr_{G, I})^{\gL^+(M)_I}$.\\
ii) $\IC^{\frac{\infty}{2}}_{P, I}$ is the extension by zero under $\bar S^0_{P, I}\hook{} \Gr_{G, I}$.\\
iii) One has $(v^0_{S, I})^! \IC^{\frac{\infty}{2}}_{P, I}\,\iso\, (v^0_{S, I})^* \IC^{\frac{\infty}{2}}_{P, I}\,\iso\, \omega$ over $S^0_{P, I}$. 
\end{Pp}
\begin{proof}
i) Let $\und{\lambda}\in\Map(I,  \Lambda^+_{M, ab})$. It suffices to show that $\IC^{\frac{\infty}{2}}_{P, I}\in Shv(\Gr_{G, I})^{H_{\und{\lambda}}}$. If $\und{\mu}\in\Map(I, \Lambda^+_{M, ab})$ with $\und{\lambda}\le\und{\mu}$ then $(-\und{\mu})\Sat_{G, I}(V^{\und{\mu}})$ is $H_{\und{\lambda}}$-equivariant, and the inclusion 
$$
(\{\mu\in \Map(I, \Lambda^+_{M, ab})\mid \und{\lambda}\le \und{\mu}\}, \le )\subset (\Map(I,  \Lambda^+_{M, ab}), \le)
$$ 
is cofinal. The claim follows.

\medskip\noindent
ii) It suffices to show that for any $\und{\lambda}\in \Map(I,\Lambda^+_{M, ab})$,  $(-\und{\lambda})\Sat_{G, I}(V^{\und{\lambda}})$ is the extension by zero from $\bar S^0_{P, I}$. Recall that 
\begin{equation}
\label{V^und(lambda)_as_a_colimit}
V^{\und{\lambda}}\,\iso\, \underset{(I\to J\to K)\in\Tw(I)}{\colim} \omega_{X^K}\otimes (\underset{j\in J}{\otimes} V^{\lambda_j})
\end{equation} 
in $\Rep(\check{G})_I$ with $\lambda_j=\sum_{i\in I_j} \und{\lambda}(i)$. We claim that $(-\und{\lambda})\Sat_{G, I}(\omega_{X^K}\otimes (\underset{j\in J}{\otimes} V^{\lambda_j}))$ is the extension by zero from $\bar S^0_{P, I}$. The latter complex identifies with
\begin{equation}
\label{complex_one_for_Pp3.1.11}
(-\und{\lambda})\Sat_{G, I}(\omega_{X^K}\otimes (\underset{k\in K}{\otimes}(\underset{j\in J_k}{\otimes} V^{\lambda_j})))\,\iso\,(-\und{\mu})\Sat_{G, K}(\omega_{X^K}\otimes (\underset{k\in K}{\otimes}(\underset{j\in J_k}{\otimes} V^{\lambda_j})))
\end{equation}
with $\und{\mu}: K\to \Lambda^+_{M, ab}$ being the direct image of $\und{\lambda}$. 
Given any field-valued point $(x_k)\in \oo{X}{}^K(\cK)$ for some field $k\subset \cK$, after the base change by $\Spec \cK\to X^K$ given by this collection, (\ref{complex_one_for_Pp3.1.11}) is the extension by zero from 
$$
\prod_{k\in K} t_{x_k}^{-\mu_k}\ov{\Gr}_{G, x_k}^{\mu_k}\subset \prod_{k\in K} \bar S^0_{P, x_k}
$$ 
We can argue similarly for other diagonal strata in $X^K$. Our claim follows.

\medskip\noindent
iii) We know already that $(v^0_{S, I})^! \IC^{\frac{\infty}{2}}_{P, I}\in \SI^0_{P, I}(S)$. By Lemma~\ref{Lm_1.4.10_now}, it suffices to establsh an isomorphism 
$$
(i^0_{S, I})^!(v^0_{S, I})^! \IC^{\frac{\infty}{2}}_{P, I}\,\iso\omega
$$ 
in $Shv(X^I)^{\gL^+(M)_I}$ for the map $i^0_{S, I}: X^I\to S^0_{P, I}$. We claim that applying $(s^{\und{0}}_G)^!$ to the diagram (\ref{formula_Def_of_IC_infty/2_P_I}) one gets the constant diagram with value $\omega_{X^I}$.

  In particiular, for $\und{\lambda}\in\Map(I, \Lambda^+_{M, ab})$ we need to establish an isomorphism 
$$
(s^{\und{\lambda}}_G)^!\Sat_{G, I}(V^{\und{\lambda}})[\<\lambda, 2\check{\rho}\>]\,\iso\,\omega_{X^I}
$$

 Argue as in ii) using (\ref{V^und(lambda)_as_a_colimit}). We get
$$
(s^{\und{\lambda}}_G)^!\Sat_{G, I}(V^{\und{\lambda}})\,\iso\,
(s^{\und{0}}_G)^!(-\und{\lambda}) \underset{(I\to J\to K)\in\Tw(I)}{\colim} \Sat_{G, I}(\omega_{X^K}\otimes (\underset{k\in K}{\otimes} (\underset{j\in J_k}{\otimes} V^{\lambda_j})))
$$
Given $(I\to J\to K)\in\Tw(I)$ using (\ref{complex_one_for_Pp3.1.11}) we get the following. Let $\und{\mu}:K\to \Lambda^+_{M, ab}$ be the direct image of $\und{\lambda}$ then
\begin{equation}
\label{complex_two_for_Pp3.1.11}
(s^{\und{\lambda}}_G)^!\Sat_{G, I}(\omega_{X^K}\otimes (\underset{k\in K}{\otimes} (\underset{j\in J_k}{\otimes} V^{\lambda_j})))\,\iso\, (s^{\und{\mu}}_G)^!\Sat_{G, K}(\omega_{X^K}\otimes (\underset{k\in K}{\otimes} (\underset{j\in J_k}{\otimes} V^{\lambda_j})))
\end{equation}
For each $k\in K$ we have a canonical direct summand $V^{\mu_k}\subset \underset{j\in J_k}{\otimes} V^{\lambda_j}$, and (\ref{complex_two_for_Pp3.1.11}) identifies canonically with
\begin{equation}
\label{complex_30_for_Pp_3.1.14}
(s^{\und{\mu}}_G)^!\Sat_{G, K}(\omega_{X^K}\otimes (\underset{k\in K}{\otimes} V^{\mu_k})).
\end{equation}
Indeed, let $\und{\mu}': K\to \Lambda^+$ with $\und{\mu}'(k)\le \und{\mu}_k$ for all $k$. Then $\Sat_{G, K}(\omega_{X^K}\otimes (\underset{k\in K}{\otimes} V^{\mu'_k}))$ identifies with $\IC(\ov{\Gr}^{\und{\mu}'}_{G,K})$ up to a shift, and the limit of the diagram 
$$
X^K\toup{s^{\und{\mu}}_G}\Gr_{G, K}\gets \ov{\Gr}^{\und{\mu}'}_{G,K}
$$ 
is empty unless $\und{\mu}=\und{\mu}'$.

By Remark~\ref{Rem_3.1.8}, (\ref{complex_30_for_Pp_3.1.14}) identifies canonically with $\omega_{X^K}[-\<\lambda, 2\check{\rho}\>]$. Finally, 
$$
(s^{\und{\lambda}}_G)^!\Sat_{G, I}(V^{\und{\lambda}})[\<\lambda, 2\check{\rho}\>]\,\iso\, \underset{(I\to J\to K)\in\Tw(I)}{\colim} \omega_{X^K}\,\iso\,\omega_{X^I}
$$ 
in $Shv(X^I)$, the latter colimit diagram is the same as for the commutative factorization algebra $\omega_{\Ran}\in Shv(\Ran)$ over $X^I$. 
\end{proof}

\sssec{} 
\label{Sect_3.1.15_now}
Since 
$$
\oblv: (Shv(\Gr_{G, I})^{\gL^+(M)_I})\otimes_{C_{X^I}} \Rep(\check{M})_I\to (Shv(\Gr_{G, I})^{\gL^+(M)_I})
$$ 
is monadic, by (\cite{Ly}, 3.4.4), $\SI^{\le 0}_{P, I}(S)\otimes_{C_{X^I}} \Rep(\check{M})_I$ is 
the preimage of the full subcategory $\SI^{\le 0}_{P, I}(S)$ under the latter functor. So, 
$$
\IC^{\frac{\infty}{2}}_{P, I}\in \SI^{\le 0}_{P, I}(S)\otimes_{C_{X^I}} \Rep(\check{M})_I.
$$ 

\begin{Def}
\label{Def_'IC_P,Ran^infty/2}
By construction, for a map $I\to J$ in $fSets$ and the diagonal $\vartriangle^{(I/J)}: \Gr_{G, J}\to \Gr_{G, I}$ we have canonically
$$
(\vartriangle^{(I/J)})^! \IC^{\frac{\infty}{2}}_{P, I}\,\iso\, \IC^{\frac{\infty}{2}}_{P, J}.
$$
So, the collection $\{\IC^{\frac{\infty}{2}}_{P, I}\}_{I\in fSets}$ is naturally an object of $\underset{I\in fSets}{\lim} \SI^{\le 0}_{P, I}(S)\,\iso\, \SI^{\le 0}_{P, \Ran}(S)$ denoted $'\!\IC^{\frac{\infty}{2}}_{P, \Ran}$. This is the main object of study of this paper.
\end{Def}

 By Remark~\ref{Rem_2.2.8}, $'\!\IC^{\frac{\infty}{2}}_{P, \Ran}$ factorizes naturally over $\Ran$, as $Shv(\Gr_{G,\Ran})^{\gL^+(M)_{\Ran}}$ has a natural factorization structure.
 
\sssec{} 
\label{Sect_3.1.16_now}
If $x\in X$ is a $k$-point of $X$ then the $!$-restriction of $\IC^{\frac{\infty}{2}}_{P, I}$ under $\Gr_{G, x}\hook{} \Gr_{G, I}$ identifies canonically with the object $\IC^{\frac{\infty}{2}}_P$ in $\SI_{P, x}=Shv(\Gr_{G, x})^{H_x}$ from (\cite{DL}, 4.1.3). Here $H_x$ is the fibre of $H_I$ at the corresponding point $\Spec k\toup{x} X\hook{} X^I$.  

\index{$\IC^{\frac{\infty}{2}}_P, \SI_{P, x}$, Section~\ref{Sect_3.1.16_now}}

\begin{Lm} 
\label{Lm_3.1.17}
i) Let $K\in Shv(\Gr_{G, I})^{\gL^+(M)_I}$ such that the functor $\Av_!^{\gL(U(P))_I}: Shv(\Gr_{G, I})^{\gL^+(M)_I}\to\SI_{P, I}$ is defined on $K$. Let $\und{\lambda}\in\Map(I, \Lambda_{M, ab})$.
Then $\Av_!^{\gL(U(P))_I}$ is defined on $\und{\lambda}K$, and 
$$
\Av_!^{\gL(U(P))_I}(\und{\lambda}K)\,\iso\, \und{\lambda}\Av_!^{\gL(U(P))_I}(K)
$$
canonically in $\SI_{P, I}$.

\smallskip\noindent
ii) The functor $\Av_!^{\gL(U(P))_I}$ is defined on $\delta_I\ast \Sat_{G, I}(V^{\und{\lambda}})$, and 
\begin{equation}
\label{second_iso_for_Lm_4.1.18}
\Av_!^{\gL(U(P))_I}(\delta_I\ast \Sat_{G, I}(V^{\und{\lambda}}))\,\iso\, \Av_!^{\gL(U(P))_I}(\delta_I)\ast \Sat_{G, I}(V^{\und{\lambda}})
\end{equation} 
canonically in $\SI_{P, I}$.
\end{Lm}
\begin{proof}
i) The object $e^{\und{\lambda}}$ is dualizable in $\Rep(\check{M}_{ab})_I$. Using Section~\ref{Sect_3.1.2_now} our claim is an application of (\cite{HA}, Lemma~4.6.1.6).

\medskip\noindent
ii) By Lemma~\ref{Lm_1.4.13_existence}, $\Av_!^{\gL(U(P))_I}(\delta_I)\,\iso\, (v^0_{S,I})_!\omega$. By definition, 
$$
V^{\und{\lambda}}\,\iso\, \underset{(I\to J\to K)\in\Tw(I)}{\colim} \omega_{X^K}\otimes (\underset{j\in J}{\otimes} V^{\lambda_j})\in \Rep(\check{G})_I,
$$
where $\lambda_j=\sum_{i\in I_j}\und{\lambda}(i)$. So,
$$
\Sat_{G,I}(V^{\und{\lambda}})\,\iso\,  \underset{(I\to J\to K)\in\Tw(I)}{\colim} \Sat_{G,I}(\omega_{X^K}\otimes (\underset{j\in J}{\otimes} V^{\lambda_j})).
$$ 
 
 One has 
$$
\Sat_{G,I}(\omega_{X^K}\otimes (\underset{j\in J}{\otimes} V^{\lambda_j}))\,\iso\, \Sat_{G,K}(\omega_{X^K}\otimes (\underset{k\in K}{\otimes}(\underset{j\in J_k}{\otimes} V^{\lambda_j}))).
$$
Here we view $\Sph_{G,K}\subset \Sph_{G,I}$ as the full subcategory of objects, which are extensions by zero under $\Gr_{G,K}\hook{} \Gr_{G,I}$.  

\smallskip\noindent
{\bf Step 1} Let $\mu:  K\to \Lambda^+$. We claim that
$$
\Av_!^{\gL(U(P))_I}(\delta_I\ast \Sat_{G,K}(\omega_{X^K}\otimes (\underset{k\in K}{\otimes} V^{\mu_k}))\,\iso\, \Av_!^{\gL(U(P))_I}(\delta_I)\ast \Sat_{G,K}(\omega_{X^K}\otimes (\underset{k\in K}{\otimes} V^{\mu_k})),
$$
in particular the LHS is well-defined. Indeed, 
$$
\delta_I\ast \Sat_{G,K}(\omega_{X^K}\otimes (\underset{k\in K}{\otimes} V^{\mu_k}))\,\iso\, \delta_K\ast \Sat_{G,K}(\omega_{X^K}\otimes (\underset{k\in K}{\otimes} V^{\mu_k}))
$$ 
canonically. So, for $F\in \SI_{P, I}$ one has
\begin{multline}
\label{multline_for_proof_Lemma_4.1.18}
\HOM(\Av_!^{\gL(U(P))_I}(\delta_I\ast \Sat_{G,K}(\omega_{X^K}\otimes (\underset{k\in K}{\otimes} V^{\mu_k})), F))\,\iso\\
\HOM(\delta_K\ast \Sat_{G,K}(\omega_{X^K}\otimes (\underset{k\in K}{\otimes} V^{\mu_k})), \vartriangle^!F), 
\end{multline}
where $\vartriangle: \Gr_{G, K}\hook{}\Gr_{G, I}$ is the natural closed immersion. Now the image of $\omega_{X^K}\otimes (\underset{k\in K}{\otimes} V^{\mu_k})$ in $\Rep(\check{G})_K$ is dualizable, and (\ref{multline_for_proof_Lemma_4.1.18}) identifies with
\begin{multline*}
\HOM(\delta_K, (\vartriangle^!F)\ast \Sat_{G,K}(\omega_{X^K}\otimes (\underset{k\in K}{\otimes} (V^{\mu_k})^*))\,\iso\\ 
\HOM(\Av_!^{\gL(U(P))_K}(\delta_K), (\vartriangle^!F)\ast \Sat_{G,K}(\omega_{X^K}\otimes (\underset{k\in K}{\otimes} (V^{\mu_k})^*))\,\iso\\
\HOM(\Av_!^{\gL(U(P))_K}(\delta_K)\ast \Sat_{G,K}(\omega_{X^K}\otimes (\underset{k\in K}{\otimes} V^{\mu_k})), F)
\end{multline*}
Since $\vartriangle^!\Av_!^{\gL(U(P))_I}(\delta_I)\,\iso\, \Av_!^{\gL(U(P))_K}(\delta_K)$, our claim follows.

\smallskip\noindent
{\bf Step 2} Now the LHS of (\ref{second_iso_for_Lm_4.1.18}) identifies with
\begin{multline*}
\underset{(I\to J\to K)\in\Tw(I)}{\colim} \Av_!^{\gL(U(P))_I}(\delta_I\ast \Sat_{G,I}(\omega_{X^K}\otimes (\underset{j\in J}{\otimes} V^{\lambda_j}))\underset{Step 1}{\iso}\\
\underset{(I\to J\to K)\in\Tw(I)}{\colim} \Av_!^{\gL(U(P))_I}(\delta_I)\ast \Sat_{G,I}(\omega_{X^K}\otimes (\underset{j\in J}{\otimes} V^{\lambda_j}))\,\iso\\
\Av_!^{\gL(U(P))_I}(\delta_I)\ast \Sat_{G, I}(V^{\und{\lambda}}).
\end{multline*}
\end{proof}

\sssec{Another presentation of  $\IC^{\frac{\infty}{2}}_{P, I}$} 
\label{Sect_3.1.18_now}
Let $I\in fSets$. Set 
$$
\bvartriangle^0_I=\Av_!^{\gL(U(P))_I}(\delta_I)\in \SI^{\le 0}_{P, I}(S)
$$
By Lemma~\ref{Lm_1.4.13_existence}, the latter object is defined, and 
$\bvartriangle^0_I\,\iso\, (v^0_{S, I})_!(\gt_{S, I}^0)^!\omega_{X^I}$ canonically. 
For $\und{\lambda}: I\to \Lambda_{M, ab}$ set
$$
\bvartriangle^{\und{\lambda}}_I=(\und{\lambda})\bvartriangle^0_I[-\<\lambda, 2\check{\rho}\>]\in \SI_{P, I}.
$$
These are analogs of the objects $\bvartriangle^{\mu}\in \SI_{P, x}$ for $\mu\in\Lambda_{M, ab}$ from (\cite{DL}, 3.3.3). 

Applying Lemma~\ref{Lm_3.1.17} to the definition of $\IC^{\frac{\infty}{2}}_{P, I}$, one gets 
$$
\IC^{\frac{\infty}{2}}_{P, I}\,\iso\, \underset{\und{\lambda}\in (\Map(I, \Lambda^+_{M, ab}),\le)}{\colim}  \bvartriangle_I^{
-\und{\lambda}}\ast \Sat_{G, I}(V^{\und{\lambda}}),
$$
where the colimit is calculated in $\SI_{P, I}$. 

\index{$\bvartriangle^0_I, \bvartriangle^{\und{\lambda}}_I$, Section~\ref{Sect_3.1.18_now}}

\ssec{Graded Satake functors}
\label{Sect_Graded Satake functor}

\sssec{} 
\label{Sect_3.2.1_now}
For $\theta\in\Lambda_{G,P}$ let $\Rep(\check{M})_{\theta}\subset \Rep(\check{M})$ be the full subcategory of those representations on which $Z(\check{M})$ acts by $\theta$. Set $\Rep(\check{M})_{neg}=\underset{\theta\in -\Lambda_{G,P}^{pos}}{\oplus} \Rep(\check{M})_{\theta}$.

 Apply the construction of (\cite{Ly10}, Appendix E) to the latter $-\Lambda_{G,P}^{pos}$-graded category. We get for each $\theta\in -\Lambda_{G,P}^{pos}$ the object $\Fact(\Rep(\check{M})_{neg})_{\theta}\in Shv((X^{-\theta}\times\Ran)^{\subset})-mod$ and the decomposition 
$$
\Fact(\Rep(\check{M})_{neg})\,\iso\, \underset{\theta\in -\Lambda_{G,P}^{pos}}{\oplus} \Fact(\Rep(\check{M})_{neg})_{\theta}
$$
in the notations of \select{loc.cit.} 

\index{$\Rep(\check{M})_{neg}, \Sat_{M,\Ran, +}$, Section~\ref{Sect_3.2.1_now}}

 For $\theta\in -\Lambda_{G,P}^{pos}$ recall the diagram
$$
X^{-\theta}\gets \Gr_{M,\Ran, +}^{\theta}\hook{} \Gr_{M,\Ran,\theta},
$$ 
where the second map is a closed immersion. By functoriality of $\Fact$, we have the natural functor $\Fact(\Rep(\check{M})_{neg})\to \Fact(\Rep(\check{M}))$. 

 In Section~\ref{Sect_append_Graded Satake functors} we construct a $-\Lambda_{G,P}^{pos}$-graded functor 
\begin{equation}
\label{functor_Sat_M_Ran_+_graded}
\Sat_{M,\Ran, +}: \Fact(\Rep(\check{M})_{neg})\to Shv(\Gr_{M,\Ran, +})^{\gL^+(M)_{\Ran}}
\end{equation}
and study its properties. It is given by a collection for $\theta\in -\Lambda_{G,P}^{pos}$ of $Shv((X^{-\theta}\times\Ran)^{\subset})$-linear functors
\begin{equation}
\label{Sat_functor_for_Pp_3.2.2}
\Sat_{M,\Ran,+}: \Fact(\Rep(\check{M})_{neg})_{\theta}\to Shv(\Gr^{\theta}_{M, \Ran, +})^{\gL^+(M)_{\Ran}}.
\end{equation}
The target here is a full subcategory of $\Sph_{M,\Ran}$ of objects extended by zero from the closed substack $\Gr_{M,\Ran, +}^{\theta}\hook{} \Gr_{M,\Ran}$. In Section~\ref{Sect_append_Graded Satake functors} we establish the following.
\begin{Pp} 
\label{Pp_3.2.2}
For each $\theta\in -\Lambda_{G,P}^{pos}$ the restriction of $\Sat_{M, \Ran}$ under 
$$
\Fact(\Rep(\check{M})_{neg})_{\theta}\hook{} \Fact(\Rep(\check{M})_{neg})\to \Fact(\Rep(\check{M}))
$$
is the functor (\ref{Sat_functor_for_Pp_3.2.2}).
\end{Pp}

\sssec{} 
\label{Sect_3.2.3_now}
Set 
$$
\Rep(\check{M})^{untl}_{neg}=\Vect\oplus\, (
\underset{\theta\in -\Lambda_{G,P}^{pos,*}}{\oplus} \Rep(\check{M})_{\theta}).
$$
We consider $\Rep(\check{M})^{untl}_{neg}\subset \Rep(\check{M})_{neg}$ as the subcategory of those representations in the RHS, whose zero's component is a trivial $\check{M}$-module. Applying the construction of (\cite{Ly10}, Appendix E) we get as above
$$
\Fact(\Rep(\check{M})^{untl}_{neg})=\underset{\theta\in -\Lambda_{G,P}^{pos}}{\oplus}  \Fact(\Rep(\check{M})^{untl}_{neg})_{\theta}
$$
with $\Fact(\Rep(\check{M})^{untl}_{neg})_{\theta}\in Shv((X^{-\theta}\times\Ran)^{\subset})-mod$.

\index{$\Rep(\check{M})^{untl}_{neg}$, Section~\ref{Sect_3.2.3_now}}

\sssec{} 
\label{Sect_3.2.4_now}
Let $\theta\in -\Lambda_{G,P}^{pos}$. Set 
$$
\Gr_{M, X^{-\theta}}^{\theta,\subset}=\Gr_{M, X^{-\theta}, +}^{\theta}\times_{X^{-\theta}} (X^{-\theta}\times\Ran)^{\subset}.
$$
Consider the projection on the first factor
$$
\Gr_{M, X^{-\theta}}^{\theta,\subset}\to \Gr_{M, X^{-\theta}, +}^{\theta}.
$$
As in Section~\ref{Sect_1.6.2}, we extend it to the map of stack quotients
$$
\tau^{\theta}_{untl}: \Gr_{M, X^{-\theta}}^{\theta,\subset}/\gL^+(M)_{\Ran}\to \Gr_{M, X^{-\theta}, +}^{\theta}/\gL^+(M)_{\theta}
$$
To be precise, for $S\in\Sch^{aff}$ an $S$-point of $\Gr_{M, X^{-\theta}}^{\theta,\subset}/\gL^+(M)_{\Ran}$ is given by an $S$-point $(D,\cI)\in (X^{-\theta}\times\Ran)^{\subset}$, $M$-torsors $\cF_M, \cF'_M$ on $\cD_{\cI}$ with an isomorphism $\beta: \cF_M\,\iso\,\cF'_M\mid_{\cD_{\cI}-\supp(D)}$ whose reduction to $M/[M,M]$ extends to an isomorphism  on $\cD_{\cI}$
$$
\cF_{M/[M, M]}(D)\,\iso\,\cF'_M.
$$
The map $\tau^{\theta}_{untl}$ is given by forgetting $\cI$ and restricting the above data to $\cD_D$. 

\index{$\Gr_{M, X^{-\theta}}^{\theta,\subset}, \tau^{\theta}_{untl}, i_{untl}^{\theta}, \Gr_{M,\Ran}^{\subset}, \Gr_{M, I}^{\subset}$, Section~\ref{Sect_3.2.4_now}}
 
As in Lemma~\ref{Lm_1.6.3}, one defines the functor
$$
(\tau^{\theta}_{untl})^!: Shv(\Gr_{M, X^{-\theta}, +}^{\theta})^{\gL^+(M)_{\theta}}\to 
Shv(\Gr_{M, X^{-\theta}}^{\theta,\subset})^{\gL^+(M)_{\Ran}}.
$$

 Write
$$
i_{untl}^{\theta}:  \Gr_{M, X^{-\theta}}^{\theta,\subset}\hook{} \Gr^{\theta}_{M, \Ran, +}
$$
for the natural closed immersion. For $I\in fSets$ set
$$
\Gr_{M,\Ran}^{\subset}=\underset{\theta\in -\Lambda_{G,P}^{pos}}{\sqcup} \Gr_{M, X^{-\theta}}^{\theta,\subset}\;\;\mbox{and}\,\;
\Gr_{M, I}^{\subset}=\Gr_{M,\Ran}^{\subset}\times_{\Ran} X^I.
$$ 

\sssec{} 
\label{Sect_3.2.5_now}
In Section~\ref{Sect_case_of_checkM_neg_untl} we construct a functor
\begin{equation}
\label{functor_Sat_M_Ran^untl}
\Sat_{M,\Ran}^{untl}:\Fact(\Rep(\check{M})^{untl}_{neg})\to Shv(\Gr_{M, \Ran}^{\subset})^{\gL^+(M)_{\Ran}}
\end{equation}
graded by $-\Lambda_{G,P}^{pos}$ and study its properties. It is a collection of $Shv((X^{-\theta}\times\Ran)^{\subset})$-linear functors for $\theta\in -\Lambda_{G,P}^{pos}$
\begin{equation}
\label{functor_Sat_M_Ran_theta^untl}
\Sat_{M,\Ran}^{untl}: \Fact(\Rep(\check{M})^{untl}_{neg})_{\theta}\to Shv(\Gr_{M, X^{-\theta}}^{\theta,\subset})^{\gL^+(M)_{\Ran}}.
\end{equation}
We also equip $Shv(\Gr_{M, \Ran}^{\subset})^{\gL^+(M)_{\Ran}}$ with the $\otimes^*$-monoidal and $\otimes^{ch}$-symmetric monoidal structures.  

\index{$\Sat_{M,\Ran}^{untl}$, Section~\ref{Sect_3.2.5_now}}

The following is established in Section~\ref{Sect_C.2.10}.
\begin{Pp}
\label{Pp_about_functor_Sat_M_Ran^untl}
For $\theta\in -\Lambda_{G,P}^{pos}$ the restriction of $\Sat_{M,\Ran, +}$ under 
$$
\Fact(\Rep(\check{M})_{neg}^{untl})_{\theta}\hook{} \Fact(\Rep(\check{M})_{neg})_{\theta}
$$
is the functor (\ref{functor_Sat_M_Ran_theta^untl}) composed with the full embedding 
$$
(i^{\theta}_{untl})_!: Shv(\Gr_{M, X^{-\theta}}^{\theta, \subset})^{\gL^+(M)_{\Ran}}
\hook{} Shv(\Gr^{\theta}_{M, \Ran, +})^{\gL^+(M)_{\Ran}}.
$$
\end{Pp}

\sssec{} 
\label{Sect_3.2.7_now}
Let $\Conf^*=\underset{\theta\in \Lambda_{G,P}^{pos,*}}{\sqcup} X^{\theta}$ and $\Conf=\Spec k\sqcup \Conf^*$, here $\Spec k=X^0$. Write $\gL^+(M)_{\Conf}$ for the group scheme on $\Conf$ whose restriction to $X^{\theta}$ is $\gL^+(M)_{-\theta}$. Set 
$$
\Gr_{M, \Conf, +}=\underset{\theta\in \Lambda_{G,P}^{pos}}{\sqcup} \Gr_{M, X^{\theta}, +}^{-\theta}\;\;\;\mbox{and}\;\;\: \Gr_{M, \Conf^*, +}=\underset{\theta\in \Lambda_{G,P}^{pos,*}}{\sqcup} \Gr_{M, X^{\theta}, +}^{-\theta},
$$ 
here $\Gr_{M, X^0, +}^0=\Spec k$. For $\theta\in\Lambda_{G,P}^{pos}$ set 
$$
\Sph_{M, X^{\theta}, +}=Shv(\Gr_{M, X^{\theta}, +}^{-\theta})^{\gL^+(M)_{-\theta}}\;\;\;\mbox{and}\;\;\; \Sph_{M, \Conf, +}=Shv(\Gr_{M, \Conf, +})^{\gL^+(M)_{\Conf}}.
$$
viewed as sheaf of categories on $X^{\theta}$ and $\Conf$ respectively.
Note that $\Sph_{M, X^0, +}=\Vect$ canonically. The restriction of $\Sph_{M, \Conf, +}$ to $\Conf^*$ is denoted $\Sph_{M, \Conf^*, +}$. 

\index{$\Conf^*, \Conf, \Gr_{M, \Conf, +}, \Gr_{M, \Conf^*, +}$, Section~\ref{Sect_3.2.7_now}}
\index{$\Sph_{M, X^{\theta}, +}, \Sph_{M, \Conf, +}, \Sph_{M, \Conf^*, +}$, Section~\ref{Sect_3.2.7_now}}

\sssec{} 
\label{Sect_3.2.8_now}
Set 
$$
\Rep(\check{M})_{<0}=\underset{\theta\in -\Lambda_{G,P}^{pos,*}}{\oplus} \Rep(\check{M})_{\theta}
$$ 
viewed as an object of $CAlg^{nu}(\DGCat_{cont})$.

In Section~\ref{Sect_case_Rep(checkM)_<0} we define the version of the Satake functor
\begin{equation}
\label{funct_Sat_M_Conf^*}
\Sat_{M, \Conf^*}: \Fact(\Rep(\check{M})_{<0})\to \Sph_{M, \Conf^*,+}
\end{equation}
compatible with the $-\Lambda_{G,P}^{pos,*}$-grading and study its basic properties. This is a collection of $Shv(X^{-\theta})$-linear functors 
\begin{equation}
\label{funct_Sat_M_X^-theta}
\Sat_{M, X^{-\theta}}: \Fact(\Rep(\check{M})_{<0})_{\theta}\to \Sph_{M, X^{-\theta},+}
\end{equation}
for $\theta\in -\Lambda_{G,P}^{pos,*}$. 

\index{$\Rep(\check{M})_{<0}, \Sat_{M, \Conf^*}, \Sat_{M, X^{\theta}}$, Section~\ref{Sect_3.2.8_now}}

 We equip the categories $\Sph_{M, \Conf^*,+}$ and $\Sph_{M,\Conf, +}$ with the $\otimes^*$-monoidal structure in Sections~\ref{Sect_C.3.1_now}-\ref{Sect_C.3.3_now}.
We equip $\Sph_{M, \Conf^*,+}$ and $\Sph_{M,\Conf, +}$ with the $\otimes^{ch}$-symmetric monoidal structure in Sections~\ref{Sect_C.3.4_now} and \ref{Sect_C.3.14_now}. The proof of the following is postponed to Section~\ref{Sect_C.4.12}. 
 
\begin{Pp} 
\label{Pp_3.2.5_Satake}
Let $B\in \Rep(\check{M})_{<0}$. Set 
$
A_{<0}=\underset{m>0}{\oplus}\Sym^m(B)\in CAlg^{nu}(\Rep(\check{M})_{<0})
$
and 
$$
A=\underset{m\ge 0}{\oplus}\Sym^m(B)\in CAlg(\Rep(\check{M})_{neg}^{untl}).
$$ 
 For $\theta\in -\Lambda_{G,P}^{pos,*}$ there is a canonical isomorphism in $Shv(\Gr_{M, X^{-\theta}}^{\theta,\subset})^{\gL^+(M)_{\Ran}}$
$$
(\tau_{untl}^{\theta})^! \Sat_{M, X^{-\theta}}(\Fact(A_{<0})_{\theta})\,\iso\, \Sat^{untl}_{M,\Ran}(\Fact(A)_{\theta}).
$$ 
\end{Pp}

\begin{Rem} Under the assumptions of Proposition~\ref{Pp_3.2.5_Satake} applying (\cite{Ly10}, Corollary~2.4.16) we get $\Fact(A_{<0})\,\iso\, \underset{m>0}{\oplus} \Sym^{m, \star}(\vartriangle_!(B\otimes\omega_X))$, where $\vartriangle: X\to \Ran$ is the natural map, and the symmetric algebra is understood for the non-unital $\star$-symmetric monoidal structure on $\Fact(\Rep(\check{M})_{<0})$ defined in (\cite{Ly10}, 2.2.6). Here $B\otimes\omega_X\in \Rep(\check{M})_{<0}\otimes Shv(X)$. 
\end{Rem}

\sssec{} 
\label{Sect_3.2.11_now}
Set
$$
\Rep(\check{M})^-_{neg}=\underset{\lambda\in \Lambda^-_{M,G}}{\oplus} \Vect\otimes U^{\lambda},
$$
here $U^{\lambda}$ is the irreducible $\check{M}$-module with highest weight $\lambda$, and $\Lambda^-_{M,G}$ was defined in Section~\ref{Sect_1.3.10}. By (\cite{DL}, Remark~4.2.16), we have a full embedding 
$$
\Rep(\check{M})^-_{neg}\subset \Rep(\check{M})^{untl}_{neg}
$$
in $\DGCat_{cont}$. The latter is a map in $CAlg(\DGCat_{cont})$. We view $\Rep(\check{M})^-_{neg}$ as $-\Lambda_{G,P}^{pos}$-graded with the grading induced by that of $\Rep(\check{M})^{untl}_{neg}$. Set
$$
\Rep(\check{M})^-_{<0}=\underset{0\ne \lambda\in \Lambda^-_{M,G}}{\oplus} \Vect\otimes U^{\lambda}.
$$
Then $\Rep(\check{M})^-_{<0}\subset \Rep(\check{M})_{<0}$ is a full embedding, this is a map in $CAlg^{nu}(\DGCat_{cont})$.

\index{$\Rep(\check{M})^-_{neg}, \Rep(\check{M})^-_{<0}, i_{\Mod}^{\theta}$, Section~\ref{Sect_3.2.11_now}}

 We have the closed immersion for $\theta\in -\Lambda_{G,P}^{pos}$ denoted
$$
i_{\Mod}^{\theta}: \Mod_M^{-,\theta}\hook{} \Gr_{M, X^{-\theta}, +}^{\theta}.
$$ 

\begin{Pp} 
\label{Pp_3.2.7_Satake}
For $\theta\in -\Lambda_{G,P}^{pos,*}$ the restriction of $\Sat_{M, X^{-\theta}}$ under the natural map 
$$
\Fact(\Rep(\check{M})^-_{<0})_{\theta}\to \Fact(\Rep(\check{M})_{<0})_{\theta}
$$ 
fits into the commutative diagram
$$
\begin{array}{ccc}
\Fact(\Rep(\check{M})_{<0})_{\theta}&\toup{\Sat_{M, X^{-\theta}}} &\Sph_{M, X^{-\theta},+}\\
\uparrow && \uparrow\lefteqn{\scriptstyle (i^{\theta}_{\Mod})_*}\\
\Fact(\Rep(\check{M})^-_{<0})_{\theta} & \toup{\Sat_{M, X^{-\theta}}} & Shv(\Mod_M^{-,\theta})^{\gL^+(M)_{\theta}}
\end{array}
$$
\end{Pp}
\begin{proof}
This follows from the fact that for $\lambda\in\Lambda^+_M$ over $\eta\in -\Lambda_{G,P}^{pos}$ the property $\Gr_M^{\lambda}\subset \Gr_M^{-, \eta}$ is equivalent to $\lambda\in \Lambda^-_{M, G}$ by (\cite{BG}, 6.2.3). 
\end{proof}

\ssec{Restrictions to strata}
\label{Sect_Restrictions to strata}

\sssec{} 
\label{Sect_3.3.1_now}
Consider the space $\cO(U(\check{P}))\in CAlg(\Rep(\check{M}))$ of regular functions on $U(\check{P})$. It is graded by $-\Lambda_{G,P}^{pos}$. The component $\cO(U(\check{P}))_{\theta}$ is given by requiring that $Z(\check{M})$ acts by $\theta$. We denote by 
$$
\Fact(\cO(U(\check{P})))\in \Fact(\Rep(\check{M}))
$$ 
the corresponding commutative factorization algebra defined in (\cite{Ly10}, 2.4), and by $\cO(U(\check{P}))_{X^I}$ its !-restriction under $X^I\to\Ran$ for $I\in fSets$. By (\cite{Ly10}, 2.5.4), 
$$
\Fact(\cO(U(\check{P})))\in CAlg(\Fact(\Rep(\check{M})), \otimes^!)
$$ 
and $\cO(U(\check{P}))_{X^I}\in CAlg(\Rep(\check{M})_{X^I}, \otimes^!)$. 
When we need the fact that 
$$
\Sat_{M,\Ran}: \Fact(M)^{\epsilon}\to \Sph_{M, \Ran}
$$ 
is compatible with the factorization structures (cf. Remark~\ref{Rem_Sat_and_factorization}), by abuse of notations we still denote by
$$
\Fact(\cO(U(\check{P})))\in CAlg(\Fact(\Rep(\check{M})^{\epsilon}), \otimes^!)
$$
the corresponding object. So, 
$$
\Sat_{M, \Ran}(\Fact(\cO(U(\check{P}))))\in \Sph_{M,\Ran}
$$
has a factorization structure. Besides, for each $I\in fSets$, $\Sat_{M, I}(\cO(U(\check{P}))_{X^I})\in Alg(\Sph_{M, I})$ is an algebra with respect to the inner convolution on $\Sph_{M, I}$. 

\index{$\cO(U(\check{P})), \cO(U(\check{P}))_{\theta}, \Fact(\cO(U(\check{P}))), (\cO(U(\check{P}))_{X^I}$, Section~\ref{Sect_3.3.1_now}}

  Write $U(\gu(\check{P}^-))$ for the universal enveloping algebra of $\gu(\check{P}^-)$. Then $U(\gu(\check{P}^-))$ identifies with the graded dual of $\cO(U(\check{P}))$ as Hopf algebras (for $P=B$ this is proved in (\cite{GLS}, Proposition 5.1)). So, for $\theta\in\Lambda_{G,P}$ one has $(\cO(U(\check{P}))_{\theta})^*\;\iso\; U(\gu(\check{P}^-))_{\theta}$.    
  
\index{$U(\gu(\check{P}^-))$, Section~\ref{Sect_3.3.1_now}}  
 
\begin{Rem} One has an isomorphism $\cO(U(\check{P}))\,\iso\,\Sym(\gu(\check{P})^*)$ in $CAlg(\Rep(\check{M}))$. It is the Hopf algebra structure on $\cO(U(\check{P}))$, which allows to distinguish it from the algebra $\Sym(\gu(\check{P})^*)$. This Hopf algebra structure is not used in this paper. 
\end{Rem} 
 
\sssec{} Note that $\cO(U(\check{P}))\in CAlg(\Rep(\check{M})^-_{neg})$. Set 
$$
\cO(U(\check{P}))_{<0}=\underset{m>0}{\oplus} \Sym^m((\gu(\check{P})^*)\in CAlg^{nu}(\Rep(\check{M})_{<0}),
$$
so $e\oplus \cO(U(\check{P}))_{<0}=\cO(U(\check{P}))$. We get 
$$
\Fact(\cO(U(\check{P}))_{<0})\in\Fact(\Rep(\check{M})_{<0})
$$ 
and for $\theta\in -\Lambda_{G,P}^{pos, *}$ the object
$$
\Sat_{M, X^{-\theta}}(\Fact(\cO(U(\check{P}))_{<0})_{\theta})\in \Sph_{M, X^{-\theta},+}.
$$

For $I\in fSets$ we equip $\Rep(\check{M})_{X^I}$ with the t-structure defined in (\cite{Ly10}, 7.2.5). By (\cite{Ly10}, 7.2.6), $\cO(U(\check{P}))_{X^I}\in \Rep(\check{M})_{X^I}$ is placed in degree $-\mid I\mid$. Since $\Sat_{M, I}$ is t-exact by (\cite{Ly10}, 7.3.15), 
$\Sat_{M, I}(\cO(U(\check{P}))_{X^I})$ is placed in perverse degree $-\mid I\mid$ in $\Sph_{M, I}$. 

 From Propositions~\ref{Pp_3.2.2}, \ref{Pp_3.2.5_Satake}, \ref{Pp_3.2.7_Satake} we learn the following. Given $\theta\in\Lambda_{G,P}$, over the component $\Gr_{M,I,\theta}$ of $\Gr_{M, I}$ the object $\Sat_{M, I}(\cO(U(\check{P}))_{X^I})$ vanishes unless $\theta\in -\Lambda_{G,P}^{pos}$, and in the latter case it is the extension by zero under the closed immersion 
\begin{equation}
\label{closed_imm_for_Sect_3.1.21}
 \Mod^{-,\theta,\subset}_{M, I}\hook{}\Gr_{M, I,\theta}.
\end{equation} 
The resulting object on $\Mod^{-,\theta,\subset}_{M, I}$ is the !-restriction under $\Mod^{-,\theta,\subset}_{M, I}\hook{} \Mod^{-,\theta,\subset}_{M, \Ran}$ of
$$
(\tau^{\theta}_{untl})^!\Sat_{M, X^{-\theta}}(\Fact(\cO(U(\check{P}))_{<0})_{\theta}).
$$

\begin{Thm}
\label{Thm_*-fibres_first} 
For $\theta\in -\Lambda_{G,P}^{pos}$, one has canonically 
$$
(v_{S, I}^{\theta})^*\IC^{\frac{\infty}{2}}_{P, I}[\<\theta, 2\check{\rho}-2\check{\rho}_M\>]
\,\iso\, (\gt_{S, I}^{\theta})^!\Sat_{M, I}(\cO(U(\check{P}))_{X^I})
$$ 
in $\SI^{\theta}_{P, I}(S)$.
Here we used the closed immersion (\ref{closed_imm_for_Sect_3.1.21}), and $\gt_{S, I}^{\theta}: S^{\theta}_{P, I}\to \Mod^{-,\theta,\subset}_{M, I}$.
\end{Thm} 
\begin{proof} {\bf Step 1}. Since $(v_{S, I}^{\theta})^*\IC^{\frac{\infty}{2}}_{P, I}\in \SI^{\theta}_{P,I}(S)$, it suffices to obtain an isomorphism
\begin{equation}
\label{iso_first_for_Thm_3.1.22}
(\gt_{S, I}^{\theta})_!(v_{S, I}^{\theta})^*\IC^{\frac{\infty}{2}}_{P, I}[\<\theta, 2\check{\rho}-2\check{\rho}_M\>]\,\iso\, \Sat_{M, I}(\cO(U(\check{P}))_{X^I})
\end{equation}
in $Shv(\Mod_{M, I}^{-,\theta, \subset})^{\gL^+(M)_I}$. By Section~\ref{Sect_1.5.17}, 
$$
(\gt_{S, I}^{\theta})_!(v_{S, I}^{\theta})^*\IC^{\frac{\infty}{2}}_{P, I}\,\iso\,  (\gt^{\theta}_{P^-, I})_* (v^{\theta}_{P^-, I})^!\IC^{\frac{\infty}{2}}_{P, I},
$$
here 
$$
\Mod^{-,\theta,\subset}_{M, I}\;\getsup{\gt^{\theta}_{P^-, I}}\; \Gr^{\theta}_{P^-, I}\cap \,\bar S^0_{P, I}\;\toup{v^{\theta}_{P^-, I}} \;\bar S^0_{P, I}.
$$

For $\theta\in\Lambda_{G,P}$ consider the diagram of natural maps
$$
\Gr_{M, I,\theta}\getsup{\gt_{P^-,\theta}} \Gr_{P^-, I,\theta} \toup{v_{P^-,\theta}} \Gr_{G, I}.
$$
 
The compatibility of the Satake functors with the corrected Jacquet functors is given by (\cite{GLys}, 9.4.3). It says the following. Let $\epsilon\in Z(\check{G})$ be the image of $-1$ under $2\check{\rho}: \Gm\to \check{T}$. Note that $Z(\check{G})\subset Z(\check{M})$, so $\epsilon$ also acts naturally by automorphisms of the identity functor on $\Rep(\check{M})$. Let $\cG^{\epsilon_{P}}$ be the multiplicative $\ZZ/2\ZZ$-gerbe on $\Gr_{M,\Ran}$ attached to $\epsilon_{P}: \Lambda\to \ZZ/2\ZZ$, $z\mapsto \<z, 2\check{\rho}-2\check{\rho}_M\>\!\mod 2$ (cf. \cite{GLys}, Section~5.2). Then
the diagram commutes naturally
$$
\begin{array}{ccc}
(\Rep(\check{G})^{\epsilon})_{X^I} & \toup{\Sat_{G,I}} & \Sph_{G,I}\\
\downarrow && \downarrow\lefteqn{\scriptstyle J^G_M}\\
(\Rep(\check{M})^{\epsilon})_{X^I} & \toup{\Sat_{M, I}} & Shv_{\cG^{\epsilon_P}}(\Gr_{M,I})^{\gL^+(M)_I}
\end{array}
$$
in the notations of (\cite{GLys}, 9.4.3). Here in the low row $\Sat_{M, I}$ denotes the corresponding \select{metaplectic} Satake functor, and $J^G_M$ is the corrected Jacquet functor. Up to ignoring the twist of the category of sheaves by $\cG^{\epsilon_P}$, our $J^G_M$ is given by 
$$
J^G_M(K)=(\gt_{P^-,\theta})_* (v_{P^-, \theta})^!K[\<2\check{\rho}-2\check{\rho}_M,\theta\>]
$$ 
over $\Gr_{M, I,\theta}$. Recall that $Shv_{\cG^{\epsilon_P}}(\Gr_{M,\Ran})^{\gL^+(M)_{\Ran}}\,\iso\, \Sph_{M, \Ran}$ as sheaves of categories on $\Ran$, but the latter isomorphism does not respect the factorization structures. 

\index{$\gt_{P^-,\theta}$, $v_{P^-,\theta}$, Proof of Theorem~\ref{Thm_*-fibres_first}}

\medskip\noindent
{\bf Step 2}. Let $\und{\lambda}: I\to \Lambda^+_{M, ab}$ and $\lambda=\sum_{i\in I} \und{\lambda}(i)$. Let us calculate 
$$
(\gt_{P^-, \theta})_* (v_{P^-, \theta})^!(-\und{\lambda}) \Sat_{G, I}(V^{\und{\lambda}}).
$$ 
By Remark~\ref{Rem_action_on_conn_components}, the latter identifies with
$$
(-\und{\lambda})(\gt_{P^-, \theta+\lambda})_* (v_{P^-, \theta+\lambda})^! \Sat_{G, I}(V^{\und{\lambda}}).
$$ 
By Step 1,
$$
(\gt_{P^-, \theta+\lambda})_* (v_{P^-, \theta+\lambda})^! \Sat_{G, I}(V^{\und{\lambda}})[\<2\check{\rho}-2\check{\rho}_M,\theta+\lambda\>]\,\iso\, \Sat_{M, I}(\Res V^{\und{\lambda}}).
$$  
Here we have denoted by $\Res V^{\und{\lambda}}\in (\Rep(\check{M})^{\epsilon})_{X^I}$ the object attached to the right-lax monoidal functor 
$$
\Lambda^+_{M, ab}\to \Rep(\check{M})^{\epsilon}, \lambda\mapsto \Res(V^{\lambda}) 
$$
as in (\cite{Ly10}, 6.1.2). So, 
$$
(\gt_{P^-, \theta})_* (v_{P^-, \theta})^!(-\und{\lambda}) \Sat_{G, I}(V^{\und{\lambda}})[\<\lambda, 2\check{\rho}\>]\,\iso\, (-\und{\lambda})\Sat_{M, I}(\Res V^{\und{\lambda}})[\<2\check{\rho}-2\check{\rho}_M, -\theta\>]
$$
over $\Gr_{M, I,\theta}$. One has
$$
(-\und{\lambda})\Sat_{M, I}(\Res V^{\und{\lambda}})\,\iso\, \Sat_{M, I}(e^{-\und{\lambda}})\ast \Sat_{M, I}(\Res V^{\und{\lambda}})\,\iso\, \Sat_{M, I}(e^{-\und{\lambda}}\otimes^! (\Res V^{\und{\lambda}})),
$$ 
where $\ast$ denotes the inner convolution on $\Sph_{M, I}$, and $e^{-\und{\lambda}}\in \Rep(\check{M})^{\epsilon}_{X^I}$ is as in Section~\ref{Sect_2.2.3_now}. Here $e^{-\und{\lambda}}\otimes^! (\Res V^{\und{\lambda}})$ denotes the product in the symmetric monoidal category $((\Rep(\check{M})^{\epsilon})_{X^I}, \otimes^!)$, see Section~\ref{Sect_2.0.1}. 

\medskip\noindent
{\bf Step 3} It remains to establish an isomorphism
\begin{equation}
\label{iso_Step3_Thm_*-fibres_first}
\underset{\und{\lambda}\in(\Map(I,\Lambda^+_{M,ab}),\le)}{\colim} e^{-\und{\lambda}}\otimes^! (\Res V^{\und{\lambda}})\,\iso\, \cO(U(\check{P}))_{X^I}
\end{equation}
in $((\Rep(\check{M})^{\epsilon})_{X^I}$. 

 Let $Prod: \Lambda^+_{M, ab}\to \Rep(\check{M})^{\epsilon}$ be the functor $\lambda\mapsto e^{-\lambda}\otimes (\Res V^{\lambda})$. It is right-lax symmetric monoidal (as the tensor product of two such). So, we have the functor 
$$
\cF_{\und{\lambda}, Prod}: \Tw(I)\to ((\Rep(\check{M})^{\epsilon})_{X^I}
$$ 
attached to our data as in (\cite{Ly10}, 6.1.2). 
It sends $(I\to J\to K)\in \Tw(I)$ to the image of 
$$
\omega_{X^K}\otimes (\underset{j\in J}{\otimes} (e^{-\lambda_j}\otimes \Res V^{\lambda_j}))\in Shv(X^K)\otimes (\Rep(\check{M})^{\epsilon})^{\otimes J}
$$
in $((\Rep(\check{M})^{\epsilon})_{X^I}$ with $\lambda_j=\sum_{i\in I_j} \und{\lambda}(i)$.
By (\cite{Ly10}, 6.1.16) we get
$$
e^{-\und{\lambda}}\otimes^! (\Res V^{\und{\lambda}})\,\iso\, \underset{\Tw(I)}{\colim} \; \cF_{\und{\lambda}, Prod}.
$$ 

 Permuting the two colimits, the LHS of (\ref{iso_Step3_Thm_*-fibres_first}) identifies with
\begin{equation}
\label{object_Step3_Thm_*-fibres_first}
\underset{(I\to J\to K)\in\Tw(I)}{\colim} \;\; \underset{\und{\lambda}\in(\Map(I,\Lambda^+_{M,ab}),\le)}{\colim} \; \omega_{X^K}\otimes (\underset{j\in J}{\otimes} (e^{-\lambda_j}\otimes \Res V^{\lambda_j}))
\end{equation}
 
 Recall that by (\cite{DL}, 2.3.10 and 2.3.12), 
$$
\underset{\mu\in\Lambda^+_{M, ab}}{\colim} e^{-\mu}\otimes \Res V^{\mu}\,\iso\, \cO(U(\check{P})) \in \Rep(\check{M}). 
$$
After passing to the colimit over $\und{\lambda}$, (\ref{object_Step3_Thm_*-fibres_first}) identifies with
$$
\underset{(I\to J\to K)\in\Tw(I)}{\colim} \;\; \omega_{X^K}\otimes \cO(U(\check{P}))^{\otimes J}\,\iso\, \cO(U(\check{P}))_{X^I}
$$
in $((\Rep(\check{M})^{\epsilon})_{X^I}$ as desired.
\end{proof}

\begin{Cor} 
\label{Cor_3.3.5}
i) For $\theta\in -\Lambda_{G,P}^{pos}$, one has canonically 
$$
(v_{S, \Ran}^{\theta})^*\, {'\! \IC^{\frac{\infty}{2}}_{P, \Ran}}[\<\theta, 2\check{\rho}-2\check{\rho}_M\>]
\,\iso\, (\gt_{S, \Ran}^{\theta})^!\Sat_{M, \Ran}(\Fact(\cO(U(\check{P})))_{\theta})
$$ 
in $\SI^{\theta}_{P, \Ran}(S)$.
Here we used the closed immersion $\Mod_{M, \Ran}^{-,\theta,\subset}\hook{} \Gr_{M,\Ran, \theta}$, and $\gt_{S, \Ran}^{\theta}: S^{\theta}_{P, \Ran}\to \Mod^{-,\theta,\subset}_{M, \Ran}$. 

\smallskip\noindent
ii) One has canonically in $\Sph_{M, \Ran}$
$$
(\tau^{\theta}_{untl})^!\Sat_{M, X^{-\theta}}(\Fact(\cO(U(\check{P}))_{<0})_{\theta})\,\iso\, \Sat_{M, \Ran}(\Fact(\cO(U(\check{P})))_{\theta}),
$$
here $\tau^{\theta}_{untl}$ is defined in Section~\ref{Sect_3.2.4_now}. 

\smallskip\noindent
iii) We have $'\!\IC^{\frac{\infty}{2}}_{P, \Ran}\in \SI^{\le 0}_{P,\Ran, untl}(S)$.
\end{Cor}
\begin{proof} i) The isomorphisms of Theorem~\ref{Thm_*-fibres_first} are compatible with !-restrictions along $X^J\hook{} X^I$ for a map $I\to J$ in $fSets$. 

\medskip\noindent
ii) This follows from Propositions~\ref{Pp_3.2.2}, \ref{Pp_3.2.5_Satake}, \ref{Pp_3.2.7_Satake}. 

\medskip\noindent
iii) For any $\theta\in -\Lambda_{G,P}^{pos}$ the functor $(\tau^{\theta})^!$ of Lemma~\ref{Lm_1.6.3} 
takes values in the unital subcategory. Our claim follows now from Corollary~\ref{Cor_1.5.21}.
\end{proof}

\begin{Pp} 
\label{Pp_3.3.6}
For $\theta\in -\Lambda_{G,P}^{pos,*}$ the image of
$$
(i_{S, \Ran}^{\theta})^!(v_{S,\Ran}^{\theta})^!('\!\IC^{\frac{\infty}{2}}_{P,\Ran})[\<\theta, 2\check{\rho}-2\check{\rho}_M\>]
$$
under
$$
\oblv: Shv(\Mod_{M, \Ran}^{-,\theta,\subset})^{\gL^+(M)_{\Ran}}_{untl}\to Shv(\Mod_{M, \Ran}^{-,\theta,\subset})_{untl}
$$
is strictly coconnective.
\end{Pp}
The proof of Proposition~\ref{Pp_3.3.6} is given in Sections~\ref{Sect_3.3.8_stratification}-\ref{Sect_3.3.11}. 

\begin{Cor} 
\label{Cor_3.3.7}
There is a unique isomorphism 
$$
\oblv('\!\IC^{\frac{\infty}{2}}_{P,\Ran})\,\iso\, \IC^{\frac{\infty}{2}}_{P,\Ran}
$$ 
in $\cS\cI^{\le 0}_{P,\Ran, untl}(S)$ extending the evident one over $S^0_{P,\Ran}$. 
\end{Cor}
\begin{proof} 
Combine Proposition~\ref{Pp_3.3.6} with Corollary~\ref{Cor_3.3.5}, Lemma~\ref{Lm_1.6.8_now}, and Section~\ref{Sect_1.9.3_now}. Our claim reduces to the fact that for $\theta\in -\Lambda_{G,P}^{pos,*}$ the object
$$
\Sat_{M, X^{-\theta}}(\Fact(\cO(U(\check{P}))_{<0})_{\theta})
$$
on $\Mod_M^{-,\theta}$ is placed in perverse degrees $<0$. This follows from Lemma~\ref{Lm_C.3.44}. 
\end{proof}

\begin{Rem} We do not know if $'\!\IC^{\frac{\infty}{2}}_{P,\Ran}$ is the intermediate extension of its restriction to $S^0_{P,\Ran}$ in the t-structure of the category $\SI^{\le 0}_{P,\Ran, untl}(S)$.
\end{Rem}

\begin{Con} 
\label{Con_3.3.8}
For $\theta\in -\Lambda_{G,P}^{pos}$ we have canonically
$$
(v_{S,\Ran}^{\theta})^!\, '\!\IC^{\frac{\infty}{2}}_{P,\Ran}[\<\theta, 2\check{\rho}-2\check{\rho}_M\>]\,\iso\, (\gt^{\theta}_{S,\Ran})^!(\tau^{\theta})^! \DD \Sat_{M, X^{-\theta}}(\Fact(\cO(U(\check{P}))_{<0})_{\theta})
$$
in $\SI^{\theta}_{P,\Ran}(S)$. Here $\DD$ is the Verdier duality functor. 
\end{Con}

 Here, as in (\cite{DL}, A.5.5), $\Sph_{M, X^{-\theta}, +}^{constr}\subset \Sph_{M, X^{-\theta}, +}$ is the full subcategory of those objects whose image under $\oblv: \Sph_{M, X^{-\theta}, +}\to Shv(\Gr^{\theta}_{M, X^{-\theta},+})$ is compact. The Verdier duality in Conjecture~\ref{Con_3.3.8} is the equivalence from \select{loc.cit.}
$$
\DD: (\Sph_{M, X^{-\theta}, +}^{constr})^{op}\,\iso\, \Sph_{M, X^{-\theta}, +}^{constr}.
$$ 

\begin{Rem} 
\label{Rem_4.3.10_now}
Conjecture~\ref{Con_3.3.8} almost follows from Theorem~\ref{Th_5.5.2} below. Namely, applying $(\pi^{\theta}_{loc})^!$ to both sides, they become canonically isomorphic.
\end{Rem}

\sssec{} 
\label{Sect_3.3.8_stratification}
Let $\theta\in\Lambda_{G,P}^{pos,*}$. Recall that $X^{\theta}$ is stratified by locally closed subschemes $\oo{X}{}^{\gU(\theta)}$. Here $\gU(\theta)$ is a way to write $\theta=\sum_k n_k\theta_k$, where $n_k>0$, $\theta_k\in \Lambda_{G,P}^{pos,*}$, and $\theta_k$ are pairwise distinct. Here 
$$
X^{\gU(\theta)}=\prod_k X^{(n_k)},
$$ 
and $\oo{X}{}^{\gU(\theta)}\subset X^{\gU(\theta)}$ is the complement to all the diagonals. For $(D_k)\in \oo{X}{}^{\gU(\theta)}$ its image in $X^{\theta}$ is $\sum_k \theta_k D_k$. 

\sssec{} Let $\theta\in -\Lambda_{G,P}^{pos}$, pick $\gU(-\theta)$ as above. Write $\Mod_M^{-, \gU(-\theta)}=\Mod_M^{-,\theta}\times_{X^{-\theta}} \oo{X}{}^{\gU(-\theta)}$. Set also
$$
\Mod^{-, \gU(-\theta),\subset}_{M, \Ran}=\Mod_M^{-, \gU(-\theta)}\times_{X^{-\theta}}(X^{-\theta}\times\Ran)^{\subset}.
$$
The latter prestack has a natural unital structure, so one has $Shv(\Mod^{-, \gU(-\theta),\subset}_{M, \Ran})_{untl}$. Write
$$
p^{\gU(-\theta)}_{\Ran}: \Mod^{-, \gU(-\theta),\subset}_{M, \Ran}\to \Mod_M^{-, \gU(-\theta)}
$$
for the projection. As in Lemma~\ref{Lm_1.4.31}, the functor
\begin{equation}
\label{iso_untl_over_gU(theta)-stratum}
(p^{\gU(-\theta)}_{\Ran})^!: Shv(\Mod_M^{-, \gU(-\theta)})\to Shv(\Mod^{-, \gU(-\theta),\subset}_{M, \Ran})_{untl}
\end{equation}
is an equivalence.


  Define $Shv(\Mod^{-, \gU(-\theta),\subset}_{M, \Ran})^{\gL^+(M)_{\Ran}}_{untl}$ as the preimage of $Shv(\Mod^{-, \gU(-\theta),\subset}_{M, \Ran})_{untl}$ under
$$
\oblv: Shv(\Mod^{-, \gU(-\theta),\subset}_{M, \Ran})^{\gL^+(M)_{\Ran}}\to Shv(\Mod^{-, \gU(-\theta),\subset}_{M, \Ran}).
$$  

 We define a t-structure on $Shv(\Mod^{-, \gU(-\theta),\subset}_{M, \Ran})^{\gL^+(M)_{\Ran}}_{untl}$ by requiring that its object $K$ is connective iff its image under
\begin{equation}
\label{compostion_for_3.1.10}
Shv(\Mod^{-, \gU(-\theta),\subset}_{M, \Ran})^{\gL^+(M)_{\Ran}}_{untl}\to Shv(\Mod^{-, \gU(-\theta),\subset}_{M, \Ran})_{untl}\toup{(\ref{iso_untl_over_gU(theta)-stratum})}Shv(\Mod_M^{-, \gU(-\theta)})
\end{equation}
lies in perverse degrees $\le 0$. We equip $Shv(\Mod_M^{-, \gU(-\theta)})$ with the perverse t-structure.
 
The following is a new ingredient used for our proof of Corollary~\ref{Cor_3.3.7}, which was absent in Gaitsgory's paper \cite{Gai19Ran}.

\begin{Lm} 
\label{Lm_3.3.11_coconnective}
i) The category of coconnective objects in $Shv(\Mod^{-, \gU(-\theta),\subset}_{M, \Ran})^{\gL^+(M)_{\Ran}}_{untl}$ is the preimage of $Shv(\Mod_M^{-, \gU(-\theta)})^{\ge 0}$ under the composition (\ref{compostion_for_3.1.10}). 

\smallskip\noindent
ii)The t-structure on $Shv(\Mod^{-, \gU(-\theta),\subset}_{M, \Ran})^{\gL^+(M)_{\Ran}}_{untl}$ is compatible with filtered colimits.
\end{Lm}
\begin{proof} Set for brevity $C=Shv(\Mod^{-, \gU(-\theta),\subset}_{M, \Ran})^{\gL^+(M)_{\Ran}}_{untl}$. For $I\in fSets$ let 
$$
p_I: \Mod^{-, \gU(-\theta),\subset}_{M, I}\to \Mod_M^{-, \gU(-\theta)}
$$ 
be the projection. Recall that 
$$
\Mod^{-, \gU(-\theta),\subset}_{M, I}=\Mod^{-, \gU(-\theta),\subset}_{M, \Ran}\times_{\Ran} X^I
$$ 
according to our convention from Section~\ref{Sect_1.1.3}. Define $C_I$ as the limit in $\DGCat_{cont}$ of the diagram
$$
Shv(\Mod^{-, \gU(-\theta),\subset}_{M, I})^{\gL^+(M)_I}\,\toup{\oblv}\, Shv(\Mod^{-, \gU(-\theta),\subset}_{M, I})\;\getsup{p_I^!}\; Shv(\Mod_M^{-, \gU(-\theta)})
$$
Given a map $I\to J$ in $fSets$, the $!$-restriction along the corresponding diagonal $X^J\to X^I$ defines a functor $C_I\to C_J$. They organize into a functor $fSets\to \DGCat_{cont}, I\mapsto C_I$. By construction, 
$$
\underset{I\in fSets}{\lim} C_I\,\iso\, C. 
$$

 Using $\Mod^{-, \gU(-\theta)}_M\in \Sch_{ft}$, pick a quotient group scheme $\gL^+(M)_I\to \cG_I$ over $X^I$ such that the $\gL^+(M)_I$-action on $\Mod^{-, \gU(-\theta),\subset}_{M, I}$ factors through $\cG_I$, where $\cG_I$ is a smooth group scheme of finite type over $X^I$, and the kernel of $\gL^+(M)_I\to \cG_I$ is prounipotent. We get the diagram
$$
(\Mod^{-, \gU(-\theta),\subset}_{M, I})/\cG_I\;\getsup{h}\;\Mod^{-, \gU(-\theta),\subset}_{M, I}\;\toup{p_I}\;\Mod_M^{-, \gU(-\theta)},  
$$
where we denoted by $h$ the stack quotient.

Let $\gU(-\theta)$ be given by $-\theta=\sum_k n_k\theta_k$, where $n_k>0$, and $\theta_k\in \Lambda^{pos,*}_{G,P}$ are pairwise distinct. Set $d=\sum_k n_k$. We have the etale map $\oo{X}{}^{\gU(-\theta)}\to \oo{X}{}^{(d)}$ obtained by restricting the sum of divisors map $\prod_k X^{(n_k)}\to X^{(d)}$. Let $r=\mid I\mid$.

 By Proposition~\ref{Pp_ULA_and_perverse_appendix}, the dualizing object $\omega$ is ULA under the projection 
$$
\oo{X}{}^{\gU(-\theta)}\times_{X^{-\theta}}(X^{-\theta}\times X^I)^{\subset}\to \oo{X}{}^{\gU(-\theta)},
$$ 
and $\omega[-r]$ is perverse on $\oo{X}{}^{\gU(-\theta)}\times_{X^{-\theta}}(X^{-\theta}\times X^I)^{\subset}$.

 By Proposition~\ref{Cor_B.2.9}, $p_I^![d-r]$ is t-exact for the perverse t-structures. This is our key observation, which makes the argument work. Equip $C_I$ with a t-structure as in Section~\ref{Sect_B.3.1_now}. From Lemma~\ref{Lm_B.3.2} we see that $C_I^{\ge 0}$ is the preimage of $Shv(\Mod_M^{-, \gU(-\theta)})^{\ge 0}$ under the projection $C_I\to Shv(\Mod_M^{-, \gU(-\theta)})$,  and the t-structure on $C_I$ is compatible with filtered colimits. 
 
 It follows that each transition functor $C_I\to C_J$ in the above is t-exact, and $C^{\le 0}\subset C$ is the full subcategory of those $K\in C$ whose image in each $C_I$ lies in $C_I^{\le 0}$. So, the t-structure on $C$ is obtained as in Lemma~\ref{Lm_B.3.3}. This implies our claim.  
\end{proof}

\sssec{} For $\gU(-\theta)$ as above we have the !-restriction functor
$$
Shv(\Mod_{M, \Ran}^{-,\theta,\subset})^{\gL^+(M)_{\Ran}}_{untl}\to Shv(\Mod_{M, \Ran}^{-,\gU(-\theta),\subset})^{\gL^+(M)_{\Ran}}_{untl}.
$$
Using the $(^!, *)$-base change for the cartesian square
$$
\begin{array}{ccc}
\Mod_{M, \Ran}^{-,\gU(-\theta),\subset} & \to & \Mod^{-, \theta, \subset}_{M,\Ran}\\
\downarrow && \downarrow\\ 
 \Mod_M^{-,\gU(-\theta)} & \to & \Mod^{-,\theta}_M.
 \end{array}
$$ 
we see that there is the $*$-extension functor
$$
Shv(\Mod_{M, \Ran}^{-,\gU(-\theta),\subset})^{\gL^+(M)_{\Ran}}_{untl}\to Shv(\Mod_{M, \Ran}^{-,\theta,\subset})^{\gL^+(M)_{\Ran}}_{untl}
$$
However, there seems no reason for the latter functor to be left t-exact. This is why in Corollary~\ref{Cor_3.3.7} we will get the desired isomorphism only in $\cS\cI^{\le 0}_{P,\Ran, untl}(S)$ and not in $\SI^{\le 0}_{P,\Ran, untl}(S)$.

\sssec{Proof of Proposition~\ref{Pp_3.3.6}}
\label{Sect_3.3.11} {\bf Step 1.} Since we know already the unitality of our object by Corollary~\ref{Cor_3.3.5}, there is $K\in Shv(\Mod^{-,\theta}_M)$ equipped with an isomorphism
\begin{equation}
\label{iso_inside_proof_Pp_3.3.6}
(i_{S, \Ran}^{\theta})^!(v_{S,\Ran}^{\theta})^!('\!\IC^{\frac{\infty}{2}}_{P,\Ran})[\<\theta, 2\check{\rho}-2\check{\rho}_M\>]\,\iso\, (p^{\theta}_{\Ran})^!K
\end{equation}
in $Shv(\Mod^{-, \theta, \subset}_{M,\Ran})$. Write $K_{\gU(-\theta)}$ for the $!$-restriction of $K$ to $\Mod^{-, \gU(-\theta)}_M$. It suffices to show that for each $\gU(-\theta)$, $K_{\gU(-\theta)}$ is placed in perverse degrees $>0$. 

\medskip\noindent
{\bf Step 2.} Assume $\gU(-\theta)$ is given by the decomposition
$-\theta=\sum_k n_k\theta_k$, where $n_k>0$, $\theta_k\in \Lambda_{G,P}^{pos,*}$, and $\theta_k$ are pairwise distinct. Set $I=\sqcup_k I_k$, where $I_k=\{1,\ldots, n_k\}$. Then $X^I=\prod_k X^{n_k}$. One has the symmetrization map $X^I\to X^{\gU(\theta)}$. Let $\oo{X}{}^I\to \oo{X}{}^{\gU(\theta)}$ be its restriction to this open part, the latter map is \'etale. 

 Consider the map $\nu: \oo{X}{}^I\to (X^{-\theta}\times\Ran)^{\subset}$ sending $(x_i)_{i\in I}$ to the pair $(D,\cI)$ with $D=\sum_k \theta_k\sum_{i\in I_k} x_i$, $\cI=(x_i)_{i\in I}$. It factors through $(X^{-\theta}\times X^I)^{\subset}$. It gives rise to the map 
$$
\Mod^{-, \gU(-\theta)}_M\times_{\oo{X}{}^{\gU(-\theta)}} \oo{X}{}^I\to \Mod^{-, \gU(-\theta)}_M\times_{X^{-\theta}} (X^{-\theta}\times\Ran)^{\subset}
$$ 
It suffices to show that the $!$-restriction of $(p^{\gU(-\theta)}_{\Ran})^!K$ under the latter map is placed in perverse degrees $>0$. For this by the factorization property of $'\!\IC^{\frac{\infty}{2}}_{P,\Ran}$, we may assume that $\gU(-\theta)$ is the trivial decomposition $-\theta=\theta_1$, that is, consists of one element. 
 In the latter case we argue exactly as in (\cite{DL}, 4.1.7(d)). We are done.
\QED 

\section{Relation between local and global objects}
\label{Sect_Relation between local and global objects}

In this section we define the global versions of the semi-infinite categories $\Inv(\Bunt_P)$ and $\Inv(_{\Ran, \infty}\Bunt_P)$, the former one is equipped with a natural t-structure. We introduce a stack $\cY_{\Ran}$, which is "less global" than $\Bunt_P$ (namely, its global aspects are related only to the Levi subgroup $M$, not to $G$). This allows to define a "less global" version of the semi-infinite category 
\begin{equation}
\label{SI_category_intro_to_Sect_Relation between local and global objects}
\Inv(\cG_{\Ran, M}\times^{\gL^+(M)_{\Ran}}\bar S^0_{P,\Ran})_{untl}.
\end{equation}
We then prove our main results. Namely, Theorem~\ref{Thm_loc_glob_equiv} claims that the global and "less global" semi-infinite categories are naturally equivalent. Theorem~\ref{Th_4.2.12_global_compatibility} provides an explicit relation between $\IC_{\Bunt_P}$ and $'\!\IC^{\frac{\infty}{2}}_{P,\Ran}$ in (\ref{SI_category_intro_to_Sect_Relation between local and global objects}). The latter is stronger than the main result of Gaitsgory's paper (\cite{Gai19Ran}, Theorem~3.3.3).   

\ssec{Global category of invariants} 
\sssec{} Let $I\in fSets$. We define a full subcategory $\Inv(_{I,\infty}\Bunt_P)\subset Shv(_{I,\infty}\Bunt_P)$ by imposing the equivariance condition under the action of some groupoid. This follows the ideas of \cite{Gai08}, \cite{GaLocWhit}, \cite{Gai19Ran}.  

\sssec{} For a collection of distinct $k$-points $\bar y=\{y_1,\ldots, y_m\}$ let $\cD_{\bar y}=\prod_{j=1}^m \cD_{y_j}$ and $\cD_{\bar y}^*=\prod_{j=1}^m \cD^*_{y_j}$. 
Let 
$$
_{I,\infty}\Bunt_{P, \,\goodat\,\bar y}\subset {_{I,\infty}\Bunt_P}
$$ 
be the open substack classifying $(\cI, \cF_G,\cF_M, \kappa)$ as in Section~\ref{Sect_1.3.5_now} given by the properties: 1) if $i\in I$ then $\Gamma_i$ is disjoint from each $y_j$; 2) for $V\in\Rep(G)^{\heartsuit}$ the maps
$$
\kappa^V: V^{U(P)}_{\cF_M}\to V_{\cF_G}
$$
over $X-\Gamma_{\cI}$ have no zeros at each $y_j$. 

 A point of $_{I,\infty}\Bunt_{P, \,\goodat\,\bar y}$ defines a $P$-torsor $\cF_P$ over $\cD_{\bar y}$ equipped with an isomorphism 
$$
\gamma: \cF_P\times_P M\,\iso\, \cF_M\mid_{\cD_{\bar y}}.
$$

 Write $\cG_{\bar y}^{reg}$ (resp., $\cG_{\bar y}^{mer}$) for the group scheme over $_{I,\infty}\Bunt_{P, \,\goodat\,\bar y}$ whose fibre over a point as above is the group of automorphisms of the $P$-torsor $(\cF_M\times_M P)\mid_{\cD_{\bar y}}\to \cD_{\bar y}$  (resp., of the $P$-torsor
$$
(\cF_M\times_M P)\mid_{\cD^*_{\bar y}}\to \cD^*_{\bar y}),
$$ 
which act trivially on the induced $M$-torsor. 
 
Let $_{I,\infty}\Bunt_{P, \bar y}$ be the stack classifying a point of $_{I,\infty}\Bunt_{P, \,\goodat\,\bar y}$ as above together with a isomorphism $\cF_P\,\iso\, \cF_M\times_M P
\mid_{D_{\bar y}}$ of $P$-torsors. This is a $\cG_{\bar y}^{reg}$-torsor over $_{I,\infty}\Bunt_{P, \,\goodat\,\bar y}$. The $\cG_{\bar y}^{reg}$-action on $_{I,\infty}\Bunt_{P, \bar y}$ naturally extends to an action of $\cG_{\bar y}^{mer}$.
 
 Indeed, view $_{I,\infty}\Bunt_{P, \bar y}$ as the stack stack classifying collections:  $\cI\in (X-\bar y)^I$, a $G$-torsor $\cF_G$ over $X-\bar y$, a $M$-torsor $\cF_M$ on $X$ equipped with an isomorphism $\epsilon_{\cF}: \cF_G\,\iso\, \cF_M\times_M G\mid_{\cD_{\bar y}^*}$, for $V\in \Rep(G)^{\heartsuit}$ finite-dimensional the maps
$$ 
\kappa^V: V^{U(P)}_{\cF_M}\to V_{\cF_G}
$$
over $X-(\Gamma_{\cI}\cup \bar y)$ satisfying the Pl\"ucker relations and compatible with $\epsilon_{\cF}$. 

 The group ind-scheme $\cG_{\bar y}^{mer}$ acts on $_{I,\infty}\Bunt_{P, \bar y}$ by changing $\epsilon_{\cF}$. 
  
\sssec{} Define $\Inv(_{I,\infty}\Bunt_{P, \,\goodat\,\bar y})$ as the full $\DG$-subcategory
\begin{equation}
\label{SI_good_at_bar_y}
Shv(_{I,\infty}\Bunt_{P, \bar y})^{\cG_{\bar y}^{mer}}\subset Shv(_{I,\infty}\Bunt_{P, \bar y})^{\cG_{\bar y}^{reg}}\,\iso\, Shv(_{I,\infty}\Bunt_{P, \,\goodat\,\bar y}).
\end{equation}
It is a full subcategory, because the group ind-scheme $\cG_{\bar y}^{mer}$ is ind-prounipotent. 

 As in (\cite{GaLocWhit}, 4.6.5), the inclusion (\ref{SI_good_at_bar_y}) admits a continuous right adjoint $\Av^{\Inv}_*$. 
 
\sssec{} If $\bar y', \bar y''$ are two collections of points, let $\bar y=\bar y'\cup \bar y''$. Then as in (\cite{GaLocWhit}, 4.6.7) one gets the following.

\begin{Lm} 
\label{Lm_4.1.5}
i) The restriction functor under $_{I,\infty}\Bunt_{P, \,\goodat\,\bar y}\to {_{I,\infty}\Bunt_{P, \,\goodat\,\bar y'}}$ sends 
$$
\Inv(_{I,\infty}\Bunt_{P, \,\goodat\,\bar y'})\to \Inv(_{I,\infty}\Bunt_{P, \,\goodat\,\bar y})
$$
ii) The diagram commutes
$$
\begin{array}{ccc}
Shv(_{I,\infty}\Bunt_{P, \,\goodat\,\bar y'}) & \to & Shv(_{I,\infty}\Bunt_{P, \,\goodat\,\bar y})\\
\downarrow && \downarrow\\
\Inv(_{I,\infty}\Bunt_{P, \,\goodat\,\bar y'}) & \to & \Inv(_{I,\infty}\Bunt_{P, \,\goodat\,\bar y}),
\end{array}
$$
where the horizontal arrows are restrictions, and the vertical arrows are the right adjoints to the corresponding inclusions. \QED
\end{Lm}

\begin{Def} We let 
\begin{equation}
\label{Inv_I_infty_Bunt_P}
\Inv(_{I,\infty}\Bunt_P)\subset Shv(_{I,\infty}\Bunt_P)
\end{equation} 
be the full subcategory spanned by those objects $K$ whose restriction under 
$$
_{I,\infty}\Bunt_{P, \,\goodat\,\bar y}\to {_{I,\infty}\Bunt_P}
$$ 
lies in 
$\Inv(_{I,\infty}\Bunt_{P, \,\goodat\,\bar y})$ for any finite collection of points $\bar y$. By Lemma~\ref{Lm_4.1.5}, it suffices to check this condition for $\bar y$ being singletons $\{y\}$. 
\end{Def}

\sssec{} As in (\cite{GaLocWhit}, 4.7.4), we see that the inclusion (\ref{Inv_I_infty_Bunt_P}) admits a continuous right adjoint denoted $\Av_*^{\Inv}$.  

 Replacing $X^I$ by $\Ran$, one similarly gets the Ran versions of the above categories of invariants, in particular $\Inv(_{\Ran,\infty}\Bunt_P)$. 

\sssec{} Proceeding in a similar way, one defines for $\theta\in -\Lambda^{pos}_{G,P}$ the categories $\Inv(\Bunt_P)$, $\Inv(_{\theta}\Bunt_P)$ by equivariance property under the corresponding groupoid. Note that $\Inv(\Bunt_P)\subset Shv(\Bunt_P)$ is the full subcategory of those $K\in Shv(\Bunt_P)$ for which $\omega_{\Ran}\boxtimes K\in \Inv(_{\Ran,\infty}\Bunt_P)$. Here $\omega_{\Ran}\boxtimes K$ is viewed as extended by zero under $\Ran\times\Bunt_P\hook{} {_{\Ran,\infty}\Bunt_P}$. 

 As in (\cite{Gai19Ran}, Section~3.2), one has the following.
\begin{Lm} 
\label{Lm_4.1.9_now}
i) An object of $Shv(\Bunt_P)$ lies in $\Inv(\Bunt_P)$ iff for any $\theta\in -\Lambda^{pos}_{G,P}$ its !-restriciton under $v^{\theta}_{glob}: {_{\theta}\Bunt_P}\to \Bunt_P$ lies in $\Inv(_{\theta}\Bunt_P)$.

\smallskip\noindent
ii) The $!$-pullback under 
$$
p^{\theta}_{glob}:
{_{\theta}\Bunt_P}\to \Mod_{\Bun_M}^{+, -\theta}
$$ 
is fully faithful and yields an equivalence $Shv(\Mod_{\Bun_M}^{+, -\theta})\,\iso\, \Inv(_{\theta}\Bunt_P)$. \QED
\end{Lm}

\sssec{} The right adjoints to $\Inv(_{\theta}\Bunt_P)\hook{} Shv(_{\theta}\Bunt_P)$ and $\Inv(\Bunt_P)\hook{} Shv(\Bunt_P)$ are still denoted $\Av_*^{\Inv}$, they are continuous. For $\theta\in -\Lambda_{G,P}^{pos}$ the functors $(v^{\theta}_{glob})^!, (v^{\theta}_{glob})_*$ commute with $\Av_*^{\Inv}$ (as in \cite{Gai19Ran}, Section~3.2.3). So, for $\theta\in -\Lambda_{G,P}^{pos}$ the diagram commutes
\begin{equation}
\label{diag_for_Sect_4.1.10}
\begin{array}{ccc}
\Inv(\Bunt_P) & \getsup{\Av_*^{\Inv}} & Shv(\Bunt_P)\\
\downarrow\lefteqn{\scriptstyle (v^{\theta}_{glob})^!} && \downarrow\lefteqn{\scriptstyle (v^{\theta}_{glob})^!}\\
\Inv(_{\theta}\Bunt_P) & \getsup{\Av_*^{\Inv}} & Shv(_{\theta}\Bunt_P). 
\end{array}
\end{equation}
The composition $Shv(_{\theta}\Bunt_P)\toup{\Av_*^{\Inv}} \Inv(_{\theta}\Bunt_P)\hook{} Shv(_{\theta}\Bunt_P)$ identifies with the functor $(p^{\theta}_{glob})^*(p^{\theta}_{glob})_*$. 

\sssec{} Write $j_{glob}: \Bun_P\hook{}\Bunt_P$ for the open immersion. Let $\tilde\gq_P: \Bunt_P\to \Bun_M$ be the map sending $(\cF_M, \cF_G,\kappa)$ to $\cF_M$. Write $\tilde p^{\theta}_{glob}: \Bunt_P\times_{\Bun_M} \Mod_{\Bun_M}^{+, -\theta}\to  \Mod_{\Bun_M}^{+, -\theta}$ for the projection on the second factor, where we used $h^{\ra}$ to form the fibred product.   

\begin{Lm}
\label{Lm_4.1.12_now}
 i) For $\theta\in -\Lambda_{G,P}^{pos}$ the partially defined left adjoint to the functor $(v^{\theta}_{glob})^!: Shv(\Bunt_P)\to Shv(_{\theta}\Bunt_P)$ is defined on $\Inv(_{\theta}\Bunt_P)$.

\smallskip\noindent
ii) The resulting functor $\Inv(_{\theta}\Bunt_P)\to Shv(\Bunt_P)$ takes values in $\Inv(\Bunt_P)$. So, we get the adjoint pair $(v^{\theta}_{glob})_!: \Inv(_{\theta}\Bunt_P)\leftrightarrows \Inv(\Bunt_P): (v^{\theta}_{glob})^!$ in $\DGCat_{cont}$. 
\end{Lm}
\begin{proof} i) It suffices to show that for $K\in Shv(\Mod_{\Bun_M}^{+, -\theta})$
the object $(j_{glob}\times\id)_!(p^{\theta}_{glob})^! K$ is defined for the diagram
$$
\Bun_P\times_{\Bun_M} \Mod_{\Bun_M}^{+, -\theta}\;\toup{j_{glob}\times\id} \;\Bunt_P\times_{\Bun_M} \Mod_{\Bun_M}^{+, -\theta}\;\toup{\pr_1}\;\Bunt_P.
$$ 
By (\cite{BG}, 5.1.5), $(j_{glob})_!\omega$ is ULA under $\tilde\gq_P: \Bunt_P\to\Bun_M$. 
By Proposition~\ref{Pp_B.1.3_ULA}, this gives $\pr_1^!(j_{glob})_!\omega,\iso\, (j_{glob}\times\id)_!\omega$. Now Lemma~\ref{Lm_4.1.13} gives for $K\in Shv(\Mod_{\Bun_M}^{+, -\theta})$ a canonical isomorphism
\begin{equation}
\label{iso_for_Lm_4.1.12}
(j_{glob}\times\id)_!(p^{\theta}_{glob})^! K\,\iso\, ((j_{glob}\times\id)_!\omega)\otimes^! (\tilde p^{\theta}_{glob})^!K.
\end{equation}

\smallskip\noindent
ii) By i) in the diagram (\ref{diag_for_Sect_4.1.10}) both functors $\Av_*^{\Inv}$ and $(v^{\theta}_{glob})^!\Av_*^{\Inv}$ admit fully faithful left adjoints. Now by (\cite{Ly}, 9.2.35), 
the left adjoint to $(v^{\theta}_{glob})^!\Av_*^{\Inv}$ factors through the full subcategory $\Inv(\Bunt_P)\hook{} Shv(\Bunt_P)$, and the resulting functor $\Inv(_{\theta}\Bunt_P)\to \Inv(\Bunt_P)$ is the left adjoint to $(v^{\theta}_{glob})^!: \Inv(\Bunt_P)\to\Inv(_{\theta}\Bunt_P)$.
\end{proof}

\begin{Lm} 
\label{Lm_4.1.13}
i) Let $j: Y_0\hook{} Y$ be an open immersion of algebraic stacks locally of finite type, $f: Y\to S$ be a map with $S\in \Sch_{ft}$ smooth such that $f_0=f\comp j$ is smooth. Assume that $j_!\omega$ is ULA over $S$. Then for any $K\in Shv(S)$, $(j_!\omega)\otimes^! f^!K\,\iso\, j_!f_0^! K$ naturally. 

\smallskip\noindent
ii) Let in addition $T\to S$ be a map in $\Sch_{ft}$. Consider the diagram
$$
\begin{array}{ccccc}
T\getsup{\pr_T^0} &Y_0\times_S T &\toup{j_T} &Y\times_S T& \toup{\pr_T} T\\
&\downarrow\lefteqn{\scriptstyle \alpha_0} && \downarrow\lefteqn{\scriptstyle \alpha}\\
& Y_0 & \toup{j} & Y,
\end{array}
$$
where $\pr_T^0, \pr_T,\alpha$ are the projections, and the square is cartesian. For any $K\in Shv(T)$ one has canonically
$$
(j_T)_!(\pr_T^0)^!K\,\iso\, (\alpha^!j_!\omega)\otimes^! \pr_T^!K.
$$
\end{Lm}
\begin{proof} i) This follows from the consequences of the ULA property stated in (\cite{BG}, 5.1.2).  

\smallskip\noindent
ii) By Proposition~\ref{Pp_B.1.3_ULA},
$$
\alpha^*(j_!\omega)\otimes \pr_T^*K[-2\dim S]\,\iso\, \alpha^!(j_!\omega)\otimes^! \pr_T^!K
$$
Using $\omega_{Y_0}\,\iso\, e[2\dim Y_0]$ and the base change, the LHS rewrites as
$$
(j_T)_!(\pr_T^0)^*K[2\dim Y_0-2\dim S]\,\iso\, (j_T)_!(\pr_T^0)^!K,
$$
because $\pr_T^0$ is smooth of relative dimension $\dim Y_0-\dim S$.
\end{proof}

\begin{Cor} For $\theta\in -\Lambda_{G,P}^{pos}$ the partially defined left adjoint $(v^{\theta}_{glob})^*$ to $(v^{\theta}_{glob})_*: Shv(_{\theta}\Bunt_P)\to Shv(\Bunt_P)$ is defined on $\Inv(\Bunt_P)$ and takes values in $\Inv(_{\theta}\Bunt_P)$. This yields an adjoint pair $(v^{\theta}_{glob})^*: \Inv(\Bunt_P)\leftrightarrows \Inv(_{\theta}\Bunt_P): (v^{\theta}_{glob})_*$.
\end{Cor}
\begin{proof} This follows from Lemma~\ref{Lm_4.1.12_now} by a Cousin argument using  (\cite{Ly9}, Lemma~1.8.15, 1.8.16).  
\end{proof}
 
\sssec{} The inclusion $\Inv(_{\theta}\Bunt_P)\hook{} Shv(_{\theta}\Bunt_P)$ is compatible with the perverse t-structure on $Shv(_{\theta}\Bunt_P)$, so $\Inv(_{\theta}\Bunt_P)$ inherits a t-structure, and the latter inclusion is t-exact. The t-structure on $\Inv(_{\theta}\Bunt_P)$ is compatible with filtered colimits.
 
\sssec{} Define the t-structure on $\Inv(\Bunt_P)$ by declaring $K\in \Inv(\Bunt_P)$ connective if it is connective in $Shv(\Bunt_P)$. This is an accessible t-structure by (\cite{HA}, 1.4.4.11). Since $\Inv(\Bunt_P)\hook{} Shv(\Bun_P)$ is right t-exact, its right adjoint $\Av_*^{\Inv}$ is left t-exact. 

 Note that $\Inv(\Bunt_P)^{\le 0}\subset \Inv(\Bunt_P)$ is the smallest full subcategory containing $(v^{\theta}_{glob})_!K$ for $K\in\Inv(_{\theta}\Bunt_P)^{\le 0}$, $\theta\in -\Lambda_{G,P}^{pos}$, stable under colimits and extensions. This implies that given $L\in \Inv(\Bunt_P)$ one has $L\in \Inv(\Bunt_P)^{> 0}$ iff for any $\theta\in -\Lambda_{G,P}^{pos}$, $(v^{\theta}_{glob})^!K\in \Inv(_{\theta}\Bunt_P)^{> 0}$. The t-structure on $\Inv(\Bunt_P)$ is compatible with filtered colimits. 
 
\begin{Lm} One has $\Inv(\Bunt_P)^{>0}=\Inv(\Bunt_P)\cap Shv(\Bunt_P)^{>0}$.
In particular, $\Inv(\Bunt_P)\hook{} Shv(\Bunt_P)$ is t-exact.
\end{Lm}
\begin{proof} Let $K\in \Inv(\Bunt_P)\cap Shv(\Bunt_P)^{>0}$. Then for any $\theta\in -\Lambda_{G,P}^{pos}$, $(v^{\theta}_{glob})^!\in Shv(_{\theta}\Bunt_P)^{>0}\cap \Inv(_{\theta}\Bunt_P)=\Inv(_{\theta}\Bunt_P)^{>0}$. So, $K\in \Inv(\Bunt_P)^{>0}$. 

 Conversely, let $M\in \Inv(\Bunt_P)^{>0}$. Then for $\theta\in -\Lambda_{G,P}^{pos}$, $(v^{\theta}_{glob})^!M\in \Inv(_{\theta}\Bunt_P)^{>0}$, hence $(v^{\theta}_{glob})^!M\in Shv(_{\theta}\Bunt_P)^{>0}$. This gives $M\in Shv(\Bunt_P)^{>0}$. 
\end{proof}
\begin{Lm} Given $K\in \Inv(\Bunt_P)$, one has $K\in \Inv(\Bunt_P)^{\le 0}$ iff for any $\theta\in -\Lambda_{G,P}^{pos}$, $(v^{\theta}_{G,P})^*K\in  \Inv(_{\theta}\Bunt_P)^{\le 0}$.
\end{Lm}
\begin{proof} The argument from (\cite{Gai19Ran}, 2.1.9) applies in this situation also.
\end{proof}
\begin{Rem} The object $\IC_{\Bunt_P}\in \Inv(\Bunt_P)$ is the intermediate extension under $j_{glob}: \Bun_P\hook{}\Bunt_P$ in the sense of the t-structure on $\Inv(\Bunt_P)$.
\end{Rem}

\ssec{The stack $\cY_{\Ran}$}
\label{Sect_The stack cY_Ran}

\sssec{}  Write $\cY_{\Ran}$ for the stack classifying $\cF_M\in\Bun_M, \cI\in\Ran, \cF_G\in\Bun_G$ and an isomorphism $\xi: \cF_M\times_M G\,\iso\, \cF_G\mid_{X-\Gamma_{\cI}}$. For us $\cY_{\Ran}$ is a tool relating $\SI_{P,\Ran}$ with global objects on $\Bunt_P$. Note that $\cY_{\Ran}\in\PreStk_{lft}$. 

 As in Section~\ref{Sect_1.3.6_Hecke_action} one defines the right action of $\Sph_{G,\Ran}$ on $Shv(\cY_{\Ran})$. One defines a left action of $\Sph_{M,\Ran}$ on $Shv(\cY_{\Ran})$ as in Section~\ref{Sect_1.3.7_Hecke_action}. These actions commute.  
 
\sssec{} Recall that $_{\Ran,\infty}\Bunt_P$ classifies: $\cI\in\Ran$, $\cF_G\in\Bun_G, \cF_M\in\Bun_M$ and a collection of injective maps for $V\in\Rep(G)^{\heartsuit}$ finite-dimensional
\begin{equation}
\label{map_kappa^V_Bunt_P}
\kappa^V: V^{U(P)}_{\cF_M}\hook{} V_{\cF_G}
\end{equation}
over $X-\Gamma_{\cI}$ satisfying the Pl\"ucker relations. 

 Let $\pi_{glob,\Ran}: \cY_{\Ran}\to {_{\Ran,\infty}\Bunt_P}$ be the map sending the above point to $(\cF_M, \cF_G,  \cI, \kappa)$, where the maps (\ref{map_kappa^V_Bunt_P}) are obtained from $\xi$. 
 
 Let $\pi_{loc,\Ran}: \cY_{\Ran}\to \gL^+(M)_{\Ran}\backslash \Gr_{G,\Ran}$ be the map sending the above collection to its restriction to $\cD_{\cI}$. 
 
  It is easy to see that both functors in the diagram
$$
Shv(_{\Ran,\infty}\Bunt_P)\,\toup{\pi_{glob,\Ran}^!}\, Shv(\cY_{\Ran})\,\getsup{\pi_{loc,\Ran}^!} \, Shv(\Gr_{G,\Ran})^{\gL^+(M)_{\Ran}}
$$
commute with the actions of $\Sph_{M,\Ran}$ and of $\Sph_{G,\Ran}$. 

\begin{Rem} In the constructible context the functor $(\pi_{glob,\Ran})_!: Shv(\cY_{\Ran})\to Shv(_{\Ran,\infty}\Bunt_P)$ is defined and commutes with the $\Sph_{M,\Ran}$ and $\Sph_{G,\Ran}$-actions.
\end{Rem}   

\sssec{} The prestacks $_{\Ran,\infty}\Bunt_P$ and $\cY_{\Ran}$ have natural unital structures. The map $\pi_{glob, \Ran}$ is equivariant under the actions of the category object $(\Ran\times\Ran)^{\subset}$. So, $\pi_{glob,\Ran}^!$ preserves the corresponding unital categories of sheaves.
 
\sssec{} Recall the $\gL^+(M)_{\Ran}$-torsor $\cG_{\Ran, M}\to \Ran\times\Bun_M$ from Section~\ref{Sect_1.2.14}. We have the $\gL^+(M)_{\Ran}$-torsor $\cG_{\Ran, M}\times_{\Ran} \Gr_{G,\Ran}\to \cY_{\Ran}$, where $\gL^+(M)_{\Ran}$ acts diagonally on the source. 

Write $\Inv(\cY_{\Ran})\subset Shv(\cY_{\Ran})$ for the full subcategory
$$
Shv(\cG_{\Ran, M}\times_{\Ran} \Gr_{G,\Ran})^{H_{\Ran}}\hook{} Shv(\cG_{\Ran, M}\times_{\Ran} \Gr_{G,\Ran})^{\gL^+(M)_{\Ran}}\,\iso\, Shv(\cY_{\Ran}).
$$
Here $H_{\Ran}$ acts diagonally on $\cG_{\Ran, M}\times_{\Ran} \Gr_{G,\Ran}$, and its action on $\cG_{\Ran, M}$ factors through $\gL^+(M)_{\Ran}$. 

 For a $\gL^+(M)_{\Ran}$-invariant locally closed subprestack $V\subset \Gr_{G,\Ran}$ we get the embedding 
$$
\cG_{\Ran, M}\times^{\gL^+(M)_{\Ran}} V\hook{} \cY_{\Ran}.
$$
In particular, we get $\cG_{\Ran, M}\times^{\gL^+(M)_{\Ran}} \bar S^0_{P,\Ran}$ in $\cY_{\Ran}$. 
If moreover $V$ is $\H_{\Ran}$-invariant 
one similarly gets the full subcategory 
$$
\Inv(\cG_{\Ran, M}\times^{\gL^+(M)_{\Ran}} V)\subset Shv(\cG_{\Ran, M}\times^{\gL^+(M)_{\Ran}} V).
$$

 The square is cartesian
$$
\begin{array}{ccc}
\cG_{\Ran, M}\times^{\gL^+(M)_{\Ran}} \bar S^0_{P,\Ran}& \hook{} & \cY_{\Ran}\\ 
\downarrow && \downarrow\lefteqn{\scriptstyle \pi_{glob, \Ran}}\\ 
\Ran\times \Bunt_P & \hook{} & {_{\Ran,\infty}\Bunt_P}
\end{array}
$$ 

\sssec{} For a smooth algebraic group $H$ of finite type let $\Bun_{H-gen}$ be defined as in (\cite{Gai19Ran}, A.1.2). We get the natural map 
$$
\zeta_{gen}: {_{\Ran,\infty}\Bunt_P}\to \Bun_M\times_{\Bun_{M-gen}} \Bun_{P-gen}\times_{\Bun_{G-gen}}\Bun_G.
$$ 
The following square is cartesian
$$
\begin{array}{ccc}
\cG_{\Ran, M}\times^{\gL^+(M)_{\Ran}} \bar S^0_{P,\Ran}& \hook{} & \cY_{\Ran}\\ 
\downarrow\lefteqn{\scriptstyle \bar\pi^0_{glob}} && \downarrow\lefteqn{\scriptstyle \zeta_{gen}\pi_{glob, \Ran}}\\
\Bunt_P & \to & \Bun_M\times_{\Bun_{M-gen}} \Bun_{P-gen}\times_{\Bun_{G-gen}}\Bun_G,
\end{array}
$$
so defining the map $\bar\pi^0_{glob}$. 
\begin{Pp} The map $\zeta_{gen}\pi_{glob, \Ran}$ is universally homologically cointractible.
\end{Pp}
\begin{proof} We decompose $\zeta_{gen}\pi_{glob, \Ran}$ as 
$$
\cY_{\Ran}\toup{\zeta_1} \Bun_M\times_{\Bun_{G-gen}} \Bun_G\toup{\zeta_2} \Bun_M\times_{\Bun_{M-gen}} \Bun_{P-gen}\times_{\Bun_{G-gen}}\Bun_G
$$
We are reduced to Proposition~\ref{Pp_zeta_i_is_UHC} below.
\end{proof}

\begin{Pp} 
\label{Pp_zeta_i_is_UHC}
The map $\zeta_i$ is universally homologically cointractible for $i=1,2$.
\end{Pp}
\begin{proof}
1) The map $\zeta_1$ is pseudo-proper. So, by (\cite{Gai19Ran}, A.2.5) it suffices to show that the fibre of $\zeta_1$ over any field-valued point is homologically constractible. This fibre identifies with $\Ran^{\supset \cI_0}$ for some field-valued point $\cI_0$ of $\Ran$. The claim follows then from (\cite{Gai19Ran}, A.2.7).

\medskip\noindent
2) The map $\zeta_2$ is obtained by base change from $\bar\zeta_2: \Bun_M\to \Bun_M\times_{\Bun_{M-gen}} \Bun_{P-gen}$. So, it suffices to show that the latter map is universally homologically cointractible.

 For $S\in\Sch^{aff}$, an $S$-point of the targer is a collection: $M$-torsor $\cF_M$ on $S\times X$, and a $U(P)_{\cF_M}$-torsor $\cL$ on some domain $\cU\subset S\times X$. 
An isomorphism of $(\cF^1_M, \cL^1, \cU^1)$ with $(\cF^2_M, \cL^2,\cU^2)$ is here an isomorphism $\cF^1_M\,\iso\, \cF^2_M$ on $S\times X$, and a compatible isomorphism $\cL^1\,\iso\, \cL^2$ over some subdomain of $\cU^1\cap\cU^2$. In particular, we may assume $\cU$ affine, and the torsor $\cL$ trivial. Indeed, the group scheme $U(P)_{\cF_M}$ over $S\times X$ is unipotent.  

 Let $a: S\to \Bun_M\times_{\Bun_{M-gen}} \Bun_{P-gen}$ be given with $S\in\Sch^{aff}$. We must show that $\bar\zeta_2$ becomes universally homologically cointractible after base change by $a$. By \cite{DS}, the map $\Spec k\times_{\Bun_{M-gen}}\Bun_M\to \Bun_M$ is a fppf surjection. Since $Shv: \PreStk_{lft}^{op}\to \DGCat_{cont}$ satisfies the fppf descent, we may assume the map $a$ is a composition 
$$
S\to \Spec k\times_{\Bun_{M-gen}}\Bun_M\to \Bun_M\times_{\Bun_{M-gen}} \Bun_{P-gen}.
$$ 
Our claim follows now from the fact that $\Spec k\to \Bun_{U(P)-gen, triv}$ is universally homologically cointractible by (\cite{Gai19Ran}, A.3.3). Here $\Bun_{U(P)-gen, triv}$ is defined in (\cite{Gai19Ran}, A.3.1). Namely, for $S\in\Sch^{aff}$ the groupoid of $S$-points  of $\Bun_{U(P)-gen, triv}$ is the full subgroupoid of those $U(P)$-torsors on some domain $\cU\subset S\times X$, which are isomorphic to the trivial one. 
\end{proof}

\begin{Cor} 
\label{Cor_4.2.7_now}
The map $\bar\pi^0_{glob}$ is universally homologically contractible, so 
\begin{equation}
\label{functor_bar_pi_0_glob_shrick_pullback}
(\bar\pi^0_{glob})^!: Shv(\Bunt_P)\to Shv(\cG_{\Ran, M}\times^{\gL^+(M)_{\Ran}} \bar S^0_{P,\Ran})
\end{equation}
is fully faithful. \QED
\end{Cor}

\sssec{} For $\theta\in -\Lambda_{G,P}^{pos}$ we have the cartesian square
\begin{equation}
\label{square_for_Sect_4.2.8}
\begin{array}{ccc}
\cG_{\Ran, M}\times^{\gL^+(M)_{\Ran}}S^{\theta}_{P, \Ran} & \toup{\id\times v^{\theta}_{S,\Ran}} & \cG_{\Ran, M}\times^{\gL^+(M)_{\Ran}}\bar S^0_{P,\Ran}\\
\downarrow\lefteqn{\scriptstyle \pi^{\theta}_{glob}} && \downarrow\lefteqn{\scriptstyle \bar\pi^0_{glob}}\\
_{\theta}\Bunt_P& \hook{v^{\theta}_{glob}} & \Bunt_P,
\end{array}
\end{equation}
this defines the map $\pi^{\theta}_{glob}$. Thus, $\pi^{\theta}_{glob}$ is also universally homologically contractible. The map $\pi^{\theta}_{glob}$ is the composition
\begin{equation}
\label{diag_decomp_of_pi^theta_glob}
\cG_{\Ran, M}\times^{\gL^+(M)_{\Ran}}S^{\theta}_{P, \Ran} \toup{\pi^{\theta,\subset}_{glob}} {_{\theta}\Bunt_P}\times_{X^{-\theta}} (X^{-\theta}\times\Ran)^{\subset}\;\toup{\id\times\pr^{-\theta}_{\Ran}} \;{_{\theta}\Bunt_P},
\end{equation}
this defines our map $\pi^{\theta,\subset}_{glob}$. 
\begin{Pp} 
\label{Pp_4.2.9_UHC_property}
Both maps in the diagram (\ref{diag_decomp_of_pi^theta_glob}) are universally homologically contractible. 
\end{Pp}
\begin{proof}
For the first map this is done as in (\cite{Gai19Ran}, 3.4.5). For the second map the claim follows from the universal homological contractibility of $\pr^{-\theta}_{\Ran}$. 
\end{proof}

\begin{Pp} 
\label{Pp_4.2.10_now}
The functor (\ref{functor_bar_pi_0_glob_shrick_pullback}) restricts to a functor between full subcategories
$$
\Inv(\Bunt_P)\to \Inv(\cG_{\Ran, M}\times^{\gL^+(M)_{\Ran}}\bar S^0_{P,\Ran}),
$$
which is also fully faithful.
\end{Pp}
\begin{proof} For $K\in \Inv(\Bunt_P)$ it suffices to check that for any $\theta\in -\Lambda_{G,P}^{pos}$, 
$$
(\id\times v^{\theta}_{S,\Ran})^!(\bar\pi^0_{glob})^!\in \Inv(\cG_{\Ran, M}\times^{\gL^+(M)_{\Ran}}S^{\theta}_{P, \Ran})
$$ 
with the notations of (\ref{square_for_Sect_4.2.8}). So, it suffices to show that $(\pi^{\theta}_{glob})^!$ sends $\Inv(_{\theta}\Bunt_P)$ to $\Inv(\cG_{\Ran, M}\times^{\gL^+(M)_{\Ran}}S^{\theta}_{P, \Ran})$. We have the commutative diagram
\begin{equation}
\label{diag_for_proof_of_4.2.10}
\begin{array}{ccc}
\cG_{\Ran, M}\times^{\gL^+(M)_{\Ran}}S^{\theta}_{P, \Ran} & \toup{\id\times\gt^{\theta}_{S,\Ran}} & \cG_{\Ran, M}\times^{\gL^+(M)_{\Ran}} \Mod^{-,\theta, \subset}_{M,\Ran}\\
\downarrow\lefteqn{\scriptstyle \pi^{\theta}_{glob}} && \downarrow\lefteqn{\scriptstyle \id\times \pr^{-\theta}_{\Ran}}\\
_{\theta}\Bunt_P & \toup{p^{\theta}_{glob}} & \Mod_{\Bun_M}^{+, -\theta}
\end{array}
\end{equation}
Our claim follows now from Lemma~\ref{Lm_4.1.9_now} ii) combined with Corollary~\ref{Cor_1.5.13_now}. The fully faithfulness follows from Corollary~\ref{Cor_4.2.7_now}. 
\end{proof}

\begin{Pp} If $K\in Shv(\Bunt_P)$ and $(\bar\pi^0_{glob})^!K\in \Inv(\cG_{\Ran, M}\times^{\gL^+(M)_{\Ran}}\bar S^0_{P,\Ran})$ then $K\in \Inv(\Bunt_P)$.
\end{Pp}
\begin{proof} Let $\theta\in -\Lambda_{G,P}^{pos}$, $K\in Shv(_{\theta}\Bunt_P)$. 
In view of Lemma~\ref{Lm_4.1.9_now}i), it suffices to show that the condition $(\pi^{\theta}_{glob})^!K\in \Inv(\cG_{\Ran, M}\times^{\gL^+(M)_{\Ran}}S^{\theta}_{P, \Ran})$ implies that $K\in \Inv(\Bunt_P)$. Let 
$$
L\in Shv(\cG_{\Ran, M}\times^{\gL^+(M)_{\Ran}} \Mod^{-,\theta, \subset}_{M,\Ran})
$$ 
be equipped with an isomorphism $(\id\times\gt^{\theta}_{S,\Ran})^!L\,\iso\, K$, the notations from the diagram (\ref{diag_for_proof_of_4.2.10}). One has canonically
$$
\cG_{\Ran, M}\times^{\gL^+(M)_{\Ran}} \Mod^{-,\theta, \subset}_{M,\Ran}\,\iso\, \Mod_{\Bun_M}^{+, -\theta}\times_{X^{-\theta}} (X^{-\theta}\times\Ran)^{\subset}.
$$ 
Consider the commutative diagram
$$
\begin{array}{ccc}
\cG_{\Ran, M}\times^{\gL^+(M)_{\Ran}}S^{\theta}_{P, \Ran}\\
\downarrow\lefteqn{\scriptstyle \pi^{\theta,\subset}_{glob}}\\
_{\theta}\Bunt_P\times_{X^{-\theta}} (X^{-\theta}\times\Ran)^{\subset} & \toup{p^{\theta}_{glob}\times\id} & \Mod_{\Bun_M}^{+, -\theta}\times_{X^{-\theta}} (X^{-\theta}\times\Ran)^{\subset}\\
\downarrow\lefteqn{\scriptstyle \id\times \pr^{-\theta}_{\Ran}} && \downarrow\lefteqn{\scriptstyle \id\times \pr^{-\theta}_{\Ran}}\\
_{\theta}\Bunt_P & \toup{p^{\theta}_{glob}} & \Mod_{\Bun_M}^{+, -\theta}
\end{array}
$$
The composition of the left most vertical arrows is $\pi^{\theta}_{glob}$. Using Remark~\ref{Rem_4.2.12_now} below and Proposition~\ref{Pp_4.2.9_UHC_property} we get
\begin{multline*}
K\,\iso\, (\pi^{\theta}_{glob})_!(\pi^{\theta}_{glob})^!K\,\iso\,(\pi^{\theta}_{glob})_!(\pi^{\theta,\subset}_{glob})^!(p^{\theta}_{glob}\times\id)^!L\xrightarrow[\sim]{(\ref{Pp_4.2.9_UHC_property})} \\ (\id\times \pr^{-\theta}_{\Ran})_! (p^{\theta}_{glob}\times\id)^!L
\,\iso\, (p^{\theta}_{glob}\times\id)^!(\id\times \pr^{-\theta}_{\Ran})_!L.
\end{multline*}
Here the last isomorphism follows from the base change, as $\pr^{-\theta}_{\Ran}$ is pseudo-proper by (\cite{G3}, 1.5.4). The result follows now from Lemma~\ref{Lm_4.1.9_now} ii).
\end{proof}

\begin{Rem} 
\label{Rem_4.2.12_now}
Let $f^!: C\to D$ be a fully faithful morphism in $\DGCat_{cont}$. Then for any $K$ in the essential image of $f^!$ the partially defined left adjoint $f_!$ of $f^!$ is defined on $K$.
\end{Rem}

\sssec{} Note that $\pi_{loc, \Ran}^!: Shv(\Gr_{G,\Ran})^{\gL^+(M)_{\Ran}}\to Shv(\cY_{\Ran})$ restricts to a morphism between full subcategories $\SI_{P,\Ran}\to \Inv(\cY_{\Ran})$, which in turn restricts to a morphism between full subcategories
\begin{equation}
\label{functor_pi_loc_Ran^!_first}
(\bar \pi^0_{loc})^!: \SI_{P,\Ran}^{\le 0}(S)\to \Inv(\cG_{\Ran, M}\times^{\gL^+(M)_{\Ran}}\bar S^0_{P,\Ran}).
\end{equation}   
Here we have denoted by 
$$
\bar \pi^0_{loc}: \cG_{\Ran, M}\times^{\gL^+(M)_{\Ran}}\bar S^0_{P,\Ran}\to \gL^+(M)_{\Ran}\backslash \bar S^0_{P,\Ran}
$$ 
the restriction of $\pi_{loc,\Ran}$.

\sssec{} The prestack $\cG_{\Ran, M}\times^{\gL^+(M)_{\Ran}}\bar S^0_{P,\Ran}$ has an evident unital structure, so we have the category $Shv(\cG_{\Ran, M}\times^{\gL^+(M)_{\Ran}}\bar S^0_{P,\Ran})_{untl}$. Note that 
$$
(\bar\pi^0_{glob})^!: Shv(\Bunt_P)\to Shv(Shv(\cG_{\Ran, M}\times^{\gL^+(M)_{\Ran}}\bar S^0_{P,\Ran})
$$ 
takes values in the full subcategory $Shv(\cG_{\Ran, M}\times^{\gL^+(M)_{\Ran}}\bar S^0_{P,\Ran})_{untl}$. 
 
Define $\Inv(\cG_{\Ran, M}\times^{\gL^+(M)_{\Ran}}\bar S^0_{P,\Ran})_{untl}$ as the intersection
$$
\Inv(\cG_{\Ran, M}\times^{\gL^+(M)_{\Ran}}\bar S^0_{P,\Ran})\cap Shv(\cG_{\Ran, M}\times^{\gL^+(M)_{\Ran}}\bar S^0_{P,\Ran})_{untl}
$$
inside $Shv(\cG_{\Ran, M}\times^{\gL^+(M)_{\Ran}}\bar S^0_{P,\Ran})$. This gives the fully faithful functor still denoted
\begin{equation}
\label{functor_bar_pi^0_glob^!_unital}
(\bar\pi^0_{glob})^!: \Inv(\Bunt_P)\to \Inv(\cG_{\Ran, M}\times^{\gL^+(M)_{\Ran}}\bar S^0_{P,\Ran})_{untl}
\end{equation}

\sssec{} Our main results on the local to global compatibility are as follows.
\begin{Thm}
\label{Thm_loc_glob_equiv}
The functor (\ref{functor_bar_pi^0_glob^!_unital}) is an equivalence.
\end{Thm}
\begin{Thm}
\label{Th_4.2.12_global_compatibility} In $\Inv(\cG_{\Ran, M}\times^{\gL^+(M)_{\Ran}}\bar S^0_{P,\Ran})_{untl}$ there is a unique isomorphism 
$$
(\bar \pi^0_{loc})^!('\IC^{\frac{\infty}{2}}_{P,\Ran})\,\iso\, (\bar\pi^0_{glob})^!\IC_{\Bunt_P}[\dim\Bun_P]
$$ 
extending the tautological identification over $\cG_{\Ran, M}\times^{\gL^+(M)_{\Ran}}S^0_{P,\Ran}$. After the shift $[-\dim\Bun_M]$ both sides are the intermediate extensions from $\cG_{\Ran, M}\times^{\gL^+(M)_{\Ran}}S^0_{P,\Ran}$.
\end{Thm} 

\begin{Rem} We don't know if the restriction of $(\bar \pi^0_{loc})^!$ to the full subcategory 
$\SI_{P,\Ran, untl}^{\le 0}(S)$ takes values in $\Inv(\cG_{\Ran, M}\times^{\gL^+(M)_{\Ran}}\bar S^0_{P,\Ran})_{untl}$.
\end{Rem}

\sssec{} We have the cartesian square
$$
\begin{array}{ccc}
\bar S^0_{P,\Ran} & \toup{\iota_{triv}} & \cG_{\Ran, M}\times^{\gL^+(M)_{\Ran}}\bar S^0_{P,\Ran} \\
\downarrow && \downarrow\\
\Spec k & \toup{\cF^0_M} & \Bun_M,
\end{array}
$$
where the right vertical arrow is the projection. The map $\iota_{triv}$ commutes with the actions of $(\Ran\times\Ran)^{\subset}$, so yields the functor 
$$
\iota_{triv}^!: \Inv(\cG_{\Ran, M}\times^{\gL^+(M)_{\Ran}}\bar S^0_{P,\Ran})_{untl}\to \cS\cI^{\le 0}_{P,\Ran, untl}(S).
$$ 
From Theorem~\ref{Th_4.2.12_global_compatibility} one may derive  Corollary~\ref{Cor_3.3.7}. It says in other words that $\iota_{triv}^!(\bar \pi^0_{loc})^!('\IC^{\frac{\infty}{2}}_{P,\Ran})$ identifies with $\IC^{\frac{\infty}{2}}_{P,\Ran}$ up to a shift. 

\ssec{Proof of Theorems~\ref{Thm_loc_glob_equiv} and \ref{Th_4.2.12_global_compatibility}}
 
\sssec{} For $\theta\in -\Lambda_{G,P}^{pos}$ consider the map $\id\times v^{\theta}_{S,\Ran}$ from the diagram (\ref{square_for_Sect_4.2.8}). As in Corollary~\ref{Cor_1.4.11_now_adjoints}, one has the following.
\begin{Lm} 
\label{Lm_4.3.2_now}
i) The left adjoint to 
$$
(\id\times v^{\theta}_{S,\Ran})_*: \Inv(\cG_{\Ran, M}\times^{\gL^+(M)_{\Ran}}S^{\theta}_{P, \Ran})\to \Inv(\cG_{\Ran, M}\times^{\gL^+(M)_{\Ran}}\bar S^0_{P,\Ran})
$$ 
is defined and denoted $(\id\times v^{\theta}_{S,\Ran})^*$. The natural left-lax $Shv(\Bun_M\times\Ran)$-structure on $(\id\times v^{\theta}_{S,\Ran})^*$ is strict.

\smallskip\noindent
ii) The left adjoint to
$$
(\id\times v^{\theta}_{S,\Ran})^!: \Inv(\cG_{\Ran, M}\times^{\gL^+(M)_{\Ran}}\bar S^0_{P,\Ran})\to  \Inv(\cG_{\Ran, M}\times^{\gL^+(M)_{\Ran}}S^{\theta}_{P, \Ran})
$$
is defined and denoted $(\id\times v^{\theta}_{S,\Ran})_!$. The natural left-lax $Shv(\Bun_M\times\Ran)$-structure on $(\id\times v^{\theta}_{S,\Ran})_!$ is strict.
\QED
\end{Lm}

\sssec{}
\label{Sect_4.3.3_now}
 By construction, the functor $(\id\times v^{\theta}_{S,\Ran})^*$ from Lemma~\ref{Lm_4.3.2_now} is the composition
$$
(\id\times\gt^{\theta}_{S,\Ran})^!(\id\times\gt^{\theta}_{P^-,\Ran})_*(\id\times v^{\theta}_{P^-,\Ran})^!
$$
for the diagram obtained from (\ref{diag_for_Sect_1.5.17}) by applying the twist $\cG_{\Ran, M}\times^{\gL^+(M)_{\Ran}}\cdot$. 

\sssec{} The base change isomorphism $(\bar\pi^0_{glob})^!(v^{\theta}_{glob})_*\,\iso\, (\id\times v^{\theta}_{S,\Ran})_*(\pi^{\theta}_{glob})^!$ for the diagram (\ref{square_for_Sect_4.2.8}) yields a natural transformation
\begin{equation}
\label{natural_transformation_for_Pp_4.3.4}
(\id\times v^{\theta}_{S,\Ran})^*(\bar\pi^0_{glob})^!\to (\pi^{\theta}_{glob})^!(v^{\theta}_{glob})^*
\end{equation}
as functors $\Inv(\Bunt_P)\to \Inv(\cG_{\Ran, M}\times^{\gL^+(M)_{\Ran}}S^{\theta}_{P, \Ran})$. 

\begin{Pp} 
\label{Pp_4.3.3_key}
The natural transformation (\ref{natural_transformation_for_Pp_4.3.4}) is an isomorphism.
\end{Pp}

 The proof of Proposition~\ref{Pp_4.3.3_key} is found in Section~\ref{Sect_proof_of_Pp_4.3.3_key}. By a formal Cousin argument Proposition~\ref{Pp_4.3.3_key} gives the following.
\begin{Cor} 
\label{Cor_4.3.5_now}
For $\theta\in -\Lambda_{G,P}^{pos}$ the natural transformation
$$
(\id\times v^{\theta}_{S,\Ran})_!(\pi^{\theta}_{glob})^!\to
(\bar\pi^0_{glob})^!(v^{\theta}_{glob})_!
$$
is an isomorphism of functors from $\Inv(_{\theta}\Bunt_P)$ to $\Inv(\cG_{\Ran, M}\times^{\gL^+(M)_{\Ran}}\bar S^0_{P,\Ran})$. \QED
\end{Cor} 

\sssec{} For each $\theta\in -\Lambda_{G,P}^{pos}$ the prestack $\cG_{\Ran, M}\times^{\gL^+(M)_{\Ran}}S^{\theta}_{P, \Ran}$ also has a natural unital structure. The functors $(\id\times v^{\theta}_{S,\Ran})_*, (\id\times v^{\theta}_{S,\Ran})^!$ preserve the corresponding unital categories. As in Corollary~\ref{Cor_1.4.15_now} one gets the following.

\begin{Cor} An object $K\in Shv(\cG_{\Ran, M}\times^{\gL^+(M)_{\Ran}}\bar S^0_{P,\Ran})$ lies in 
$$
Shv(\cG_{\Ran, M}\times^{\gL^+(M)_{\Ran}}\bar S^0_{P,\Ran})_{untl}
$$ iff for any $\theta\in -\Lambda_{G,P}^{pos}$, $(\id\times v^{\theta}_{S,\Ran})^! K\in Shv(\cG_{\Ran, M}\times^{\gL^+(M)_{\Ran}}S^{\theta}_{P, \Ran})_{untl}$. \QED
\end{Cor}

\sssec{} The prestack $\cG_{\Ran, M}\times^{\gL^+(M)_{\Ran}} \Mod^{-,\theta, \subset}_{M,\Ran}$ has a natural unital structure, so one gets the corresponding category of unital sheaves on it. Define
$$
\Inv(\cG_{\Ran, M}\times^{\gL^+(M)_{\Ran}}S^{\theta}_{P, \Ran})_{untl}
$$
as $\Inv(\cG_{\Ran, M}\times^{\gL^+(M)_{\Ran}}S^{\theta}_{P, \Ran})\cap Shv(\cG_{\Ran, M}\times^{\gL^+(M)_{\Ran}}S^{\theta}_{P, \Ran})_{untl}$.
As in Corollary~\ref{Cor_1.5.13_now} one gets the following.

\begin{Cor} 
\label{Cor_4.3.9_now}
Let $\theta\in -\Lambda_{G,P}^{pos}$. The functor 
$$
(\id\times\gt^{\theta}_{S,\Ran})^!: Shv(\cG_{\Ran, M}\times^{\gL^+(M)_{\Ran}} \Mod^{-,\theta, \subset}_{M,\Ran})\to Shv(\cG_{\Ran, M}\times^{\gL^+(M)_{\Ran}}S^{\theta}_{P, \Ran})
$$
is fully faithful,  and its essential image is $\Inv(\cG_{\Ran, M}\times^{\gL^+(M)_{\Ran}}S^{\theta}_{P, \Ran})$. It induces an equivalence between the corresponding full subcategories
$$
Shv(\cG_{\Ran, M}\times^{\gL^+(M)_{\Ran}} \Mod^{-,\theta, \subset}_{M,\Ran})_{untl}\,\iso\, \Inv(\cG_{\Ran, M}\times^{\gL^+(M)_{\Ran}}S^{\theta}_{P, \Ran})_{untl}
$$
of unital objects. The composition
\begin{multline*}
\Inv(\cG_{\Ran, M}\times^{\gL^+(M)_{\Ran}}S^{\theta}_{P, \Ran})\hook{} Shv(\cG_{\Ran, M}\times^{\gL^+(M)_{\Ran}}S^{\theta}_{P, \Ran})\\
\toup{(\id\times i^{\theta}_{S,\Ran})^!} Shv(\cG_{\Ran, M}\times^{\gL^+(M)_{\Ran}}\Mod^{-,\theta, \subset}_{M,\Ran})
\end{multline*}
is an equivalence. It restricts to an equiavlence between the full subcategories
$$
\Inv(\cG_{\Ran, M}\times^{\gL^+(M)_{\Ran}}S^{\theta}_{P, \Ran})_{untl}\,\iso\, Shv(\cG_{\Ran, M}\times^{\gL^+(M)_{\Ran}}\Mod^{-,\theta, \subset}_{M,\Ran})_{untl}.
$$
\end{Cor}

\sssec{} For $\theta\in -\Lambda_{G,P}^{pos}$ consider the map 
$$
\cG_{\Ran, M}\times^{\gL^+(M)_{\Ran}}\Mod^{-,\theta, \subset}_{M,\Ran}\,\iso\, \Mod^{+, -\theta}_{\Bun_M}\times_{X^{-\theta}}(X^{-\theta}\times\Ran)^{\subset}\;\toup{\id\times \pr^{-\theta}_{\Ran}}\; \Mod^{+, -\theta}_{\Bun_M}.
$$
As in Lemma~\ref{Lm_1.4.31}, one has the following.
\begin{Lm}
\label{Lm_4.3.11}
The functor 
$$
(\id\times \pr^{-\theta}_{\Ran})^!: Shv(\Mod^{+, -\theta}_{\Bun_M})\to Shv(\Mod^{+, -\theta}_{\Bun_M}\times_{X^{-\theta}}(X^{-\theta}\times\Ran)^{\subset})
$$ 
is fully faithful, and its essential image is $Shv(\Mod^{+, -\theta}_{\Bun_M}\times_{X^{-\theta}}(X^{-\theta}\times\Ran)^{\subset})_{untl}$. \QED
\end{Lm}

\sssec{One more t-structure} Define a t-structure on $\Inv(\cG_{\Ran, M}\times^{\gL^+(M)_{\Ran}}\bar S^0_{P,\Ran})_{untl}$ along the lines of Section~\ref{Sect_t-structure_on_SI_untl}. 

 In details, for $\theta\in -\Lambda_{G,P}^{pos}$ we equip 
\begin{equation}
\label{category_Inv_untl_loc_theta_stratum} 
\Inv(\cG_{\Ran, M}\times^{\gL^+(M)_{\Ran}}S^{\theta}_{P, \Ran})_{untl}
\end{equation}
with a t-structure as follows. An object $K$ of the latter category is connective/coconnective iff the image of 
$$
(\id\times i^{\theta}_{S,\Ran})^!K[\<\theta, 2\check{\rho}-2\check{\rho}_M\>]
$$ 
under the equivalence 
\begin{equation}
\label{equiv_for_Sect_4.3.12}
Shv(\cG_{\Ran, M}\times^{\gL^+(M)_{\Ran}}\Mod^{-,\theta, \subset}_{M,\Ran})_{untl}\,\iso\, Shv(\Mod^{+, -\theta}_{\Bun_M})
\end{equation}
of Lemma~\ref{Lm_4.3.11} is connective/coconnective. 
 
 Now the full subcategory of connective objects in $\Inv(\cG_{\Ran, M}\times^{\gL^+(M)_{\Ran}}\bar S^0_{P,\Ran})_{untl}$ is defined as the smallest full subcategory closed under colimits, under extensions, and containing for each $\theta\in -\Lambda_{G,P}^{pos}$ the objects $(\id\times v^{\theta}_{S,\Ran})_!K$ for
$$
K\in (\Inv(\cG_{\Ran, M}\times^{\gL^+(M)_{\Ran}}S^{\theta}_{P, \Ran})_{untl})^{\le 0}.
$$ 
This is an accessible t-structure on $\Inv(\cG_{\Ran, M}\times^{\gL^+(M)_{\Ran}}\bar S^0_{P,\Ran})_{untl}$.

\sssec{} By definition, $K\in \Inv(\cG_{\Ran, M}\times^{\gL^+(M)_{\Ran}}\bar S^0_{P,\Ran})_{untl}$ is coconnective iff for any $\theta\in -\Lambda_{G,P}^{pos}$ the object $(\id\times v^{\theta}_{S,\Ran})^!K\in \Inv(\cG_{\Ran, M}\times^{\gL^+(M)_{\Ran}} S^{\theta}_{P,\Ran})_{untl}$ is coconnective. 

\sssec{Proof of Theorem~\ref{Thm_loc_glob_equiv}} We only have to check that (\ref{functor_bar_pi^0_glob^!_unital}) is essentially surjective. The category $\Inv(\cG_{\Ran, M}\times^{\gL^+(M)_{\Ran}}\bar S^0_{P,\Ran})_{untl}$ is generated by objects of the form
$$
(\id\times v^{\theta}_{S,\Ran})_*K
$$
for $K\in \Inv(\cG_{\Ran, M}\times^{\gL^+(M)_{\Ran}}S^{\theta}_{P, \Ran})_{untl}$. By Corollary~\ref{Cor_4.3.9_now} and Lemma~\ref{Lm_4.3.11}, there is $L\in Shv(\Mod^{+, -\theta}_{\Bun_M})$ equipped with an isomorphism 
$$
(\id\times\gt^{\theta}_{S,\Ran})^!
(\id\times \pr^{-\theta}_{\Ran})^!L\,\iso\, K.
$$
Here the LHS identifies with $(\pi^{\theta}_{glob})^!(p^{\theta}_{glob})^! L$ using the diagram (\ref{diag_for_proof_of_4.2.10}). Since (\ref{square_for_Sect_4.2.8}) is cartesian, one has 
$$
(\id\times v^{\theta}_{S,\Ran})_*K\,\iso\, (\bar \pi^0_{glob})^! (v^{\theta}_{glob})_*(p^{\theta}_{glob})^! L.
$$
So, $\bar K=(v^{\theta}_{glob})_*(p^{\theta}_{glob})^! L\in \Inv(\Bunt_P)$ is equipped with $
(\bar\pi^0_{glob})^!\bar K\,\iso\, (\id\times v^{\theta}_{S,\Ran})_*K
$ as desired. \QED
 
\begin{Pp}
\label{Pp_4.3.16_now}
 The functor (\ref{functor_bar_pi^0_glob^!_unital}) composed with the cohomollogical shift $[d]$ is a t-exact equivalence, so identifies the t-structures on these categories. Here $d$ is the function of the connected component of $\Bunt_P$ given over $\Bunt_P^{\eta}$, $\eta\in\Lambda_{G,P}$ by 
$$
d=d(\eta)=(\dim U)(g-1)+\<\eta, 2\check{\rho}-2\check{\rho}_M\>.
$$  
\end{Pp} 
\begin{proof}
Note that $d$ is the relative dimension of $\tilde\gq_P: \Bunt_P\to \Bun_M$. For $\theta\in -\Lambda_{G,P}^{pos}$ one has the commutative diagram
$$
\begin{array}{ccc}
\Mod^{+, -\theta}_{\Bun_M}\times_{X^{-\theta}}(X^{-\theta}\times\Ran)^{\subset} & \toup{\id\times i^{\theta}_{S,\Ran}}& \cG_{\Ran, M}\times^{\gL^+(M)_{\Ran}} S^{\theta}_{P,\Ran}\\
\downarrow\lefteqn{\scriptstyle\id\times \pr^{-\theta}_{\Ran}} && \downarrow\lefteqn{\scriptstyle \pi^{\theta}_{glob}}\\
\Mod^{+, -\theta}_{\Bun_M} & \toup{s^{\theta}_{glob}} & _{\theta}\Bunt_P,
\end{array}
$$
where we denoted by $s^{\theta}_{glob}$ the natural section of $p^{\theta}_{glob}$. The  map $p^{\theta}_{glob}: {_{\theta}\Bunt_P}\to \Mod^{+, -\theta}_{\Bun_M}$ is smooth of relative dimension $d(\eta)+\<\theta, 2\check{\rho}-2\check{\rho}_M\>$ over the connected component of $\Mod^{+, -\theta}_{\Bun_M}$ classifying $(\cF_M, \cF'_M, \beta, D)$ with $\cF_M\in\Bun_M^{\eta}, \cF'_M\in \Bun_M^{\eta+\theta}$. This explains the shifts by $\<\theta, 2\check{\rho}-2\check{\rho}_M\>$ in the definition of the t-structure on (\ref{category_Inv_untl_loc_theta_stratum}). 

 The above immediately implies that (\ref{functor_bar_pi^0_glob^!_unital}) composed with the cohomollogical shift $[d]$ is left t-exact. To see that it is right t-exact let $\theta\in -\Lambda_{G,P}^{pos}$, $K\in \Inv(_{\theta}\Bunt_P)^{\le 0}$. It suffices to show that 
$$
(\bar\pi^0_{glob})^!(v^{\theta}_{glob})_!K[d]
$$
is connective in $\Inv(\cG_{\Ran, M}\times^{\gL^+(M)_{\Ran}}\bar S^0_{P,\Ran})_{untl}$. 
By Corollary~\ref{Cor_4.3.5_now}, the latter complex identifies with
$(\id\times v^{\theta}_{S,\Ran})_!(\pi^{\theta}_{glob})^!K[d]$. So, it remains to show that
$(\pi^{\theta}_{glob})^!K[d]$ is connective in 
$$
\Inv(\cG_{\Ran, M}\times^{\gL^+(M)_{\Ran}}S^{\theta}_{P,\Ran})_{untl},
$$ 
which means by definition that the image of $(\id\times i^{\theta}_{S,\Ran})^!
(\pi^{\theta}_{glob})^!K[d+\<\theta, 2\check{\rho}-2\check{\rho}_M\>]$ under the equivalence (\ref{equiv_for_Sect_4.3.12}) is connective. This is true, because 
$$
(s^{\theta}_{glob})^![d+\<\theta, 2\check{\rho}-2\check{\rho}_M\>]: \Inv(_{\theta}\Bunt_P)\to Shv(\Mod^{+, -\theta}_{\Bun_M})
$$ 
is a t-exact equivalence.
\end{proof}

\sssec{} Define $\pi^{\theta}_{loc}$ by the cartesian square
$$
\begin{array}{ccc}
\cG_{\Ran, M}\times^{\gL^+(M)_{\Ran}}S^{\theta}_{P,\Ran} & \toup{\id\times v^{\theta}_{S,\Ran}} & \cG_{\Ran, M}\times^{\gL^+(M)_{\Ran}} \bar S^0_{P,\Ran}\\
\downarrow\lefteqn{\scriptstyle \pi^{\theta}_{loc}} && \downarrow\lefteqn{\scriptstyle \bar \pi^0_{loc}}\\
\gL^+(M)_{\Ran}\backslash S^{\theta}_{P,\Ran} & \toup{v^{\theta}_{S,\Ran}} & \gL^+(M)_{\Ran}\backslash \bar S^0_{P,\Ran}.
\end{array}
$$
The $(_*, ^!)$-base change isomorphism for the latter diagram yields a natural transformation 
\begin{equation}
\label{transformation_for_Sect_4.3.16}
(\id\times v^{\theta}_{S,\Ran})^*(\bar\pi^0_{loc})^!\to (\pi^{\theta}_{loc})^!(v^{\theta}_{S,\Ran})^*
\end{equation}
of functors $\SI^{\le 0}_{P,\Ran}(S)\to \Inv(\cG_{\Ran, M}\times^{\gL^+(M)_{\Ran}}S^{\theta}_{P,\Ran})$.
\begin{Lm} 
\label{Lm_4.3.18}
The natural transformation (\ref{transformation_for_Sect_4.3.16}) is an isomorphism.
\end{Lm}
\begin{proof} This follows from the description of the functor (\ref{functor_(v^theta_S,Ran)^*}) in Section~\ref{Sect_1.5.17}, the description of $(\id\times v^{\theta}_{S,\Ran})^*$ in Section~\ref{Sect_4.3.3_now}, and the $(_*, ^!)$-base change.
\end{proof}

\sssec{} It is immediate to see that the proof of Lemma~\ref{Lm_1.4.27} and of Corollary~\ref{Cor_1.5.19}
hold also after the twist $\cG_{\Ran, M}\times^{\gL^+(M)_{\Ran}}\cdot$ (cf. also Section~\ref{Sect_4.4.5_now} below). This gives the following.
\begin{Lm} 
\label{Lm_4.3.20}
Let $\theta\in -\Lambda_{G,P}^{pos}$. The functors in the dual pair 
$$
(\id\times v^{\theta}_{S,\Ran})^*: \Inv(\cG_{\Ran, M}\times^{\gL^+(M)_{\Ran}}\bar S^0_{P,\Ran})\leftrightarrows \Inv(\cG_{\Ran, M}\times^{\gL^+(M)_{\Ran}}S^{\theta}_{P, \Ran}): (\id\times v^{\theta}_{S,\Ran})_*
$$
preserve the correspinding full subcategories of unital sheaves. The same is true for the dual pair
$$
(\id\times v^{\theta}_{S,\Ran})_!: \Inv(\cG_{\Ran, M}\times^{\gL^+(M)_{\Ran}}S^{\theta}_{P, \Ran}) \leftrightarrows \Inv(\cG_{\Ran, M}\times^{\gL^+(M)_{\Ran}}\bar S^0_{P,\Ran}): (\id\times v^{\theta}_{S,\Ran})^!.
$$
\QED
\end{Lm}

\sssec{} 
\label{Sect_5.3.21_now}
Denote by 
$$
\pi^{\theta}_{\div, loc}: \Mod_{\Bun_M}^{+, -\theta}\to \gL^+(M)_{\theta}\backslash \Mod_M^{-,\theta}
$$ 
the map given by restriction to $\cD_D$ of an object $(\cF_M, \cF'_M, D, \cF_M\,\iso\, \cF'_M\mid_{X-\supp(D)})$ of the source. 
\index{$\pi^{\theta}_{\div, loc}$, Section~\ref{Sect_5.3.21_now}}

\sssec{Proof of Theorem~\ref{Th_4.2.12_global_compatibility}} By Proposition~\ref{Pp_4.3.16_now}, $(\bar\pi^0_{glob})^!\IC_{\Bunt_P}[d]$ is the intermediate extension in $\Inv(\cG_{\Ran, M}\times^{\gL^+(M)_{\Ran}}\bar S^0_{P,\Ran})_{untl}$ from $\cG_{\Ran, M}\times^{\gL^+(M)_{\Ran}} S^0_{P,\Ran}$. Apriori, 
$$
(\bar \pi^0_{loc})^!('\IC^{\frac{\infty}{2}}_{P,\Ran})[-\dim\Bun_M]\in \Inv(\cG_{\Ran, M}\times^{\gL^+(M)_{\Ran}}\bar S^0_{P,\Ran}).
$$
We show that the latter object lies in the unital subcategory and also is the intermediate extension of its restriction to the same open stratum. 

\medskip\noindent 
{\bf Step 1}. We claim that 
$$
(\bar \pi^0_{loc})^!('\IC^{\frac{\infty}{2}}_{P,\Ran})[-\dim\Bun_M]\in
(\Inv(\cG_{\Ran, M}\times^{\gL^+(M)_{\Ran}}\bar S^0_{P,\Ran})_{untl})^{\le 0},
$$
and its $*$-restriction to non-open strata is placed in strictly negative degrees.

Indeed, by Lemma~\ref{Lm_4.3.20} and a version of Corollary~\ref{Cor_1.5.21}, it suffices to show that for $\theta\in -\Lambda_{G,P}^{pos}$, 
$$
(\id\times v^{\theta}_{S,\Ran})^*(\bar \pi^0_{loc})^!('\IC^{\frac{\infty}{2}}_{P,\Ran})[-\dim\Bun_M]
\in (\Inv(\cG_{\Ran, M}\times^{\gL^+(M)_{\Ran}} S^{\theta}_{P,\Ran})_{untl})^{\le 0},
$$ 
and the latter object lies in $(\Inv(\cG_{\Ran, M}\times^{\gL^+(M)_{\Ran}} S^{\theta}_{P,\Ran})_{untl})^{< 0}$ unless $\theta=0$.

By Lemma~\ref{Lm_4.3.18}, the latter object identifies with
\begin{equation}
\label{object_for_proof_in_4.3.21}
(\pi^{\theta}_{loc})^!(v^{\theta}_{S,\Ran})^*('\IC^{\frac{\infty}{2}}_{P,\Ran})[-\dim\Bun_M].
\end{equation}
Consider the diagram
$$
\begin{array}{ccc}
&& \cG_{\Ran, M}\times^{\gL^+(M)_{\Ran}} S^{\theta}_{P,\Ran}\\
&& \downarrow\lefteqn{\scriptstyle \id\times\gt^{\theta}_{S,\Ran}}\\
\Mod_{\Bun_M}^{+, -\theta} & \getsup{\id\times \pr^{-\theta}_{\Ran}} & 
\cG_{\Ran, M}\times^{\gL^+(M)_{\Ran}} \Mod^{-,\theta, \subset}_{M,\Ran}\\
\downarrow\lefteqn{\scriptstyle \pi^{\theta}_{\div, loc}} && \downarrow\lefteqn{\scriptstyle \pi^{\theta}_{\Mod, loc}}\\
\gL^+(M)_{\theta}\backslash \Mod_M^{-,\theta} & \getsup{\tau^{\theta}} & \gL^+(M)_{\Ran}\backslash \Mod^{-,\theta, \subset}_{M,\Ran}
\end{array} 
$$ 
where $\pi^{\theta}_{\Mod, loc}$ is given by restrictions to $\cD_{\cI}$, and $\tau^{\theta}$ is defined in Section~\ref{Sect_1.6.2}. One has 
$$
\gt^{\theta}_{S,\Ran}\pi^{\theta}_{loc}=\pi^{\theta}_{\Mod, loc} (\id\times \gt^{\theta}_{S,\Ran}).
$$ 
So, by Corollary~\ref{Cor_3.3.5}, (\ref{object_for_proof_in_4.3.21}) identifies with
$$
(\id\times \gt^{\theta}_{S,\Ran})^! (\pi^{\theta}_{\Mod, loc})^! (\tau^{\theta})^!\Sat_{M, X^{-\theta}}(\Fact(\cO(U(\check{P}))_{<0})_{\theta})
[-\dim\Bun_M-\<\theta, 2\check{\rho}-2\check{\rho}_M\>].
$$
The latter object is unital by  Lemma~\ref{Lm_4.3.11}. 

 It remains to show that 
\begin{equation}
\label{complex_for_Sect_4.3.21}
(\pi^{\theta}_{\div, loc})^!\Sat_{M, X^{-\theta}}(\Fact(\cO(U(\check{P}))_{<0})_{\theta})[-\dim\Bun_M]
\end{equation}
is placed in perverse degrees $\le 0$, and the inequality is strict unless $\theta=0$. For $\theta=0$ the map $\pi^{\theta}_{\div, loc}$ is the map $\Bun_M\to \Spec k$, and (\ref{complex_for_Sect_4.3.21}) is $\IC_{\Bun_M}$. If $\theta\ne 0$ then the claim follows from Lemma~\ref{Lm_C.3.44} as in Corollary~\ref{Cor_3.3.7}. Namely, $\Sat_{M, X^{-\theta}}(\Fact(\cO(U(\check{P}))_{<0})_{\theta})$ is placed in perverse degrees $<0$. 

\medskip\noindent
{\bf Step 2} Let $\theta\in -\Lambda_{G,P}^{pos,*}$. It remains to show that 
$$
(\id\times v^{\theta}_{S,\Ran})^!(\bar \pi^0_{loc})^!('\IC^{\frac{\infty}{2}}_{P,\Ran})[-\dim\Bun_M]\in \Inv(\cG_{\Ran, M}\times^{\gL^+(M)_{\Ran}}S^{\theta}_{P, \Ran})_{untl}
$$ 
is strictly coconnective. 

 Since our object is unital, there is $\cK\in Shv(\Mod^{+, -\theta}_{\Bun_M})$ equipped with an isomorphism
$$
(\id\times\pr^{-\theta}_{\Ran})^!\cK\,\iso\, (\id\times i^{\theta}_{S,\Ran})^!(\id\times v^{\theta}_{S,\Ran})^!(\bar \pi^0_{loc})^!('\IC^{\frac{\infty}{2}}_{P,\Ran})[-\dim\Bun_M+\<\theta, 2\check{\rho}-2\check{\rho}_M\>]
$$
for $\id\times\pr^{-\theta}_{\Ran}: \cG_{\Ran, M}\times^{\gL^+(M)_{\Ran}}\Mod^{-,\theta,\subset}_{M,\Ran}\to \Mod^{+, -\theta}_{\Bun_M}$. It remains to show that $\cK$ is placed in perverse degrees $>0$. 

 We argue as in Proposition~\ref{Pp_3.3.6}. Namely, consider a stratum $\oo{X}{}^{\gU(-\theta)}\subset X^{-\theta}$ as in Section~\ref{Sect_3.3.8_stratification}. Set 
$$
\Mod_{\Bun_M}^{+, \gU(-\theta)}=\Mod^{+, -\theta}_{\Bun_M}\times_{X^{-\theta}} \oo{X}{}^{\gU(-\theta)}.
$$ 
Denote by $\cK^{\gU(-\theta)}$ the !-restriction of $\cK$ to $\Mod_{\Bun_M}^{+, \gU(-\theta)}$. It suffices to show that $\cK^{\gU(-\theta)}$ is placed in perverse degrees $>0$. 

 Assume $\gU(-\theta)$ is given by the decomposition $-\theta=\sum_k n_k\theta_k$, where $n_k>0$, $\theta_k\in \Lambda_{G,P}^{pos, *}$, and $\theta_k$ are pairwise distinct. Set $I=\sqcup_k I_k$, where $I_k=\{1,\ldots, n_k\}$, so $X^I=\prod_k X^{n_k}$. We have the etale symmetrization map $\oo{X}{}^I\to \oo{X}{}^{\gU(-\theta)}$. Let us show that the !-restriction of $\cK^{\gU(-\theta)}$ to 
$$
\Mod_{\Bun_M}^{+, \gU(-\theta)}\times_{\oo{X}{}^{\gU(-\theta)}} \oo{X}{}^I
$$ 
is placed in perverse degrees $>0$.
 
 Let $\gL^+(M)_{\oo{I}}$ denote the restriction of $\gL^+(M)_I$ to $\oo{X}{}^I$. Consider the map 
$$
\nu: \oo{X}{}^I\to (X^{-\theta}\times\Ran)^{\subset}
$$ 
sending $(x_i)_{i\in I}$ to $(D,\cI)$ with $D=\sum_k \theta_k \sum_{i\in I_k} x_i$ and $\cI=(x_i)_{i\in I}$. It gives rise to the composition
\begin{equation}
\label{composition_for_Step2_proof_of_Th_4.2.19}  
\gL^+(M)_{\oo{I}}\backslash (\Mod^{-,\gU(-\theta)}_M\times_{\oo{X}{}^{\gU(-\theta)}} \oo{X}{}^I)\to \gL^+(M)_{\Ran}\backslash \Mod^{-,\theta,\subset}_{M,\Ran}\toup{v^{\theta}_{S,\Ran} i^{\theta}_{S,\Ran}}
 \gL^+(M)_{\Ran}\backslash \bar S^0_{P,\Ran}
\end{equation}
We have seen in the proof of Proposition~\ref{Pp_3.3.6} that the $!$-restriction under (\ref{composition_for_Step2_proof_of_Th_4.2.19}) of $'\IC^{\frac{\infty}{2}}_{P,\Ran}[\<\theta, 2\check{\rho}-2\check{\rho}_M\>]
$ is placed in perverse degrees $>0$. 
Consider the natural map 
$$
a: \Mod_{\Bun_M}^{+, \gU(-\theta)}\times_{\oo{X}{}^{\gU(-\theta)}} \oo{X}{}^I\to \gL^+(M)_{\oo{I}}\backslash (\Mod^{-, \gU(-\theta)}_M\times_{\oo{X}{}^{\gU(-\theta)}} \oo{X}{}^I)
$$
given by restricting the corresponding datum to $\cD_{\cI}$ for $\cI\in \oo{X}{}^I$.
The functor $a^![-\dim\Bun_M]$ is t-exact for the perverse t-structures. Our claim follows.
\QED

\ssec{Proof of Proposition~\ref{Pp_4.3.3_key}}
\label{Sect_proof_of_Pp_4.3.3_key}

\sssec{} The argument is close to (\cite{Gai19Ran}, 3.6.3). For $\theta\in -\Lambda_{G,P}^{pos}$ consider the diagram
\begin{equation}
\label{diag_for_Sect_4.1.1}
\begin{array}{ccc}
\cG_{\Ran, M}\times^{\gL^+(M)_{\Ran}} \Mod^{-, \theta,\subset}_{M,\Ran} & \toup{\id\times i^{\theta}_{S,\Ran}} &\cG_{\Ran, M}\times^{\gL^+(M)_{\Ran}}S^{\theta}_{P, \Ran}\\
 & \searrow\lefteqn{\scriptstyle \id} & \downarrow\lefteqn{\scriptstyle \id\times \gt^{\theta}_{S,\Ran}}\\
&& \cG_{\Ran, M}\times^{\gL^+(M)_{\Ran}} \Mod^{-, \theta,\subset}_{M,\Ran} 
\end{array}
\end{equation}

The following is obtained similarly to Lemma~\ref{Lm_1.4.10_now}.
\begin{Lm}
\label{Lm_4.4.2}
Let $\theta\in -\Lambda_{G,P}^{pos}$. For the diagram (\ref{diag_for_Sect_4.1.1})
the functor 
$$
(\id\times \gt^{\theta}_{S,\Ran})^!: Shv(\cG_{\Ran, M}\times^{\gL^+(M)_{\Ran}} \Mod^{-, \theta,\subset}_{M,\Ran})\to Shv(\cG_{\Ran, M}\times^{\gL^+(M)_{\Ran}}S^{\theta}_{P, \Ran})
$$ 
is fully faithful, its essential image is the full subcategory
$\Inv(\cG_{\Ran, M}\times^{\gL^+(M)_{\Ran}}S^{\theta}_{P, \Ran})$. The composition
\begin{multline*}
\Inv(\cG_{\Ran, M}\times^{\gL^+(M)_{\Ran}}S^{\theta}_{P, \Ran})\hook{} Shv(\cG_{\Ran, M}\times^{\gL^+(M)_{\Ran}}S^{\theta}_{P, \Ran})\toup{(\id\times i^{\theta}_{S,\Ran})^!} \\
Shv(\cG_{\Ran, M}\times^{\gL^+(M)_{\Ran}} \Mod^{-, \theta,\subset}_{M,\Ran})
\end{multline*}
is an equivalence. \QED
\end{Lm}

\sssec{} Let $\cF\in \Inv(\Bunt_P)$. By Lemma~\ref{Lm_4.4.2}, we have to show that the natural map 
\begin{equation}
\label{map_Sect_4.4.3} 
(\id\times i^{\theta}_{S,\Ran})^!(\id\times v^{\theta}_{S,\Ran})^*(\bar\pi^0_{glob})^!\cF\to (\id\times i^{\theta}_{S,\Ran})^!(\pi^{\theta}_{glob})^!(v^{\theta}_{glob})^*\cF 
\end{equation}
is an equivalence. By construction, the functor 
$$
(\id\times i^{\theta}_{S,\Ran})^!(\id\times v^{\theta}_{S,\Ran})^*: \Inv(\cG_{\Ran, M}\times^{\gL^+(M)_{\Ran}}\bar S^0_{P,\Ran})\to Shv(\cG_{\Ran, M}\times^{\gL^+(M)_{\Ran}} \Mod^{-, \theta,\subset}_{M,\Ran})
$$ 
rewrites as
$$
(\id\times \gt^{\theta}_{P^-,\Ran})_*(\id\times v^{\theta}_{P^-,\Ran})^!
$$
for the diagram
\begin{equation}
\label{diag_for_Sect_4.4.3}
\begin{array}{ccc}
\cG_{\Ran, M}\times^{\gL^+(M)_{\Ran}} \Mod^{-, \theta,\subset}_{M,\Ran} &\getsup{\id\times  \gt^{\theta}_{P^-,\Ran}} & \cG_{\Ran, M}\times^{\gL^+(M)_{\Ran}} (\Gr^{\theta}_{P^-,\Ran}\cap \bar S^0_{P,\Ran})\\ & \swarrow\lefteqn{\scriptstyle \id\times v^{\theta}_{P^-,\Ran}} \\ \cG_{\Ran, M}\times^{\gL^+(M)_{\Ran}}  \bar S^0_{\Ran},
\end{array}
\end{equation}
compare with Section~\ref{Sect_1.5.17}. 

\sssec{} Recall the definition of the Zastava space $\cZ^{\theta}$. It classifies collections: 
\begin{itemize}
\item a point 
$$
(\cF_M, \cF'_M, D, \beta: \cF_M\,\iso\, \cF'_M\mid_{X-\supp(D)})\in \Mod_{\Bun_M}^{+,-\theta},
$$ 
\item a $P^-$-torsor $\cF'_{P^-}$ on $X$ with an isomorphism ${\cF'_{P^-}\times_{P^-} M}\,\iso\, \cF'_M$ on $X$, 
\item a trivialization $\tilde\beta: \cF_M\times_M P^-\,\iso\, \cF'_{P^-}\mid_{X-\supp(D)}$ whose extension of scalars via $P^-\to M$ is $\beta$. It is required that for any $V\in\Rep(G)^{\heartsuit}$ finite-dimensional the map
$$
\kappa^V: V^{U(P)}_{\cF_M}\toup{\tilde\beta} V_{\cF'_{P^-}},
$$
initially defined over $X-\supp(D)$, is regular over $X$. 
\end{itemize}
Write $q: \cZ^{\theta}\to \Bunt_P$ for the map sending the above point to $(\cF_M, \cF'_G, \kappa)$, where $\cF'_G=\cF'_{P^-}\times_{P^-} G$, and $\kappa$ is obtained from $\tilde\beta$ as above. Write $p: \cZ^{\theta}\to \Mod_{\Bun_M}^{+,-\theta}$ for the projection sending the above point to $(\cF_M, \cF'_M, D, \beta)$.

 Let $\gs: \Mod_{\Bun_M}^{+,-\theta}\to \cZ^{\theta}$ be the section of $p$ sending $(\cF_M, \cF'_M, D, \beta)$ to itself together with $(\cF'_{P^-}, \tilde\beta)$, where $\cF'_{P^-}=\cF'_M\times_M P^-$, and $\tilde\beta$ is obtained from $\beta$ via the extensions of scalars. Set $\oo{\cZ}{}^{\theta}=q^{-1}(\Bun_P)$. Let $j_{\cZ}: \oo{\cZ}{}^{\theta}\hook{} \cZ^{\theta}$ be the open immersion. 
 
\sssec{} 
\label{Sect_4.4.5_now}
We have a natural identification of $\id\times \gt^{\theta}_{P^-, \Ran}$ in (\ref{diag_for_Sect_4.4.3}) with the map
$$
p\times\id: \cZ^{\theta}\times_{X^{-\theta}} (X^{-\theta}\times\Ran)^{\subset}\to \Mod_{\Bun_M}^{+,-\theta}\times_{X^{-\theta}} (X^{-\theta}\times\Ran)^{\subset}
$$
So, the LHS of (\ref{map_Sect_4.4.3}) rewrites as
\begin{equation}
\label{complex_for_Sect_4.4.5}
(\id\times \gt^{\theta}_{P^-,\Ran})_*(\id\times v^{\theta}_{P^-,\Ran})^!(\bar\pi^0_{glob})^!\cF\,\iso\, (p\times\id)_*(\id\times \pr^{-\theta}_{\Ran})^!q^!\cF.
\end{equation}
We used that $\bar\pi^0_{glob}\comp (\id\times v^{\theta}_{P^-,\Ran})$ equals the composition in the top line of the diagram
$$
\begin{array}{cccc}
\cZ^{\theta}\times_{X^{-\theta}} (X^{-\theta}\times\Ran)^{\subset} & \toup{\id\times \pr^{-\theta}_{\Ran}} & \cZ^{\theta} & \toup{q} \Bunt_P\\
\downarrow\lefteqn{\scriptstyle p\times\id} && \downarrow\lefteqn{\scriptstyle p}\\
\Mod_{\Bun_M}^{+, -\theta}\times_{X^{-\theta}} (X^{-\theta}\times\Ran)^{\subset} & \toup{\id\times \pr^{-\theta}_{\Ran}} & \Mod_{\Bun_M}^{+, -\theta}
\end{array}
$$
By base change, (\ref{complex_for_Sect_4.4.5}) rewrites as 
$$
(\id\times \pr^{-\theta}_{\Ran})^!p_*q^!\cF. 
$$
By the contraction principle along the fibres of $p$ (\cite{DG1}, 3.2.2), the latter complex rewrites as $(\id\times \pr^{-\theta}_{\Ran})^!\gs^*q^!\cF$. 

\sssec{} Now we rewrite the RHS of (\ref{map_Sect_4.4.3}). We have the cartesian square
\begin{equation}
\label{cart_square_for_Sect_4.4.6}
\begin{array}{ccc}
\Mod_{\Bun_M}^{+, -\theta} & \toup{\gs} & \cZ^{\theta}\\
\downarrow\lefteqn{\scriptstyle q^{\theta}} && \downarrow\lefteqn{\scriptstyle q}\\
_{\theta}\Bunt_P & \toup{v^{\theta}_{glob}} &\Bunt_P,
\end{array}
\end{equation}
here the map $q^{\theta}$ is defined so that it sends a point $(\cF_M, \cF'_M, D,\beta)$ as above to itself together with $\cF'_P=\cF'_M\times_M P$. 
 Note that 
$$
\pi^{\theta}_{glob}(\id\times i^{\theta}_{S,\Ran})
$$ 
equals the composition
$$
\Mod_{\Bun_M}^{+, -\theta}\times_{X^{-\theta}} (X^{-\theta}\times\Ran)^{\subset}\toup{\id\times \pr^{-\theta}_{\Ran}}\Mod_{\Bun_M}^{+, -\theta}\toup{q^{\theta}} {_{\theta}\Bunt_P}.
$$
So, the RHS of (\ref{map_Sect_4.4.3}) identifies with
$$
(\id\times \pr^{-\theta}_{\Ran})^!(q^{\theta})^!(v^{\theta}_{glob})^*\cF 
$$
We conclude that the map (\ref{map_Sect_4.4.3}) is obtained from the natural transformation
\begin{equation}
\label{transformation_for_Sect_4.4.6}
\gs^*q^!\cF\to (q^{\theta})^!(v^{\theta}_{glob})^*\cF
\end{equation}
coming from the cartesian square (\ref{cart_square_for_Sect_4.4.6}) by applying $(\id\times \pr^{-\theta}_{\Ran})^!$. 

 So, it suffices to show that for $\cF\in \Inv(\Bunt_P)$ the natural transformation (\ref{transformation_for_Sect_4.4.6}) is an isomorphism. 
 
\sssec{} The category $\Inv(\Bunt_P)$ is generated by objects $(v^{\theta'}_{glob})_!(p^{\theta'}_{glob})^!K$ for $\theta'\in -\Lambda_{G,P}^{pos}$, $K\in Shv(\Mod^{+, -\theta'}_{\Bun_M})$. We check that  (\ref{transformation_for_Sect_4.4.6}) is an isomorphism for $\cF=(v^{\theta'}_{glob})_!(p^{\theta'}_{glob})^!K$.

 The RHS of (\ref{transformation_for_Sect_4.4.6}) vanishes unless $\theta=\theta'$, and in the latter case identifies with $K$. Consider the diagram
$$
\Bun_P\times_{\Bun_M} \Mod^{+,-\theta'}_{\Bun_M}\;\,\hook{j_{glob}\times\id} \;\,\Bunt_P\times_{\Bun_M} \Mod^{+,-\theta'}_{\Bun_M}\;\,\toup{\bar v^{\theta'}_{glob}} \;\,\Bunt_P,
$$ 
where we used $h^{\ra}$ to form the fibred products, and $\bar v^{\theta'}_{glob}$ was defined in Section~\ref{Sect_1.3.19_now}. By (\ref{iso_for_Lm_4.1.12}),
$$
(j_{glob}\times\id)_!(p^{\theta'}_{glob})^!K\,\iso\, ((j_{glob}\times\id)_!\omega)\otimes^! (\tilde p^{\theta'}_{glob})^!K
$$
So, 
$$
\cF\,\iso\, (\bar v^{\theta'}_{glob})_*(((j_{glob}\times\id)_!\omega)\otimes^! (\tilde p^{\theta'}_{glob})^!K).
$$ 
Set $\theta''=\theta-\theta'$. The base change of $\bar v^{\theta'}_{glob}$ by $q: \cZ^{\theta}\to \Bunt_P$ is empty unless $\theta''\in -\Lambda^{pos}_{G,P}$. From now on assume $\theta''\in -\Lambda^{pos}_{G,P}$, as otherwise both sides of (\ref{transformation_for_Sect_4.4.6}) vanish and our claim is evident.

  Consider the diagram
$$
\begin{array}{cccc}
\Mod^{+,-\theta'}_{\Bun_M}\;\getsup{\pr_2}&
\cZ^{\theta''}\times_{\Bun_M} \Mod^{+,-\theta'}_{\Bun_M}& \toup{u} & \cZ^{\theta}\\
& \downarrow\lefteqn{\scriptstyle q\times\id} && \downarrow\lefteqn{\scriptstyle q}\\
& \Bunt_P\times_{\Bun_M} \Mod^{+,-\theta'}_{\Bun_M} & \toup{\bar v^{\theta'}_{glob}} & \Bunt_P
\end{array}
$$
where we used the maps $h^{\ra}: \Mod^{+,-\theta'}_{\Bun_M}\to\Bun_M$ to form the fibred products, the square is cartesian (so defining the map $u$), and $\pr_2$ is the projection. We get
$$
(q \times\id)^! (((j_{glob}\times\id)_!\omega)\otimes^! (\tilde p^{\theta'}_{glob})^!K)\,\iso\, \pr_2^!K\otimes^! (q \times\id)^! (j_{glob}\times\id)_!\omega.
$$
As in (\cite{DL}, 4.2.25), one has canonically
$$
(j_{\cZ}\times\id)_!\omega\,\iso\, (q \times\id)^! (j_{glob}\times\id)_!\omega
$$
for  $j_{\cZ}\times\id: \oo{\cZ}{}^{\theta''}\times_{\Bun_M} \Mod^{+,-\theta'}_{\Bun_M}\hook{} \cZ^{\theta''}\times_{\Bun_M} \Mod^{+,-\theta'}_{\Bun_M}$. Thus, the LHS of (\ref{transformation_for_Sect_4.4.6}) rewrites as
\begin{equation}
\label{complex_for_Sect_4.4.7}
\gs^*u_!(\pr_2^!K\otimes^! (j_{\cZ}\times\id)_!\omega)
\end{equation}
Recall that the composition $\cZ^{\theta''}\to \Bunt_P\to \Bun_M$ is a locally trivial fibration in the smooth topology, and for this reason $(j_{\cZ})_!\omega$ is ULA under the latter composition. So, by Lemma~\ref{Lm_4.1.13}, 
$$
\pr_2^!K\otimes^! (j_{\cZ}\times\id)_!\omega\,\iso\, (j_{\cZ}\times\id)_!(\pr_2^0)!K
$$
with $\pr_2^0=\pr_2(j_{\cZ}\times\id)$. 

 
 Consider the diagram
$$
\oo{\cZ}{}^{\theta''}\times_{\Bun_M} \Mod^{+,-\theta'}_{\Bun_M}\, \toup{u^0} \, \cZ^{\theta}\,\getsup{\gs} \, \Mod^{+, -\theta}_{\Bun_M}, 
$$
where $u^0$ is the restriction of $u$. The limit of this diagram is empty unless $\theta''=0$.  
It follows that the LHS of (\ref{transformation_for_Sect_4.4.6}) vanishes unless $\theta''=0$. 
In the case $\theta=\theta'$ the map $u$ identifies with $\gs$, so that (\ref{complex_for_Sect_4.4.7}) becomes 
$$
\gs^*\gs_!K\,\iso\, K.
$$
So, (\ref{transformation_for_Sect_4.4.6}) is an isomorphism for $\cF$ of our form. Proposition~\ref{Pp_4.3.3_key} is proved. \QED

\ssec{Restriction to strata revisited}

\sssec{} For $\theta\in -\Lambda_{G,P}^{pos}$ recall the morphisms
$$
\Mod^{-,\theta}_M/\gL^+(M)_{\theta}\getsup{\pi^{\theta}_{\div, loc}} \Mod^{+, -\theta}_{\Bun_M} \getsup{p^{\theta}_{glob}} {_{\theta}\Bunt_P}\toup{v^{\theta}_{glob}} \Bunt_P
$$
Here $\pi^{\theta}_{\div, loc}$ is defined in Section~\ref{Sect_5.3.21_now}, and $p^{\theta}_{glob}$ in Section~\ref{Sect_1.3.19_now}. Our Theorems~\ref{Thm_loc_glob_equiv}, \ref{Th_4.2.12_global_compatibility} combined with Corollary~\ref{Cor_3.3.5} yield the following. 

\begin{Thm} 
\label{Th_5.5.2}
For $\theta\in -\Lambda_{G,P}^{pos, *}$ one has canonical isomorphisms in $\Inv(_{\theta}\Bunt_P)$
\begin{multline*}
(v^{\theta}_{glob})^*\IC_{\Bunt_P}[\dim\Bunt_P+\<\theta, 2\check{\rho}-2\check{\rho}_M\>]\,\iso\\ (p^{\theta}_{glob})^!(\pi^{\theta}_{\div, loc})^! \Sat_{M, X^{-\theta}}(\Fact(\cO(U(\check{P}))_{<0})_{\theta})
\end{multline*}
and
\begin{multline*}
(v^{\theta}_{glob})^!\IC_{\Bunt_P}[\dim\Bunt_P+\<\theta, 2\check{\rho}-2\check{\rho}_M\>]\,\iso\\ (p^{\theta}_{glob})^!(\pi^{\theta}_{\div, loc})^! \DD\Sat_{M, X^{-\theta}}(\Fact(\cO(U(\check{P}))_{<0})_{\theta}).
\end{multline*}
\end{Thm} 
\begin{proof}
Consider the diagram
$$
\begin{array}{ccccc}
&& \cG_{\Ran, M}\times^{\gL^+(M)_{\Ran}}\Mod^{-,\theta,\subset}_{M,\Ran} & \getsup{\id\times\gt^{\theta}_{S,\Ran}} & \cG_{\Ran, M}\times^{\gL^+(M)_{\Ran}} S^{\theta}_{P,\Ran}\\
& \swarrow\lefteqn{\scriptstyle \id\times\pr^{-\theta}_{\Ran}} & \downarrow\lefteqn{\scriptstyle \pi^{\theta}_{\Mod, loc}} && \downarrow\lefteqn{\scriptstyle \pi^{\theta}_{loc}}\\
\Mod^{+, -\theta}_{\Bun_M} && \gL^+(M)_{\Ran}\backslash \Mod^{-,\theta,\subset}_{M,\Ran} & \getsup{\gt^{\theta}_{S,\Ran}} & \gL^+(M)_{\Ran}\backslash S^{\theta}_{P,\Ran}\\
& \searrow\lefteqn{\scriptstyle \pi^{\theta}_{\div, loc}} & \downarrow\lefteqn{\scriptstyle \tau^{\theta}}\\
&& \Mod^{-,\theta}_M/\gL^+(M)_{\theta},
\end{array}
$$
here $\tau^{\theta}$ is defined in Section~\ref{Sect_1.6.2}. 
From Theorem~\ref{Th_4.2.12_global_compatibility} we get canonical isomorphisms
\begin{multline*}
(\pi^{\theta}_{loc})^!(\gt^{\theta}_{S,\Ran})^!(\tau^{\theta})^!\Sat_{M, X^{-\theta}}(\Fact(\cO(U(\check{P}))_{<0})_{\theta})[\<\theta, 2\check{\rho}_M-2\check{\rho}\>]
\,\overset{\ref{Cor_3.3.5}}{\iso}\\
(\pi^{\theta}_{loc})^!(v^{\theta}_{S,\Ran})^*('\IC^{\frac{\infty}{2}}_{P,\Ran})
\, \overset{\ref{Lm_4.3.18}}{\iso}\,
(\id\times v^{\theta}_{S,\Ran})^*(\bar\pi^0_{loc})^!('\IC^{\frac{\infty}{2}}_{P,\Ran})
\,\overset{\ref{Th_4.2.12_global_compatibility}}{\iso}\\
(\id\times v^{\theta}_{S,\Ran})^*(\bar\pi^0_{glob})^!\IC_{\Bunt_P}[\dim\Bunt_P]\,\overset{\ref{Pp_4.3.3_key}}{\iso}\,(\pi^{\theta}_{glob})^!(v^{\theta}_{glob})^*\IC_{\Bunt_P}[\dim\Bunt_P] 
\end{multline*}
in 
$$
\Inv(\cG_{\Ran, M}\times^{\gL^+(M)_{\Ran}} S^{\theta}_{P,\Ran})_{untl}\,\iso\, Shv(\Mod^{+, -\theta}_{\Bun_M})
$$ 
where the latter equivalence is given by Corollary~\ref{Cor_4.3.9_now} and Lemma~\ref{Lm_4.3.11}. Now chasing the above diagram yields the first isomorphism, since 
$$
(\id\times\gt^{\theta}_{S,\Ran})^!(\id\times\pr^{-\theta}_{\Ran})^!
$$ 
is fully faithful.

 The second one is obtained by passing to the Verdier duality. We used the fact that the $\gL^+(M)_{\theta}$-action on $\Mod_M^{-,\theta}$ factors through an action of a smooth group scheme of finite type over $X^{-\theta}$, so $(\pi_{\div, loc}^{\theta})^![-2\dim\Bun_M]\,\iso\, (\pi_{\div, loc}^{\theta})^*$ with our conventions.
\end{proof} 
 
\begin{Rem} Theorem~\ref{Th_5.5.2} after further $*$-restriction to a stratum
$$
\Mod_{\Bun_M}^{+, \gU(-\theta)}\times_{\Bun_M}\Bun_P\hook{} {_{\theta}\Bunt_P}
$$ has been established before in a {\rm non-canonical form} in (\cite{BFGM}, Theorem~1.12). In the case of $B=P$ this is the result of Gaitsgory (\cite{Gai19Ran}, Theorem~3.9.3). 
\end{Rem}
 
\appendix

\section{Generalities}

\ssec{Invariants under category objects}
\label{Sect_Invariants_under category objects}
\sssec{} By a \select{category object} in $C\in 1-\Cat$ we mean a map $\cX: \bfitDelta^{op}\to C$ such that for any $n\ge 0$ the morphisms $[1]\toup{i, i+1} [n]$ yield an isomorphism 
$$
\cX[n]\,\iso\, \cX[1]\times_{\cX[0]} \cX[1]\times_{\cX[0]}\ldots \cX[1],
$$ 
where $[1]$ appears $n$ times. Then we say that $\cX[1]$ \select{acts on $\cX[0]$ on the right}. For $0,1: [0]\to [1]$ we denote by $s,t: \cX[1]\to \cX[0]$ the corresponding maps. Here $s,t$ stand for the \select{source} and the \select{target}. 

 Given a map $\tau: c\to \cX[0]$ in $C$, we say that the $\cX[1]$-action on $\cX[0]$ \select{is extended} to a right $\cX$-action on $c$ if the following holds. We are given a category object $\cX': \bfitDelta^{op}\to C$ and a map $\cX'\to \cX$ of category objects in $C$ such that  $\cX'[0]\to \cX[0]$ is the map $\tau$, and the square is cartesian
$$
\begin{array}{ccc}
\cX'[0] &\getsup{s} & \cX'[1]\\
\downarrow\lefteqn{\scriptstyle\tau} && \downarrow\\
\cX[0] & \getsup{s} & \cX[1]
\end{array}
$$
In this situation the square
$$
\begin{array}{ccc}
\cX'[0] &\getsup{t} & \cX'[1]\\
\downarrow\lefteqn{\scriptstyle\tau} && \downarrow\\
\cX[0] & \getsup{t} & \cX[1]
\end{array}
$$
is not necessarily cartesian. We think of $t: \cX'[1]\to c$ as the action map.  

\begin{Rem} Suppose $C$ admits finite limits, let $Corr(C)$ be the category of correspondences in $C$ (cf. \cite{R}). The following is established in (\cite{G}, published version, Cor. 4.4.5 of Chapter 9). Let $\cX$ be a category object in $C$. Then $\cX[1]\in Alg(Corr(C))$ with the product given by the diagram 
$$
\cX[1]\times \cX[1]\gets \cX[1]\times_{t, \cX[0], s} \cX[1]\toup{m} \cX[1], 
$$
where $m$ is the composition map. The unit is the diagram $*\gets X[0]\toup{u} X[1]$, where $u$ is the unit. 
\end{Rem}

\sssec{} Recall the involution $rev: \bfitDelta\,\iso\, \bfitDelta$ defined in (\cite{G}, I.1, 1.1.10). It preserves the category objects inside $\Fun(\bfitDelta^{op}, C)$. For a category object $\cX$ the composition with $rev$ is denoted $\cX^{rm}$, the same category object with reversed multiplication. 

\sssec{} 
\label{Sect_A.0.4}
Let $\cX: \bfitDelta^{op}\to \PreStk_{lft}$ be a category object with $S=\cX[0], H=\cX[1]$, so $H$ acts on $S$ on the right. Then set 
$$
S/H=\underset{[n]\in\bfitDelta^{op}}{\colim} \cX[n].
$$
We think of it as \select{the quotient of $S$ by $H$} in the sense of prestacks.

 Consider the subcategory $\bfitDelta_s\subset \bfitDelta$ with all objects and morphisms which are injective maps $[n]\to [m]$ (\cite{HTT}, 6.5.3.6). By (\cite{HTT}, 6.5.3.7), $\bfitDelta_s^{op}\hook{} \bfitDelta^{op}$ is cofinal, so in the definition of $S/H$ we may equivalenty take the colimit over $\bfitDelta_s^{op}$.

 Define the category of $H$-equivariant objects $Shv(S)^H$ of $Shv(S)$ as $\Tot(Shv(\cX([\bullet]))$. Here we applied the functor $Shv: (\PreStk_{lft})^{op}\to\DGCat_{cont}$ to $\cX$. So, by definition $Shv(S)^H\,\iso\, Shv(S/H)$. 
In this generality, $\oblv: Shv(S)^H\to Shv(S)$ is comonadic by (\cite{HA}, 4.7.5.1). 

Let's call the unit category object acting on $S$ the constant functor $\bfitDelta^{op}\to \PreStk_{lft}$ with value $S$. The unit section yields a morphism of category objects $S\to H$ acting on $S$. Note that $S/S\,\iso\, S$, so $Shv(S)^S\,\iso\, Shv(S)$. Applying the invariants, we get a functor $Shv(S)^H\to Shv(S)^S\,\iso\,Shv(S)$, which is the evaluation at $[0]\in\bfitDelta$. 

 

\sssec{} The target map $t: H\to S$ is naturally lifted to a right action of $H$ on itself, so that the map $t: H\to S$ attached to $[0]\toup{1} [1]$ is $H$-equivariant. Namely, we have the functor $\eta: \bfitDelta\to \bfitDelta$, $[n]\mapsto [0]\star [n]$ together with the natural transformation $[n]\to [0]\star [n]$ functorial in $[n]\in\bfitDelta$. The right action of $H$ on itself is obtained as the composition 
$$
\bfitDelta^{op}\toup{\eta}\bfitDelta^{op}\toup{\cX} C
$$ 

 Recall the category $\bfitDelta_{-\infty}$ from (\cite{HA}, 4.7.2.1). Its objects are $[n]$ for integers $n\ge -1$, and morphisms from $[n]$ to $[m]$ are nondecreasing maps $\alpha:  \{-\infty\}\cup [n]\to \{-\infty\}\cup [m]$ with $\alpha(-\infty)=-\infty$, here we view $-\infty$ as the least element of both categories.  
 
 Let $\phi: \bfitDelta_{-\infty}^{op}\to \bfitDelta^{op}$ be the functor given by $\{-\infty\}\cup [n]\mapsto \{-\infty\}\star[n]$. A map $f: \{-\infty\}\cup [n]\to \{-\infty\}\cup [m]$ in $\bfitDelta_{-\infty}$ is sent to the induced map $\{-\infty\}\star[n]\to \{-\infty\}\star[m]$.
 
 Restricting $\cX$ along $\phi$, we get a split augmented simplicial object in the sense of (\cite{HA}, 4.7.2.2). The corresponding augmented simplicial object is a colimit diagram by (\cite{HA}, 4.7.2.3), namely, 
$$
\underset{[n]\in\bfitDelta^{op}}{\colim} H\times_{t, S, s} H^n_S\,\iso\, S
$$ 
in $\PreStk_{lft}$. This says that the quotient of $H$ by the natural right action of $H$ on itself identifies with $S$. Here 
$$
H^n_S=H\times_{t, S, s} H\times_{t, S, s}\ldots\times_{t, S, s} H, 
$$
where $H$ appears $n$ times. 
   
 Consider the inclusion $[n]\hook{} \{-\infty\}\star[n]$ functorial in $[n]\in\bfitDelta$, it gives a morphism of simplicial diagrams 
$$
\alpha_n: H\times_{t, S, s} H^n_S\to H^n_S
$$ 
(functorial in $[n]\in\bfitDelta$). Passing to the colimit, this gives the map 
$$
S\,\iso\, \mathop{\colim}\limits_{[n]\in\bfitDelta^{op}} H\times_{t, S, s} H^n_S\to  \mathop{\colim}\limits_{[n]\in\bfitDelta^{op}} H^n_S \,\iso\, S/H,
$$
which is the natural map $f: S\to S/H$.

\sssec{} 
\label{Sect_A.0.6}
Assume $t: H\to S$ universally homologically contractible in the sense of (\cite{Gai19Ran}, A.1.8). So, for $n\ge 0$ the functor $\alpha_{n}^!: Shv(H^n_S)\to Shv(H\times_{t, S, s} H^n_S)$ is fully faithful. Passing to the limit we conclude that $f^!: Shv(S/H)\to Shv(S)$ is fully faithful. 

 Note that $K\in Shv(S)$ lies in the full subcategory $Shv(S/H)$ iff $s^!(K)$ lies in the essential image of the full embedding $t^!: Shv(S)\to Shv(H)$. We also write $Shv(S)_{untl}=Shv(S/H)$. 

\begin{Rem} 
\label{Rem_A.1.7}
Let $\cX$ be a category object in $\PreStk_{lft}$ such that the source map $s: H\to S$ is universally homologically contractible. Applying the above for $\cX^{rm}$, we see that $f^!: Shv(S/H)\to Shv(S)$ is fully faithful. An object $K\in Shv(S)$ lies in the essential image of $f^!$ iff $t^!K$ lies in the essential image of $s^!: Shv(S)\to Shv(H)$.

 Assume in addition $h: Z\to S$ is a map in $\PreStk_{lft}$, and the right $H$-action on $S
$ is extended to a right $H$-action on $Z$. Then $h^!: Shv(S)\to Shv(Z)$ induces a functor $h^!: Shv(S)_{untl}\to Shv(Z)_{untl}$. Assume, moreover, that $h$ is ind-schematic (of ind-finite type), and the square is cartesian
$$
\begin{array}{ccc}
Z\times_{h, S,s} H & \toup{t} & Z\\
\downarrow && \downarrow\lefteqn{\scriptstyle h}\\
H & \toup{t} & S.
\end{array}
$$
then $h_*: Shv(Z)\to Shv(S)$ induces a functor $h_*: Shv(Z)_{untl}\to Shv(S)_{untl}$.  
\end{Rem}

\ssec{Some coinvariants}

\sssec{} Let $I\in fSets$, $G$ be a smooth affine algebraic group of finite type. It is well known that $\gL(G)_I\to X^I$ is a placid group ind-scheme. For the convenience of the reader, we give an argument here.

 For $S\in\Sch^{aff}$ let $\cI\subset \Map(S, X)$ be an $S$-point of $\Ran$. For $n\ge 1$ let $n\Gamma_{\cI}\subset S\times X$ be the closed subscheme given by the $n$-th power of the sheaf of ideals defining $\Gamma_{\cI}\subset S\times X$. So, $\underset{n\in \NN}{\colim} \; n\Gamma_{\cI}\,\iso\,\cD_{\cI}$, where the colimit is calculated in $\Sch^{aff}$. 
Let $\gL^+_n(G)_I\to X^I$ be the group scheme sending an $S$-point $\cI$ of $X^I$ to $\Map(n\Gamma_{\cI}, G)$. In particular, $\gL^+_0(G)_I$ is trivial. Note that $\gL^+(G)_I\,\iso\, \underset{n\in \NN^{op}}{\lim} \gL^+_n(G)_I$. Let 
$$
K_{n, I}=\Ker(\gL^+(G)_I\to \gL^+_n(G)_I),
$$ 
in particular $K_{0, I}=\gL^+(G)_I$. For $n>0$ the group scheme $K_{n, I}\to X^I$ is prounipotent.
 
  For each $n\ge 0$, the quotient in the sense of etale stacks $\gL(G)_I/K_{n, I}$ is an ind-scheme of ind-finite type over $X^I$. For $n<m$ the map $f_{n,m}: \gL(G)_I/K_{m, I}\to \gL(G)_I/K_{n, I}$ is a $K_{n, I}/K_{m,I}$-torsor. Then $\gL(G)\,\iso\, \underset{n\in \NN^{op}}{\lim} \gL(G)_I/K_{n, I}$ taken in $\PreStk$. To see this is a placid ind-scheme over $X^I$ argue as in (\cite{Ly4}, 0.0.51, 0.0.36).   

 Namely, pick a presentation $\Gr_{G, I}\,\iso\, \underset{i\in I}{\colim} Y_i$, where $I$ is small filtered, $Y_i\in\Sch_{ft}$, and for $i\to i'$ in $Y_i\to Y_{i'}$ is a closed immersion. Let 
$$
\cZ_i=\underset{n\in \NN^{op}}{\lim} (\gL(G)_I/K_{n, I})\times_{\Gr_{G, I}} Y_i.
$$ 
Then $\cZ_i$ is a placid scheme over $X^I$. If $i\to i'$ is a map in $I$ then $\cZ_i\to \cZ_{i'}$ is a placid closed immersion, and $\underset{i\in I}{\colim} \cZ_i\,\iso\, \gL(G)$ is $\PreStk$. 
 
 Note that 
$$
Shv(\gL(G)_I)\,\iso\, \underset{n\in \NN^{op}}{\lim} Shv(\gL(G)_I/K_{n, I})
$$
with respect to the transition functors $(f_{n,m})_*$ for $n<m$. One has also
$$
Shv(\gL(G)_I)^{\gL^+(G)_I}\,\iso\, Shv(\Gr_{G,I}),
$$
where we used the $\gL^+(G)_I$-action on $\gL(G)_I$ by right translations. This is a particular case of the following claim established in (\cite{Ly4}, 0.0.36-0.0.40).
\begin{Lm} Let $S\in \Sch_{ft}$, $H\to S$ be a placid group ind-scheme over $S$, $G\subset H$ be a closed placid group subscheme of $H$ over $S$. Assume $G$ is prosmooth over $S$. Then $Shv(H)^G\,\iso\, Shv(H/G)$ canonically, so that $\oblv: Shv(H)^G\to Shv(H)$ identifies with $q^*: Shv(H/G)\to Shv(H)$ for $q: H\to H/G$. Here we used the action of $G$ on $H$ by right translations. 
\QED
\end{Lm}

\ssec{Some invariants}

\begin{Lm} 
\label{Lm_A.1.3}
Let $S\in\Sch_{ft}$, let $G, U$ be placid group schemes over $S$ with $U$ prounipotent. Assume $G$ acts on $U$ by automorphisms of group schemes over $S$. Let $H=G\rtimes U$ be the semi-direct product. Let $Y\,\iso\, \underset{i\in I}{\colim} Y_i$ be an ind-scheme of ind-finite type over $S$, where $I$ is small filtered, $Y_i\in (\Sch_{ft})_{/S}$, and for $i\to i'$ in $I$, $f_{ii'}: Y_i\to Y_{i'}$ is a closed immersion (over $S$). Assume $H$ acts on $Y$ over $S$ preserving each $Y_i$ and acting on $Y_i$ through a finite-dimensional quotient group scheme (with a prounipotent kernel). Then in the constructible context the functor $\oblv: Shv(Y)^H\to Shv(Y)^G$ admits a left adjoint $\Av^U_!$.
\end{Lm}
\begin{proof}
{\bf Step 1} Pick $i\in I$. Let us show that $\oblv: Shv(Y_i)^H\to Shv(Y_i)^G$ admits a left adjoint $\Av^i_!$. Assume $H$ acts on $Y_i$ through its finite-dimensional quotient $K$ with $\Ker(H\to K)$ prounipotent. Let $\bar G$ be the image of $G$ in $K$. The above functor identifies with $f_i^!: Shv(Y_i/K)\to Shv(Y_i/\bar G)$ for $f_i: Y_i/\bar G\to Y_i/K$. So, $f_i^!$ has a left adjoint $(f_i)_!$.

\medskip\noindent
{\bf Step 2} Consider the diagram $\tau: I^{op}\times [1]\to \DGCat_{cont}$ sending $i$ to $\oblv: Shv(Y_i)^H\to Shv(Y_i)^G$ with the transition functors for $i\to i'$ being $f_{ii'}^!$. The functor $\oblv: Shv(Y)^H\to Shv(Y)^G$ is obtained by passing to the limit over $I^{op}$ in the
diagram $\tau$. We may pass to the left adjoints in the diagram $\tau$ and get a diagram
$\tau^L: I\times [1]^{op}\to \DGCat_{cont}$ sending $i$ to $\Av^i_!: Shv(Y_i)^G\to Shv(Y_i)^H$. Let $\Av_!: Shv(Y)^G\to Shv(Y)^H$ be obtained from $\tau ^L$ by passing to the colimit over $I$. By (\cite{Ly}, 9.2.39), $\Av^!$ is the desired left adjoint.
\end{proof}

\ssec{About t-structures}

\sssec{} 
\label{Sect_A.4.1}
Let $\alpha: S'\to S$ be a closed immersion in $\Sch_{ft}$. $Y\to S$ be a ind-scheme of ind-finite type over $S$, $Y'=Y\times_S S'$. Let $G$ be a placid prosmooth group scheme over $S$ acting on $Y$ over $S$. Assume $Y\,\iso\, \underset{i\in I}{\colim} Y_i$, where $I$ is small filtered, and for $i\to j$ in $I$ the map $Y_i\to Y_j$ is $G$-equivariant closed immersion in $\Sch_{ft}$. Let $\bar\alpha: Y'\to Y$ be obtained by base change from $\alpha$. Let $G'=G\times_S S'$. Assume that the $G$-action on each $Y_i$ factors through a quotient group scheme $G\to G_i$ over $S$, where $G_i$ is smooth and of finite type over $S$, and such that $\Ker(G\to G_i)$ is a prounipotent group scheme over $S$. It is shown in (\cite{Ly9}, Section~1.3.24)
that the diagram canonically commutes
$$
\begin{array}{ccc}
Shv(Y)^G & \toup{\bar\alpha^!} & Shv(Y')^{G'}\\
\downarrow\lefteqn{\scriptstyle \oblv[dimrel]} && \downarrow\lefteqn{\scriptstyle \oblv[dimrel]}\\
Shv(Y) & \toup{\bar\alpha^!} & Shv(Y')
\end{array}
$$  
Here the functors $\oblv[\dimrel]$ are those defined in (\cite{DL}, A.5.3). 

\sssec{Example} Consider the stack quotient $q_{\Ran}: \Ran\to \Ran/\gL^+(M)_{\Ran}$. For $I\in fSets$ let also $q_I: X^I\to X^I/\gL^+(M)_I$ be the corresponding stack quotient. By Section~\ref{Sect_A.4.1}, the functors $\oblv[\dimrel]: Shv(X^I)^{\gL^+(M)_I}\to Shv(X^I)$ are t-exact for the perverse t-structures and compatible with the $!$-restrictions with respect to the closed immersions $X^J\to X^I$ for a map $I\to J$ in $fSets$. Passing to the limit over $I\in fSets$, the latter functors yield the functor $\oblv[\dimrel]: Shv(\Ran)^{\gL^+(M)_{\Ran}}\to Shv(\Ran)$.

\ssec{Properties of some functors}

\sssec{} 
\label{Sect_A.5.1_now}
Let $H$ be an affine algebraic group of finite type over $e$. $C=\Rep(H)$. Let $\cA\in CAlg(C)$ and $E=\cA-mod(C)$. Then $C$ is rigid by (\cite{G}, I.3, 3.7.4), and $E$ is rigid by (\cite{G}, I.1, 9.1.3). We get the adjoint pair $\ind: C\leftrightarrows E: \oblv$ in $C-mod$. Set $C(X)=C\otimes Shv(X), E(X)=E\otimes Shv(X)$. 

 Since $\ind$ is symmetric monoidal, it yields by functoriality a symmetric monoidal functor $\ind_I: (C_{X^I}, \otimes^!)\to (E_{X^I}, \otimes^!)$. Let $\oblv_I$ denote the right adjoint to $\ind_I$. By (\cite{Ly10}, 2.7.4), the dual pair $(\ind_I, \oblv_I)$ takes place in $(C_{X^I}, \otimes^!)-mod$. 
 
 From (\cite{Ly10}, 2.7.5) we conclude that $\oblv_I$ is monadic, and the corresponding monad identifies canonically with $\cA_{X^I}\in CAlg(C_{X^I})$. So, $E_{X^I}\,\iso\, \cA_{X^I}-mod(C_{X^I})$. 

\sssec{} Let $\Gamma\hook{} H$ be a closed group subscheme. Let $f: B(\Gamma)\to B(H)$ be the corresponding map. Assume $H/\Gamma$ quasi-affine (so, also quasi-compact by definition). We get an adjoint pair $f^*: \Rep(H)\leftrightarrows \Rep(\Gamma): f_*$ in $\Rep(H)-mod$, because $f$ is schematic quasi-compact. Set $D=\Rep(\Gamma)$. Write $e\in \Rep(\Gamma)$ for the trivial representation. 
Let $A=f_*e\in CAlg(C)$. Note that $D$ is rigid by (\cite{G}, I.3, 3.7.4). 

 Set $D(X)=D\otimes Shv(X)$. Let $I\in fSets$. Since $f^*$ is symmetric monoidal, as in Section~\ref{Sect_2.0.1}, it yields by functoriality a symmetric monoidal functor 
$$
L_I: (C_ {X^I},\otimes^!)\to (D_{X^I},\otimes^!).
$$ 
Let $R_I: D_{X^I}\to C_{X^I}$ denote the right adjoint to $L_I$. By (\cite{Ly10}, 2.7.4), the dual pair $(L_I, R_I)$ takes place in $(C_{X^I},\otimes^!)-mod$. By (\cite{G}, I.3, 3.3.3), $f_*: D\to C$ is monadic, so $D\,\iso\, A-mod(C)$ canonically. 
 
\begin{Lm} 
\label{Lm_A.5.2}
The functor $R_I: D_{X^I}\to C_{X^I}$ is monadic, and the corresponding monad identifies canonically with $R_IL_I(e_{X^I})\,\iso\, A_{X^I}\in CAlg(C_{X^I}, \otimes^!)$, where $e_{X^I}\in (C_{X^I}, \otimes^!)$ is the unit. So, $D_{X^I}\,\iso\, A_{X^I}-mod(C_{X^I})$.
\end{Lm}
\begin{proof} We apply (\cite{Ly10}, Lemma~2.7.5). In the notations of \select{loc.cit.}, we have the functor $\cF: \Tw(I)\times [1]\to Shv(X^I)-mod$ sending $(I\to J\toup{\phi} K)\in \Tw(I)$ to the functor 
$$
\alpha_{\phi}: Shv(X^K)\otimes C^{\otimes J}\to Shv(X^K)\otimes D^{\otimes J}.
$$ 
Here $\alpha_{\phi}$ comes from the direct image functor 
\begin{equation}
\label{functor_f^J_*_appendix}
(f^J)_*: \Rep(\Gamma^J)\,\iso\,\Rep(\Gamma)^{\otimes J}\to \Rep(H)^{\otimes J}\,\iso\, \Rep(H^J)
\end{equation} 
under the map $f^J: B(\Gamma^J)\to B(H^J)$. 

 Let $\alpha^R_{\phi}$ denote the right adjoint to $\alpha_{\phi}$. To apply (\cite{Ly10}, Lemma~2.7.5), it remains to show that for each $(I\to J\toup{\phi} K)\in\Tw(I)$ the functor $\alpha^R_{\phi}$ is conservative.

Since $(H^J)/(\Gamma^J)\,\iso\, (H/\Gamma)^J$ is quasi-affine, so (\ref{functor_f^J_*_appendix}) is conservative. Since $Shv(X^K)$ is dualizable in $\DGCat_{cont}$, this shows that $\alpha_{\phi}^R$  is conservative, as desired.
\end{proof} 

\ssec{Placid group schemes over $\Ran$}
\label{Sect_Placid group schemes over Ran}

\sssec{} Our conventions about the theory of placid group (ind)-schemes acting on categories is that of (\cite{DL}, A.4). Assume $Y\to \Ran\gets G$ are maps in $\PreStk$, where $G\in \Grp(\PreStk_{/\Ran})$. For $I\in fSets$ let $G_I=G\times_{\Ran} X^I, Y_I=Y\times_{\Ran} X^I$. Assume the projectioin $f_I: G_I\to X^I$ is a placid prosmooth group scheme over $X^I$, $Y_I$ is an ind-scheme of ind-finite type over $X^I$. So, for $I\to J$ in $fSets$ the corresponding maps $G_J\to G_I$ and $Y_J\to Y_I$ are closed immersions. 

 Assume $G$ acts on $Y$ over $\Ran$. The purpose of this subsection is to discuss the t-structures on $Shv(Y)^G$. 

\sssec{} For a monoidal category $\cA$ the category $Alg+Mod(\cA)$ is defined in (\cite{G}, I.1, 3.5.4). For $I\in fSets$ the category of invariants $Shv(Y_I)^{G_I}$ is defined in (\cite{Ly9}, 1.3.12) as $\Fun_{Shv(G_I)}(Shv(X^I), Shv(Y_I))\in Shv(X^I)-mod$. Note that $(Shv(G_I), Shv(Y_I))\in Alg+Mod(Shv(X^I)-mod)$. In particular, the action map 
$$
Shv(G_I)\otimes Shv(Y_I)\to Shv(Y_I)
$$
factors as $Shv(G_I)\otimes_{Shv(X^I)} Shv(Y_I)\to Shv(Y_I)$. Recall that $f_I^*\omega_{X^I}\in CoAlg(Shv(G_I))$, and 
$$
Shv(Y_I)^{G_I}\iso f_I^*\omega_{X^I}-comod(Shv(Y_I))
$$
by (\cite{Ly9}, 1.3.17). Since the comonad $f_I^*\omega_{X^I}$ acting on $Shv(Y_I)$ is $Shv(X^I)$-linear, $f_I^*\omega_{X^I}-comod(Shv(Y_I))\in Shv(X^I)-mod$ naturally. 
By definition, 
$$
Shv(Y)^G:=\underset{I\in fSets}{\lim} Shv(Y_I/G_I),
$$ 
where the transition functors are !-pullbacks. 

 Write $\oblv_{\Ran}: Shv(Y)^G\to Shv(Y)$ for the oblivion functor, $\oblv^R_{\Ran}$ for its right adjoint. Write $\oblv_I: Shv(Y_I)^{G_I}\to Shv(Y_I)$ for the oblivion functor, $\oblv_I^R$ for its right adjoint. Set $\cA=\oblv_{\Ran}\oblv_{\Ran}^R$ and $\cA_I=\oblv_I\oblv_I^R$.

\begin{Lm} The functor $\oblv_{\Ran}: Shv(Y)^G\to Shv(Y)$ is comonadic. Besides, for a map $I\to J$ in $fSets$ the diagram in $Shv(\Ran)-mod$ commutes
$$
\begin{array}{ccc}
Shv(Y_I) & \toup{\cA_I} & Shv(Y_I)\\
\downarrow && \downarrow\\
Shv(Y_J) & \toup{\cA_J} & Shv(Y_J).
\end{array}
$$
There is a functor $fSets\times [1]\to Shv(\Ran)-mod$, sending $I$ to $\cA_I: Shv(Y_I)\to Shv(Y_I)$, and the comonad $\cA$ on $Shv(Y)$ is obtained by passing to the limit over $fSets$ in $\cA_I: Shv(Y_I)\to Shv(Y_I)$.
\end{Lm}
\begin{proof} Since each 
$\oblv_I: Shv(Y_I/G_I)\to Shv(Y_I)$ is conservative, $\oblv_{\Ran}$ is also conservative by (\cite{Ly}, 2.5.3). 

 For $I\in fSets$ the adjoint pair $\oblv_I: f_I^*\omega_{X^I}-comod(Shv(Y_I))\leftrightarrows Shv(Y_I): \coind$ takes place in $Shv(X^I)-mod$. For a map $I\to J$ in $fSets$ the similar adjoint pair for $J$ is obtained from the previous one by applying $\_\otimes_{Shv(X^I)} Shv(X^J)$ by (\cite{Ly9}, 1.3.24, Claim 1). So, the diagram commutes
$$
\begin{array}{ccc}
Shv(Y_J/G_J) & \getsup{\oblv_J^R} & Shv(Y_J)\\
\uparrow && \uparrow\\
Shv(Y_I/G_I) & \getsup{\oblv_I^R} & Shv(Y_I).
\end{array}
$$ 
Now, by (\cite{G}, ch I.1, 2.6.4), for each $I$ the diagram commutes
$$
\begin{array}{ccc}
Shv(Y)^G & \getsup{\oblv_{\Ran}^R} & Shv(Y)\\
\downarrow && \downarrow\\
Shv(Y_I/G_I) & \getsup{\oblv_I^R} & Shv(Y_I),
\end{array}
$$
where the vertical arrow are the structure maps in our limit diagrams.

 Let now $V$ be a simplicial object of $(Shv(Y)^G)^{op}$ which is $\oblv_{\Ran}$-split in $Shv(Y)^{op}$. Then its image in $Shv(Y_I)^{op}$ is a split simplicial object. So, the image $V_I$ of $V$ in $Shv(Y_I/G_I)^{op}$ admits a colimit denoted $v_I$, and this colimit is preserved by $\oblv_I$ in $Shv(Y_I)^{op}$. 
 
 For each map $I\to J$ in $fSets$ let $\vartriangle: X^J\to X^I$ be the corresponding diagonal. The transition map $Shv(Y_I)\to Shv(Y_J)$ is $\vartriangle^!: Shv(Y_I)\to Shv(Y_I)\otimes_{Shv(X^I)} Shv(X^J)$. It has a left adjoint, hence preserves limits. Similarly, by (\cite{Ly4}, 0.3.6), $Shv(Y_I/G_I)\to Shv(Y_J/G_J)$ is the functor 
$$
\vartriangle^!: Shv(Y_I/G_I)\to Shv(Y_I/G_I)\otimes_{Shv(X^I)} Shv(X^J),
$$ 
so this functor preserves limits. 
 
 We see that each transition functor $Shv(Y_I/G_I)^{op}\to Shv(Y_J/G_J)^{op}$ in the diagram defining $(Shv(Y)^G)^{op}$ sends $v_I$ to $v_J$. Now by (\cite{Ly}, 2.2.69) we conclude that $V$ admits a colimit $v$ in $(Shv(Y)^G)^{op}$, whose image in $Shv(Y_I/G_I)^{op}$ is $v_I$.   
 
  We see that $\colim \oblv_I(V_I)\,\iso\, \oblv_I(v_I)$ for each $I$ in $Shv(Y_I)^{op}$, and the transition functors $Shv(Y_I)^{op}\to Shv(Y_J)^{op}$ preserve these colimits. Thus, $\oblv_{\Ran}(V)$ has a colimit in $Shv(Y)^{op}$, whose image in each $Shv(Y_I)^{op}$ is $\oblv_I(v_I)$. So, in $Shv(Y)^{op}$ we get 
$$
\colim \oblv_{\Ran}(V)\,\iso\, \oblv_{\Ran}(v).
$$
Thus, $\oblv_{\Ran}$ is comonadic by Barr-Beck theorem (\cite{HA}, 4.7.3.5).  
\end{proof}  

\sssec{} Equip each $Shv(Y_I/G_I)$ with a perverse t-structure as in (\cite{DL}, A.5.3). 
We get 
$$
Shv(Y)^G\,\iso\, \underset{I\in fSets^{op}}{\colim} Shv(Y_I/G_I),
$$ 
where the transition map for $I\to J$ in $fSets$ is the map $\vartriangle_!: Shv(Y_J/G_J)\to Shv(Y_I/G_I)$, the latter functor is fully faithful and t-exact.

 Define a t-structure on $Shv(Y)$ by requiring that $Shv(Y)^{\le 0}$ is the smallest full subcategory closed under extensions, colimits, and containing for each $I\in fSets$ the image of  $Shv(Y_I)^{\le 0}$. By (\cite{HA}, 1.4.4.11), this is an accessible t-structure on $Shv(Y)^G$. 
 
  Note that $K\in Shv(Y)$ lies in $Shv(Y)^{>0}$ iff for any $I$ the image of $K$ in $Shv(Y_I)$ lies in $Shv(Y_I)^{>0}$. So, the t-structure on $Shv(Y)$ is compatible with filtered colimits. 

\begin{Lm} 
\label{Lm_A.6.5}
For each $I\in fSets$ the functor $\cA_I: Shv(Y_I)\to Shv(Y_I)$ is left t-exact.
\end{Lm}
\begin{proof}
It suffices to show that for any $K\in Shv(Y_I)^c\cap Shv(Y_I)^{\ge 0}$, 
$\cA_I(K)\in Shv(Y_I)^{\ge 0}$. So, we may assume $K$ is the extension by zero from a closed subscheme of finite type $Z\hook{} Y_I$ stable under $G_I$. Moreover, the action of $G_I$ on $Shv(Z)$ factors through an action of a quotient group scheme of finite type $H\to X^I$ smooth over $X^I$ such that $\Ker(G_I\to H)$ is prounipotent over $X^I$. Our claim follows now from Lemma~\ref{Lm_A.6.6} below.  
\end{proof}
\begin{Lm} 
\label{Lm_A.6.6}
Let $S\in \Sch_{ft}, f: \cG\to S$ be a group scheme of finite type, where $f$ is smooth of relative dimension $d$. Let $\cY\to S$ be a map in $\Sch_{ft}$. Assume $\cG$ acts on $\cY$ over $S$. Then the action of $f^*(\omega_S)\in CoAlg(Shv(\cG))$ on $Shv(\cY)$ is a left t-exact functor on $Shv(\cY)$.
\end{Lm}
\begin{proof}
Write $p: \cY\to \cY/\cG$ for the projection, where $\cY/\cG$ denotes the stack quotient (over $S$). Then the corresponding comonad is $p^*p_*: Shv(\cY)\to Shv(\cY)$. Since all the fibres of $f: \cG\to S$ are of dimension $\le d$, the perverse amplitude of $p_*$ is $\ge -d$ by (\cite{Asterisque100}, 4.2.4). Since. $p^*[d]$ is t-exact, we are done.
\end{proof}

\begin{Lm} The functor $\cA: Shv(Y)\to Shv(Y)$ is left t-exact.
\end{Lm}
\begin{proof} Let $K\in Shv(Y)^{>0}$. Viewing $Shv(Y)\,\iso\,\underset{I\in fSets}{\lim} Shv(Y_I)$, assume $K$ is given by a collection $K=\{K_I\}_{I\in fSets}$, here $K_I$ is the !-restriction of $K$ to $Y_I$. Then $\cA(K)=\{\cA_I(K_I)\}_{I\in fSets}$. For each $I$, $\cA_I(K_I)\in Shv(Y_I)^{>0}$ by Lemma~\ref{Lm_A.6.5}, so $\cA(K)\in Shv(Y)^{>0}$ also.
\end{proof}

\sssec{} Note that for a map $I\to J$ in $fSets$ the adjoint pair $\oblv_J: Shv(Y_J/G_J)\leftrightarrows Shv(Y_J): \oblv_J^R$ in $Shv(X^J)-mod$ is obtained from the adjoint pair 
$$
\oblv_I: Shv(Y_I/G_I)\leftrightarrows Shv(Y_I): \oblv_I^R
$$ 
in $Shv(X^I)-mod$ by applying $Shv(X^J)\otimes_{Shv(X^I)}$. 

\sssec{} Finally, we apply (\cite{Ly}, 9.3.26) to get a t-structure on $Shv(Y)^G$ given by 
$$
(Shv(Y)^G)^{>0}=\oblv_{\Ran}^{-1}(Shv(Y)^{>0}), \;\; \;\; (Shv(Y)^G)^{\le 0}=\oblv_{\Ran}^{-1}(Shv(Y)^{\le 0}).
$$ 
The functor $\oblv_{\Ran}: Shv(Y)^G\to Shv(Y)$ is t-exact, and $\coind: Shv(Y)\to Shv(Y)^G$ is left t-exact. By (\cite{Ly}, 9.3.26), the t-structure on $Shv(Y)^G$ is compatible with filtered colimits.

\section{The ULA property and perversity}
\label{Sect_The ULA property and perversity}

\ssec{Relation to the perverse t-structures} 
\label{Sect_Relation to the perverse t-structures} 

The purpose of Section~\ref{Sect_Relation to the perverse t-structures} is to establish Proposition~\ref{Pp_B.1.4_t-exactness} below. It can be seen as a complement to (\cite{Ly_GAFA}, Section~4.8). 

\sssec{} For a map $p: Y\to S$ in $\Sch_{ft}$, $F\in Shv(Y)^c$ we say that $F$ is ULA with respect to $p$ if it satisfies (\cite{Del77}, Definition 2.12). In particular, this notion is stable under any base change $S'\to S$ by a map in $\Sch_{ft}$. By (\cite{BG}, Theorem~B.2),
for $S$ smooth this definition is equivalent to (\cite{BG}, Definition 5.1).  

\sssec{} 
\label{Sect_B.1.2_now}
Let $S\toup{q_1} S_1\getsup{p_1} Y_1$ be a diagram in $\Sch_{ft}$ with $S_1$ smooth. Set $Y=Y_1\times_{S_1} S$. Let $p: Y\to S, q: Y\to Y_1$ be the projections. For $L\in Shv(Y_1)$ in (\cite{Ly_GAFA}, Section~4.8) we introduced the functor $\cF_L: Shv(S)\to Shv(Y)$ given by 
$$
\cF_L(K)=p^*K\otimes q^*L[-\dim S_1].
$$ 
Here $\dim S_1$ is viewed as a function of a connected component of $S_1$. 

 Given $L\in Shv(Y_1)^c$ in (\cite{Ly_GAFA}, Definition~4.8.2) we called $L$ \select{locally acyclic with respect to the diagram} $S\getsup{p} Y\toup{q} Y_1$ if the natural map
$$
\cF_{\DD(L)}(\DD(K))\to \DD(\cF_L(K))
$$
in $Shv(Y)$ defined in \select{loc.cit.} is an equivalence for any $K\in Shv(S)^c$. Here $\DD$ denotes the Verdier duality. We say that $L$ is \select{universally locally acyclic with respect to this diagram} if the same property holds after any smooth base change $S'_1\to S_1$.   

Recall the following.
\begin{Pp}[\cite{Ly_GAFA}, 4.8.3] 
\label{Pp_B.1.3_ULA}
In the situation of Section~\ref{Sect_B.1.2_now} assume $L$ is ULA with respect to $p_1$. Then $L$ is ULA with respect to the diagram $S\getsup{p} Y\toup{q} Y_1$. \QED
\end{Pp} 

\begin{Pp} 
\label{Pp_B.1.4_t-exactness}
In the situation of Section~\ref{Sect_B.1.2_now} assume $L$ is ULA with respect to $p_1$. If $L$ is placed in perverse degrees $\le 0$ (resp., $\ge 0$) then $\cF_L$ is right t-exact (resp., left t-exact) for the perverse t-structures. If $L$ is perverse then $\cF_L$ is t-exact for the perverse t-structures.
\end{Pp}
\begin{proof}
Assume we are in the constructible context. Assume $L$ is placed in perverse degrees $\le 0$. It suffices to show that $\cF_L$ is right t-exact. The other properties follow by duality using Proposition~\ref{Pp_B.1.3_ULA}.  

Assume $K\in Shv(S)^c$ is placed in perverse degrees $\le 0$. Let $y\in Y$ be a non necessarily closed point, let $s=p(y), y_1=q(y)$, $s_1=q_1(s)$. By $\dim y$ we mean the dimension of the closure of $y$. It suffices to show that the $*$-fibre $(\cF_L(K))_y$ is placed in degrees $\le -\dim y$. 

Let $S'_1$ be the closure of $s_1$. Applying (\cite{BG}, 5.1.2) to $L$, we see that the functor 
$$
Shv(S_1)\to Shv(Y_1), \, \cT\mapsto L\otimes p_1^*(\cT)[-\dim S_1]
$$
is right t-exact for the perverse t-structures. In particular, $L\otimes p_1^*\IC(S'_1)[-\dim S_1]$ is placed in perverse degrees $\le 0$. So, the $*$-fibre $L_{y_1}[\dim S'_1-\dim S_1]$ is placed in degrees $\le -\dim y_1$. Here $\dim S'_1=\dim s_1$. Besides, the $*$-fibre $K_s$ is placed in degrees $\le -\dim s$. Note that $\dim y\le \dim (s\times_{s_1} y_1)=\dim s+\dim y_1-\dim s_1$. Thus, 
$$
(\cF_L(K))_y\,\iso\, K_s\otimes L_y[-\dim S_1]
$$
is placed in degrees $\le -\dim s-\dim y_1+\dim s_1\le -\dim y$. We are done.

 For $\cD$-modules we can argue similarly considering the $!$-fibres of $\DD(\cF_L(K))$ in view of Proposition~\ref{Pp_B.1.3_ULA}. 
\end{proof}

\ssec{Spaces of points and divisors}
\label{Sect_Spaces of points and divisors}
\sssec{} The purpose of section~\ref{Sect_Spaces of points and divisors} is to establish Proposition~\ref{Pp_ULA_and_perverse_appendix} and Corollary~\ref{Cor_B.2.9}.

\sssec{} Let $I\in fSets$ with $r=\mid I\mid$. Let $d>0$. Write $\oo{X}{}^{(d)}\subset X^{(d)}$ for the complement to all the diagonals, it classifies reduced effective divisors $D$ on $X$ of degree $d$.

Write $Y:=(\oo{X}{}^{(d)}\times X^I)^{\subset}\subset \oo{X}{}^{(d)}\times X^I$ for the closed subscheme classifying $(D, (x_i))\in \oo{X}{}^{(d)}\times X^I$ such that $D\le \sum_i x_i$. Let $p: Y\to \oo{X}{}^{(d)}$ be the projection. 

\begin{Pp} 
\label{Pp_ULA_and_perverse_appendix}
i) The dualizing object $\omega_Y$ is ULA with respect to $p$.\\
ii) The object $\omega_Y[-r]$ is perverse.
\end{Pp}
 The proof of Proposition~\ref{Pp_ULA_and_perverse_appendix} is given in the rest of Section~\ref{Sect_The ULA property and perversity}.
 
\sssec{Proof of Proposition~\ref{Pp_ULA_and_perverse_appendix} for $d=1$}  
i) The inclusion $Y\subset X\times X^I$ is a divisor with normal crossings. For $K\subset I$ let $Y_K=\{(x, (x_i))\in Y\mid$  $x_i=x$ for all $i\in K\}$.  Write $S_m(I)$ for the set of subsets $K\subset I$ with $\mid K\mid=m$. The irreducible components of $Y$ are $Y_K$ for $K\in S_1(I)$. 

By Lemma~\ref{Lm_divisor_normal_crossings} below, we have an exact sequence of constructible sheaves on $Y$
\begin{equation}
\label{complex_for_d=1_case_Pp_ULA_and_perverse_appendix}
0\to e_Y\to \underset{K\in S_1(I)}{\oplus} e_{Y_K}\to \underset{K\in S_2(I)}{\oplus} e_{Y_K}\to\ldots\to \underset{K\in S_r(I)}{\oplus} e_{Y_K}\to 0
\end{equation}
Since each composition $Y_K\hook{} Y\toup{p} X$ is smooth, we conclude that both $e_Y$ and $\omega_Y$ are ULA with respect to $p$.

\medskip\noindent
ii) Since $Y$ is a locally complete intersection, $\omega_Y[-r]$ is perverse. \QED

\begin{Lm} 
\label{Lm_divisor_normal_crossings}
Let $Z\in\Sch_{ft}$ be smooth irreducible scheme, $Y\subset Z$ be a divisor with normal crossings. Write $Y_i$ ($i\in I$) for the irreducible components of $Y$. Set $r=\mid I\mid$. Let $S_m(I)$ be the set of subsets $K\subset I$ with $\mid K\mid=m$. For a nonempty subset $K\subset I$ write $Y_K=\cap_{k\in K} Y_k$. Then there is an exact sequence of constructible sheaves on $Y$
$$
e_Y\to \underset{K\in S_1(I)}{\oplus} e_{Y_K}\to \underset{K\in S_2(I)}{\oplus} e_{Y_K}\to\ldots\to \underset{K\in S_r(I)}{\oplus} e_{Y_K}\to 0.
$$
\end{Lm}
\begin{proof}
We define instead the Verdier dual complex 
$$
\cK=(0\to \underset{K\in S_r(I)}{\oplus} \omega_{Y_K}\to \ldots\to 
\underset{K\in S_2(I)}{\oplus} \omega_{Y_K}\to
\underset{K\in S_1(I)}{\oplus} \omega_{Y_K}\to \omega_Y\to 0)
$$
as follows. For $K\subset K'\subset I$ we have the closed immersion $Y_{K'}\hook{} Y_K$, hence a canonical map $\omega_{Y_{K'}}\to \omega_{Y_K}$.

Pick a linear order on $I$, that is, a bijection $I\,\iso\, \{1,\ldots, r\}$. Note that any $K\subset I$ gets a linear order induced by that of $I$. The restriction of the differential 
$$
d: \underset{K\in S_m(I)}{\oplus} \omega_{Y_K}\to \underset{L\in S_{m-1}(I)}{\oplus} \omega_{Y_{L}}
$$ 
to the object $\omega_{Y_K}$ with $K\in S_m(I)$ is the map
$$
\omega_{Y_K}\to \underset{L\subset K, \, L\in S_{m-1}(I)}{\oplus} \omega_{Y_{L}},
$$
whose $L$-component equals $(-1)^{j-1}$ times the canonical map $\omega_{Y_K}\to \omega_{Y_L}$. Here $j\in \{1,\ldots, n\}$ is such that if $K=\{k_1,\ldots, k_m\}$ in its linear order that $K-L=\{k_j\}$. The last differential
$$
\underset{K\in S_1(I)}{\oplus} \omega_{Y_K}\to \omega_Y
$$
is just the natural map. It is easy to see that  $d^2=0$. 

We show, by induction on $n$, that $\DD(\cK)$ is an exact complex in the abelian category of constructible sheaves on $Y$. The !-restriction of $\cK$ to $\cap_{i\in I} Y_i$ is the usual Koszul complex. The exactness of our sequence outside $\cap_{i\in I} Y_i$ follows from the induction hypothesis. We are done.
\end{proof}

\sssec{} Write $Q(I)$ for the set of equivalence relations on $I$.
Let $\cS$ be the set of pairs: $\tilde K\subset I$ and an equivalence relation $(\tilde K\toup{\phi} K)\in Q(\tilde K)$ such that $\mid K\mid=d$. For $(\tilde K\toup{\phi} K)\in\cS$ consider the map 
$$
\alpha_{\tilde K,\phi}: \oo{X}{}^K\times X^{I-\tilde K}\to Y
$$
sending $(y_k)\in \oo{X}{}^K, (x_i)\in X^{I-\tilde K}$ to the collection: $D=\sum_{k\in K} y_k$, $(x_i)\in X^I$ with $(x_i)_{i\in \tilde K}$ beling the image of $(y_k)$ under $\oo{X}{}^K\to X^{\tilde K}$. Set 
$$
Y_{\tilde K,\phi}=\oo{X}{}^K\times X^{I-\tilde K}.
$$ 
Note that $\alpha_{\tilde K,\phi}$ is a closed immersion, and $Y_{\tilde K,\phi}$ is smooth of dimension $d+\mid I-\tilde K\mid$. 

We make $\cS$ into a partially ordered set as follows. Given $(\tilde K_1\toup{\phi_1} K_1)\in \cS$, $(\tilde K_2\toup{\phi_2} K_2)\in \cS$ say that the first object is less or equal to the second if $\tilde K_1\subset \tilde K_2$, and for each $k_1\in K_1$ there is $k_2\in K_2$ such that $\phi_1^{-1}(k_1)\subset \phi_2^{-1}(k_2)$. In this case there is a unique isomorphism $K_1\,\iso\, K_2$ making the diagram commute
$$
\begin{array}{ccc}
\tilde K_1 & \hook{} & \tilde K_2\\
\downarrow\lefteqn{\scriptstyle \phi_1} && \downarrow\lefteqn{\scriptstyle \phi_2}\\ 
K_1 & \iso & K_2.
\end{array}
$$
 
  If $(\tilde K_1, \phi_1)\le (\tilde K_2, \phi_2)$ then $Y_{\tilde K_2,\phi_2}\hook{} Y$ factors uniquely through $Y_{\tilde K_1,\phi_1}\hook{} Y$.  

\sssec{} Write $\cS^m=\{(\tilde K,\phi)\in \cS\mid \dim Y_{\tilde K,\phi}=m\}$. 
We will see that there is an exact sequence of constructible sheaves on $\tilde Y:=Y\times_{\oo{X}{}^{(d)}} \oo{X}{}^{d}$
\begin{multline}
\label{complex_cK_hopefully}
0\to e_{\tilde Y}\to \underset{(\tilde K,\phi)\in \cS^r}{\oplus} e_{\tilde Y_{\tilde K,\phi}}\to \underset{(\tilde K,\phi)\in \cS^{r-1}}{\oplus} e_{\tilde Y_{\tilde K,\phi}}\to \underset{(\tilde K,\phi)\in \cS^{r-2}}{\oplus} e_{\tilde Y_{\tilde K,\phi}}\to\ldots\\ \to \underset{(\tilde K,\phi)\in \cS^d}{\oplus} e_{\tilde Y_{\tilde K,\phi}}\to 0.
\end{multline} 
Here $\tilde Y_{\tilde K,\phi}=Y_{\tilde K,\phi}\times_{\oo{X}{}^{(d)}} \oo{X}{}^{d}$. 

\sssec{End of the proof of Proposition~\ref{Pp_ULA_and_perverse_appendix}}
ii) It suffices to show that $Y$ is a locally complete intersection. It suffices to prove this after the \'etale base change $\oo{X}{}^d\to\oo{X}{}^{(d)}$. For $1\le i\le d$ let 
$$
\tilde Y^j=\{((y_j), (x_i)_{i\in I})\in X^d\times X^I\mid y_j\le \sum_i x_i\}.
$$ 
Since $\tilde Y^j\subset\tilde Y$ is an effective Cartier divisor, $\tilde Y=\cap_{1\le j\le d} \tilde Y^j$ is a locally given by $d$ equations. Each of the irreducible components of $\tilde Y$ is smooth of dimension $\mid I\mid$. So, $\tilde Y$ is a locally complete intersection. Our claim follows.

\medskip\noindent
i) The argument below is due to Dima Arinkin\footnote{We are grateful to him for suggesting this idea to us}. For $1\le j\le d$ the constant sheaf on $\tilde Y^j$ is represented by the complex
$$
\underset{K\in S_1(I)}{\oplus} e_{\tilde Y^j_K}\to \underset{K\in S_2(I)}{\oplus} e_{\tilde Y^j_K}\to\ldots\to \underset{K\in S_r(I)}{\oplus} e_{\tilde Y^j_K}\to 0
$$
given by (\ref{complex_for_d=1_case_Pp_ULA_and_perverse_appendix}). Tensoring all these complexes for $1\le j\le d$ we see that $e_Y$ is represented by a complex of the form 
(\ref{complex_cK_hopefully}). For any $(\tilde K,\phi)\in \cS$ the composition $\tilde Y_{\tilde K,\phi}\hook{} \tilde Y\to Y\toup{p} \oo{X}{}^{(d)}$ is smooth, $e_{\tilde Y}$ is ULA over $\oo{X}{}^{(d)}$. The ULA property is local in the smooth topology of the source. We are done. \QED

\begin{Cor}
\label{Cor_B.2.9} Let $Z\to \oo{X}{}^{(d)}$ be a map in $\Sch_{ft}$, let $p_Z: Y\times_{\oo{X}{}^{(d)}} Z\to Z$ be the projection. Then the functor $p_Z^![d-r]$ is t-exact for the perverse t-structures. Here $r=\dim Y$. 
\end{Cor}
\begin{proof}
Apply Proposition~\ref{Pp_B.1.4_t-exactness} to the perverse sheaf $e_Y[r]$ using Proposition~\ref{Pp_ULA_and_perverse_appendix}. 
\end{proof}

\ssec{Property of some t-structures}
\sssec{} 
\label{Sect_B.3.1_now}
Let $S\gets{} Y\toup{p} Z$ be a diagram in $\Sch_{ft}$, let $\cG$ be a smooth group scheme of finite type over $S$ acting on $Y$ over $S$. Write $h: Y\to Y/\cG$ for the stack quotient. Let $C$ be the limit in $\DGCat_{cont}$ of the diagram
$$
Shv(Y)^{\cG}\toup{\oblv} Shv(Y)\getsup{p^!} Shv(Z)
$$
Assume $a\in\ZZ$ is such that $p^![a]$ is t-exact for the perverse t-structures.

 An object of $C$ is a collection $K\in Shv(Y)^{\cG}, K'\in Shv(Z)$ together with an isomorphism $\gamma: \oblv(K)\,\iso\, p^! K'$. Equip $Shv(Z)$ with the perverse t-structure.
Equip $C$ with the t-structure such that $(K, K', \gamma)\in C^{\le 0}$ iff $K'\in Shv(Z)^{\le 0}$. This is an accessible t-structure on $C$. 
 
\begin{Lm} 
\label{Lm_B.3.2}
We have $C^{>0}=\{(K, K',\gamma)\in C\mid K'\in Shv(Z)^{>0}\}$. The t-structure on $C$ is compatible with filtered colimits.
\end{Lm}
\begin{proof}
Applying the autoequivalence $(K, K',\gamma)\mapsto (K[a], K',\gamma)$
of $C$, we may view $C$ as the limit of the diagram
$$
Shv(Y)^{\cG}\toup{\oblv} Shv(Y)\getsup{p^![a]} Shv(Z)
$$

By (\cite{Ly9}, 1.3.4), $\oblv: Shv(Y)^{\cG}\to Shv(Y)$ is comonadic. By Lemma~\ref{Lm_A.6.6}, the corresponding comonad on $Shv(Y)$ is left t-exact for the perverse t-structures. Equip now $Shv(Y)^{\cG}$ with the t-structure as in (\cite{Ly}, 9.3.22), that is, we set $(Shv(Y)^{\cG})^{\le 0}=\oblv^{-1}(Shv(Y)^{\le 0})$. By (\cite{Ly}, 9.3.26), $(Shv(Y)^{\cG})^{\ge 0}=\oblv^{-1}(Shv(Y)^{\ge 0})$, the functor $\oblv: Shv(Y)^{\cG}\to Shv(Y)$ is t-exact, and the t-structure on $Shv(Y)^{\cG}$ is compatible with filtered colimits. So, the t-structure on $C$ is the one obtained as in Lemma~\ref{Lm_B.3.3} below. The claim follows.
\end{proof}  

\begin{Lm}[\cite{G}, I.3, 1.5.8] 
\label{Lm_B.3.3}
Let $I\in 1-\Cat$, $I\to \DGCat_{cont}, i\mapsto C_i$ be a diagram. Assume each $C_i$ is equipped with an accessible t-structure, and all the transition functors $C_i\to C_j$ are t-exact. Set $C=\lim_{i\in I} C_i$ in $\DGCat_{cont}$. Then $C$ acquires a unique t-structure such that every evaluation functor $C\to C_i$ is t-exact. For. this t-structure one has $C^{\le 0}\,\iso\, \lim_{i\in I} C_i^{\le 0}$ and $C^{>0}\,\iso\, \lim_{i\in I} C_i^{>0}$. If, moreover, the t-structure on each $C_i$ is compatible with filtered colimits then the same holds for the t-structure on $C$.
\QED
\end{Lm}  

\section{Graded Satake functors}
\label{Sect_append_Graded Satake functors}

\sssec{} In this section we define some graded versions of the usual Satake functors for the Levi subgroup $M$ of $G$ and relate them to the usual Satake functors.

\ssec{Case of $\check{M}_{neg}$}
\label{Sect_case_checkM_neg}

\sssec{} Set $\Rep(\check{M})_{pos}=\underset{\theta\in\Lambda_{G,P}^{pos}}{\oplus}\Rep(\check{M})_{\theta}$. The canonical pairing $\Rep(\check{M})\otimes\Rep(\check{M})\to\Vect$ restricts to a perfect pairing 
\begin{equation}
\label{perfect_pairing_case_checkM_neg}
\Rep(\check{M})_{neg}\otimes\Rep(\check{M})_{pos}\to\Vect.
\end{equation}

For $\lambda\in\Lambda^+_M$ write $U^{\lambda}$ for the irreducible $\check{M}$-module of h.w. $\lambda$. Let $\Lambda^+_{M, neg}$ be the set of those $\lambda\in\Lambda^+_M$ whose image in $\Lambda_{G,P}$ lies in $-\Lambda_{G,P}^{pos}$. Then 
$$
\Rep(\check{M})_{neg}\,\iso\, \underset{\lambda\in\Lambda^+_{M, neg}}{\oplus} \Vect\otimes U^{\lambda}.
$$ 

 The product $m: \Rep(\check{M})_{neg}\otimes \Rep(\check{M})_{neg}\to \Rep(\check{M})_{neg}$ admits a continuous right adjoint $m^R$. Indeed, it is given by the property that for $\nu\in \Lambda^+_{M, neg}$, 
$$
m^R(U^{\nu})=\underset{\lambda,\mu\in \Lambda^+_{M, neg}}{\oplus} U^{\lambda}\otimes U^{\mu}\otimes\Hom_{\check{M}}(U^{\lambda}\otimes U^{\mu}, U^{\nu})\in \Rep(\check{M})_{neg}.
$$
Another way is to note that the object $m(U^{\lambda}\otimes U^{\mu})\in\Rep(\check{M}_{neg})$ is compact. 

\sssec{} Write $\cO(\check{M})$ (resp., $\cO(Z(\check{M}))$ for the algebra of regular functions on $\check{M}$ (resp., on $Z(\check{M})$). For $\theta\in\Lambda_{G,P}$ write $e^{\theta}\in \cO(Z(\check{M}))$ for the character $\theta: Z(\check{M})\to \Gm$.  Let 
$$
\cO(Z(\check{M}))_{neg}\subset \cO(Z(\check{M}))
$$ 
be the $e$-subalgebra generated by $e^{\theta}$ for $\theta\in -\Lambda_{G,P}^{pos}$. Set $Z(\check{M})_{neg}=\Spec \cO(Z(\check{M}))_{neg}$. Then $Z(\check{M})_{neg}$ is naturally a monoid in $\Sch_{ft}$, namely the coproduct map for $\cO(Z(\check{M}))$ restricts to one for $\cO(Z(\check{M}))_{neg}$. The map $Z(\check{M})\to Z(\check{M})_{neg}$ is a morphism of monoids and an open immersion.   

\sssec{} Set 
$$
\cO(\check{M})_{neg}=\underset{\theta\in -\Lambda_{G,P}^{pos}}{\oplus} \cO(\check{M})_{\theta},
$$ 
where $\cO(\check{M})_{\theta}\subset \cO(\check{M})$ is the subspace of those functions on which $Z(\check{M})$ acts by $\theta$. That is, $f: \check{M}\to e$ lies in $\cO(\check{M})_{\theta}$ iff for $z\in Z(\check{M}), g\in \check{M}$ one has $f(gz)=\theta(z)f(g)$. Note that
$$
\cO(\check{M})_{neg}=\underset{\lambda\in \Lambda^+_{M, neg}}{\oplus} U^{\lambda}\otimes (U^{\lambda})^*.
$$

One has naturally $\cO(\check{M})_{neg}\in CAlg(\Rep(\check{M}\times \check{M}))$. Set
$$
\check{M}_{neg}=\Spec \cO(\check{M})_{neg}
$$
 
  Under the comultiplication map $\cO(\check{M})\to \cO(\check{M})\otimes\cO(\check{M})$ for $\lambda\in\Lambda^+_M$ the subspace $U^{\lambda}\otimes (U^{\lambda})^*$ is sent inside 
$$
(U^{\lambda}\otimes (U^{\lambda})^*)\otimes (U^{\lambda}\otimes (U^{\lambda})^*)\subset  \cO(\check{M})\otimes\cO(\check{M}).
$$ 
This shows that $\cO(\check{M})_{neg}$ is a bialgebra, the comultiplication map for $\cO(\check{M})$ preserves the corresponding subspaces. So, $\check{M}_{neg}:=
\Spec(\check{M})_{neg}$ is naturally a monoid, and $\cO(\check{M})_{neg}\to \cO(\check{M})$ is a map of bialgebras, which gives the map of monoids $\check{M}\to \check{M}_{neg}$. 

The restriction along $Z(\check{M})\to \check{M}$ is a morphism of algebras $\cO(\check{M})\to \cO(Z(\check{M}))$. It restricts to a morphism of algebras
$\cO(\check{M})_{neg}\to \cO(Z(\check{M}))_{neg}$. The corresponding morphism $Z(\check{M})_{neg}\to \check{M}_{neg}$ is a morphism of monoids. 

\sssec{} 
\label{Sect_C.1.4}
Let $B(\check{M}_{neg})\,\iso\, \underset{[n]\in\bfitDelta^{op}}{\colim} \check{M}_{neg}^n$ in $\PreStk$, this is the classifying space of this monoid. We have the natural map $q: \Spec e\to B(\check{M}_{neg})$. By definition, 
$$
\QCoh(B(\check{M}_{neg}))\,\iso\, \Tot(\QCoh(\check{M}_{neg}^{\bullet})).
$$  
Let $B_{et}(\check{M}_{neg})\in\Stk$ be the \'etale sheafification of $B(\check{M}_{neg})$. 
Note that $B_{et}(\check{M}_{neg})$ is a 1-Artin stack. By (\cite{G}, ch. I.3),
$$
\QCoh(B_{et}(\check{M}_{neg}))\,\iso\, \QCoh(B(\check{M}_{neg})).
$$ 

As in (\cite{Ga1}, 5.5.3), for any $D\in \DGCat_{cont}$ the co-simplicial category $\Tot(D\otimes \QCoh(\check{M}_{neg}^{\bullet}))$ satisfies the comonadic Beck-Chevalley conditions, so the functor 
$$
q^*: D\otimes \QCoh(B(\check{M}_{neg}))\to D
$$ 
is comonadic, and the corresponding comonad is given by the action of $\cO(\check{M}_{neg})\in coAlg(\Vect)$. 

\sssec{} The right adjoint to the inclusion $\Rep(\check{M})_{neg}\to\Rep(\check{M})$ is continuous, it is the projection 
$$
\underset{\lambda\in \Lambda^+_M}{\oplus} U^{\lambda}\otimes\Vect\to \underset{\lambda\in \Lambda^+_{M, neg}}{\oplus} U^{\lambda}\otimes\Vect
$$ 
on the components corresponding to $\Lambda^+_{M, neg}$. Denote by $\oblv_{neg}$ the composition $\Rep(\check{M})_{neg}\to\Rep(\check{M})\toup{\oblv}\Vect$. So, $\oblv_{neg}$ has a continuous right adjoint $\oblv_{neg}^R: \Vect\to \Rep(\check{M})_{neg}$. 

 View $\cO(\check{M}_{neg})$ as an object of $\Rep(\check{M})$ via the $\check{M}$-action on $\check{M}_{neg}$ by left translations. The functor $\oblv_{neg}^R$ is given by
$\Vect\to \Rep(\check{M})_{neg}$, $V\mapsto \cO(\check{M}_{neg})\otimes V$. The comonad $\oblv_{neg}\oblv_{neg}^R$ is given by the action of $\cO(\check{M}_{neg})\in coAlg(\Vect)$ on $\Vect$. 

\begin{Lm} 
\label{Lm_oblv_neg_is_comonadic}
The functor $\oblv_{neg}$ is comonadic. 
\end{Lm}
\begin{proof} The functor $\oblv_{neg}$ is conservative. The inclusion $\Rep(\check{M})_{neg}\hook{} \Rep(\check{M})$ admits both left and right adjoints, and they coincide. Let now $x^{\bullet}$ be a simplicial object of $\Rep(\check{M})_{neg}^{op}$ whose image in $\Vect^{op}$ is a split simplicial obect. By Lurie-Barr-Beck theorem (\cite{HA}, 4.7.3.5), $x^{\bullet}$ has a colimit $x$ in $\Rep(\check{M})^{op}$, and this colimit is preserved by $\oblv$. Since $\Rep(\check{M})_{neg}^{op}$ admits colimits, and the inclusion $\Rep(\check{M})_{neg}^{op}\hook{} \Rep(\check{M})^{op}$ preserves colimits, $x\in \Rep(\check{M})_{neg}$. Our claim follows now from (\cite{HA}, 4.7.3.5).
\end{proof}

So, $\Rep(\check{M})_{neg}\,\iso\, \QCoh(B(\check{M}_{neg}))$ canonically.

\sssec{}  Note that (\cite{Ly10}, 4.1.22) applies to $C=\Rep(\check{M})_{neg}$, so $\Fact(C)\to \ov{\Fact}(C)$ is an equivalence, here we used the notations of (\cite{Ly10}, Section 4.1). 
 
 For any $I\in fSets$ the natural functor $(\Rep(\check{M})_{neg})_{X^I}\to \Rep(\check{M})_{X^I}$ is fully faithful. Indeed, in the $\ov{\Fact}$ realization it is obtained by passing to the limit over $(I\toup{p} J\to K)\in \Tw(I)^{op}$ of the fully faithful functors
$$
C^{\otimes K}\otimes Shv(X^I_{p, d})\to (\Rep(\check{M}))^{\otimes K}\otimes Shv(X^I_{p, d})
$$
in the notations of \select{loc.cit}. Here 
$X^I_{p,d}=\{(x_i)\in X^I\mid \mbox{if} \; p(i)\ne p(i')\;\mbox{then}\; x_i, x_{i'}\;\mbox{are disjoint}\}$.   

 Since $\ov{\Fact}(C)\,\iso\, \underset{I\in fSets}{\lim} (\ov{\Fact}(C)\otimes_{Shv(\Ran)} Shv(X^I))$, the natural functor 
$$
\Fact(\Rep(\check{M})_{neg})\to \Fact(\Rep(\check{M}))
$$
is fully faithful.

\sssec{} Define the $-\Lambda_{G,P}^{pos}$-graded functor (\ref{functor_Sat_M_Ran_+_graded}) as follows. It is given by a compatible collections of functors for $I\in fSets$
\begin{equation}
\label{functor_Sat_M_I_+}
\Sat_{M, I, +}: (\Rep(\check{M})_{neg})_{X^I}\to Shv(\Gr_{M,I, +})^{\gL^+(M)_I}
\end{equation}
constructed below.

 Take 
$$
A=\cO(\check{M})_{neg}\otimes\omega_X\in CAlg(\Rep(\check{M})\otimes Shv(X)).
$$ 
Here we view $\cO(\check{M})_{neg}\in CAlg(\Rep(\check{M}))$ via the $\check{M}$-action on $\check{M}_{neg}$ by left translations. Applying the construction of (\cite{Ly10}, Section 7), we get the chiral algebra on $\Gr_{M,X}$ denoted $\cT_A$ in \select{loc.cit.}, as well as the corresponding Chevalley-Cousin  complex $\cC(\cT_A)\in \Sph_{M,\Ran}$. For $I\in fSets$ we denote by $\cC(\cT_A)_{X^I}\in \Sph_{M, I}$ its !-restriciton to $\Gr_{M, I}$. 

 Since $A$ and hence $\cT_A$ are $-\Lambda_{G,P}^{pos}$-graded, for each $I$ the object $\cC(\cT_A)_{X^I}$ inherits a $-\Lambda_{G,P}^{pos}$-grading. The resulting grading of $\cC(\cT_A)$ is given by restriction to the component $\Gr_{M,\Ran,\theta}$ of $\Gr_{M,\Ran}$ for $\theta\in\Lambda_{G,P}$. 

 If $\theta\in \Lambda_{G,P}$ then over $\Gr_{M,\Ran,\theta}$ the complex $\cC(\cT_A)$ vanishes unless $\theta\in -\Lambda_{G,P}^{pos}$ and in the latter case it is the extension by zero under $\Gr^{\theta}_{M, \Ran, +}\hook{} \Gr_{M,\Ran,\theta}$. 
 
 Set $\Sph_{M, I, +}=Shv(\Gr_{M, I, +})^{\gL^+(M)_I}$ for $I\in fSets$. They form naturally a sheaf of categories over $\Ran$ (for any of our sheaf theories) denoted $\Sph_{M, \Ran, +}$. Note that 
$$
\Sph_{M, \Ran, +}\subset \Sph_{M,\Ran}
$$ 
is fully faithful, this is the full subcategory of those sheaves which are extensions by zero under $\Gr_{M,\Ran, +}\hook{} \Gr_{M,\Ran}$. 

\sssec{} Note that $\Gr_{M,\Ran, +}$ is a factorization prestack over $\Ran$, and $\Sph_{M,\Ran, +}$ is a weak factorization sheaf of categories (for $\cD$-modules this is a true factorization sheaf of categories). 

 For $I,J\in fSets$ the usual exterior convolution functors 
$$
\Sph_{M, I}\boxtimes \Sph_{M, J}\to \Sph_{M, I\sqcup J}
$$ 
are considered in (\cite{Ly10}, 7.3.11), they are $Shv(X^{I\sqcup J})$-linear. They restrict to the exterior convolution functors 
$$
\Sph_{M, I, +}\boxtimes \Sph_{M, J, +}\to \Sph_{M, I\sqcup J, +}
$$ 
on the corresponding full subcategories. 
 
 Now for $I\in fSets$ the monoidal operation on $\Sph_{M, I}$ defined in (\cite{Ly10}, 7.3.11) also restricts to a monoidal operation on $\Sph_{M, I, +}$. Moreover, $\Sph_{M, I, +}$ has a natural $-\Lambda_{G,P}^{pos}$ grading compatible with the monoidal operation. Thus, $\Sph_{M, \Ran, +}$ is a weak factorization sheaf of $-\Lambda_{G,P}^{pos}$-graded monoidal categories over $\Ran$.  
 
\sssec{} 
\label{Sect_C.1.10_now}
Arguing as in (\cite{Ly10}, 7.3.8-7.3.9), one promotes for $I\in fSets$ the object $\cC(\cT_A)\in\Sph_{M, I, +}$ to an object of $(\Rep(\check{M})_{pos})_{X^I}\otimes_{Shv(X^I)} \Sph_{M, I, +}$. As in (\cite{Ly10}, 7.3.10), one gets canonically
$$
\Fun_{Shv(X^I)}((\Rep(\check{M})_{neg})_{X^I}, \Sph_{M, I, +})\,\iso\, (\Rep(\check{M})_{pos})_{X^I}\otimes_{Shv(X^I)} \Sph_{M, I, +}
$$
in view of the pairing (\ref{perfect_pairing_case_checkM_neg}). The object $\cC(\cT_A)_{X^I}$ for $I\in fSets$ gives now the desired functor (\ref{functor_Sat_M_I_+}). 

 Formal properties of these functors are as follows. For $I\in fSets$, $\Sat_{M, I, +}$ is a monoidal functor compatible with $-\Lambda_{G,P}^{pos}$-grading. The commutative chiral product on $\Fact(\Rep(\check{M})_{neg})$ is compatible with the exterior convolutions on $\Sph_{M, \Ran, +}$ as in (\cite{Ly10}, 7.3.16).

 Let $\epsilon\in Z(\check{M})$ be the image of $-1$ under $2\check{\rho}_M: \Gm\to \check{T}$. As in (\cite{Ly10}) we get the twists $\Rep(\check{M})_{neg}^{\epsilon}\in CAlg(\DGCat_{cont})$ of the category $\Rep(\check{M})_{neg}$. Though 
$$
\Fact(\Rep(\check{M})_{neg}^{\epsilon})\,\iso\, \Fact(\Rep(\check{M})_{neg})
$$ 
in $Shv(\Ran)-mod$, this isomorphism is not compatible with the factorization structures on these categories. The functor
$$
\Sat_{M, \Ran, +}: \Fact(\Rep(\check{M})_{neg}^{\epsilon})\to \Sph_{M,\Ran, +}
$$
is compatible with the factorization structures. 

\sssec{} Let $\theta\in -\Lambda_{G,P}^{pos}$. As in (\cite{Ly10}, Appendix E), write $\cTw(fSets)_{\theta}$ for the category of $(J\to K)\in\cTw(fSets)$ and a map $\und{\theta}: J\to -\Lambda_{G,P}^{pos}$ such that $\sum_{j\in J} \und{\theta}(j)=\theta$. A morphism from $(\und{\theta}^1, J_1\to K_1)$ to $(\und{\theta}^2, J_2\to K_2)$ is a map $(J_1\to K_1)\to (J_2\to K_2)$ in $\cTw(fSets)$ such that for the corresponding $\tau: J_1\to J_2$ in $fSets$ one has $\tau_*\und{\theta}^1=\und{\theta}^2$.
 
\sssec{Proof of Proposition~\ref{Pp_3.2.2}} 
\label{Sect_Proof_of_Pp_3.2.2}
We claim that (\ref{Sat_functor_for_Pp_3.2.2}) is $Shv((X^{-\theta}\times\Ran)^{\subset})$-linear.

 In the notations of (\cite{Ly10}, Appendix E), one has 
$$
\Fact(\Rep(\check{M})_{neg})_{\theta}\,\iso\, 
\underset{(J\to K, \, \und{\theta})\in \cTw(fSets)_{\theta}}{\colim} (\underset{j\in J}{\otimes} \Rep(\check{M})_{\und{\theta}(j)})\otimes Shv(X^K)
$$ 
in $Shv((X^{-\theta}\times\Ran)^{\subset})-mod$. Thus, it suffices to show that for an object $(K\toup{\id} K, \und{\theta})\in \cTw(fSets)_{\theta}$ the composition of $\Sat_{M,\Ran,+}$ with
$$
(\underset{k\in K}{\otimes} \Rep(\check{M})_{\und{\theta}(k)})\otimes Shv(X^K)\to \Fact(\Rep(\check{M})_{neg})_{\theta}
$$
is $Shv((X^{-\theta}\times\Ran)^{\subset})$-linear. 

 For $K=*$ this is clear. Indeed, if $\lambda\in\Lambda^+_M$ whose image $\theta$ in $\Lambda_{G,P}$ lies in $-\Lambda_{G,P}^{pos}$ then $\ov{\Gr}_{M, x}^{\lambda}\subset \Gr_{M, x, +}^{\theta}$. Here $\Gr_{M, x, +}^{\theta}$ is the fibre of $\Gr^{\theta}_{M,\Ran,+}$ over $\{x\}\in\Ran$. 
For general $K$ this follows from the fact that the 
chiral multplication on $\Fact(\Rep(\check{M})_{neg})$ is compatible with the exterior product on $\Sph_{M, \Ran, +}$ as in (\cite{Ly10}, 7.3.16). Namely, apply this for the exterior product map $\underset{k\in K}{\boxtimes}\Sph_{M, *, +}\to \Sph_{M, K, +}$. \QED

\sssec{Example} Let us illustrate Proposition~\ref{Pp_3.2.2} in the case $M=T$. In this case $\Lambda_{G,P}^{pos}=\Lambda^{pos}$ and $\Rep(\check{M})_{neg}\,\iso\, \underset{\theta\in -\Lambda^{pos}}{\oplus} \Vect$. So, for $\theta\in -\Lambda^{pos}$ we have canonically 
$$
\Fact(\Rep(\check{M})_{neg})_{\theta}\,\iso\, Shv((X^{-\theta}\times\Ran)^{\subset})
$$ 
by (\cite{Ly10}, E.2.2). In this case for $\theta\in -\Lambda^{pos}$ the functor (\ref{Sat_functor_for_Pp_3.2.2}) becomes
$$
\Sat_{T,\Ran, +}: Shv((X^{-\theta}\times\Ran)^{\subset})\to Shv((X^{-\theta}\times\Ran)^{\subset})^{\gL^+(T)_{\Ran}}.
$$
It is the t-exact functor $\pr^!$ for the projection 
$$
\pr: (X^{-\theta}\times\Ran)^{\subset}/\gL^+(T)_{\Ran}\to (X^{-\theta}\times\Ran)^{\subset},
$$ 
here we use the conventions of (\cite{DL}, A.5).

\ssec{Case of $\check{M}_{neg}^{untl}$}
\label{Sect_case_of_checkM_neg_untl}

\sssec{} Set
$\cO(\check{M}_{neg}^{untl})=e\oplus (\underset{\theta\in -\Lambda_{G,P}^{pos,*}}{\oplus} \cO(\check{M})_{\theta})$.
Write also $\Lambda^{+, untl}_{M, neg}$ be the set of those $\lambda\in\Lambda^+_M$ which satisfy two properties: 
\begin{itemize}
\item the image of $\lambda$ in $\Lambda_{G,P}$ lies in $-\Lambda_{G,P}^{pos}$;
\item if the image of $\lambda$ in $\Lambda_{G,P}$ vanishes then $\lambda=0$. 
\end{itemize}
This is a subsemi-group in $\Lambda^+_M$. We have 
$$
\cO(\check{M}_{neg}^{untl})=\underset{\lambda\in \Lambda^{+, untl}_{M, neg}}{\oplus} U^{\lambda}\otimes (U^{\lambda})^*.
$$
Note that $\cO(\check{M}_{neg}^{untl})\subset \cO(\check{M}_{neg})$ is a subalgebra. Set $\check{M}_{neg}^{untl}=\Spec \cO(\check{M}_{neg}^{untl})$. The coproduct on $\cO(\check{M})$ preserves the corresponding subspaces, so induces a coproduct map on $\cO(\check{M}_{neg}^{untl})$ making it into a bialgebra. The morphisms $\cO(\check{M}_{neg}^{untl})\hook{} \cO(\check{M}_{neg})\hook{}\cO(\check{M})$ are those of bialgebras, thus we get a diagram of monoids
$$
\check{M}\to \check{M}_{neg}\to \check{M}_{neg}^{untl}.
$$
\begin{Rem} If $M=T$ then $\check{M}_{neg}\to \check{M}_{neg}^{untl}$ is an isomorphism.
In general, the latter map is not an isomorphism, this happens already for $\check{M}$ being a maximal parabolic of $\check{G}=\SL_3$.
\end{Rem}

\sssec{} Let $B(\check{M}_{neg}^{untl})=\underset{[n]\in\bfitDelta^{op}}{\colim} (\check{M}^{untl}_{neg})^n$ in $\PreStk$, the classifying space of this monoind. By definition, 
$$
\QCoh(B(\check{M}_{neg}^{untl}))\,\iso\, \Tot(\QCoh((\check{M}^{untl}_{neg})^{\bullet})).
$$ 

 Consider the map $q: \Spec e\to B(\check{M}_{neg}^{untl})$. Write $B_{et}(\check{M}_{neg}^{untl})$ for the etale sheafification of $B(\check{M}_{neg}^{untl})$, this is a 1-Artin stack. As in Section~\ref{Sect_C.1.4}, for any $D\in\DGCat_{cont}$ the co-simplicial category $\Tot(D\otimes \QCoh((\check{M}^{untl}_{neg})^{\bullet}))$ satisfies the comonadic Beck-Chevalley conditions. So, the functor
$$
q^*: D\otimes \QCoh(B(\check{M}_{neg}^{untl}))\to D
$$
is comonadic, and the corresponding comonad is given by the action of $\cO(\check{M}_{neg}^{untl})\in coAlg(\Vect)$. 

\sssec{} Set 
$$
\Rep(\check{M})^{untl}_{pos}=\Vect\oplus (\underset{\theta\in \Lambda_{G,P}^{pos, *}}{\oplus} \Rep(\check{M})_{\theta}).
$$
The restriction of the canonical pairing $\Rep(\check{M})\otimes\Rep(\check{M})\to\Vect$ (coming from the fact that the latter category is rigid) yields a perfect pairing
$$
\Rep(\check{M})^{untl}_{pos}\otimes \Rep(\check{M})^{untl}_{neg}\to\Vect.
$$

 Let $\oblv_{neg}^{untl}: \Rep(\check{M})_{neg}^{untl}\to\Vect$ be the composition $\Rep(\check{M})_{neg}^{untl}\to\Rep(\check{M})\toup{\oblv}\Vect$. As in Lemma~\ref{Lm_oblv_neg_is_comonadic}, one checks that $\oblv_{neg}^{untl}$ is comonadic, and
$$
\Rep(\check{M})_{neg}^{untl}\,\iso\, \cO(\check{M}_{neg}^{untl})-comod(\Vect)
$$ 
canonically. Thus, $\Rep(\check{M})_{neg}^{untl}\,\iso\, \QCoh(B(\check{M}_{neg}^{untl}))$, so we also write $\Rep(\check{M}_{neg}^{untl}):=\Rep(\check{M})_{neg}^{untl}$. 

\sssec{} The product functor $m: \Rep(\check{M}_{neg}^{untl})\otimes \Rep(\check{M}_{neg}^{untl})\to \Rep(\check{M}_{neg}^{untl})$ preserves compact objects, so its right adjoint $m^R$ is continuous. This also follows from the fact that the diagonal morphism 
$$
B_{et}(\check{M}_{neg}^{untl})\to B_{et}(\check{M}_{neg}^{untl})\times B_{et}(\check{M}_{neg}^{untl})
$$ 
is schematic quasi-compact.

 Now (\cite{Ly10}, 4.1.22) applies to $C:=\Rep(\check{M}_{neg}^{untl})$ showing that $\Fact(C)\to \ov{\Fact}(C)$ is an equivalence in the notations of \select{loc.cit.}
 
\sssec{} For any $I\in fSets$ the natural functor $\Rep(\check{M}_{neg}^{untl})_{X^I}\to \Rep(\check{M}_{neg})_{X^I}$ is fully faithful. Indeed, in the $\ov{\Fact}$ realization it is obtained as the limit over $(I\toup{p} J\to K)\in \Tw(I)^{op}$ of the fully faithful functors
$$
\Rep(\check{M}_{neg}^{untl})^{\otimes K}\otimes Shv(X^I_{p, d})\to \Rep(\check{M}_{neg})^{\otimes K}\otimes Shv(X^I_{p, d}).  
$$
Similarly, $\Fact(\Rep(\check{M}_{neg}^{untl}))\to \Fact(\Rep(\check{M}_{neg}))$ is fully faithful. 

Let $\theta\in -\Lambda_{G,P}^{pos}$. By functoriality, we get a canonical functor $\Fact(\Rep(\check{M}_{neg}^{untl}))_{\theta}\to \Fact(\Rep(\check{M}_{neg}))_{\theta}$, which is also fully faithful.  

\sssec{} For $I\in fSets$ the prestacks $\Gr_{M,\Ran}^{\subset}, \Gr_{M, I}^{\subset}$ are defined in Section~\ref{Sect_3.2.4_now}. Set 
$$
\Sph_{M, I, +}^{untl}=Shv(\Gr_{M, I}^{\subset})^{\gL^+(M)_I}.
$$ 
As $I$ varies in $fSets$, they form a sheaf of categories on $\Ran$ denoted 
$$
\Sph_{M,\Ran, +}^{untl}=Shv(\Gr_{M, \Ran}^{\subset})^{\gL^+(M)_{\Ran}}.
$$ 
This sheaf of categories is $-\Lambda_{G,P}^{pos}$-graded, the component attached to $\theta\in -\Lambda_{G,P}^{pos}$ is 
$$
\Sph_{M,\Ran, +}^{\theta, untl}:=Shv(\Gr_{M, X^{-\theta}}^{\theta, \subset})^{\gL^+(M)_{\Ran}}.
$$ 
Note that $\Sph_{M,\Ran, +}^{0, untl}\,\iso\, Shv(\Ran)^{\gL^+(M)_{\Ran}}$. We have the full embedding 
$$
(i_{untl})_!: \Sph_{M,\Ran, +}^{untl}\hook{} \Sph_{M,\Ran, +},
$$ 
which is also $-\Lambda_{G,P}^{pos}$-graded.
 
\sssec{} 
\label{Sect_C.2.8_now}
Define the functors (\ref{functor_Sat_M_Ran^untl}) along the lines of Section~\ref{Sect_case_checkM_neg}. Namely, consider $\cO(\check{M}_{neg}^{untl})\in \Rep(\check{M})$ via the $\check{M}$-action on $\check{M}_{neg}^{untl}$ by left translations. Set
$$
A=\cO(\check{M}_{neg}^{untl})\otimes\omega_X\in CAlg(\Rep(\check{M})\otimes Shv(X)).
$$

Applying the construction of (\cite{Ly10}, Section~7), we get the chiral algebra $\cT_A$ on $\Gr_{M,X}$, which is a $\gL^+(M)_X$-equivariant perverse sheaf graded by $-\Lambda_{G,P}^{pos}$. Write $\cC(\cT_A)\in\Sph_{M,\Ran}$ for the corresponding Chevalley-Cousin complex. For $I\in fSets$ we denote by $\cC(\cT_A)_{X^I}\in \Sph_{M, I}$ its $!$-restriction to $\Gr_{M, I}$. Since $A$ is unital, $\cC(\cT_A)_{X^I}$ is a shifted perverse sheaf. It inherits a grading by $-\Lambda_{G,P}^{pos}$.

 Given $\theta\in \Lambda_{G,P}$, the object $\cC(\cT_A)$ over $\Gr_{M,\Ran, \theta}$ vanishes unless $\theta\in -\Lambda_{G,P}^{pos}$, and in the latter case it
is the extension by zero under the closed immersion 
$$
\Gr^{\theta,\subset}_{M, X^{-\theta}}\hook{} \Gr_{M,\Ran, \theta}.
$$ 
For $\theta\in -\Lambda_{G,P}^{pos}$ the $\theta$-component of $\cC(\cT_A)$ is $\cC(\cT_A)$ restricted to $\Gr_{M,\Ran, \theta}$. 

\sssec{} Note that $\Gr_{M, \Ran}^{\subset}$ is a factorization prestack over $\Ran$, and $\Sph_{M,\Ran, +}^{untl}$ is naturally a weak factorization sheaf of categories over $\Ran$. 

For $I, J\in fSets$ the exterior convolution functor
$$
\Sph_{M,I, +}\boxtimes \Sph_{M, J, +}\to \Sph_{M, I\sqcup J, +}
$$ 
restricts to the exterior convolution functors
$$
\Sph_{M,I, +}^{untl}\boxtimes \Sph_{M, J, +}^{untl}\to \Sph_{M, I\sqcup J, +}^{untl}
$$

 For $I\in fSets$ the monoidal operation on $\Sph_{M, I, +}$ also restricts to a monoidal operation on $\Sph_{M,I, +}^{untl}$ similarly. So, $\Sph_{M,\Ran, +}^{untl}$ is a weak factorization sheaf of monoidal categories on $\Ran$. 
 
 As in Section~\ref{Sect_C.1.10_now}, one promotes $\cC(\cT_A)\in \Sph_{M,I,+}^{untl}$ to an object of 
$$
(\Rep(\check{M}_{pos}^{untl}))_{X^I}\otimes_{Shv(X^I)} \Sph_{M,I,+}^{untl}\,\iso\, 
\Fun_{Shv(X^I)}((\Rep(\check{M}_{neg}^{untl}))_{X^I}, \Sph_{M,I,+}^{untl}).
$$
This is our functor $\Sat_{M, I}^{untl}$. As $I$ varies in $fSets$, this compatible system of functors defines the desired functor (\ref{functor_Sat_M_Ran^untl}). By Proposition~\ref{Pp_3.2.2}, the obtained functor (\ref{functor_Sat_M_Ran_theta^untl}) is $Shv((X^{-\theta}\times\Ran)^{\subset})$-linear. 

\sssec{Proof of Proposition~\ref{Pp_about_functor_Sat_M_Ran^untl}} 
\label{Sect_C.2.10} It follows from the construction of $\Sat_{M, \Ran}^{untl}$. \QED

\sssec{} The full embedding $\Sph_{M,\Ran, +}^{untl}\hook{} \Sph_{M,\Ran}$ is stable under the chiral symmetric monoidal structure $\otimes^{ch}$ on $\Sph_{M,\Ran}$. This defines the chiral symmetric monoidal structure $\otimes^{ch}$ on $\Sph_{M,\Ran, +}^{untl}$.

\sssec{} The following explicit description of the exterior convolution for $\Sph^{untl}_{M,\Ran, +}$ will be used below in Lemma~\ref{Lm_C.3.48_now}. 

 For $\theta,\theta'\in -\Lambda_{G,P}^{pos}$ we define the diagram
\begin{equation}
\label{diag_conv_for_Sph_M_Ran_untl}
(\gL^+(M)_{\Ran}\backslash \Gr^{\theta, \subset}_{M, X^{-\theta}})\times (\gL^+(M)_{\Ran}\backslash \Gr^{\theta', \subset}_{M, X^{-\theta'}})\getsup{p} \Conf^{\theta,\theta',\subset}_M \toup{m}  \gL^+(M)_{\Ran}\backslash \Gr^{\theta+\theta', \subset}_{M, X^{-\theta-\theta'}}
\end{equation}
as follows. 

 The stack $\gL^+(M)_{\Ran}\backslash \Gr^{\theta, \subset}_{M, X^{-\theta}}$ classifies $(D, \cJ)\in (X^{-\theta}\times\Ran)^{\subset}$, $M$-torsors $\cF_M, \cF'_M$ on $\cD_{\cJ}$ with an isomorphism $\cF_M\,\iso\,\cF'_M\mid_{D_{\cJ}-\supp(D)}$ whose reduction to $M/[M,M]$ is extended to an isomorphism on $\cD_{\cJ}$
$$
\cF_{M/[M,M]}(D)\,\iso\, \cF'_{M/[M,M]}.
$$

 The stack $\Conf^{\theta,\theta',\subset}_M$ classifies: $(D,\cJ)\in (X^{-\theta}\times\Ran)^{\subset}$, $(D',\cJ')\in (X^{-\theta'}\times\Ran)^{\subset}$, $M$-torsors $\cF_M,\cF'_M,\cF''_M$ on $\cD_{\cJ\cup \cJ'}$ together with isomorphisms
$$
\beta: \cF_M\,\iso\, \cF'_M\mid_{\cD_{\cJ\cup\cJ'}-\supp(D)}, \;\;\;\beta': \cF'_M\,\iso\,\cF''_M\mid_{\cD_{\cJ\cup\cJ'}-\supp(D')}
$$ 
whose reductions to $M/[M,M]$ are extended to isomorphisms
$$
\cF_{M/[M,M]}(D)\,\iso\, \cF'_{M/[M,M]}, \;\;\; 
\cF'_{M/[M,M]}(D')\,\iso\, \cF''_{M/[M,M]} 
$$
over $\cD_{\cJ\cup\cJ'}$. 

 The map $m$ sends the above point to $(D+D', \cJ\cup\cJ')\in (X^{-\theta-\theta'}\times\Ran)^{\subset}$, $M$-torsors $\cF_M, \cF''_M$ on $\cD_{\cJ\cup\cJ'}$ together with an isomorphism $\beta'\beta: \cF_M\,\iso\, \cF'_M\mid_{\cD_{\cJ\cup\cJ'}-\supp(D+D')}$ whose reduction to $M/[M,M]$ is extended to the isomorphism
$$
\cF_{M/[M,M]}(D+D')\,\iso\, \cF''_{M/[M,M]} 
$$
over $\cD_{\cJ\cup\cJ'}$. 

 The map $p$ sends the above point to the collection
$$
(D,\cJ, \cF_M\mid_{\cD_{\cJ}}, \cF'_M\mid_{\cD_{\cJ}}, \bar\beta: \cF_M\,\iso\, \cF'_M\mid_{\oo{\cD}_{\cJ}})\in  \gL^+(M)_{\Ran}\backslash \Gr^{\theta, \subset}_{M, X^{-\theta}},
$$
$$
(D',\cJ', \cF'_M\mid_{\cD_{\cJ'}}, \cF''_M\mid_{\cD_{\cJ'}}, \bar\beta': \cF_M\,\iso\, \cF'_M\mid_{\oo{\cD}_{\cJ'}})\in  \gL^+(M)_{\Ran}\backslash \Gr^{\theta, \subset}_{M, X^{-\theta}}.
$$
Here $\bar\beta$ is the restriction of $\beta$, and $\bar\beta'$ is the restriction of $\beta'$. 

 The exterior convolution of $K\in \Sph^{\theta, untl}_{M,\Ran,+}$, $K'\in \Sph^{\theta', untl}_{M,\Ran, +}$ is
$$
m_*p^!(K\boxtimes K')
$$
This defines a non-unital monoidal category $(\Sph^{untl}_{M,\Ran, +},\otimes^*)$. The monoidal operation $\otimes^*$ is $-\Lambda_{G,P}^{pos}$-graded.   

\ssec{Case of $\Rep(\check{M})_{<0}$}
\label{Sect_case_Rep(checkM)_<0}

\sssec{} Recall that $\Rep(\check{M})_{<0}$ is defined in Section~\ref{Sect_3.2.8_now}. Let $\Lambda^+_{M, <0}$ be the preimage of $-\Lambda_{G,P}^{pos, *}$ under $\Lambda^+_M\to \Lambda_{G,P}$. Note that 
$$
\Rep(\check{M})_{<0}\,\iso\, \underset{\lambda\in \Lambda^+_{M, <0}} \Vect\otimes U^{\lambda}
$$ is compactly generated. The product $m: \Rep(\check{M})_{<0}\otimes \Rep(\check{M})_{<0}\to \Rep(\check{M})_{<0}$ admits a continuous right adjoint given by
$$
m^R(U^{\nu})\,\iso\, \underset{\lambda_1,\lambda_2\in \Lambda^+_{M, <0}}{\oplus} U^{\lambda_1}\otimes U^{\lambda_2}\otimes\Hom_{\check{M}}(U^{\lambda_1}\otimes U^{\lambda_2}, U^{\nu}).
$$

Set 
$
\Rep(\check{M})_{>0}=\underset{\theta\in \Lambda_{G,P}^{pos,*}}{\oplus} \Rep(\check{M})_{\theta}.
$ 
The canonical self-duality on $\Rep(\check{M})$ restricts to a perfect pairing $\Rep(\check{M})_{>0}\otimes \Rep(\check{M})_{<0}\to \Vect$. Under this duality the dual $(m^R)^{\vee}$ of $m^R$ identifies with the product map on $\Rep(\check{M})_{>0}$. 

 Now by (\cite{Ly10}, 2.6.4, 2.6.8), $\Fact(\Rep(\check{M})_{<0})$ is dualizable in $Shv(\Ran)-mod$, and its dual identifies canonically with $\Fact(\Rep(\check{M})_{>0})$. This duality is compatible with the $-\Lambda_{G,P}^{pos,*}$-gradings. 
 
\sssec{} For $\theta\in -\Lambda_{G,P}^{pos, *}$ write $\cTw(fSets)^{\theta}$ for the category defined as in (\cite{Ly10}, E.1.2). Namely, its objects are collections $(\und{\theta}, J\to K)$, where $(J\to K)\in \cTw(fSets)$, and $\und{\theta}: J\to -\Lambda_{G,P}^{pos, *}$ a map satisfying $\sum_j \und{\theta}(j)=\theta$. A morphism from $(\und{\theta}^1, J_1\to K_1)$ to $(\und{\theta}^2, J_2\to K_2)$ is a morphism from $(J_1\to K_1)$ to $(J_2\to K_2)$ in $\cTw(fSets)$ such that $\phi_*\und{\theta}^1=\und{\theta}^2$ for the corresponding morphism $\phi: J_1\to J_2$ in $fSets$.
 
\sssec{} 
\label{Sect_C.3.3_about_cF_Ran,C}
For $C=\Rep(\check{M})_{<0}$ consider the functor $\cF_{\Ran, C}: \cTw(fSets)\to Shv(\Ran)-mod$ defined in (\cite{Ly10}, 2.1.14). By (\cite{Ly10}, 2.6.3), we may pass to $Shv(\Ran)$-linear continuous right adjoints in the latter functor and get $\cF^R_{\Ran, C}: \cTw(fSets)^{op}\to Shv(\Ran)-mod$, so 
$$
\Fact(C)\,\iso\, \underset{\cTw(fSets)}{\colim} \cF_{\Ran, C}\,\iso\, \underset{\cTw(fSets)}{\lim} \cF^R_{\Ran, C}
$$ 
in $Shv(\Ran)-mod$.

 For $\theta\in -\Lambda_{G,P}^{pos, *}$ consider the functor 
$$
\cF_{\Ran, C}^{\theta}: \cTw(fSets)^{\theta}\to Shv((X^{-\theta}\times\Ran)^{\subset})-mod
$$ 
defined as in (\cite{Ly10}, E.1.3), so that  $\Fact(C)_{\theta}\,\iso\, \underset{\cTw(fSets)^{\theta}}{\colim} \cF_{\Ran, C}^{\theta}$. Explicitly, $\cF_{\Ran, C}^{\theta}$ sends $(\und{\theta}, J\to K)$ to $(\underset{j\in J}{\otimes} C_{\und{\theta}(j)})\otimes Shv(X^K)$. For the same reasons we may pass to continuous right adjoints in the functor $\cF_{\Ran, C}^{\theta}$ and get the functor
$$
\cF_{\Ran, C}^{\theta, R}: (\cTw(fSets)^{\theta})^{op}\to Shv((X^{-\theta}\times\Ran)^{\subset})-mod,
$$ 
so that $\Fact(C)_{\theta}\,\iso\, \underset{(\cTw(fSets)^{\theta})^{op}}{\lim} \cF_{\Ran, C}^{\theta, R}$. 

\sssec{} Set for brevity $C=\Rep(\check{M})_{<0}$ and $D=\Rep(\check{M})_{neg}^{untl}$. Set $E=\underset{\theta\in -\Lambda_{G,P}^{pos, *}}{\oplus}\Vect$. We view $E$ as a $-\Lambda_{G,P}^{pos, *}$-graded non-unital commutative algebra in $\DGCat_{cont}$. 

In (\cite{Ly10}, 4.1.10) we defined a canonical map in $Shv(\Ran)-mod$
\begin{equation}
\label{functor_for_Sect_C.3.1} 
\Fact(C)\to \ov{\Fact}(C),
\end{equation}
and similarly for $C$ replaced by $E$.
 
\begin{Lm} 
\label{Lm_C.3.5_now}
i) The functor (\ref{functor_for_Sect_C.3.1}) is an equivalence. For $I\in fSets$ the natural functor $C_{X^I}\to D_{X^I}$ 
is fully faithful. Similarly, $\Fact(C)\to \Fact(D)$ is fully faithful and compatible with $-\Lambda_{G,P}^{pos}$-gradings.

\smallskip\noindent
ii) The natural functor $\Fact(E)\to \ov{\Fact}(E)$ is an equivalence.
\end{Lm}
\begin{proof}
i) Let $I\in fSets$. 
The natural functor $\ov{\Fact}(C)_I\to \ov{\Fact}(D)_I$ is fully faithful. Indeed, it is obtained by passing to the limit over $(I\toup{p} J\to K)\in\Tw(I)^{op}$ in the fully faithful functors
$$
C^{\otimes K}\otimes Shv(X^I_{p, d})\to D^{\otimes K}\otimes Shv(X^I_{p, d})
$$
Passing to the limit over $I\in fSets$ in the fully faithful functors $\ov{\Fact}(C)_I\to \ov{\Fact}(D)_I$ we conclude that $\ov{\Fact}(C)\to \ov{\Fact}(D)$ is fully faithful.

 Consider the functor 
$$
\cF_{\Ran, C}: \cTw(fSets)\to Shv(\Ran)-mod
$$ 
defined in (\cite{Ly10}, 2.1.14), it sends $(J\to K)$ to $C^{\otimes J}\otimes Shv(X^K)$. By definition, one has $\Fact(C)\,\iso\, \underset{\cTw(fSets)}{\colim} \cF_{\Ran, C}$. By (\cite{Ly10}, 2.6.2) we may pass to the right adjoints in $\cF_{\Ran, C}$ and get the functor $$
\cF^R_{\Ran, C}: \cTw(fSets)^{op}\to Shv(\Ran)-mod.
$$ 
So,  $\Fact(C)\,\iso\, \underset{\cTw(fSets)^{op}}{\lim}\cF^R_{\Ran, C}$, and the same holds for $C$ replaced by $D$.
 
 The natural functor $a: \Fact(C)\to \Fact(D)$ admits a continuous $Shv(\Ran)$-linear right adjoint $a^R$. Indeed, $a$ is obtained from morphism of functors $\cF_{\Ran, C}\to \cF_{\Ran, D}$ from $\cTw(fSets)$ to $Shv(\Ran)-mod$. We may pass to right adjoints in the corresponding diagram $\cTw(fSets)\times [1]\to Shv(\Ran)-mod$ and get a diagram $(\cTw(fSets)\times [1])^{op}\to Shv(\Ran)-mod$. The desired claim about $a^R$ follows now from (\cite{Ly}, 9.2.39). 
 
 Recall that $\Fact(D)\to \ov{\Fact}(D)$ is an equivalence by (\cite{Ly10}, 4.1.22).
 Now we apply Lemma~\ref{Lm_C.3.6_now_cont_right_adjoint} below to the composition $\Fact(C)\to\ov{\Fact}(C)\to \ov{\Fact}(D)$ and conclude that (\ref{functor_for_Sect_C.3.1}) has a continuous $Shv(\Ran)$-linear right adjoint. 

From (\cite{Ly10}, 4.1.15) we know that $\ov{\Fact}(C)$ is a factorization sheaf of categories over $\Ran$. By (\cite{Ly10}, 2.1.18), $\Fact(C)$ is also a factorization sheaf of categories. The functor (\ref{functor_for_Sect_C.3.1}) is compatible with factorizations. Since (\ref{functor_for_Sect_C.3.1}) has a continuous $Shv(\Ran)$-linear right adjoint and  induces an equivalence after restriction under $X\to \Ran$, 
our claim follows from (\cite{Ly10}, B.1.11).

\medskip\noindent
ii) Similar to i). 
\end{proof}

\begin{Lm} 
\label{Lm_C.3.6_now_cont_right_adjoint}
Let $A\in CAlg(\DGCat_{cont})$ and 
$D\toup{L^0} C^0\hook{j} C$ be a diagram in $A-mod$. Set $L=j\comp L^0$. Assume $L$ has a right adjoint $R: C\to D$ in $A-mod$, and $j$ is fully faithful. Then $Rj$ is the right adjoint of $L^0$ in $A-mod$, so $L^0$ admits a right adjoint in $A-mod$.
\end{Lm}
\begin{proof} Let $j^R: C\to C^0$ be the right adjoint to $j$, $R^0: C^0\to D$ be the right adjoint to $L^0$. Then $R^0\comp j^R\,\iso\, R$ and $R\comp j\,\iso\, R^0\comp j^R\comp j\,\iso\, R^0$.
\end{proof}


\sssec{} Set 
$$
\cO(\check{M})_{<0}=\underset{\theta\in -\Lambda_{G,P}^{pos,*}}{\oplus} \cO(\check{M})_{\theta}
$$
One has $\cO(\check{M})_{<0}\in CAlg^{nu}(\Rep(\check{M}\times \check{M}))$ and $\cO(\check{M}^{untl}_{neg})=e\oplus \cO(\check{M})_{<0}$. The coproduct map on $\cO(\check{M})$ preserves the corresponding subspaces, so induces a coalgebra structure on $\cO(\check{M})_{<0}$, so $\cO(\check{M}_{<0})\to \cO(\check{M}^{untl}_{neg})$ is a morphism of bialgebras in $\Vect$ (such that both the underlying algebras and coalgebras are non-unital). 

 The inclusion $C\hook{} \Rep(\check{M})$ has a continuous right adjoint given by the projection
\begin{equation}
\label{functor_projection_on_Rep(checkM)_structly_neg}
\underset{\lambda\in\Lambda^+_M}{\oplus} U^{\lambda}\otimes V_{\lambda}\mapsto \underset{\lambda\in\Lambda^+_{M, <0}}{\oplus} U^{\lambda}\otimes V_{\lambda}
\end{equation}
for $V_{\lambda}\in\Vect$. The functor (\ref{functor_projection_on_Rep(checkM)_structly_neg}) is also left adjoint to the inclusion. 

 Let $\oblv_{<0}$ be the composition $C\hook{} \Rep(\check{M})\toup{\oblv} \Vect$. Then $\oblv_{<0}$ has a continuous right adjoint $\oblv_{<0}^R$, and the comonad $\oblv_{<0}\oblv_{<0}^R$ on $\Vect$ identifies with the coalgebra $\cO(\check{M})_{<0}$ in $\Vect$. As in Lemma~\ref{Lm_oblv_neg_is_comonadic}, one shows that $\oblv_{<0}$ is comonadic, so $C\,\iso\, \cO(\check{M})_{<0}-comod(\Vect)$. 

\sssec{} 
\label{Sect_C.3.8_now_E_introduced}
Let $\oblv_{gr}: C=\Rep(\check{M})_{<0}\to E$ be the forgetful functor, which forgets the $\check{M}$-action but keeps the grading (the subscript $gr$ stands for `graded'). This is a map in $CAlg^{nu}(\DGCat_{cont})$ compatible with $-\Lambda_{G,P}^{pos, *}$-gradings. The right adjoint $\oblv_{gr}^R: E\to C$ to $\oblv_{gr}$ is continuous and $-\Lambda_{G,P}^{pos, *}$-graded. Let $e_{\theta}\in E$ be the object $e$ placed in degree $\theta$. Then for $\theta\in -\Lambda_{G,P}^{pos,*}$ one has canonically
$$
\oblv_{gr}^R(e_{\theta})\,\iso\, \cO(\check{M})_{\theta}.
$$
As in Lemma~\ref{Lm_oblv_neg_is_comonadic}, for $\theta\in -\Lambda_{G,P}^{pos,*}$ one gets $C_{\theta}\,\iso\,  \cO(\check{M})_{\theta}-comod(\Vect)$. 

 Note that $C, E$ are equipped with natural t-structures compatible with gradings, and both $\oblv_{gr}, \oblv_{gr}^R$ are t-exact. 

\sssec{} The functor $\oblv_{gr}$ yields the functor 
$$
\oblv_{gr, \Fact}: \Fact(C)\to \Fact(E)
$$ 
compatible with the $-\Lambda_{G,P}^{pos, *}$-gradings defined in (\cite{Ly10}, E.1.1). In the notations of (\cite{Ly10}, 2.2.15), the functor $\oblv_{gr,\Fact}: (\Fact(C), \otimes^!)\to (\Fact(E),\otimes^!)$ is non-unital symmetric monoidal. 

 For $\theta\in -\Lambda_{G,P}^{pos, *}$ denote by 
\begin{equation}
\label{functor_oblv_gr,theta}
\oblv_{gr, \theta}: \Fact(C)_{\theta}\to \Fact(E)_{\theta}\,\iso\, Shv(X^{-\theta})
\end{equation} 
the $\theta$-component of $\oblv_{gr, \Fact}$. The isomorphism used in (\ref{functor_oblv_gr,theta}) is given by (\cite{Ly10}, E.1.11). As in (\cite{Ly10}, E.1.3), the functor (\ref{functor_oblv_gr,theta}) is $Shv((X^{-\theta}\times\Ran)^{\subset})$-linear. 

\begin{Lm} The functor (\ref{functor_oblv_gr,theta}) has a $Shv((X^{-\theta}\times\Ran)^{\subset})$-linear continuous right adjoint $\oblv_{gr, \theta}^R$ compatible with factorization. 
\end{Lm}
\begin{proof} By functoriality, the map $\oblv_{gr}$ yields a natural transformation $\cF^{\theta}_{\Ran, C}\to \cF^{\theta}_{\Ran, E}$ of functors $\cTw(fSets)^{\theta}\to Shv((X^{-\theta}\times\Ran)^{\subset})-mod$. We may view the latter as a diagram
$[1]\times \cTw(fSets)^{\theta}\to Shv((X^{-\theta}\times\Ran)^{\subset})-mod$. We may pass to right adjoints in the latter diagram. Indeed, using Section~\ref{Sect_C.3.3_about_cF_Ran,C} this follows from the fact that for any $(\und{\theta}, J\to K)\in \cTw(fSets)^{\theta}$ the corresponding forgetful functor
$$
o: Shv(X^K)\otimes (\underset{j\in J}{\otimes} C_{\und{\theta}(j)})\to Shv(X^K)
$$
has a $Shv((X^{-\theta}\times\Ran)^{\subset})$-linear continuous right adjoint $o^R$. 
Indeed, the forgetful functor $\underset{j\in J}{\otimes} C_{\und{\theta}(j)}\to\Vect$ has a continuous right adjoint. 

By (\cite{Ly}, 9.2.39) the desired functor $\oblv_{gr, \theta}^R$ is obtained by passing to the limit over $(\und{\theta}, J\to K)\in (\cTw(fSets)^{\theta})^{op}$ in the functors $o^R$. 
\end{proof}

\sssec{} Let $K\in fSets$ and $\und{\theta}: K\to -\Lambda_{G,P}^{pos, *}$ be a map with $\theta=\sum_k \und{\theta}(k)$. For each $k$ we get the composition $C_{\und{\theta}(k)}\hook{} C\toup{\oblv}\Vect$. Consider their tensor products $\underset{k\in K}{\otimes} C_{\und{\theta}(k)}\to\Vect$. For $D\in \DGCat_{cont}$ tensoring by $D$ we get the functor
\begin{equation}
\label{functor_hopefully_comonadic_graded_version}
(\underset{k\in K}{\otimes} C_{\und{\theta}(k)})\otimes D\to D
\end{equation}

\begin{Lm} 
\label{Lm_first_for_graded_comonadicity}
For any $D\in \DGCat_{cont}$ the functor (\ref{functor_hopefully_comonadic_graded_version}) is comonadic.
\end{Lm}
\begin{proof}
The inclusion $\underset{k\in K}{\otimes} C_{\und{\theta}(k)}\hook{} C^{\otimes K}$ is fully faithful, and its right adjoint concides with its left adjoint by Lemma~\ref{Lm_for_comonadicity_now}. The same remains true after tensoring by $D$. 
The oblivion $C^{\otimes K}\otimes D\to D$ is comonadic by (\cite{Ga1}, Lemma~5.5.4). This implies that (\ref{functor_hopefully_comonadic_graded_version}) is conservative, as composition of conservative functors 
$$
(\underset{k\in K}{\otimes} C_{\und{\theta}(k)})\otimes D\to C^{\otimes K}\otimes D\to D.
$$ 
Let now $x^{\bullet}$ be a simplicial object of $((\underset{k\in K}{\otimes} C_{\und{\theta}(k)})\otimes D)^{op}$ whose image in $D^{op}$ is a split simplicial object. By (\cite{HA}, 4.7.3.5), $x^{\bullet}$ has a limit in $C^{\otimes K}\otimes D$, and this limit is preserved by the projection $C^{\otimes K}\otimes D\to D$. Now $(\underset{k\in K}{\otimes} C_{\und{\theta}(k)})\otimes D$ admits limits, and the inclusion $(\underset{k\in K}{\otimes} C_{\und{\theta}(k)})\otimes D\to C^{\otimes K}\otimes D$ preserves limits. Our claim follows now from (\cite{HA}, 4.7.3.5). 
\end{proof}
\begin{Lm} 
\label{Lm_for_comonadicity_now}
Let $I\subset I'$ be small sets. For $i\in I'$ let $D_i\in\DGCat_{cont}$. The inclusion functor $L: \underset{i\in I}{\oplus} D_i\to \underset{i\in I'}{\oplus} D_i$ has a right adjoint $R$ given by the projection $R(\underset{i\in I'}{\oplus} d_i)=\underset{i\in I}{\oplus} d_i$. Besides, $R$ is also left adjoint to $L$. \QED
\end{Lm}

\sssec{} Given $\EE\in \DGCat_{cont}$, which is compactly generated and equipped with a t-structure, we say that the t-structure on $\EE$ is \select{compactly generated} if $\EE^{\le 0}$ is generated under filtered colimits by $\EE^{\le 0}\cap \EE^c$. Recall that such t-structure is compatible with filtered colimits by (\cite{Ly}, 9.3.5).   

\begin{Lm} 
\label{Lm_oblv_gr_theta_is_comonadic} 
i) Let $\theta\in -\Lambda_{G,P}^{pos, *}$. The functor 
$$
\oblv_{gr,\theta}: \Fact(C)_{\theta}\to Shv(X^{-\theta})
$$ 
is comonadic. \\
ii) Equip $Shv(X^{-\theta})$ with the perverse t-structure. There is a natural t-structure on $\Fact(C)_{\theta}$ such that $\oblv_{gr,\theta}$ is t-exact. This t-structure is accessible, compatible with filtered colimits, left and right complete. Moreover, for $I\in fSets$ there is a natural t-structures on $\bar C_{X^I,\theta}, \bar E_{X^I,\theta}$ such that $\bar C_{X^I,\theta}\to \bar E_{X^I,\theta}$ is t-exact. These t-structures are accessible, left complete, and compatible with filtered colimits. For a map $I\to J$ in $fSets$ the morphism $\vartriangle_!: \bar C_{X^J, \theta}\to \bar C_{X^I,\theta}$ is t-exact. 
\end{Lm}
\begin{proof} i) Let $I\in fSets$. Recall the category $\bar C_{X^I}$ defined in (\cite{Ly10}, 4.1.1). Recall that 
$$
\bar C_{X^I}\,\iso\, \underset{\theta'\in -\Lambda_{G,P}^{pos,*}}{\oplus} \bar C_{X^I, \theta'}
$$ 
is graded naturally by (\cite{Ly10}, E.1.30). First, we check that $\bar C_{X^I, \theta}\to \bar E_{X^I, \theta}$ is comonadic. Given $(I\toup{p} J\to K)\in \Tw(I)$, the functor 
\begin{equation}
\label{functor_for_Lm_oblv_gr_theta_is_comonadic}
(C^{\otimes K})_{\theta}\otimes Shv(X^I_{p, d})\to (E^{\otimes K})_{\theta}\otimes Shv(X^I_{p, d})
\end{equation}
is comonadic by (\cite{Ly10}, 7.2.4), as it is obtained by passing to the product over $\und{\theta}: K\to -\Lambda_{G,P}^{pos, *}$ with $\sum_{k\in K} \und{\theta}(k)=\theta$ in the functors
$$
(\underset{k\in K}{\otimes} C_{\und{\theta}(k)})\otimes Shv(X^I_{p, d})\to Shv(X^I_{p, d}),
$$
which are comonadic by Lemma~\ref{Lm_first_for_graded_comonadicity}. The functor $\bar C_{X^I, \theta}\to \bar E_{X^I, \theta}$ is now comonadic by (\cite{Ly10}, 7.2.4), as it is obtained by passing to the limit over $(I\to J\to K)\in \Tw(I)^{op}$ in the comonadic functors (\ref{functor_for_Lm_oblv_gr_theta_is_comonadic}).  

 Passing further to the limit over $I\in fSets$, we see that 
$$
\oblv_{gr, \theta}: \ov{\Fact}(C)_{\theta}\,\iso\, \underset{I\in fSets}{\lim} \bar C_{X^I, \theta}\to \underset{I\in fSets}{\lim} \bar E_{X^I, \theta}\,\iso\, \ov{\Fact}(E)_{\theta}
$$ 
is also comonadic. Our claim follows now from Lemma~\ref{Lm_C.3.5_now}. 

\medskip\noindent 
ii) Given finite collection of objects of $\DGCat_{cont}$ with accessible t-structures
we equip their tensor product with a t-structure as in (\cite{Ly}, 9.3.17). Note that $C^{\otimes K}, E^{\otimes K}$ are equipped with natural t-structures compatible with gradings, so $(C^{\otimes K})_{\theta}, (E^{\otimes K})_{\theta}$ are equipped with the induced t-structures and 
$$
(C^{\otimes K})_{\theta}\to (E^{\otimes K})_{\theta}
$$ 
is t-exact. The t-structures on $C^{\otimes K}, E^{\otimes K}$ are both left and right complete. 

We equip $Shv(X^I_{p, d})$ with the perverse t-structure, it is compactly generated. So, by (\cite{Ly}, 9.3.10), the functor (\ref{functor_for_Lm_oblv_gr_theta_is_comonadic}) is t-exact. The t-structure on $(C^{\otimes K})_{\theta}\otimes Shv(X^I_{p, d})$ is compatible with filtered colimits by (\cite{Ly}, 9.3.5 and 9.3.7). Since $Shv(X^I_{p, d})\,\iso\,\Ind(\D^b(\Perv(X^I_{p, d})))$, from (\cite{AGKRRV}, E.9.6) we see that the t-structure on $(C^{\otimes K})_{\theta}\otimes Shv(X^I_{p, d})$ is left complete.   

The transition functors in the diagram 
$$
\Tw(I)^{op}\to Shv(X^I)-mod, \; (I\toup{p} J\to K)\mapsto (C^{\otimes K})_{\theta}\otimes Shv(X^I_{p, d})
$$ 
are t-exact. So, by (\cite{G}, I.3, 1.5.8), there is a unique t-structure on 
$$
\bar C_{X^I, \theta}\,\iso\, \underset{(I\toup{p} J\to K)\in\Tw(I)^{op}}{\lim} (C^{\otimes K})_{\theta}\otimes Shv(X^I_{p, d})
$$
compatible with filtered colimits such that each evaluation functor 
$$
\bar C_{X^I, \theta}\to (C^{\otimes K})_{\theta}\otimes Shv(X^I_{p, d})
$$ 
is t-exact. This t-structure is accessible as 
$$
\underset{(I\toup{p} J\to K)\in\Tw(I)^{op}}{\lim} ((C^{\otimes K})_{\theta}\otimes Shv(X^I_{p, d}))^{\le 0}
$$
is presentable, and left complete. The same holds for $C$ replaced by $E$. By construction, $\bar C_{X^I,\theta}\to \bar E_{X^I,\theta}$ is t-exact.  

 We may pass to left adjoints in the diagram $fSets\to Shv(\Ran)-mod$, $I\mapsto \bar C_{X^I, \theta}$. So, $\ov{\Fact}(C)_{\theta}\,\iso\, \underset{I\in fSets^{op}}{\colim} \bar C_{X^I, \theta}$ taken in $Shv(\Ran)-mod$, here the transition functors are !-extensions. By construction, for a map $I\to J$ in $fSets$ the transition functor $\bar C_{X^J, \theta}\to \bar C_{X^I, \theta}$ in the above diagram is t-exact. 
 
 Define $\Fact(C)_{\theta}^{\le 0}\subset \Fact(C)_{\theta}$ as the smallest full subcategory closed under extensions, colimits, and containing for each $I\in fSets$ the image of $C_{X^I, \theta}^{\le 0}$. By (\cite{HA}, 1.4.4.11), this defines an accessible t-structure on $\Fact(C)_{\theta}$. 
 
 Note that an object $K\in \ov{\Fact}(C)_{\theta}$ is coconnective iff its $!$-restriction under $X^I\to \Ran$ is coconnective in $\bar C_{X^I,\theta}$ for any $I\in fSets$. This shows that the t-structure on $\ov{\Fact}(C)_{\theta}$ is compatible with filtered colimits.
 
 The same construction applies with $C$ replaced by $E$. The so obtained t-structure on $Shv(X^{-\theta})$ coincides with the perverse one by (\cite{Ly10}, E.1.39). By construction, $\oblv_{gr,\theta}$ is t-exact. 
 
  Let us show that the t-structure on $\Fact(C)_{\theta}$ is right complete. By (\cite{Ly}, 4.0.10), it suffices to show that for $K\in \Fact(C)_{\theta}$ the natural map $\underset{n\in\NN}{\colim} \tau^{\le n} K\to K$ in $\Fact(C)_{\theta}$ is an isomorphism. Since $\oblv_{gr, \theta}$ is conservative, continuous and t-exact, our claim follows from the fact that the t-structure on $Shv(X^{-\theta})$ is right complete by (\cite{Ly4}, 0.0.10). 
  
  Let us show that the t-structure on $\Fact(C)_{\theta}$ is left complete. For $n\in\ZZ$ we have $\Fact(C)_{\theta}^{\ge n}\,\iso\, \underset{I\in fSets}{\lim} (\bar C_{X^I,\theta})^{\ge n}$. Since the t-structure on each $\bar C_{X^I, \theta}$ is left complete, we get
$$
\underset{n\in\ZZ^{op}}{\lim} \Fact(C)_{\theta}^{\ge-n}\,\iso\, \underset{n\in\ZZ^{op}}{\lim} \;\underset{I\in fSets}{\lim} (\bar C_{X^I,\theta})^{\ge -n}\,\iso\, \underset{I\in fSets}{\lim} \bar C_{X^I,\theta}\,\iso\, \Fact(C)_{\theta}
$$  
as desired. We are done.
\end{proof}

\sssec{Example} Let for a moment $\theta$ be one of generators of the semi-group $-\Lambda_{G,P}^{pos,*}$. For $I\in fSets$ one has canonically $\bar C_{X^I,\theta}\,\iso\, C_{\theta}\otimes Shv(X)$, which is viewed as a $Shv(X^I)$-module via the inclusion of the main diagonal $X\hook{} X^I$. Similarly, $\bar E_{X^I, \theta}\,\iso\, Shv(X)$ as a $Shv(X^I)$-module. This is obtained from (\cite{Ly10}, 4.1.4). So, $\Fact(C)_{\theta}\,\iso\, C_{\theta}\otimes Shv(X)$. 

\sssec{} Set $B=\underset{\theta\in -\Lambda_{G,P}^{pos, *}}{\oplus} e_{\theta}\in E$. We view $B$ as an object of $CAlg^{nu}(E)$, which is $-\Lambda_{G,P}^{pos,*}$-graded compatibly with the non-unital commutative algebra structure. By (\cite{Ly10}, E.1.23), for $\theta\in -\Lambda_{G,P}^{pos, *}$ one has canonically $\Fact(B)_{\theta}\,\iso\, \omega_{X^{-\theta}}$ in $Shv(X^{-\theta})$. 

\sssec{} For $I\in fSets$, $\theta\in -\Lambda_{G,P}^{pos, *}$ denote by $\Tw(I)^{\theta}$ the category of collections $(I\to J\to K)\in\Tw(I)$ and $\und{\theta}: J\to -\Lambda_{G,P}^{pos,*}$. A morphism from $(I\to J_1\to K_1, \und{\theta}^1)$ to $(I\to J_2\to K_2, \und{\theta}^2)$ is a morphism from $(I\to J_1\to K_1)$ to $(I\to J_2\to K_2)$ in $\Tw(I)$ such that for the underlying map $\phi: J_1\to J_2$ one has $\phi_*\und{\theta}^1=\und{\theta}^2$. 

\begin{Lm} 
\label{Lm_C.3.19_now}
Let 
$$
\cA=\underset{\theta\in -\Lambda_{G,P}^{pos,*}}{\oplus} \cA_{\theta}\in CAlg^{nu}(C)
$$ 
with $\cA_{\theta}\in C^{\heartsuit}_{\theta}$ for any $\theta$. Recall that, by (\cite{Ly10}, E.1.9), $\Fact(\cA)\in\Fact(C)$ is $-\Lambda_{G,P}^{pos, *}$-graded, as well as its !-restriction $\cA_{X^I}\in C_{X^I}$ under $X^I\to\Ran$. Let $\theta\in -\Lambda_{G,P}^{pos, *}$. 

\smallskip\noindent
i) If $I\in fSets$ then $\cA_{X^I, \theta}\in \bar C_{X^I, \theta}$ is strictly connective. 

\smallskip\noindent
ii) The object $\Fact(\cA)_{\theta}\in \Fact(C)_{\theta}$ is strictly connective. 
\end{Lm}
\begin{proof}
i) As in (\cite{Ly10}, E.1.9), 
$$
\cA_{X^I, \theta}\,\iso\, \underset{(I\to J\to K, \, \und{\theta})\in \Tw(I)^{\theta}}{\colim} (\underset{j\in J}{\otimes} \cA_{\und{\theta}(j)})\otimes\omega_{X^K},
$$
where the colimit is taken in $\bar C_{X^I,\theta}$. For each $(I\to J\to K, \und{\theta})\in \Tw(I)^{\theta}$ the functor
$$
(\underset{j\in J}{\otimes}C_{\und{\theta}(j)})\otimes Shv(X^K)\to \bar C_{X^I, \theta}
$$
is t-exact, and $(\underset{j\in J}{\otimes} \cA_{\und{\theta}(j)})\otimes\omega_{X^K}$ is strictly connective in $(\underset{j\in J}{\otimes}C_{\und{\theta}(j)})\otimes Shv(X^K)$. 
Here we used an analog of (\cite{Ly10}, 4.1.13). Since $(\bar C_{X^I,\theta})^{<0}$ is closed under colimits, our claim follows.

\smallskip\noindent
ii) 
We have 
$$
\Fact(\cA)_{\theta}\,\iso\, \underset{(\und{\theta}, J\to K)\in \cTw(fSets)^{\theta}}{\colim} \; (\underset{j\in J}{\otimes} \cA_{\und{\theta}(j)})\otimes\omega_{X^K}
$$ 
taken in $\Fact(C)_{\theta}$. For each $(\und{\theta}, J\to K)\in \cTw(fSets)^{\theta}$ in the diagram 
$$
(\underset{j\in J}{\otimes} C_{\und{\theta}(j)})\otimes Shv(X^K)\to \,\bar C_{X^K, \theta}\,\toup{\vartriangle^K_!} \,\ov{\Fact}(C)_{\theta}
$$ 
the first arrow is t-exact, the second one is right t-exact. Here $\vartriangle^K: X^K\to \Ran$ is the natural map. Our claim follows now as in i).
\end{proof}
\begin{Rem} The version of Lemma~\ref{Lm_C.3.19_now} with $C$ replaced by $E$ also holds. 
\end{Rem}

\sssec{} 
\label{Sect_C.3.17_now}
View $\cO(\check{M})$ as an object of $\Rep(\check{M})$ via the $\check{M}$-action on itself by left translations, and $\cO(\check{M})_{<0}\subset \cO(\check{M})$ as a subobject. Then 
$$
\cO(\check{M})_{<0}\,\iso\, \underset{\lambda\in \Lambda_{M, <0}^+}{\oplus} U^{\lambda}\otimes \und{U}^{\lambda}\in CAlg^{nu}(C),
$$ 
here for $V\in\Rep(\check{M})$ we denote by $\und{V}$ its image under $\oblv: \Rep(\check{M})\to\Vect$. So, we get $\Fact(\cO(\check{M})_{<0})\in\Fact(C)$. 

 The map 
$$
\id: \cO(\check{M})_{<0}\to \cO(\check{M})_{<0}\,\iso\, \oblv_{gr}^R(B)
$$ 
gives by adjointness the counit map $\oblv_{gr}(\cO(\check{M})_{<0})\to B$, which is $-\Lambda_{G,P}^{pos, *}$-graded morphism in $CAlg^{nu}(E)$. For $\lambda\in \Lambda^+_{M, <0}$ lying over $\theta\in -\Lambda_{G,P}^{pos, *}$ its restriction to $U^{\lambda}\otimes (\und{U^{\lambda}})^*$ factors as $U^{\lambda}\otimes (\und{U^{\lambda}})^*\toup{\ev} e_{\theta}\hook{} B$, where $\ev$ is the canonical pairing. This counit map induces a morphism 
$$
\Fact(\oblv_{gr}(\cO(\check{M})_{<0}))_{\theta}\,\iso\,
\oblv_{gr,\theta}(\Fact(\cO(\check{M})_{<0})_{\theta})\to \omega_{X^{-\theta}}
$$
of factorization algebras in $\Fact(E)$. By adjointness, it gives rise to a morphism
\begin{equation}
\label{morphism_in_Fact(C)_theta_SectC.3}
\Fact(\cO(\check{M})_{<0})_{\theta}\to \oblv_{gr,\theta}^R(\omega_{X^{-\theta}})
\end{equation}
in $\Fact(C)_{\theta}$ compatible with factorization. The map (\ref{morphism_in_Fact(C)_theta_SectC.3}) is an isomorphism. Indeed, by factorization, it suffices to check this after !-restriction to the main diagonal $X\to X^{-\theta}$, where it follows from the description of $\oblv_{gr}^R$ in Section~\ref{Sect_C.3.8_now_E_introduced}. 

 From Lemma~\ref{Lm_oblv_gr_theta_is_comonadic} we get
$$
\Fact(C)_{\theta}\,\iso\, \oblv_{gr,\theta}(\Fact(\cO(\check{M})_{<0})_{\theta})-comod(Shv(X^{-\theta}))
$$ 
By Lemma~\ref{Lm_C.3.19_now}, $\Fact(\cO(\check{M})_{<0})_{\theta}\in \Fact(C)_{\theta}$ is strictly connective. 

\sssec{Question} For $\theta\in -\Lambda_{G,P}^{pos, *}$ what are the perverse cohomology sheaves of the complex $\oblv_{gr,\theta}(\Fact(\cO(\check{M})_{<0})_{\theta})$ on $X^{-\theta}$? 

\sssec{} 
\label{Sect_C.3.1_now}
Define the exterior convolution on the sheaf of categories $\Sph_{M, \Conf,+}$ over $\Conf$ as follows. Given $\theta,\theta'\in \Lambda_{G,P}^{pos,*}$ consider the diagram
\begin{multline*}
\Gr_{M, X^{\theta}, +}^{-\theta}\times \Gr_{M, X^{\theta'}, +}^{-\theta'}\getsup{p} \wt{\Gr_{M, X^{\theta}, +}^{-\theta}\times \Gr_{M, X^{\theta'}, +}^{-\theta'}}\toup{q} 
\\ \Gr_{M, X^{\theta}, +}^{-\theta}\ttimes \Gr_{M, X^{\theta'}, +}^{-\theta'} \toup{m}\Gr_{M, X^{\theta+\theta'}, +}^{-\theta-\theta'}
\end{multline*}

 Here $\wt{\Gr_{M, X^{\theta}, +}^{-\theta}\times \Gr_{M, X^{\theta'}, +}^{-\theta'}}$ is the prestack classifying collections: $D, D'\in {X^{\theta}\times X^{\theta'}}$, $M$-torsors $\cF_M, \cF'_M$ on $X$ with trivializations 
$$
\beta: \cF_M\,\iso\, \cF^0_M\mid_{X-\supp(D)}, \; \beta': \cF'_M\,\iso\, \cF^0_M\mid_{X-\supp(D')}
$$ 
inducing isomorphisms on $X$
$$
\bar\beta: \cF_{M/[M,M]}\,\iso\,\cF^0_{M/[M,M]}(D), \;\; \bar\beta': \cF'_{M/[M,M]}\,\iso\,\cF^0_{M/[M,M]}(D'),
$$ 
and another trivialization $\gamma: \cF_M\,\iso\, \cF^0_M\mid_{\cD_{D’}}$. 

 The map $p$ forgets the trivialization $\gamma$. Here $\Gr_{M, X^{\theta}, +}^{-\theta}\ttimes \Gr_{M, X^{\theta'}, +}^{-\theta'}$ is the prestack classifying collections: $D, D'\in {X^{\theta}\times X^{\theta'}}$, $M$-torsors $\cF_M, \cF''_M$ on $X$ with trivializations 
$$
\beta: \cF_M\,\iso\, \cF^0_M\mid_{X-\supp(D)}, \; \eta: \cF''_M\,\iso\, \cF_M\mid_{X-\supp(D')}
$$ 
inducing isomorphisms on $X$
$$
\bar\beta: \cF_{M/[M,M]}\,\iso\, \cF^0_{M/[M,M]}(D), \;\;\; \bar\eta: \cF''_{M/[M,M]}\,\iso\, \cF_{M/[M,M]}(D').
$$
  
  The morphism $q$ sends the above collection to $(D,D', \cF_M,\cF''_M, \beta, \eta)$, where $\cF''_M$ is obtained as the gluing of $\cF_M\mid_{X-\supp(D')}$ with $\cF'_M\mid_{\cD_{D'}}$ via the isomorphism 
$$
(\beta')^{-1}\gamma: \cF_M\mid_{\oo{\cD}_{D'}}\,\iso\, \cF'_M\mid_{\oo{\cD}_{D'}}.
$$ 

 The map $m$ sends the above collection to $(D+D', \cF''_M)$ with the trivialization $\beta\eta: \cF''_M\,\iso\, \cF^0_M\mid_{X-\supp(D+D')}$. 
 
  The above diagram has an evident extension to the cases when $\theta=0$ or $\theta'=0$.

\sssec{} 
\label{Sect_C.3.2_now}
For $\theta,\theta'\in \Lambda_{G,P}^{pos}$ the exterior convolution functor
$$
\Sph_{M, X^{\theta},+}\otimes \Sph_{M, X^{\theta'},+}\to \Sph_{M, X^{\theta+\theta'},+}
$$
sends the pair $\cB\in \Sph_{M, X^{\theta},+}, \cB'\in \Sph_{M, X^{\theta'},+}$ to $m_*(\cB\tboxtimes \cB')$, where $\cB\tboxtimes \cB'$ is given informally by $q^*(\cB\tboxtimes \cB')\,\iso\, p^*(\cB\boxtimes\cB')$. 

 For $\theta=0$ or $\theta'=0$ we get the usual action of $\Vect$ on the corresponding object of $\DGCat_{cont}$. 
 
\sssec{} 
\label{Sect_C.3.3_now}
To make the latter definition rigorous, one considers the diagram
\begin{multline*}
(\gL^+(M)_{-\theta}\backslash\Gr_{M, X^{\theta}, +}^{-\theta})\times (\gL^+(M)_{-\theta'}\backslash\Gr_{M, X^{\theta'}, +}^{-\theta'})\getsup{\tilde p} \Conv_{M, \theta, \theta',+} \\\toup{\tilde m} \gL^+(M)_{-\theta-\theta'}\backslash \Gr_{M, X^{\theta+\theta'}, +}^{-\theta-\theta'}.
\end{multline*}
Here $\Conv_{M, \theta, \theta',+}$ is the prestack classifying: $D\in X^{\theta}, D'\in X^{\theta'}$, $M$-torsors $\cF_M, \cF'_M, \cF''_M$ on $\cD_{D+D'}$ with isomorphisms 
$$
\cF_M\,\iso\, \cF'_M\mid_{\cD_{D+D'}-\supp(D)}, \;\;\; \cF'_M\,\iso\, \cF''_M\mid_{\cD_{D+D'}-\supp(D')}
$$ 
inducing isomorphisms $\cF_{M/[M,M]}(D)\,\iso\, \cF'_{M/[M,M]}$ and $\cF'_{M/[M,M]}(D')\,\iso\, \cF''_{M/[M,M]}$ over the whole of $\cD_{D+D'}$. 

 The map $\tilde m$ sends the above collection to $D+D'\in X^{\theta+\theta'}$, $M$-torsors $\cF_M, \cF''_M$ on $\cD_{D+D'}$ together with compatible isomorphisms $\cF_M\,\iso\, \cF''_M\mid_{\cD_{D+D'}-\supp(D+D')}$
and  
$$
\cF_{M/[M,M]}(D+D')\,\iso\, \cF''_{M/[M,M]}\mid_{\cD_{D+D'}}.
$$ 

 The map $\tilde p$ sends the above collection to the pair 
$$
(D, \cF_M\mid_{\cD_D}, \cF'_M\mid_{\cD_D}, \cF_M\,\iso\, \cF'_M\mid_{\oo{\cD}_D}, \cF_{M/[M,M]}(D)\,\iso\, \cF'_{M/[M,M]}\mid_{\cD_D})
$$
and
$$
(D, \cF'_M\mid_{\cD_{D'}}, \cF''_M\mid_{\cD_{D'}}, \cF'_M\,\iso\, \cF''_M\mid_{\oo{\cD}_{D'}}, \cF'_{M/[M,M]}(D')\,\iso\, \cF''_{M/[M,M]}\mid_{\cD_D'})
$$
The map $\tilde m$ is representable and ind-proper. The exterior convolution of $\cB\in \Sph_{M, X^{\theta},+}$, $\cB'\in \Sph_{M, X^{\theta'},+}$ is defined as 
$$
\tilde m_*\tilde p^!(\cB\boxtimes\cB').
$$  
This defines the $\otimes^*$-monoidal structure on $\Sph_{M, \Conf,+}$. It is not symmetric. The induced functor
$$
\Sph_{M, X^{\theta},+}\boxtimes \Sph_{M, X^{\theta'},+}\to \Sph_{M, X^{\theta+\theta'},+}
$$
is $Shv(X^{\theta+\theta'})$-linear. The usual semi-smallness argiment shows that the latter functor is t-exact. 

\sssec{} 
\label{Sect_C.3.4_now}
Let us define the chiral symmetric monoidal structure on $\Sph_{M,\Conf^*,+}$. Given $J\in fSets$ let $\und{\theta}: J\to -\Lambda_{G,P}^{pos,*}$, $j\mapsto \theta_j$ be a map and $\theta=\sum_j \theta_j$. 
Let 
$$
(\underset{j\in J}{\prod} X^{-\theta_j})_d\subset \underset{j\in J}{\prod} X^{-\theta_j}
$$ 
be the open subscheme classifying collections $(D_j)_{j\in J}$ with pairwise disjoint supports. Consider the diagram
$$
\underset{j\in J}{\prod} (\Gr_{M, X^{-\theta_j},+}^{\theta_j}/\gL^+(M)_{\theta_j})\;\getsup{\tilde p_d} \;\Conv_{M, \und{\theta}, +}\;\toup{\tilde m_d}\; \Gr_{M, X^{-\theta},+}^{\theta})/\gL^+(M)_{\theta}
$$  
Here $\Conv_{M, \und{\theta}, +}$ is the prestack classifying: $(D_j)\in (\underset{j\in J}{\prod} X^{-\theta_j})_d$ for which we set $D=\sum_j D_j$, $M$-torsors $\cF_M, \cF'_M$ on $\cD_D$ with compatible isomorphisms 
\begin{equation}
\label{iso_for_Sect_C.3.4}
\cF_M\,\iso\, \cF'_M\mid_{\oo{\cD}_D}\;\;\;\mbox{and}\;\;\; \cF_{M/[M,M]}(D)\,\iso\, \cF'_{M/[M,M]}\mid_{\cD_D}.
\end{equation}

The map $\tilde m_d$ sends the above collection to: $\cF_M,\cF'_M, D$ and the isomorphisms (\ref{iso_for_Sect_C.3.4}). The $j$-th component of $\tilde p_d$ sends the above collection to $(D_j, \cF_M\mid_{\cD_{D_j}}, \cF'_M\mid_{\cD_{D_j}})$ with the induced isomorphisms
$$
\cF_M\,\iso\, \cF'_M\mid_{\oo{\cD}_{D_j}}\;\;\;\mbox{and}\;\;\; \cF_{M/[M,M]}(D_j)\,\iso\, \cF'_{M/[M,M]}\mid_{\cD_{D_j}}
$$  
Note that $\tilde p_d$ is an open immersion, and $\tilde m_d$ is obtained by base change from the \'etale map $(\underset{j\in J}{\prod} X^{-\theta_j})_d\to X^{-\theta}$ given by the sum of divisors. The functor
$$
\underset{j\in J}{\otimes}^{ch}: (\underset{j\in J}{\boxtimes} \Sph_{M, X^{-\theta_j}, +})\to \Sph_{M, X^{-\theta}, +}
$$ 
sends $\underset{j\in J}{\boxtimes} K_j$ to 
$(\tilde m_d)_*(\tilde p_d)^!(\underset{j\in J}{\boxtimes} K_j)$ for $K_j\in \Sph_{M, X^{-\theta_j}, +}$. It is $Shv(X^{-\theta})$-linear. 

Given a linear order on $J$, we get a natural transformation 
$$
-\otimes^*-\;\to\; -\otimes^{ch}-.
$$
The above allows to imitate the theory of $\Lie^*$ and chiral algebras in $\Sph_{M,\Conf^*, +}$. 

\sssec{Pseudo-tensor structures obtained by restriction} 
\label{Sect_Pseudo-tensor structures obtained by restriction}
For $\theta\in -\Lambda_{G,P}^{pos, *}$
set 
$$
\Gr^{\theta}_{M, X, +}=(\Gr^{\theta}_{M, X^{-\theta}, +})\times_{X^{-\theta}} X,
$$ 
where the map $X\to X^{-\theta}$ is the main diagonal. We have the closed embeddings
$$
\Gr^{\theta,\subset}_{M, X^{-\theta}}
\;\gets\;\Gr^{\theta}_{M, X, +}\;\to\; \Gr^{\theta}_{M, X^{-\theta}, +},
$$
where the left map is obtained from $X\to (X^{-\theta}\times\Ran)^{\subset}$ by the base change along $\Gr^{\theta,\subset}_{M, X^{-\theta}}\to (X^{-\theta}\times\Ran)^{\subset}$. Set
$$
\Gr_{M, X^{pos, *}}=\underset{\theta\in -\Lambda_{G,P}^{pos, *}}{\sqcup} \Gr^{\theta}_{M, X, +}
\;\;\;\mbox{and}\;\;\; \Sph_{M, X^{pos, *}}=Shv(\Gr_{M, X^{pos, *}})^{\gL^+(M)_X}.
$$ 
We view $\Sph_{M, X^{pos, *}}$ as a $-\Lambda_{G,P}^{pos, *}$-graded sheaf of categories on $X$. 

Set $\Gr_{M, X, +}^0=X$ viewed as the unit section inside $\Gr_{M, X}$. Let 
$$
\Gr_{M, X^{pos}}=\underset{\theta\in -\Lambda_{G,P}^{pos}}{\sqcup} \Gr^{\theta}_{M, X, +}
\;\;\;\mbox{and}\;\;\; \Sph_{M, X^{pos}}=Shv(\Gr_{M, X^{pos}})^{\gL^+(M)_X}.
$$
For $\theta\in -\Lambda_{G,P}^{pos, *}$ set $\Sph_{M, X^{pos, *}}^{\theta}=Shv(\Gr^{\theta}_{M, X, +})^{\gL^+(M)_X}$. 

\index{$\Gr^{\theta}_{M, X, +}$, $\Gr_{M, X^{pos, *}}$, $\Gr_{M, X^{pos}}$, 
$\Sph_{M, X^{pos, *}}$, $\Sph_{M, X^{pos, *}}^{\theta}$, $\Sph_{M, X^{pos}}$, Section~\ref{Sect_Pseudo-tensor structures obtained by restriction}}

\begin{Rem} 
\label{Rem_C.3.6}
If $\cO^{\otimes}\to \Fin_*$ is an $\infty$-operad in the sense of \cite{HA}, and $\cO\subset \cO^{\otimes}$ is the underlying $\infty$-category, let $D\subset \cO$ be a full subcategory stable under isomorphisms. The construction of (\cite{HA}, 2.2.1) extends $D$ to an $\infty$-operad $D^{\otimes}\to\Fin_*$ together with a full embedding $D^{\otimes}\hook{} \cO^{\otimes}$, which is a morphism of $\infty$-operads. So, for $\<n\>\in\Fin_*$ the fibre $D^{\otimes}_{\<n\>}\subset \cO^{\otimes}_{\<n\>}$ is the full subcategory generated by objects of the form $x_1\oplus\ldots\oplus x_n$ with $x_i\in D$ (in the notations of \cite{HA}, 2.1.1.15). We also refer to $D^{\otimes}\to \Fin_*$ as the pseudo-tensor structure on $D$ obtained by restriction.
\end{Rem}

\sssec{} 
\label{Sect_C.3.7_now}
The chiral symmetric monoidal structures on $\Sph_{M, \Ran, +}^{untl}$ and on  $\Sph_{M,\Conf^*, +}$ yield by restriction (in the sense of Remark~\ref{Rem_C.3.6})
the same pseudo-tensor structure on $\Sph_{M, X^{pos, *}}$ also denoted $\otimes^{ch}$. 

 The $\otimes^*$-monoidal structures on $\Sph_{M, \Ran, +}^{untl}$ and on  $\Sph_{M,\Conf^*, +}$ yield by restriction (in the sense of Remark~\ref{Rem_C.3.6})
the same pseudo-tensor structure on $\Sph_{M, X^{pos, *}}$ also denoted $\otimes^*$. 

 This follows from the fact that for $n>0$ the product $X^n\to X^{(n)}\gets X$ identifies with $X$ up to nilpotents (that we may ignore when considering the categories of sheaves). 
 
 In the constructible context the $\otimes^*$ (resp., $\otimes^{ch}$)-pseudo-tensor structure on $\Sph_{M, X^{pos, *}}$ is actually a non-unital monoidal structure (resp., a non-unital symmetric monoidal structure). 

\sssec{} Given $\theta\in -\Lambda^{pos,*}_{G,P}$, consider the category $fSets^{\theta}$ defined as in (\cite{Ly10}, E.1.4). Namely, its objects are pairs $(K, \und{\theta})$, where $K\in fSets$ and $\und{\theta}: K\to -\Lambda^{pos,*}_{G,P}$ with $\theta=\sum_j \und{\theta}(j)$. A morphism from $(K_1, \und{\theta}^1)$ to $(K_2, \und{\theta}^2)$ in $fSets^{\theta}$ is a map $\tau: K_1\to K_2$ in $fSets$ such that $\tau_*\und{\theta}^1=\und{\theta}^2$. 

 As in (\cite{Ly10}, E.1.22) we set $\und{X}^{-\theta}=\underset{(K, \und{\theta})\in (fSets^{\theta})^{op}}{\colim} X^K$. By \select{loc.cit.}, we have a canonical map $\und{X}^{-\theta}\to X^{-\theta}$, which induces an isomorphism $Shv(X^{-\theta})\,\iso\, Shv(\und{X}^{-\theta})$. 
 
\sssec{} 
\label{Sect_C.3.30_now}
 For $(\und{\theta}, J)\in fSets^{\theta}$ set 
$$
\Gr_{M, J, \und{\theta}}=\Gr_{M, X^{-\theta}, +}^{\theta}\times_{X^{-\theta}} X^J,
$$ 
where we have used the map $X^J\to X^{-\theta}$ sending $(x_j)$ to $-\sum_j \und{\theta}(j)x_j$. This gives
$$
\underset{(\und{\theta}, J)\in (fSets^{\theta})^{op}}{\colim} \Gr_{M, J, \und{\theta}}\,\iso\, \Gr_{M, \und{X}^{-\theta}, +}^{\theta}
$$
in $\PreStk$.  Given a map $\phi: (\und{\theta}, J)\to (\und{\mu}, K)$ in $fSets^{\theta}$ let $\vartriangle^{(\und{\theta}/\und{\mu})}: \Gr_{M, K, \und{\mu}}\to \Gr_{M, J, \und{\theta}}$ denote the transition map in the above diagram, it is obtained by base change from $X^K\to X^J$. The corresponding morphism
$$
\Gr_{M, K, \und{\mu}}/\gL^+(M)_{K}\to \Gr_{M, J, \und{\theta}}/\gL^+(M)_J
$$
is also denoted by $\vartriangle^{(\und{\theta}/\und{\mu})}$ by abuse of notations. Write 
$$
\vartriangle^{J,\und{\theta}}:  \Gr_{M, J, \und{\theta}}\to \Gr_{M, \und{X}^{-\theta}, +}^{\theta}$$ 
for the natural map. As in (\cite{Ly10}, Section~E.4), one gets now
$$
\Sph_{M, X^{-\theta}, +}\,\iso\, \underset{(\und{\theta}, J)\in fSets^{\theta}}{\lim} Shv(\Gr_{M, J, \und{\theta}})^{\gL^+(M)_J},
$$
the transition functors being given by $!$-pullbacks. 

\index{$\vartriangle^{(\und{\theta}/\und{\mu})}, \vartriangle^{J,\und{\theta}}$, Section~\ref{Sect_C.3.30_now}}

\sssec{} 
\label{Sect_C.3.32_now}
The functor (\ref{funct_Sat_M_X^-theta})  is given by a version of the chiral Hecke algebra $\cC(\cT_A)_{X^{-\theta}}$ on $\Gr_{M, X^{-\theta},+}^{\theta}$ that we are going to define. By the above, as a mere object of $\Sph_{M, X^{-\theta}, +}$
it is a datum of a compatible system of objects
\begin{equation}
\label{complex_cC(cT_A)_J_und(theta)}
\cC(\cT_A)_{J,\und{\theta}}\in Shv(\Gr_{M, J, \und{\theta}})^{\gL^+(M)_J}
\end{equation} 
for $(\und{\theta}, J)\in fSets^{\theta}$. Namely, for a map $(J_1, \und{\theta}^1)\to (J_2,\und{\theta}^2)$ in $fSets^{\theta}$ and the corresponding morphism
$
\vartriangle^{(\und{\theta}^1/\und{\theta}^2)}: \Gr_{M, J_2,\und{\theta}^2}\to \Gr_{M, J_1,\und{\theta}^1}
$
we will have an isomorphism 
$$
(\vartriangle^{(\und{\theta}^1/\und{\theta}^2)})^!\cC(\cT_A)_{J_1,\und{\theta}^1}\,\iso\, \cC(\cT_A)_{J_2,\und{\theta}^2}
$$ 
together with higher compatibilities.

\sssec{} Set for brevity 
$$
A=\cO(\check{M})_{<0}\in CAlg^{nu}(\Rep(\check{M}))
$$ 
viewed as an object of $\Rep(\check{M})$ as in Section~\ref{Sect_C.3.17_now}.  

 In (\cite{Ly10}, 7.3.4-7.3.5) we attached to $A$ the chiral algebra $\cT_A\in \Perv(\Gr_{M, X})^{\gL^+(M)_X}$, more precisely a Lie algebra in $(\Sph_{M, \Ran},\otimes^{ch})$. It is actually graded by $-\Lambda_{G,P}^{pos,*}$. For $\theta\in -\Lambda_{G,P}^{pos,*}$ the component $(\cT_A)_{\theta}$ is an object of $\Perv(\Gr^{\theta}_{M, X,+})^{\gL^+(M)_X}$. 

By construction, $\cT_A$ is a Lie algebra object in $(\Sph_{M, X^{pos,*}}, \, \otimes^{ch})$. Using Section~\ref{Sect_C.3.7_now}, view $\cT_A$ as a Lie algebra object in $(\Sph_{M, \Conf^*, +},\otimes^{ch})$ via the pseudo-tensor embedding 
$$
(\Sph_{M, X^{pos,*}}, \otimes^{ch})\to (\Sph_{M, \Conf^*, +}, \otimes^{ch}).
$$ 

 Define  $\cC(\cT_A)$ as the Chevalley-Cousin complex of $\cT_A$ in $(\Sph_{M, \Conf^*, +}, \otimes^{ch})$. We write $\cC(\cT_A)_{X^{-\theta}}$ for its restriction to $\Gr_{M, X^{-\theta},+}^{\theta}$.
 
\sssec{} It is instructive to do the following calculation. Let $\cF\in \Sph_{M, X^{pos, *}}$ and $n>0$. Write $\Sym^{n, ch}(\cF)$ for the $n$-th symmetric power of $\cF$ in the non-unital symmetric monoidal category $(\Sph_{M, \Conf^*, +}, \,\otimes^{ch})$. For $\theta\in -\Lambda_{G,P}^{pos, *}$ denote by $\cF_{\theta}$ the $\theta$-component of $\cF$. 

 Pick $\theta\in -\Lambda_{G,P}^{pos, *}$ and $(J, \und{\theta})\in fSets^{\theta}$. The object $(\vartriangle^{J, \und{\theta}})^!\Sym^{n, ch}(\cF)$ is described as follows. 
Write $Q(J, n)$ for the set of equivalence relations on $J$ with $n$ equivalence classes, we view them as maps $J\to K$ in $fSets$. For a morphism $\phi: (\und{\theta}, J)\to (\und{\mu}, K)$ in $fSets^{\theta}$ given by $\phi: J\to K$ we have the open immersions
$$
(\prod_{k\in K}\Gr_{M, X,+}^{\und{\mu}(k)})\;\,\getsup{j_{\phi}}\;\,
(\prod_{k\in K}\Gr_{M, X,+}^{\und{\mu}(k)})\times_{X^K} \oo{X}{}^K\;\, \hook{j^{\phi}}\;\,
\Gr_{M, K, \und{\mu}}.
$$
Both of them are obtained by base change from the open immersion $\oo{X}{}^K\hook{} X^K$, the complement to all the diagonals. Then
$$
(\vartriangle^{J, \und{\theta}})^!\Sym^{n, ch}(\cF)\,\iso\, \underset{(J\toup{\phi} K)\in Q(J, n)}{\oplus} \vartriangle^{(\und{\theta}/\und{\mu})}_! (j^{\phi})_*(j_{\phi})^!(\underset{k\in K}{\boxtimes} \cF_{\und{\mu}(k)})
$$
in $Shv(\Gr_{M, J,\und{\theta}})^{\gL^+(M)_J}$, where in the latter formula we let $\und{\mu}=\phi_*\und{\theta}$. 

\sssec{} 
\label{Sect_C.3.35_description}
 Let us make explicit the definition of (\ref{complex_cC(cT_A)_J_und(theta)}). Given $(\und{\theta}, J)\in fSets^{\theta}$ and a map $\phi: J\to K$ in $fSets$ let $\und{\mu}=\phi_*\und{\theta}$. One has
$$
\cC(\cT_A)_{J,\und{\theta}}\,\iso\, \underset{(J\toup{\phi} K)\in Q(J)}{\oplus} \;\vartriangle^{(\und{\theta}/\und{\mu})}_*(j^{\phi})_*(j_{\phi})^!(\underset{k\in K}{
\boxtimes} (\cT_A)_{\und{\mu}(k)}[1]),
$$
here $\und{\mu}=\phi_*\und{\theta}$, and the differential is defined as for the usual Chevalley-Cousin complex in (\cite{BD_chiral}, 3.4.11). 

 In particular, for $(*, \theta)\in fSets^{\theta}$ we get $\cC(\cT_A)_{*,\theta}\,\iso\, (\cT_A)_{\theta}[1]$. 

\sssec{} 
\label{Sect_C.3.36}
For $\theta\in -\Lambda_{G,P}^{pos, *}$ 
the natural maps $X^J\to\Ran$ define by passing to the colimit over $(J, \und{\theta})\in fSets^{\theta}$ a morphism denoted $\eta_{\theta}: \und{X}^{-\theta}\to \Ran$. Note that $\eta_{\theta}$ is psedo-proper. View $\Sph_{M, X^{-\theta}, +}$ as a $Shv(\Ran)$-module via the restriction of scalars through $\eta_{\theta}^!$.
 
\index{$\eta_{\theta}$, Section~\ref{Sect_C.3.36}}

We promote $\cC(\cT_A)_{X^{-\theta}}$ to an object of 
$$
\Fact(\Rep(\check{M})_{>0})_{-\theta}\otimes_{Shv(\Ran)} \Sph_{M, X^{-\theta}, +}
$$ 
as follows. First, for $(J, \und{\theta})\in fSets^{\theta}$ we promote $\cC(\cT_A)_{J,\und{\theta}}$ to an object of
$$
\Fact(\Rep(\check{M})_{>0})_{J, -\theta}\otimes_{Shv(X^J)} Shv(\Gr_{M, J, \und{\theta}})^{\gL^+(M)_J}
$$
as in (\cite{Ly10}, 7.3.8-7.3.9). This construction is compatible with the pullbacks $(\vartriangle^{(\und{\theta}/\und{\mu})})^!$ for a map $(J, \und{\theta})\to (K, \und{\mu})$ (including higher compatibilities). The corresponding compatible system is an object of
\begin{multline*}
\underset{(J, \und{\theta})\in fSets^{\theta}}{\lim} \Fact(\Rep(\check{M})_{>0})_{J, -\theta}\otimes_{Shv(X^J)} Shv(\Gr_{M, J, \und{\theta}})^{\gL^+(M)_J}
\,\iso\\ 
\underset{(J, \und{\theta})\in fSets^{\theta}}{\lim} (\Fact(\Rep(\check{M})_{>0})_{-\theta}\otimes_{Shv(\Ran)} Shv(\Gr_{M, J, \und{\theta}})^{\gL^+(M)_J}\,\iso\\ 
\Fact(\Rep(\check{M})_{>0})_{-\theta}\otimes_{Shv(\Ran)} \Sph_{M, X^{-\theta}, +},
\end{multline*}
here for the second isomorphism we used the fact that $\Fact(\Rep(\check{M})_{>0})_{-\theta}$ is dualizable in $Shv(\Ran)-mod$. 


 We have the canonical equivalence
$$
\Fact(\Rep(\check{M})_{>0})_{-\theta}\otimes_{Shv(\Ran)} \Sph_{M, X^{-\theta}, +}\,\iso\; \Fun_{Shv(\Ran)}(\Fact(C)_{\theta}, \Sph_{M, X^{-\theta}, +}),
$$ 
and we denote by 
$$
\Sat_{M, X^{-\theta}}: \Fact(C)_{\theta}\to\Sph_{M, X^{-\theta}, +}
$$ 
the $Shv(\Ran)$-linear functor given by $\cC(\cT_A)_{X^{-\theta}}$. It is compatible with factorizations of $\Ran$. 

\sssec{} Recall the canonical equivalences $\Fact(C)_{\theta}\otimes_{Shv(\Ran)} Shv(X)\,\iso\, C_{\theta}\otimes Shv(X)$ and 
$$
\Sph_{M, X^{-\theta}, +}\otimes_{Shv(\Ran)} Shv(X)\,\iso\, \Sph_{M, X^{pos, *}}^{\theta}.
$$ 
After the base change by $X\to\Ran$, the functor $\Sat_{M, X^{-\theta}}$ identifies with the $Shv(X)$-linear t-exact functor denoted
$$
\Sat_{M, X^{pos, *}}^{\theta}: 
C_{\theta}\otimes Shv(X)\to \Sph_{M, X^{pos, *}}^{\theta}.
$$
For $V\in C_{\theta}$ it sends $V\otimes\IC_X$ to $(V\otimes(\cT_A)_{\theta})^{\check{G}}$, where $\check{G}$ acts diagonally (so, this is the usual Satake functor). 

 The following version of (\cite{Ly10}, 7.3.16) holds. 

\begin{Lm} 
\label{Lm_C.3.38_comm_diagram}
i) Let $(J, \und{\theta})\in fSets^{\theta}$. Once a linear order on $J$ is chosen, one has a commutativity datum for the the diagram 
\begin{equation}
\label{diag_for_Lm_C.3.38}
\begin{array}{ccc}
\underset{j\in J}{\boxtimes} (C_{\und{\theta}(j)}\otimes Shv(X)) & \to &\Fact(C)_{\theta}\\ \\
\downarrow\lefteqn{\scriptstyle \underset{j\in J}{\boxtimes} \Sat_{M, X^{pos, *}}^{\und{\theta}(j)}} && \downarrow\lefteqn{\scriptstyle \Sat_{M, X^{-\theta}}} 
\\ \\ 
\underset{j\in J}{\boxtimes} \Sph_{M, X^{pos, *}}^{\und{\theta}(j)} & \to & \Sph_{M, X^{-\theta}, +}.
\end{array}
\end{equation}
Here the low horizontal arrow is the exterior convolution for $\Sph_{M, \Conf^*, +}$, and the top horizontal arrow is the structure map for $\underset{\cTw(fSets)^{\theta}}{
\colim} \cF^{\theta}_{\Ran, C}$ for $(\und{\theta}, J\toup{\id} J)\in \cTw(fSets)^{\theta}$
given in Section~\ref{Sect_C.3.3_about_cF_Ran,C}. The composition of functors in (\ref{diag_for_Lm_C.3.38}) is t-exact. \\
ii) The functor 
$$
\Sat_{M,\Conf^*}: (\Fact(\Rep(\check{M})_{<0}), \star)\to (\Sph_{M,\Conf^*,+},\otimes^*)
$$ 
is non-unital monoidal. 
\end{Lm}
\begin{proof}
i) The first claim follows from the description of $\cC(\cT_A)_{J, \und{\theta}}$ in Section~\ref{Sect_C.3.35_description}. The t-exactness of the composition follows from the fact that the exterior convolution functors for $\Sph_{M, \Conf^*, +}$ are t-exact.
\end{proof}

\sssec{} As in the proof of Proposition~\ref{Pp_3.2.2} (cf. Section~\ref{Sect_Proof_of_Pp_3.2.2}), Lemma~\ref{Lm_C.3.38_comm_diagram} 
implies that the functor $\Sat_{M, X^{-\theta}}$ is naturally $Shv(X^{-\theta})$-linear, and thus defines (\ref{funct_Sat_M_X^-theta}).  

\sssec{} 
\label{Sect_C.3.14_now}
Define the chiral symmetric monoidal structure on 
$$
\Sph_{M, \Conf, +}=\Sph_{M,\Conf^*,+}\oplus\Vect
$$ 
by applying the functor 
$$
CAlg^{nu}(\DGCat_{cont})\to CAlg(\DGCat_{cont}), \; M\mapsto M\oplus\Vect
$$ 
to $(\Sph_{M, \Conf^*, +}, \, \otimes^{ch})\in CAlg^{nu}(\DGCat_{cont})$.

\sssec{Example} Let us illustrate the functors (\ref{funct_Sat_M_X^-theta}) in the special case $M=T$. In this case $C\,\iso\, \underset{\theta\in -\Lambda_{G,P}^{pos, *}}{\oplus}\Vect$. For $\theta\in -\Lambda_{G,P}^{pos, *}$ we get $\Fact(C)_{\theta}\,\iso\, Shv(X^{-\theta})$ by (\cite{Ly10}, E.1.11), 
$$
\Gr^{\theta}_{M, X^{-\theta}, +}=X^{-\theta}\;\;\;\mbox{and}\;\;\; \Sph_{M, X^{-\theta}, +}\,\iso\, Shv(X^{-\theta})^{\gL^+(M)_{\theta}}.
$$  
In this case the functor
$$
\Sat_{T, X^{-\theta}}: Shv(X^{-\theta})\to Shv(X^{-\theta})^{\gL^+(M)_{\theta}} 
$$
is the t-exact functor $\pr^!$ for the projection $X^{-\theta}/\gL^+(M)_{\theta}\to X^{-\theta}$, here we use the conventions of (\cite{DL}, A.5). 

\sssec{Questions}
\label{Sect_C.3.42_Questions}
 i) For $\theta\in -\Lambda_{G,P}^{pos, *}$ is the functor $\Sat_{M, X^{-\theta}}: \Fact(C)_{\theta}\to \Sph_{M, X^{-\theta}, +}$ t-exact?

\medskip\noindent
ii) For $\theta\in -\Lambda_{G,P}^{pos, *}$ is it true that $\cC(\cT_A)_{X^{-\theta}}$ is placed in perverse degree $-\dim X^{-\theta}$? 

\sssec{} Here is a partial answer to Question~\ref{Sect_C.3.42_Questions} i) sufficient for our purposes.
\begin{Lm} 
\label{Lm_C.3.44}
Let $\cA$ be as in Lemma~\ref{Lm_C.3.19_now}. Then for any $\theta\in -\Lambda_{G,P}^{pos, *}$ the object $\Sat_{M, X^{-\theta}}(\Fact(\cA)_{\theta})$ in strictly connective in $\Sph_{M, X^{-\theta}, +}$.
\end{Lm}
\begin{proof}
Recall the isomorphism
$$
\Fact(\cA)_{\theta}\,\iso\, \underset{(\und{\theta}, J\to K)\in \cTw(fSets)^{\theta}}{\colim} \; (\underset{j\in J}{\otimes} \cA_{\und{\theta}(j)})\otimes\omega_{X^K},
$$ 
where the colimit is taken in $\Fact(C)_{\theta}$. Since $\Sat_{M, X^{-\theta}}$ preserves colimits, it suffices to show that for any $(\und{\theta}, J\toup{\phi} K)\in \cTw(fSets)^{\theta}$ the object
$$
\Sat_{M, X^{-\theta}}((\underset{j\in J}{\otimes} \cA_{\und{\theta}(j)})\otimes\omega_{X^K})\in  \Sph_{M, X^{-\theta}, +}
$$
is strictly connective. We have the natural map $(\und{\theta}, J\toup{\phi} K)\to (\phi_*\und{\theta}, K\to K)$ in $\cTw(fSets)^{\theta}$, so we may and do assume that $\phi$ is the map $\id: J\to J$. Our result follows from the t-exactness claim in Lemma~\ref{Lm_C.3.38_comm_diagram}. 
\end{proof}

\sssec{} Here is an alternative construction of the functor (\ref{funct_Sat_M_Conf^*}). 
 For $\theta\in \Lambda_{G,P}^{pos}$ consider the map 
$$
s^{-\theta}_{untl}: \und{X}^{\theta}\to (X^{\theta}\times\Ran)^{\subset}
$$ 
sending $(x_i, \theta_i)$ to $D=\sum_i \theta_i x_i$, $(x_i)\in\Ran$. 
 
\begin{Lm}  Let $\theta\in -\Lambda_{G,P}^{pos, *}$. The the base change by $s^{\theta}_{untl}$ gives
$$
\Sph^{\theta, untl}_{M,\Ran, +}\otimes_{Shv((X^{-\theta}\times\Ran)^{\subset})} Shv(\und{X}^{-\theta})\,\iso\, \Sph_{M, X^{-\theta}, +} 
$$
canonically.
\end{Lm} 
\begin{proof}
This is similar to (\cite{Ly10}, E.4.7). Namely, $Shv(\und{X}^{-\theta})\,\iso\ \underset{(J, \und{\theta})\in (fSets^{\theta})^{op}}{\colim} Shv(X^J))$, so
\begin{multline*}
Shv(\gL^+(M)_{\Ran}\backslash \Gr^{\theta, \subset}_{M, X^{-\theta}})\otimes_{Shv((X^{-\theta}\times\Ran)^{\subset})} Shv(\und{X}^{-\theta})\,\iso\\ 
\underset{(J, \und{\theta})\in (fSets^{\theta})^{op}}{\colim} Shv(\gL^+(M)_{\Ran}\backslash \Gr^{\theta, \subset}_{M, X^{-\theta}})\otimes_{Shv((X^{-\theta}\times\Ran)^{\subset})}Shv(X^J)\,\iso\\
\underset{(J, \und{\theta})\in (fSets^{\theta})^{op}}{\colim} Shv(\gL^+(M)_{X^J}\backslash \Gr_{M, J, \und{\theta}})\,\iso\, \Sph_{M, X^{-\theta}, +}.
\end{multline*}
\end{proof}  

\sssec{} For $\theta\in -\Lambda_{G,P}^{pos}$ the restriction of $\gL^+(M)_{\theta}$ under $\und{X}^{-\theta}\to X^{-\theta}$ is still denoted $\gL^+(M)_{\theta}$ by abuse of notations. For $\theta\in -\Lambda_{G,P}^{pos, *}$ one has canonically 
$$
\Gr^{\theta,\subset}_{M, X^{-\theta}}\times_{(X^{-\theta}\times\Ran)^{\subset}} \und{X}^{-\theta}\,\iso\, \Gr^{\theta}_{M, \und{X}^{-\theta}, +}
$$  
and
$$
(\Gr^{\theta,\subset}_{M, X^{-\theta}}/\gL^+(M)_{\Ran})\times_{(X^{-\theta}\times\Ran)^{\subset}} \und{X}^{-\theta}\,\iso\, \Gr^{\theta}_{M, \und{X}^{-\theta}, +}/\gL^+(M)_{\theta}.
$$
We still denote for $\theta\in -\Lambda_{G,P}^{pos, *}$ by
$$
s^{\theta}_{untl}: \Gr^{\theta}_{M, \und{X}^{-\theta}, +}/\gL^+(M)_{\theta}\to \Gr^{\theta,\subset}_{M, X^{-\theta}}/\gL^+(M)_{\Ran}
$$
the map obtained by the base change. This is the left inverse of the map $\tau^{\theta}_{untl}$ from Section~\ref{Sect_3.2.4_now}. 

\begin{Lm} 
\label{Lm_C.3.48_now}
The collection of functors $(s^{\theta}_{untl})^!: \Sph^{\theta, untl}_{M,\Ran, +}\to \Sph_{M, X^{-\theta}, +}$ for $\theta\in -\Lambda_{G,P}^{pos, *}$
combine into a non-unital monoidal functor
$$
(s_{untl})^!: (\Sph^{untl}_{M,\Ran, +},\otimes^*)\to (\Sph_{M,\Conf^*, +},\otimes^*).
$$
\end{Lm}
\begin{proof}
This is similar to (\cite{Ly10}, E.4.16). In details, for $\theta,\theta'\in -\Lambda_{G,P}^{pos, *}$ we make the base change of $m$ in (\ref{diag_conv_for_Sph_M_Ran_untl}) under $s^{\theta+\theta'}_{untl}$ is the map $\tilde m$ from Section~\ref{Sect_C.3.3_now}  (up to sheafification in the topology of finite surjective maps). The base change of $p$ in (\ref{diag_conv_for_Sph_M_Ran_untl}) under $s^{\theta}_{untl}\times s^{\theta'}_{untl}$ is the map $\tilde p$ from Section~\ref{Sect_C.3.3_now} (up to sheafification in the topology of finite surjective maps). The claim follows.
\end{proof}

\begin{Pp} 
\label{Pp_C.3.49}
For $\theta\in -\Lambda_{G,P}^{pos, *}$ the base change of the functor (\ref{functor_Sat_M_Ran_theta^untl}) by 
$$
-\otimes_{Shv((X^{-\theta}\times\Ran)^{\subset})} Shv(\und{X}^{-\theta})
$$ 
identifies canonically with the functor (\ref{funct_Sat_M_X^-theta}). 
\end{Pp}
\begin{proof} 
By (\cite{Ly10}, Theorem~E.3.12), one has for $\theta\in -\Lambda_{G,P}^{pos, *}$
$$
\Fact(\Rep(\check{M})^{untl}_{neg})_{\theta}\,\iso\, \Fact(\Rep(\check{M})_{<0}))_{\theta}\otimes_{Shv(X^{-\theta})} Shv((X^{-\theta}\times\Ran)^{\subset})
$$   
So, 
$$
\Fact(\Rep(\check{M})^{untl}_{neg})_{\theta}\otimes_{Shv((X^{-\theta}\times\Ran)^{\subset})}Shv(\und{X}^{-\theta})\,\iso\, \Fact(\Rep(\check{M})_{<0}))_{\theta}
$$  

 Consider the $\theta$-component of Chevalley-Cousin complex $\cC(\cT_A)_{\theta}$ on $\Gr^{\theta, \subset}_{M, X^{-\theta}}$ defined in Section~\ref{Sect_C.2.8_now}. Then 
$(s^{\theta}_{untl})^!\cC(\cT_A)_{\theta}$ identifies canonically with the Chevalley-Cousin complex $\cC(\cT_A)_{X^{-\theta}}$ from Section~\ref{Sect_C.3.32_now}. 
\end{proof}  

\sssec{} Let us illustrate Proposition~\ref{Pp_C.3.49} in the case $M=T$. In this case for $\theta\in -\Lambda_{G,P}^{pos, *}$ consider the cartesian square
$$
\begin{array}{ccc}
\und{X}^{-\theta}/\gL^+(T)_{\theta} & \to & (X^{-\theta}\times\Ran)^{\subset}/\gL^+(T)_{\Ran}\\
\downarrow\lefteqn{\scriptstyle \pr}
 && \downarrow\lefteqn{\scriptstyle \pr}\\
\und{X}^{-\theta} & \to & (X^{-\theta}\times\Ran)^{\subset}
\end{array}
$$
Proposition~\ref{Pp_C.3.49} says in this case that the base change of the functor 
$$
\pr^!: Shv((X^{-\theta}\times\Ran)^{\subset})\to Shv((X^{-\theta}\times\Ran)^{\subset}/\gL^+(T)_{\Ran})
$$ 
by $-\otimes_{Shv((X^{-\theta}\times\Ran)^{\subset})} Shv(\und{X}^{-\theta})$ yields the functor 
$$
\pr^!: Shv(\und{X}^{-\theta})\to Shv(\und{X}^{-\theta}/\gL^+(T)_{\theta}). 
$$

\ssec{The ranification functor} 

\sssec{} In Section~\ref{Sect_3.2.4_now} we defined for $\theta\in -\Lambda_{G,P}^{pos}$ the map 
$$
\tau_{untl}^{\theta}: \Gr_{M, X^{-\theta}}^{\theta,\subset}/\gL^+(M)_{\Ran}\to \Gr_{M, X^{-\theta}, +}^{\theta}/\gL^+(M)_{\theta}.
$$
For $\theta\in -\Lambda_{G,P}^{pos}$ set $\cR^{\theta}=(\tau^{\theta}_{untl})^!: \Sph_{M, X^{-\theta}, +}\to\Sph_{M,\Ran, +}^{\theta, untl}$. As $\theta$ varies in $-\Lambda_{G,P}^{pos}$ these functors form the ranification functor
\begin{equation}
\label{Ranification_functor}
\cR: \Sph_{M, \Conf, +}\to \Sph_{M,\Ran, +}^{untl}
\end{equation}
graded by $-\Lambda_{G,P}^{pos}$. 

\begin{Rem} 
\label{Rem_gluing_torsor_using_divisors}
Let $S\in\Sch^{aff}$, let $(D,\cI)$ be an $S$-point of $(\Conf\times\Ran)^{\subset}$. The groupoid of $M$-torsors on $\cD_{\cI}$ is equivalent to the groupoid of triples $(\cF_M^1, \cF_M^2, \beta)$, where $\cF_M^1$ (resp., $\cF_M^2$) is a $M$-torsor on $\cD_D$ (resp., on $\cD_{\cI}-\supp D$), and 
$$
\beta: \cF_M^1\,\iso\, \cF_M^2\mid_{\oo{\cD}_D}
$$ 
is an isomorphism.
\end{Rem}
 
\begin{Pp} 
\label{Pp_from_Raskin_rlax}
i) The functor (\ref{Ranification_functor}) is naturally right-lax monoidal with respect to the $\otimes^*$-monoidal structures on the source and on the target.

\smallskip\noindent
ii) The functor (\ref{Ranification_functor}) is naturally left-lax symmetric monoidal with respect to the $\otimes^{ch}$-symmeetric monoidal structures on the source and on the target.
\end{Pp}
\begin{proof} i) This is an analog of (\cite{Ras2}, 4.9.1) in our setting. Set temporary 
$$
\cC=\gL^+(M)_{\Ran}\backslash \Gr_{M,\Ran}^{\subset}\;\;\;\mbox{and}\;\;\; \cM=\gL^+(M)_{\Conf}\backslash \Gr_{M,\Conf, +}.
$$ 
Write for short $\cC\times\cC\gets\Conv_{\cC}\to \cC$ and $\cM\times\cM\gets \Conv_{\cM} \to \cM$ for the convolution diagrams for the $\otimes^*$-monoidal structures on $\Sph_{M,\Ran, +}^{untl}$
and $\Sph_{M, \Conf, +}$ respectively.

They fit into the commutative diagram
$$
\begin{array}{ccccc}
&&\Conv_{\cC} \\
& \swarrow & \downarrow\lefteqn{\scriptstyle \alpha} & \searrow\\
\cC\times\cC && \Conv_{\cM}\times_{\cM} \cC & \to & \cC\\
\downarrow\lefteqn{\scriptstyle \tau_{untl}\times\tau_{untl}}  && \downarrow && \downarrow\lefteqn{\scriptstyle \tau_{untl}}\\
\cM\times\cM &\gets & \Conv_{\cM} & \to &\cM,
\end{array}
$$
where the right square is cartesian. Using Remark~\ref{Rem_gluing_torsor_using_divisors}, one checks that $\alpha$ is obtained by base change from the natural map
$$
(\Conf\times\Ran)^{\subset}\times (\Conf\times\Ran)^{\subset}\to (\Conf\times\Conf)\times_{sum, \Conf} (\Conf\times\Ran)^{\subset},
$$
where $sum: \Conf\times\Conf\to\Conf$ is the sum of divisors. Since $\alpha$ is pseudo-proper, one has a natural morphism $\alpha_!\alpha^!\to\id$, which yields the desired natural transformation 
$$
\cR(\cB_1)\otimes^*\cR(\cB_2)\to \cR(\cB_1\otimes^*\cB_2)
$$ 
functorial in $\cB_1,\cB_2\in \Sph_{M,\Conf, +}$. 

\medskip\noindent
ii) This is an analog of (\cite{Ras2}, 4.9.2) in our setting. Write for short $\cC\times\cC \gets\Conv_{\cC, d}\to \cC$ and $\cM\times\cM\gets \Conv_{\cM, d}\to \cM$ for the convolution diagrams for the $\otimes^{ch}$-symmetric monoidal structures on $\Sph_{M,\Ran, +}^{untl}$ and $\Sph_{M, \Conf, +}$ respectively. 
Here $\Conv_{\cM, d}\subset \Conv_{\cM}$ and $\Conv_{\cC, d}\subset \Conv_{\cC}$ are open immersions, and the subscript $d$ stands for `disjoint'

They fit into a commutative diagram
$$
\begin{array}{ccccc}
&&\Conv_{\cC,d} \\
& \swarrow & \downarrow\lefteqn{\scriptstyle \beta} & \searrow\\
\cC\times\cC && \Conv_{\cM, d}\times_{\cM} \cC & \to & \cC\\
\downarrow\lefteqn{\scriptstyle \tau_{untl}\times\tau_{untl}}  && \downarrow && \downarrow\lefteqn{\scriptstyle \tau_{untl}}\\
\cM\times\cM &\gets & \Conv_{\cM,d} & \to &\cM,
\end{array}
$$
where the right square is cartesian. The map $\beta$ is obtained by base change from the natural map
\begin{multline*}
((\Conf\times\Ran)^{\subset}\times (\Conf\times\Ran)^{\subset})\times_{\Ran\times\Ran} (\Ran\times\Ran)_d\to \\
(\Conf\times\Conf)_d\times_{\Conf} (\Conf\times\Ran)^{\subset}
\end{multline*}
So, as in (\cite{Ras2}, 4.9.2), $\beta$ is schematic and \'etale. The map $\id\to \beta_*\beta^*$ yields the desired natural transformation
$$
\cR(\cB_1\otimes^{ch} \cB_2)\to \cR(\cB_1)\otimes^{ch}\cR(\cB_2)
$$
functorial in $\cB_1,\cB_2\in \Sph_{M,\Conf, +}$. 
\end{proof}

\begin{Lm} For $\cB_i\in \Sph_{M,\Conf, +}$ the diagram  
$$
\begin{array}{ccc}
\cR(\cB_1)\otimes^*\cR(\cB_2) & \toup{\ref{Pp_from_Raskin_rlax}i)} & \cR(\cB_1\otimes^*\cB_2)\\
\downarrow && \downarrow\\
\cR(\cB_1)\otimes^{ch} \cR(\cB_2) & \getsup{\ref{Pp_from_Raskin_rlax}ii)} & \cR(\cB_1\otimes^{ch}\cB_2)
\end{array}
$$
commutes. Here the vertical arrows come from the natural transformations $-\otimes^*-\to -\otimes^{ch}-$. Moreover, the higher version of this statement encoding the compatibility with the monoidal structures holds. 
\end{Lm}  
\begin{proof}
This is analogous to (\cite{Ras2}, 4.9.3). It follows from the commutativity of the diagram
$$
\begin{array}{ccc}
\Conv_{\cC} & \toup{\alpha} & \Conv_{\cM}\times_{\cM} \cC\\
\uparrow && \uparrow\\
\Conv_{\cC, d} & \toup{\beta} & \Conv_{\cM, d}\times_{\cM} \cC, 
\end{array}
$$
where we used the notations from the proof of Proposition~\ref{Pp_from_Raskin_rlax}. 
\end{proof}

\sssec{} As in \cite{FG}, write $Com-coalg^{ch}(\Sph_{M,\Conf, +})$ and $Com-coalg^{ch}(\Sph_{M,\Ran, +}^{untl})$ for the corresponding categories of cocommutative coalgebras for the chiral symmetric monoidal structures. We also get the corresponding full subcategories of factorization coalgebras 
$$
Com-coalg^{ch, fact}(\Sph_{M,\Conf, +})\subset Com-coalg^{ch}(\Sph_{M,\Conf, +})
$$
and 
$$
Com-coalg^{ch, fact}(\Sph_{M,\Ran, +}^{untl})\subset Com-coalg^{ch}(\Sph_{M,\Ran, +}^{untl})
$$
as in \select{loc.cit}, and similarly for $\Conf$ replaced by $\Conf^*$. By Proposition~\ref{Pp_from_Raskin_rlax}ii), (\ref{Ranification_functor}) extends to the functor
$$
\cR: Com-coalg^{ch}(\Sph_{M,\Conf, +})\to Com-coalg^{ch}(\Sph_{M,\Ran, +}^{untl}).
$$
One checks that it preserves the corresponding full subcategories of factorization coalgebras. 

\begin{Rem} Though $(\Sph_{M,\Conf^*, +}, \otimes^{ch})$ is pro-nilpotent in the sense of (\cite{FG}, 4.1.1), we warn the reader that $(\Sph_{M,\Conf, +}, \otimes^{ch})$ is not pro-nilpotent in that sense.
\end{Rem}

\sssec{} Consider the functor 
$$
Add-unit: Com-coalg^{ch, fact}(\Sph_{M,\Conf^*, +})\to Com-coalg^{ch, fact}(\Sph_{M,\Conf, +})
$$ 
sending $M$ to $e\oplus M\in \Vect\oplus \Sph_{M,\Conf^*, +}\,\iso\, \Sph_{M,\Conf, +}$, here $e$ is viewed as the constant sheaf on $\Spec k\subset \Conf$. 

\begin{Rem} 
\label{Rem_C.4.8}
Recall from \cite{BD_chiral} that if $K\in CAlg(Shv(X),\otimes^!)$ then $\vartriangle_!K[-1]\in \Lie-alg^{ch}(Shv(\Ran))$ naturally, here $\vartriangle: X\to\Ran$ is the natural map. Indeed, if for $I\in fSets$ one lets $\lambda_I=(e[1])^{\otimes I}[-\mid I\mid]$ then for $I$ of order 2 
$$
\lambda_I\,\iso\, \Hom_{Shv(X^I)}(j_*j^*(\IC_X)^{\boxtimes I}), i_!\IC_X)
$$ 
canonically\footnote{According to our conventions, we ignore the Tate twists}, here $i: X\to X^I$ and $j: \oo{X}{}^I\to X^I$ is the complement to the diagonal. This defines the Lie algebra structure on $\vartriangle_!\IC_X$. Tensoring the Lie bracket $j_*j^*(\IC_X)^{\boxtimes I})\to i_!\IC_X$ with $K^{\boxtimes I}$ and composing with the product map $K\otimes^!K\to K$, one gets the canonical morphism $\lambda_I\to \Hom_{Shv(X^I)}(j_*j^*((K[-1])^{\boxtimes I}), i_!K[-1])$ defining the Lie bracket on $K[-1]$.
\end{Rem}

\sssec{} Let also
$$
Add-unit_X: \Sph_{M, X^{pos, *}}\to \Sph_{M, X^{pos}}
$$
be the functor sending $M$ to $\IC_X\oplus M$, where $\IC_X\in Shv(\Gr_{M, X, +}^0)^{\gL^+(M)_X}$. 

It extends naturally to the functor
$$
Add-unit_X: \Lie-alg^{ch}(\Sph_{M, X^{pos, *}})\to \Lie-alg^{ch}(\Sph_{M, X^{pos}})
$$
between the corresponding categories of Lie algebras with respect to the chiral symmetric monoidal structures. 

\sssec{} As in \cite{FG}, we consider the functor of `homological Chevalley complex' (and its restriction to the full subcategories in the diagrams below)
$$
\begin{array}{ccc}
\Lie-alg^{ch}(\Sph^{untl}_{M,\Ran, +}) & \toup{C^{ch}} & Com-coalg^{ch}(\Sph_{M,\Ran, +}^{untl})\\
\cup && \cup\\
\Lie-alg^{ch}(\Sph_{M, X^{pos}}) & \toup{C^{ch}} & Com-coalg^{ch, fact}(\Sph_{M,\Ran, +}^{untl})
\end{array}
$$
and
$$
\begin{array}{ccc}
\Lie-alg^{ch}(\Sph_{M,\Conf^*, +}) & \toup{C^{ch}} & Com-coalg^{ch}(\Sph_{M,\Conf^*, +})\\
\cup && \cup\\
\Lie-alg^{ch}(\Sph_{M, X^{pos, *}}) & \toup{C^{ch}} & Com-coalg^{ch, fact}(\Sph_{M,\Conf^*, +}).
\end{array}
$$
The fact that $C^{ch}$ preserves the corresponding full subcategories is obtained as in (\cite{FG}, 5.2.1). 

\begin{Pp} 
\label{Pp_C.4.10}
The diagram commutes
$$
\begin{array}{ccc}
\Lie-alg^{ch}(\Sph_{M, X^{pos}}) & \toup{C^{ch}} & Com-coalg^{ch, fact}(\Sph_{M,\Ran, +}^{untl})\\
\uparrow\lefteqn{\scriptstyle Add-unit_X} && \uparrow\lefteqn{\scriptstyle \cR}\\
\Lie-alg^{ch}(\Sph_{M, X^{pos, *}}) & \toup{Add-unit\,\comp\, C^{ch}} & Com-coalg^{ch, fact}(\Sph_{M,\Conf, +})
\end{array}
$$
\end{Pp} 
\begin{proof}
This is obtained as in (\cite{Ras2}, Proposition-Construction~4.10.1). Over the connected component $\Ran\,\iso\, \Gr^{0, \subset}_{M, X^0}\subset \Gr^{\subset}_{M,\Ran}$ the isomorphism is evident. It remains to obtain it over the complement open subset. This uses the commutative diagram
$$
\begin{array}{ccc}
\Gr_{M, X^{pos, *}}/\gL^+(M)_X & \hook{} & \Gr^{\subset}_{M,\Ran}/\gL^+(M)_{\Ran}\\
\downarrow && \downarrow\lefteqn{\scriptstyle \tau_{untl}}\\
\Gr_{M, \Conf^*, +}/\gL^+(M)_{\Conf^*} & \hook{} &\Gr_{M, \Conf, +}/\gL^+(M)_{\Conf},
\end{array}
$$
where we the unnamed maps are the corresponding closed immersions. 
\end{proof}

\sssec{} 
\label{Sect_C.4.11}
As in \cite{FG}, write $\Ind^{*\to ch}_{\Lie}$ for the left adjoint to the oblivion functor 
$$
\Lie-alg^{ch}(\Sph_{M, X^{pos, *}})\to \Lie-alg^*(\Sph_{M, X^{pos, *}}),
$$ 
here the superscript $*$ means that we consider the $\otimes^*$-monoidal structure on the corresponding category. As in \cite{FG}, we write 
$$
\triv: \Sph_{M, X^{pos, *}}\to \Lie-alg^*(\Sph_{M, X^{pos, *}})
$$ 
for the functor sending $K$ to itself with the zero Lie bracket. As in (\cite{FG}, 6.5.2), one has the PBW-filtration on $\Ind^{*\to ch}_{\Lie}(M)$ for $M\in \Lie-alg^*(\Sph_{M, X^{pos, *}})$. Moreover, for $K\in \Sph_{M, X^{pos, *}}$ and $M=\triv(K)$ this filtration splits canonically, and we get 
$$
\Ind^{*\to ch}_{\Lie}(\triv(K))[1]\,\iso\, \underset{m>0}{\oplus} \Sym^{m, !}(K[1]),
$$
where $\Sym^{m, !}$ is calculated in the symmetric monoidal category $(\Sph_{M, X^{pos, *}}, \otimes^!)$. Thus, 
$$
Add-unit_X(\Ind^{*\to ch}_{\Lie}(\triv(K)))\,\iso\, \left(\underset{m\ge 0}{\oplus} \Sym^{m, !}(K[1])\right)[-1]\in \Lie-alg(\Sph_{M, X^{pos}}),
$$
where the symmetric algebra is calculated in $(\Sph_{M, X^{pos}}, \otimes^!)$.

\sssec{} 
\label{Sect_C.4.12_toadd}
If $C(X)\in CAlg^{nu}(Shv(X)-mod)$ then $\Fact(C)$ is equipped with the chiral symmetric monoidal structure defined in (\cite{Ly10}, 2.2.5). If $\cE\in CAlg(C(X)), \otimes^!)$ then  
$\cE[-1]\in \Lie-alg^{ch}(\Fact(C))$ naturally as in Remark~\ref{Rem_C.4.8}, and one has canonically 
$$
C^{ch}(\cE[-1])\,\iso\,\Fact(\cE)\in Com-coalg^{ch}(\Fact(C)).
$$ 
 
\sssec{Proof of Proposition~\ref{Pp_3.2.5_Satake}} 
\label{Sect_C.4.12}
Combine Proposition~\ref{Pp_C.4.10} and 
Sections~\ref{Sect_C.4.11}-\ref{Sect_C.4.12_toadd}. We also use the fact that the functor $\Sat^{untl}_{M,\Ran}$ is a factorization functor, hence is symmetric monoidal with respect to the symmetric monoidal structures $\otimes^{ch}$ on both categories (and similarly for $\Sat_{M, \Conf^*}$). 
\QED 

\printindex

\end{document}